\documentclass[12pt,a4paper]{book}
\usepackage[french, english]{babel}
\usepackage[utf8]{inputenc}
\usepackage[T1]{fontenc}
\usepackage{amsmath,amsthm,amssymb}
\usepackage{amsfonts}
\usepackage{amssymb}
\usepackage{caption}
\usepackage{enumitem}
\usepackage{stmaryrd}
\usepackage{esint}
\usepackage{soul}
\usepackage{floatrow}
\usepackage{dsfont}
\usepackage[defaultsans,oldstyle, scale=1]{opensans}
\usepackage[a4paper,top=3.5cm,bottom=4cm,left=1.5cm,right=3cm, heightrounded,bindingoffset=5mm]{geometry}
\usepackage{subeqnarray}
\usepackage[all]{xy}
\usepackage{varioref}
\usepackage{tikz-cd}
\usepackage{xcolor}
\definecolor{Prune}{RGB}{99,0,60} 
\definecolor{B1}{RGB}{49,62,72} 
\definecolor{C1}{RGB}{124,135,143}
\definecolor{D1}{RGB}{213,218,223}
\definecolor{A2}{RGB}{198,11,70}
\definecolor{B2}{RGB}{237,20,91}
\definecolor{C2}{RGB}{238,52,35}
\definecolor{D2}{RGB}{243,115,32}
\definecolor{A3}{RGB}{124,42,144}
\definecolor{B3}{RGB}{125,106,175}
\definecolor{C3}{RGB}{198,103,29}
\definecolor{D3}{RGB}{254,188,24}
\definecolor{A4}{RGB}{0,78,125}
\definecolor{B4}{RGB}{14,135,201}
\definecolor{C4}{RGB}{0,148,181}
\definecolor{D4}{RGB}{70,195,210}
\definecolor{A5}{RGB}{0,128,122}
\definecolor{B5}{RGB}{64,183,105}
\definecolor{C5}{RGB}{140,198,62}
\definecolor{D5}{RGB}{213,223,61}
\usepackage{mdframed}
\usepackage{multirow} 
\usepackage{multicol} 
\usepackage{tikz}
\usepackage{graphicx}
\usepackage[absolute]{textpos} 
\usepackage{colortbl}
\usepackage{array}
\usepackage[pagebackref]{hyperref}
\renewcommand*\backref[1]{\ifx#1\relax \else (Cité p.#1) \fi}
\hypersetup{ 
	colorlinks=true,
	linkcolor=black,
	urlcolor=A5}

\usepackage[french, english]{minitoc}
\usepackage{fancyhdr}

\usepackage{titlesec}

\usepackage{textcase}

\fancyhfoffset[RO,LE]{0.5cm}
\fancyhfoffset[LE,RO]{0.5cm}

\usepackage{pgfplots}

\newcommand{\cat}{\text{CAT}}
\newcommand{\ch}{\text{ch}}

\newcommand{\diam}{\text{diam}}
\newcommand{\Map}{\text{Map}}
\newcommand{\bnd}{\text{bnd}}
\newcommand{\bdd}{\text{Bdd}}
\newcommand{\Har}{\text{Har}}
\newcommand{\Lip}{\text{Lip}}
\newcommand{\bdg}{\partial_{\text{Grom}}}

\newcommand{\iso}{\text{Isom}}
\newcommand{\bdinf}{\Delta^\infty}

\newcommand{\Aut}{\text{Aut}}
\newcommand{\dt}{\partial T}
\newcommand{\Out}{\text{Out}}
\newcommand{\supp}{\text{supp}}
\newcommand{\haar}{Haar}
\newcommand{\R}{\mathbb{R}}
\newcommand{\Opp}{\text{Opp}}
\newcommand{\conv}{\text{Conv}}
\newcommand{\K}{\mathbb{K}}
\newcommand{\Ell}{\text{Ell}}

\newcommand{\proj}{\text{proj}}
\newcommand{\germ}{\text{germ}}
\newcommand{\rang}{\text{rang}}
\newcommand{\SO}{\text{SO}}
\newcommand{\res}{\text{Res}}
\newcommand{\dep}{\text{dep}}
\newcommand{\psl}{\text{PSL}}
\newcommand{\glnr}{\text{GL}_n(\mathbb{R})}
\newcommand{\gl}{\text{GL}}
\newcommand{\stab}{\text{Stab}}
\newcommand{\Id}{\text{Id}}
\newcommand{\bd}{\partial_\infty}
\newcommand{\nui}{\check{\nu}}
\newcommand{\mui}{\check{\mu}}

\newcommand{\prob}{\text{Prob}}
\newcommand{\Xhu}{\widetilde{X}^{u}}
\newcommand{\Xh}{\widetilde{X}}
\newcommand{\Xhb}{\widetilde{X}^f}
\newcommand{\Xhuc}{\widehat{X}^{u}}
\newcommand{\Xhc}{\widehat{X}}
\newcommand{\XG}{\overline{X}^{\text{Grom}}}
\newcommand{\Xhbc}{\widehat{X}^f}
\newcommand{\bary}{\text{bar}}
\newcommand{\tnu}{\tilde{\nu}}

\theoremstyle{plain}
\newtheorem{thm}{Théorème}[section]
\newtheorem*{nthm}{Théorème}
\newtheorem{ithm}{Théorème}
\newtheorem{conj}{Conjecture}
\newtheorem{cor}[thm]{Corollaire}
\newtheorem{icor}[ithm]{Corollaire}
\newtheorem{lem}[thm]{Lemme}
\newtheorem{prop}[thm]{Proposition}
\newtheorem{ethm}{Theorem}[section]

\newtheorem{ecor}[ethm]{Corollary}
\newtheorem{elem}[ethm]{Lemma}
\newtheorem{eprop}[ethm]{Proposition}

\theoremstyle{definition} 
\newtheorem{Def}[thm]{Définition}

\newtheorem{ex}[thm]{Exemple}

\newtheorem{eDef}[ethm]{Definition}

\theoremstyle{remark}
\newtheorem{erem}[ethm]{Remark}
\newtheorem{rem}[thm]{Remarque}

\title{Marches aléatoires et éléments contractants sur des espaces $\cat$(0)}
\author{Corentin Le Bars}

\begin{document}
	\sffamily
	\begin{titlepage}

		\newgeometry{left=6cm,bottom=2cm, top=1cm, right=1cm}
		
		\tikz[remember picture,overlay] \node[opacity=1,inner sep=0pt] at (-13mm,-135mm){\includegraphics{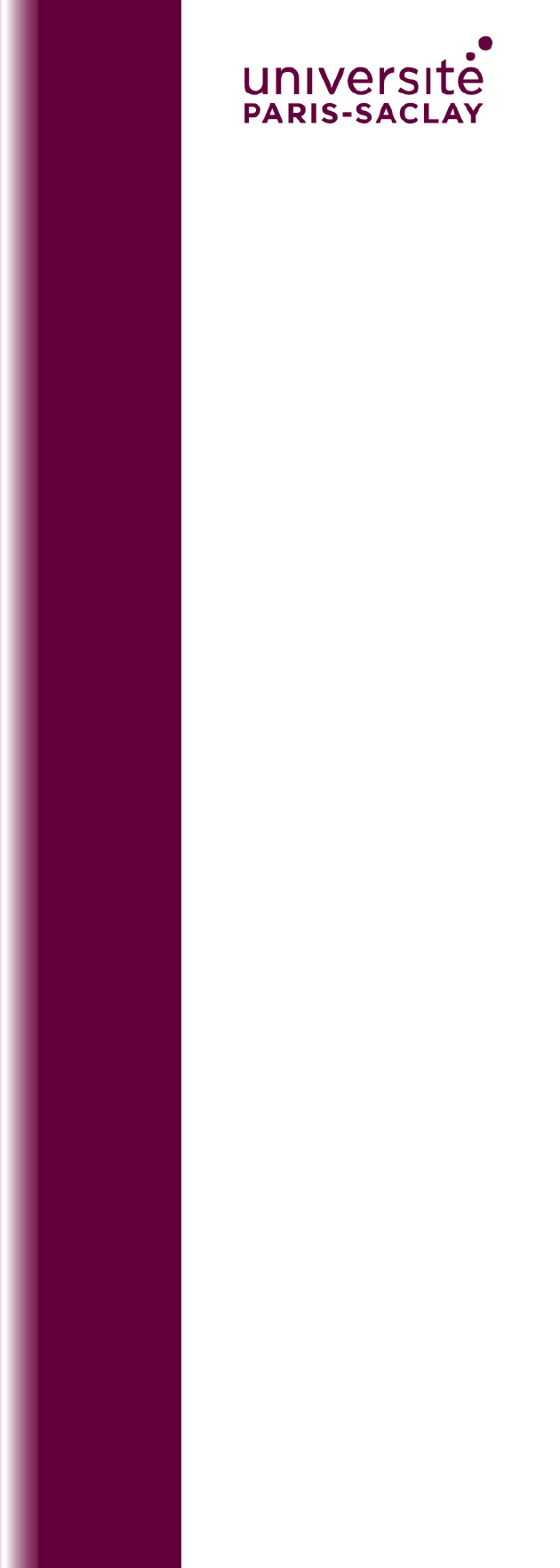}};
		
		
		\color{white}
		
		\begin{picture}(0,0)
			\put(-152,-743){\rotatebox{90}{\Large \textsc{THESE DE DOCTORAT}}} \\
			\put(-120,-743){\rotatebox{90}{NNT : 2023UPASAM017}}
		\end{picture}
		
		\color{black}
		\vspace{5mm}
		\begin{flushright}
			\color{Prune}
			
			\fontsize{22}{26}\selectfont
			\Huge Marches aléatoires et éléments contractants sur des espaces {CAT}(0) \\
			\vspace{5mm}
			\normalsize
			\color{black}
			\Large{\textit{Random walks and contracting elements on {CAT}(0) spaces}} \\

			\fontsize{8}{12}\selectfont
			
			\vspace{1.5cm}
			
			\textbf{Thèse de doctorat de l'université Paris-Saclay} \\
			
			\vspace{6mm}
			
			\small École doctorale n$^{\circ} 574$ MH | Mathématiques Hadamard (EDMH)\\
			\small Spécialité de doctorat : Mathématiques Fondamentales\\
			\small Graduate School : Mathématiques. Référent : Faculté des sciences d’Orsay \\
			\vspace{6mm}
			
			\footnotesize Thèse préparée dans l'unité de recherche \textbf{Laboratoire de mathématiques d'Orsay (Université Paris-Saclay, CNRS)}, sous la direction de \textbf{Jean LÉCUREUX}, maître de conférences.\\
			\vspace{15mm}
			
			\textbf{Thèse soutenue à Paris-Saclay, le 31 août 2023, par}\\
			\bigskip
			\Large {\color{Prune} \textbf{Corentin LE BARS}} 
		\end{flushright}

		\vspace{\fill} 
		
		\flushleft
			\small {\color{Prune} \textbf{Composition du jury}}\\
			{\color{Prune} \scriptsize {Membres du jury avec voix délibérative}} \\
			\vspace{2mm}
			\scriptsize
			\begin{tabular}{|p{7cm}l}
				\arrayrulecolor{Prune}
				\textbf{Yves BENOIST} &  Président du jury \\ 
				Professeur, Laboratoire de Mathématiques d'Orsay, Université Paris-Saclay   &   \\ 
				\textbf{Peter HAÏSSINSKY} &  Rapporteur \& Examinateur \\ 
				Professeur, Institut de Mathématiques de Marseille, Aix-Marseille Université   &   \\ 
				\textbf{Alessandro SISTO} &   Rapporteur \& Examinateur \\ 
				Assistant Professor, Heriott-Watt University &   \\
				
				\textbf{Pierre-Emmanuel CAPRACE} &  Examinateur \\ 
				Professeur, Institut de Recherche en Mathématiques et Physique, UCLouvain   & \\
				\textbf{Anne PARREAU} & Examinatrice \\ 
				Maître de conférences, Institut Fourier, Université Grenoble Alpes  &   \\ 
		\end{tabular} 
		
	\end{titlepage}

	\thispagestyle{empty}
	\newgeometry{top=1.5cm, bottom=1.25cm, left=2cm, right=2cm}

	\noindent 
	\includegraphics[height=2.45cm]{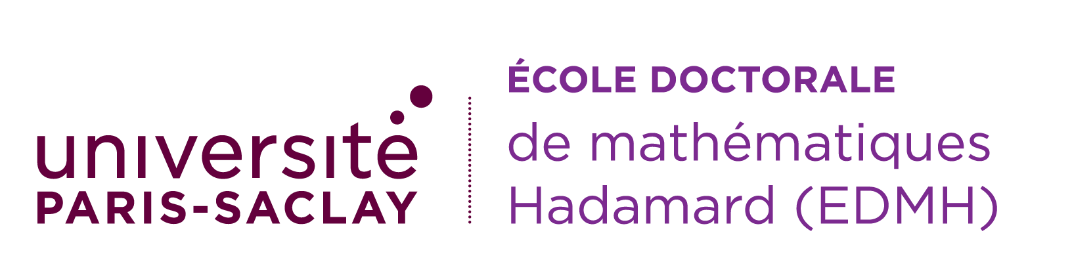}
	\vspace{1cm}
	
	\small
	
	\begin{mdframed}[linecolor=Prune,linewidth=1]
		
		\textbf{Titre:} Marches aléatoires et éléments contractants sur des espaces $\cat$(0). 
		
		\noindent \textbf{Mots clés:} Théorie des groupes, Géométrie, Topologie, Dynamique, Courbure négative. 
		
		\vspace{-.5cm}
		\begin{multicols}{2}
			\noindent \textbf{Résumé:} Dans cette thèse, on étudie des marches aléatoires induites par des actions de groupes probabilisés sur des espaces de courbure négative ou nulle. On décrit en particulier le cas d'actions sur des espaces $\cat$(0) admettant des éléments contractants, c'est-à-dire dont les propriétés imitent ceux des isométries loxodromiques dans les espaces Gromov-hyperboliques. On montre dans ce cadre plusieurs lois limites, comme la convergence presque sûre vers le bord sans hypothèse de moment, la positivité de la dérive, et un théorème de la limite centrale. Dans un second temps, on étudie les mesures stationnaires sur un immeuble affine de type $\tilde{A}_2$, et on montre qu'il existe toujours des isométries hyperboliques pour une action non élémentaire et par automorphismes sur un tel espace. Notre approche implique l'utilisation de modèles hyperboliques pour les espaces $\cat$(0), dont la construction est due à H.~Petyt, D.~Spriano et A.~Zalloum, et des techniques de théorie des bords mesurés dont les principes remontent à H.~Furstenberg. 
		\end{multicols}
		
	\end{mdframed}
	
	\vspace{8mm}
	
	\begin{mdframed}[linecolor=Prune,linewidth=1]
		
		\textbf{Title:} Random walks and contracting elements on $\cat$(0) spaces. 
		
		\noindent \textbf{Keywords:} Group theory, Geometry, Topology, Dynamics, Negative curvature. 
		
		\begin{multicols}{2}
			\noindent \textbf{Abstract:} This thesis is dedicated to random walks on spaces with non-positive curvature. In particular, we study the case of group actions on $\cat$(0) spaces that admit contracting elements, that is, whose properties mimic those of loxodromic isometries in Gromov-hyperbolic spaces. In this context, we prove several limit laws, among which the almost sure convergence to the boundary without moment assumption, positivity of the drift and a central limit theorem. In a second part, we study boundary maps and stationary measures on affine buildings of type $\tilde{A}_2$, ans we show that there always exists a hyperbolic isometry for a non-elementary action by isometries on such a space. Our approach involves the use of hyperbolic models for $\cat$(0) spaces, which were constructed by H.~Petyt, D.~Spriano and A.~Zalloum, and measured boundary theory, whose principles come from H.~Furstenberg. 
		\end{multicols}
	\end{mdframed}
	\vfill
	\newgeometry{top=3.5cm,bottom=4cm,left=1.5cm,right=3cm, heightrounded,bindingoffset=5mm}
	\makeatother
	
	\fontsize{12}{15}\selectfont
	\rmfamily
	
	\selectlanguage{french}
	
	\mtcselectlanguage{french}
	
	\begin{flushleft}
		\huge{\textbf{Remerciements}}
	\end{flushleft}
	\normalsize
	\vspace{1cm}
	 Cette thèse est le produit d'un travail collectif qui, pour une large part, est invisible. J'ai écrit ce manuscrit après plusieurs années de travail sous les conseils de Jean Lécureux, qui a été un directeur de thèse exceptionnellement attentionné et bienveillant, et à qui je souhaiterais adresser mes premiers remerciements. Jean, tu as réussi à créer un équilibre parfait entre ton accompagnement et la liberté que tu me laissais, et tu m'as orienté vers des questions mathématiques passionnantes et riches. Ton écoute, ton savoir et tes conseils m'ont énormément aidé, et j'espère de tout cœur pouvoir continuer à discuter et travailler avec toi dans les prochaines années. 
	 \newline
	 
	 Je tiens à remercier Peter Haïssinsky et Alessandro Sisto, qui m'ont fait l'honneur de rapporter cette thèse. Je suis particulièrement reconnaissant à Peter Haïssinsky pour ses commentaires et questions qui ont permis d'en améliorer le contenu. Je remercie également chaleureusement Yves Benoist, Pierre-Emmanuel Caprace et Anne Parreau pour avoir accepté de faire partie de mon jury. 
	 \newline 
	 
	 J'ai bénéficié de la gentillesse et du savoir de nombreux\textperiodcentered ses chercheuses et chercheurs du laboratoire d'Orsay. Je remercie notamment Frédéric Paulin, Bruno Duchesne et Yves Benoist, qui ont pris du temps pour m'aider dans mes recherches ou dans mon parcours. Le laboratoire d'Orsay est un lieu intimidant mais au sein duquel j'ai rencontré de très belles mathématiques et où j'ai appris la pratique scientifique. J'ai de plus eu l'honneur de faire partie de son équipe de foot, qui avec ses attaquants Rachid Fahlaoui et Pascal Massart n'avait pas d'adversaires sérieux sur le plateau. Un grand merci aux personnes travaillant dans l'administration, au service informatique et à l'École Doctorale, qui font un travail peu reconnu bien qu'indispensable. Plus généralement, je pense à tous\textperiodcentered tes les non-titulaires de l'université, précaires qui la font tenir debout en souffrant.  
	 \newline 
	 
	 Je suis très reconnaissant envers Thomas Haettel, qui a été un formidable directeur de mémoire pour mon Master, et qui m'a fait découvrir le domaine des espaces hyperboliques et des complexes cubiques. Thomas a ensuite continué de m'aider en m'offrant ses conseils et idées, et j'ai toujours plaisir à discuter avec lui. Je remercie également Camille Horbez pour m'avoir proposé d'étudier les groupes modulaires de surfaces, le complexe des courbes et la belle construction des quasi-arbres d'espaces métriques associés aux complexes de projections, lors d'un stage pendant le premier confinement. Ces objets n'apparaissent pas explicitement dans ce manuscrit, mais certaines idées m'ont donné des perspectives intéressantes, et les rencontrer m'a permis d'élargir ma culture scientifique. 
	 \newline
	 
	 Je souhaiterais remercier beaucoup de jeunes chercheur\textperiodcentered ses, qui ont fait de ma thèse quelque chose d'un peu moins solitaire et avec qui j'espère pouvoir travailler dans les prochaines années. Je pense notamment à Hermès Lajoinie, Paige Helms, Ella Blair, Mireille Soergel, Oussama Bensaid, Antoine Goldsborough, Pénélope Azuelos, Baptiste Cerclé, Yusen Long et David Xu. Je suis particulièrement reconnaissant envers Harry Petyt, Davide Spriano et Abdul Zalloum d'avoir gentiment accepté de discuter de leur construction des modèles hyperboliques. Félix Loubaton, en plus d'être un ami qui m'est cher, a attentivement relu des parties de ce manuscrit, et m'a beaucoup aidé sur la fin de sa rédaction. 
	 \newline 
	 
	 Je souhaiterais enfin manifester toute ma gratitude et mon amour à mes proches, ma famille et mes ami\textperiodcentered es. A mes parents, sans qui, pour des raisons évidentes et bien d'autres encore, je n'aurais pas pu faire cette thèse. À mes frères Alain et Yann, avec qui je partage une relation belle et précieuse, et à ma petite sœur Garance, qui a très tôt compris qu'avoir un grand frère qui fait des maths avait des inconvénients, surtout en vacances. Au reste de ma famille enfin, merci pour votre affection. Si travailler sur cette thèse à Paris ces dernières années a été une expérience heureuse, c'est aussi beaucoup grâce aux merveilleuses personnes de mon entourage, mes ami\textperiodcentered es qui le resteront malgré l'éloignement. 
	 \newline 
	 
	 Une bonne partie des pages de ce manuscrit a été écrite avec Léonard à mes côtés, beau et affectueux chat blanc doué dans les manipulations informatiques complexes --- on pourrait dire que ce texte lui a survécu. 
	 \newline 
	 
	 Ma dernière pensée, celle qui compte le plus, vient à Juliette, évidemment, que j'aime et qui est ma joie. 
	 \newpage
	\dominitoc
	
	\tableofcontents
	
	\chapter{Introduction}
	\minitoc


	\section{Marches aléatoires en courbure non-positive}\label{intro lois lim}

	Cette thèse est consacrée à certains problèmes de théorie géométrique des groupes, et leur lien avec les systèmes dynamiques sur des espaces de courbure négative. En particulier, nous étudierons certaines propriétés de marches aléatoires induites par des actions de groupes probabilisés sur des espaces qui manifestent un comportement hyperbolique. Plus précisément, soit $G \curvearrowright (X,d,o)$ une action de groupe sur un espace métrique marqué, et soit $(Z_n) $ une marche aléatoire sur $G$, on étudie les questions suivantes : 
	\begin{itemize}
		\item En fonction des propriétés géométriques de $X$, que peut-on dire du comportement asymptotique de la suite de variables aléatoires $(Z_n o)_{n \in \mathbb{N}}$ ? 
		\item Réciproquement, ces lois limites peuvent-elles nous en apprendre sur le groupe $G$ ? 
	\end{itemize}
	
	Ces questions ont été étudiées dans des contextes variés et ont mené à des avancées importantes en théorie des groupes. En particulier, l'étude de ces marches aléatoires et de ces systèmes dynamiques mesurés a permis de démontrer certains résultats fondamentaux n'ayant a priori rien de probabilistes. Le théorème de super-rigidité de Margulis en est un exemple remarquable. On verra que les propriétés asymptotiques des marches aléatoires sur un groupe ont en effet des liens étroits avec, par exemple, des questions d'analyse harmonique, de croissance et de moyennabilité du groupe ou encore de rigidité. L'étude de ces systèmes dynamiques aléatoires fait donc l'objet d'une recherche active, et dont les ramifications sont nombreuses.  
	\newline
	
	\subsection*{Émergence du sujet}
	
	Soit $X_1, \dots, X_n$ des variables aléatoires réelles, indépendantes et identiquement distribuées. L'étude des propriétés asymptotiques de la somme $S_n = X_1 + \dots + X_n$ est un sujet fondamental et bien compris de la théorie des probabilités. Dès le début du XVIIIe siècle, Bernoulli prouve en effet que si $X_1$ a une espérance finie, $S_n/n $ converge en probabilité vers $\mathbb{E}[X_1]$. Dans un travail qui a fondé la théorie moderne des probabilités, Kolmogorov montre en 1928 que la convergence a également lieu presque sûrement \cite{kolmogorov28}. Tout au long du vingtième siècle, des résultats limites de cette nature (théorèmes de récurrence, loi des grands nombres, théorème de la limite centrale...) sont prouvés. Plus récemment, on a cherché à étendre ce type de résultats pour d'autres groupes, dans un contexte non commutatif : on s'intéresse au produit $S_n = X_1  \dots X_n$, où les $X_i$ sont des variables aléatoires à valeurs dans un groupe $G$, distribuées selon une même mesure de probabilité $\mu  $. Les premiers résultats dans ce cadre remontent à Furstenberg et Kesten, qui étudient dans \cite{furstenberg_kesten60} et \cite{furstenberg63} des produits de matrices dans $\gl(n ,\mathbb{R})$. Sous certaines conditions de moment pour la mesure $\mu$, les auteurs montrent qu'il existe $\lambda >0 $ tel que pour tout $v \in \R^n \setminus \{0\}$, 
	\begin{eqnarray}
		\frac{1}{n} \log \| X_n \dots X_1 v\| \underset{n}{\longrightarrow} \lambda. \nonumber
	\end{eqnarray}
	Ce résultat est l'analogue de la loi des grands nombres pour la marche aléatoire $(X_n \dots X_1 v)_{n }$. Sous des conditions de moments plus restrictives, la marche aléatoire vérifie d'autres propriétés asymptotiques comme le théorème de la limite centrale, \cite{lepage82}, \cite{guivarch_raugi85} : il existe $\sigma_\mu >0 $ tel que pour tout $v \in \mathbb{R}^n \setminus \{0\}$, 
	\begin{equation*}
		\frac{\log \| X_n \dots X_1 v\| - n  \lambda}{\sqrt{n}} \underset{n}{\longrightarrow} \mathcal{N}(0, \sigma_\mu^2), 
	\end{equation*}
	où $\mathcal{N}(0, \sigma_\mu^2)$ désigne une loi gaussienne centrée, non-dégénérée. 
	\newline
	
	L'étude de ces lois limites dans des groupes présentant de la courbure négative s'est également beaucoup développée, notamment depuis l'introduction des groupes hyperboliques par Gromov dans \cite{gromov87}. Dans l'article fondamental \cite{kaimanovich00}, V. Kaimanovich démontre que si $G$ est un groupe Gromov-hyperbolique, alors la marche aléatoire $(Z_n ) $ converge presque sûrement vers un point du bord $\bdg G$. Ce résultat de convergence a été amélioré par Maher et Tiozzo dans l'article \cite{maher_tiozzo18}, où les auteurs montrent cette convergence pour des marches aléatoires issues d'actions de groupes sur des espaces hyperboliques non localement compacts. Encore plus récemment, Gouëzel montre que dans ce cadre, la marche aléatoire $(Z_n o)$ satisfait une loi des grands nombres, et ce sans hypothèse de premier moment fini, \cite{gouezel22}. Ces progrès concernent également les lois limites à l'ordre supérieur. Par exemple, sous des hypothèses de moments exponentiels finis, des théorèmes de la limite centrale sont démontrés pour les groupes libres \cite{sawyer_steger87}, puis pour une large classe de groupes hyperboliques \cite{bjorklund09}. Plus récemment, Y.~Benoist et J-F.~Quint ont prouvé des théorèmes de la limite centrale pour les groupes linéaires \cite{benoist_quint16CLTlineargroups} et hyperboliques \cite{benoist_quint16}. 
	\newline 
	
	Il existe une grande variété d'approches pour étudier ces lois limites. L'une d'entre elles  concerne l'utilisation de la théorie des bords. Dans l'article fondamental \cite{furstenberg63}, Furstenberg introduit la notion de bord de Poisson-Furstenberg $B(G, \mu)$ d'un groupe probabilisé $(G, \mu)$. L'espace $B$ est un espace probabilisé qui représente l'ensemble des directions asymptotiques de la marche aléatoire $(Z_n) $ sur $G$ engendrée par $\mu$. On en donnera une présentation plus approfondie dans la Section \ref{section bord poisson}, ainsi que ses principales propriétés. On verra notamment que déterminer le bord de Poisson-Furstenberg associé à une mesure $\mu$ peut renseigner sur des propriétés géométriques et analytiques du groupe $G$. Par exemple, les fonctions mesurables bornées sur $B(G, \mu)$ sont canoniquement associées aux fonctions $\mu$-harmoniques sur $G$ \cite{furstenberg63}. Un autre exemple est fourni par le résultat suivant, lié à la croissance dans le groupe : $G$ est moyennable si et seulement s'il existe une mesure admissible $\mu$ telle que $B(G, \mu)$ est trivial, voir \cite{kaimanovich_vershik} et \cite{rosenblatt81}. Enfin, il existe une action naturelle de $G$ sur son bord de Poisson, et cette action possède des propriétés particulièrement intéressantes de moyennabilité et d'ergodicité \cite{zimmer78}. 
	
	Les propriétés associées à l'action de $G$ sur son bord de Poisson ont été utilisées de manière cruciale pour étudier certains phénomènes de rigidité de groupe. L'exemple emblématique de ce genre de résultat est le théorème de super-rigidité de Margulis, dont le théorème de Mostow est une des applications les plus spectaculaires. 
	
	\begin{nthm}[{Rigidité de Margulis-Mostow, \cite[Section 5]{margulis91}}]
		Soit $G$ et $G'$ deux groupes de Lie semisimples de rangs supérieurs, connexes, de centres triviaux et sans facteurs compacts. Soit $\Gamma$ un réseau irréductible de $G$ et $\Gamma'$ un réseau de $G'$. Alors tout isomorphisme $\phi : \Gamma \rightarrow \Gamma'$ s'étend en un isomorphisme rationnel $\tilde{\phi} : G \rightarrow G'$.
	\end{nthm} 
	
	L'un des ingrédients principaux de la preuve est l'utilisation d'une \textit{application de bord}, c'est-à-dire une application équivariante $\psi$ entre des bords bien choisis $B$ et $B'$ de $\Gamma$ et $\Gamma'$ respectivement. Grâce aux hypothèses sur $G$ et sur la représentation  $\phi : \Gamma \rightarrow \Gamma'$, $\psi$ s'étend en un homomorphisme rationnel de groupes. En revanche, l'existence de cette application est non triviale, et provient justement des bonnes propriétés ergodiques de l'action $\Gamma \curvearrowright B$. Les applications de bords jouent des rôles importants dans de nombreuses situations, et sont liées aux propriétés asymptotiques des marches aléatoires sur le groupe $G$. Plus récemment, U.~Bader et A.~Furman ont introduit la notion de $G$-bord et développé de nouvelles techniques de théorie ergodique des représentations inspirées de la théorie des bords de Furstenberg, \cite{bader_furman14}. Grâce à ces méthodes, les auteurs ont étendu le théorème de super-rigidité de Margulis à des groupes algébriques définis sur des corps non nécessairement locaux, \cite{bader_furman22}. 
	\newline 
	
	Mon travail s'inscrit dans cette approche des lois limites par la théorie des bords. Je me suis particulièrement intéressé aux systèmes dynamiques dans des espaces $\cat$(0), c'est-à-dire présentant des caractéristiques de courbure non-positive, et dans lesquels on peut trouver des comportements hyperboliques malgré l'existence de sous-espaces de courbure nulle. 
	\newline

	\subsection*{Espaces de courbure non-positive}
	
	Un espace métrique $(X,d)$ est dit de courbure non-positive dans un sens généralisé, ou $\cat$(0), si \og ses triangles sont plus fins que ses analogues euclidiens \fg. C'est une notion globale, et elle implique que l'espace $(X,d)$ est simplement connexe et uniquement géodésique. Ces espaces ont été particulièrement étudiés en théorie géométrique des groupes lors des vingt dernières années, notamment de par la variété des situations dans lesquelles ils apparaissent. Les variétés riemanniennes simplement connexes de courbure partout négative ou nulle, et en particulier les espaces symétriques de type non compact, sont des familles très importantes d'exemples. En effet, ces derniers peuvent être décrits comme des (produits de) quotients $G/K$, où $G$ est un groupe de Lie semisimple sans facteur compact et $K$ est un sous-groupe compact maximal de $G$, e.g. $\text{SL}_n(\mathbb{R})/\text{SO}(n)$. Une autre famille est celle des immeubles affines, qui peuvent être vus comme les analogues des espaces symétriques pour des groupes de Lie sur des corps locaux non-archimédiens. Par exemple, $\text{PGL}_n(\mathbb{Q}_p)/\text{PGL}_n(\mathbb{Z}_p)$ est l'ensemble des sommets d'un immeuble affine associé canoniquement à $G = \text{PGL}_n(\mathbb{Q}_p)$, et est appelé immeuble de Bruhat-Tits de $G$. Une partie significative de mon travail a justement été consacrée à étudier des lois limites sur certains immeubles affines. Enfin, de nombreux développements récents concernent les complexes cubiques $\cat$(0) et leurs propriétés combinatoires très riches. 
	\newline 
	
	Lorsqu'on étudie un espace métrique $X$, il est souvent très utile de se munir d'un espace qui \og représente \fg{} la géométrie asymptotique de $X$. Un tel objet est appelé \emph{bord}. Lorsque $X$ est un espace $\cat$(0), il existe une définition naturelle de bord que l'on appelle \textit{bord visuel} et que l'on note $\bd X$. L'espace $\bd X$ est construit à partir des classes d'équivalences des rayons géodésiques sur $X$. Lorsque $X$ est propre, il existe une topologie sur $\overline{X} := X \cup \bd X$ qui en fait un espace compact. Il existe néanmoins des constructions prenant en compte d'autres caractéristiques de la géométrie de $X$, et on sera amené à travailler avec ces différentes notions en fonction du contexte.

	\subsection*{Géodesiques et isométries contractantes}
	
	Lorsqu'on considère des espaces $\cat$(0), il est fondamental de pouvoir distinguer des directions qui ont un comportement hyperbolique. La notion de géodésique contractante donne un sens précis à cette idée. D'après Bestvina et Fujiwara \cite{bestvina_fujiwara09}, qui s'appuient eux-mêmes sur des idées de Masur et Minsky \cite{masur_minsky99}, on dit qu'une géodésique $\gamma$ est contractante si toute boule disjointe de $\gamma$ se projette sur $\gamma$ en un ensemble de diamètre uniformément borné. Cette propriété est vraie dans les espaces Gromov-hyperboliques. Une isométrie d'un espace $\cat$(0) est ainsi dite contractante si ses axes le sont. 
	
	Si l'espace $X$ est de plus localement compact, les isométries de rang 1, c'est-à-dire dont aucun axe ne borde de demi-plat, sont contractantes, \cite{bestvina_fujiwara09}. La notion de rang 1 entre en jeu dans plusieurs thèmes de recherches contemporains, parmi lesquels la conjecture du rang \cite{ballman95} ou l'alternative de Tits, voir \cite{caprace_sageev11}. D'après la conjecture du rang, qui a été prouvée dans plusieurs cas particuliers, un groupe agissant géométriquement sur un espace $\cat$(0) quelconque possède génériquement des éléments contractants. Il est donc particulièrement intéressant d'étudier ces actions. Par ailleurs, les isométries contractantes présentent des caractéristiques typiquement vérifiées par les isométries loxodromiques dans les espaces Gromov-hyperboliques. Elles induisent par exemple elles-aussi une dynamique Nord-Sud sur le bord visuel $\bd X$, \cite{hamenstadt09}. Cette analogie est rendue encore plus claire après l'introduction des modèles hyperboliques $\{X_L\}_L$ par Petyt, Spriano et Zalloum \cite{petyt_spriano_zalloum22}, pour lesquels une isométrie contractante de $X$ agit de manière loxodromique sur son modèle $X_L$. Plus généralement, un des thèmes récurrents de ce travail a été de m'inspirer des techniques et des idées de la théorie des marches aléatoires sur les espaces hyperboliques pour étudier des actions de groupes sur des espaces $\cat$(0) admettant des isométries contractantes.

	\section{Présentation des résultats de recherche}
	
	\subsection*{Espaces $\cat$(0) avec des éléments contractants}

	Soit $G$ un groupe dénombrable discret et $\mu$ une mesure de probabilité sur $G$. On supposera toujours que $\mu$ est admissible, c'est-à-dire telle que son support engendre $G$ comme semi-groupe. Soit $(X,d)$ un espace $\cat$(0) séparable complet, et soit $G\curvearrowright (X,d)$ une action par isométries. On suppose de plus que $G$ admet deux isométries contractantes pour cette action, dont les points fixes à l'infini sont disjoints. La première partie de mon travail a été de montrer que dans ce contexte, les lois limites évoquées dans la partie précédente sont bien vérifiées, et que la marche aléatoire adopte des comportements typiquement hyperboliques. Ces résultats ont été obtenus dans \cite{LeBars22} et \cite{le-bars22b} avec l'hypothèse supplémentaire que l'espace $(X,d)$ est propre. Ici, on les montre pour des espaces $\cat$(0) séparables et complets, mais plus nécessairement localement compacts. Il est à noter que nous énonçons ces résultats pour des groupes dénombrables discrets, mais ceux-ci restent vrais si on suppose seulement $G$ localement compact à base dénombrable, en adaptant les preuves avec une attention particulière sur les hypothèses de mesurabilité. 
	\newline 
	
	Le premier résultat que l'on obtient est la convergence de la marche aléatoire vers le bord visuel. 
	
	\begin{ithm}[{Convergence vers le bord visuel}]
		Soit $G$ un groupe dénombrable discret et $\mu$ une mesure de probabilité admissible sur $G$. Soit $(X,d) $ un espace $\cat$(0) séparable complet et soit $G\curvearrowright (X,d)$ une action par isométries telle que $G$ admet deux éléments contractants indépendants. Soit $o \in X$ un point-base de $X$. Alors la marche aléatoire $(Z_n o)_n $ converge presque sûrement vers un point du bord visuel $\bd X$ de $X$. 
	\end{ithm}

	Lorsque la mesure $\mu$ a un premier moment fini, i.e. $\int_G d(g x, x) d\mu(g) < \infty$, le Théorème sous-additif de Kingman implique que la limite $\lambda := \lim_n \frac{1}{n} d(o, Z_n (\omega)o ) $ existe et est essentiellement constante. Cette limite est la \emph{dérive} de la marche aléatoire, et peut se comprendre comme la vitesse moyenne à laquelle $(Z_n o)_n$ part à l'infini. Puisque l'action est isométrique, $\lambda$ ne dépend pas du choix du point-base $o \in X$. Sous l'hypothèse d'un premier moment fini et d'une dérive strictement positive, Karlsson et Margulis montrent que $(Z_n o)_n$ converge presque sûrement vers le bord visuel \cite[Theorem 2.1]{karlsson_margulis}. Dans le cas où $G$ est non-moyennable, Guivarc'h montre que la marche aléatoire engendrée par une métrique des mots sur $G$ a toujours une dérive strictement positive \cite{guivarch80}. Dans notre cas, on ne suppose pas que l'action est propre, et on montre la convergence vers le bord sans hypothèse de moment. De plus, la positivité de la dérive peut être difficile à obtenir à priori. En fait, on l'obtient dans un second temps. 
	
	\begin{ithm}[{Positivité de la dérive}]
		Soit $G$ un groupe dénombrable discret et $\mu$ une mesure de probabilité admissible sur $G$. Soit $(X,d) $ un espace $\cat$(0) séparable complet et soit $G\curvearrowright (X,d)$ une action par isométries admettant deux éléments contractants indépendants. Soit $o \in X$ un point-base de $X$. Soit $l_X(\mu) \in [0, \infty]$ défini par 
		\begin{eqnarray}
			l_X(\mu) :=\liminf_{n\rightarrow \infty} \frac{1}{n}\int_\Omega d(Z_n(\omega)o, o) d \mathbb{P}(\omega) \nonumber.
		\end{eqnarray}
		Alors $l_X(\mu) >0$. Si de plus $\mu$ a un premier moment fini, on a la convergence presque sûre
		\begin{eqnarray}
			 \frac{1}{n} d(o, Z_n o )  \underset{n \rightarrow \infty}{\longrightarrow} l_X(\mu) >0. \nonumber
		\end{eqnarray}
	\end{ithm}
	
	En particulier, le Théorème de Karlsson et Margulis \cite[Theorem 2.1]{karlsson_margulis} s'applique, et on a immédiatement un suivi sous-linéaire de la marche aléatoire par une géodésique. 
	
	\begin{icor}[{Suivi sous-linéaire}]
		Soit $G$ un groupe dénombrable discret et $\mu$ une mesure de probabilité admissible et de premier moment fini sur $G$. Soit $(X,d) $ un espace $\cat$(0) séparable complet et soit $G\curvearrowright (X,d)$ une action par isométries admettant deux éléments contractants indépendants. Soit $o \in X$ un point-base de $X$. Alors pour presque tout $\omega \in \Omega$, il existe un rayon géodésique $\gamma^\omega : [0, \infty) \rightarrow X$ issu de $o$ tel que
		\begin{equation}
			\lim_{n\rightarrow \infty} \frac{1}{n} d(\gamma^\omega(l_X(\mu) n), Z_n(\omega)o) = 0, \nonumber
		\end{equation}
		où $l_X(\mu) >0 $ est la dérive de la marche aléatoire. 
	\end{icor}
	
	On peut de plus donner des précisions sur les éléments de la marche aléatoire. Dans \cite[Theorem 1.4]{maher_tiozzo18}, Maher et Tiozzo montrent que la longueur de translation des éléments d'une marche aléatoire sur un espace hyperbolique croît presque sûrement au moins linéairement. Ceci implique que la proportion des éléments de la marche qui sont loxodromiques tend vers 1. Le résultat analogue dans notre contexte est le suivant. 
	
	\begin{ithm}[{Proportion d'éléments contractants}]\label{thm prop contra intro}
		Soit $G$ un groupe dénombrable discret et $\mu$ une mesure de probabilité admissible sur $G$. Soit $(X,d) $ un espace $\cat$(0) séparable complet et soit $G\curvearrowright (X,d)$ une action par isométries admettant deux éléments contractants indépendants. Alors la proportion d'isométries contractantes parmi les éléments de la marche aléatoire $(Z_n)_n$ tend vers $1$ quand $n $ tend vers $\infty$ : 
		\begin{eqnarray}
			\mathbb{P}(\omega \, : \, Z_n (\omega) \text{ est une isométrie contractante }) \underset{n\rightarrow \infty}{\longrightarrow} 1 \nonumber.
		\end{eqnarray}
		Si $\mu$ est à support fini, alors il existe $c < 1$ et $ K, L>0$ tels que pour tout $n $, 
		\begin{eqnarray}
			\mathbb{P}(\omega \, : \, Z_n (\omega) \text{ n'est pas contractant }) \leq Kc^n  \nonumber.
		\end{eqnarray}
	\end{ithm}

	Dans \cite{petyt_spriano_zalloum22}, Petyt, Spriano et Zalloum introduisent pour tout espace $\cat$(0) une famille de distances $\{d_L\}_L$ sur $X$ telles que les espaces $X_L = (X, d_L)$ sont hyperboliques, et telles qu'une isométrie contractante de $X$ est une isométrie loxodromique du modèle $X_L$. De plus, ces espaces traduisent bien les comportements hyperboliques de $X$, et permettent de faire des estimations géométriques précises quant à ces comportements. La démonstration des théorèmes précédents se fait en étudiant la marche aléatoire $(Z_n ) $ agissant sur un modèle hyperbolique bien choisi $(X,d_L)$. Au cours de la démonstration, on est amené à étudier le lien entre le bord de Gromov $\bdg X_L$ de ces espaces hyperboliques et un sous-ensemble $B_L$ du bord visuel. On montre le théorème suivant, qui était connu dans le cas propre \cite[Proposition 8.9]{petyt_spriano_zalloum22}.
	
	\begin{ithm}
		Soit $(X,d)$ un espace $\cat$(0) complet. Alors l'application identité $\iota : (X,d_L) \rightarrow (X, d)$ induit un plongement homéomorphe et $\iso(X)$-équivariant entre les bords $\iota_L : \bdg X_L \hookrightarrow \bd X$.
	\end{ithm}
	
	L'utilisation de ces modèles $X_L$ permet de convoquer la théorie, très riche, des marches aléatoires sur les espaces hyperboliques. En particulier, lorsqu'un groupe agit de manière non-élémentaire sur un espace hyperbolique, la convergence de la marche aléatoire $(Z_n) $ vers le bord de Gromov est connue, \cite{maher_tiozzo18} et \cite{gouezel22}. On en redonne une démonstration, qui s'appuie sur des idées de théorie ergodique et de rigidité des applications de bords présentées dans \cite{bader_caprace_furman_sisto22}. Ce résultat n'est pas nouveau en lui-même, mais la simplicité de l'argument et les propriétés très générales qui sont impliquées dans la preuve pourraient s'avérer utiles dans des contextes différents. Nous pensons que cela pourrait par exemple être le cas dans l'étude de cocycles ergodiques, c'est-à-dire de produits aléatoires d'isométries identiquement distribuées mais non-indépendantes. 
	\newline 
	
	Lorsqu'on étudie une marche aléatoire $(Z_n o)$ sur $X$, il est souvent très utile de doter une compactification bien choisie $\overline{X}$ d'une mesure de probabilité invariante sous l'action du système dynamique aléatoire sous-jacent $(G, \mu, X)$. Plus précisément, on munit $\overline{X}$ d'une mesure de probabilité $\mu$-stationnaire, c'est-à-dire préservée sous l'action du produit de convolution par $\mu$. Puisque dans notre contexte la marche aléatoire $(Z_n o)$ converge presque sûrement vers le bord visuel $\bd X$, la mesure de sortie $\nu$ définie pour tout mesurable $F \subseteq \overline{X}$ par 
	\begin{eqnarray}
		\nu(F) = \mathbb{P}\big(\lim_n Z_n o \in F\big) \nonumber
	\end{eqnarray} 
	est une mesure stationnaire. En fait, il s'agit de la seule. 
	
	\begin{ithm}[{Unicité de la mesure stationnaire}]
		Soit $G$ un groupe dénombrable discret et $\mu$ une mesure de probabilité admissible sur $G$. Soit $(X,d) $ un espace $\cat$(0) séparable complet et soit $G\curvearrowright (X,d)$ une action par isométries qui admet deux éléments contractants indépendants. Alors il existe une unique mesure de probabilité $\mu$-stationnaire sur $\overline{X}$. 
	\end{ithm}
	
	Dans le contexte d'une marche aléatoire sur un groupe hyperbolique, Blachère, Haissinski et Mathieu donnent dans \cite{blachere_haissinsky_mathieu11} une formule explicite pour calculer la dimension de Hausdorff de la mesure harmonique $\nu$. Ce résultat a été récemment amélioré par Dussaule et Yang dans \cite{dussaule_yang20} qui montrent une formule analogue pour des groupes relativement hyperboliques. L'étude de la mesure $\nu$ et de ses propriétés est un sujet intéressant en lui-même, et on montrera qu'elle ne charge que les directions \og les plus hyperboliques \fg. 
	
	L'unicité de la mesure stationnaire $\nu$ est de plus indispensable pour démontrer notre théorème de la limite centrale, dont on fait maintenant la présentation. 
	\newline 
	
	Le résultat de positivité de la dérive correspond à une loi des grands nombres pour les variables aléatoires $(Z_n o )$ dans le contexte d'un produit non-commutatif. Il est naturel de se demander si on peut également démontrer un théorème de la limite centrale (TLC). Furstenberg et Kesten le montrent pour des produits de matrices aléatoires, \cite{furstenberg_kesten60}, et Björklund pour des espaces hyperboliques. Dans ces deux cas, l'hypothèse sur la mesure $\mu$ est forte, car on suppose un moment exponentiel fini, c'est-à-dire l'existence d'un nombre $\alpha >0 $ tel que $\int_G \exp(\alpha d(o, g o)) d\mu(g)<\infty$. Les techniques employées se ramènent à montrer l'existence d'un trou spectral pour l'opérateur de Markov $P_\mu$ associé à la marche aléatoire. 
	
	Plus récemment, Y. Benoist et J-F. Quint ont développé une nouvelle approche pour cette question, qui permet de s'affranchir de l'hypothèse de moment exponentiel. Sous la condition (optimale) où $\mu$ a un second moment fini $\int_G d(o, g o)^2 d\mu(g)<\infty$, Benoist et Quint prouvent un théorème de la limite centrale dans les contextes linéaires \cite{benoist_quint16CLTlineargroups} et hyperboliques \cite{benoist_quint16}. Grâce à cette approche, C. Horbez démontre un résultat similaire pour les groupes modulaires de surfaces orientables et pour $\Out(F_N)$ \cite{horbez18}. Plus récemment, Fern\'os, Lécureux et Mathéus ont montré le TLC pour les groupes agissant de manière non-élémentaire sur un complexe cubique $\cat$(0) de dimension finie \cite{fernos_lecureux_matheus21}. Nous utilisons également l'approche de Benoist-Quint pour montrer le résultat suivant. 
	
	\begin{ithm}[{Théorème de la Limite Centrale}]
		Soit $G$ un groupe dénombrable discret et $\mu$ une mesure de probabilité admissible et de second moment fini sur $G$. Soit $(X,d) $ un espace $\cat$(0) séparable complet et soit $G\curvearrowright (X,d)$ une action par isométries admettant deux éléments contractants indépendants. Soit $o \in X$ un point-base de $X$. Soit $\lambda$ la dérive (strictement positive) de la marche aléatoire $(Z_n o)$. Alors les variables aléatoires $\frac{1}{\sqrt{n}}(d(Z_n o, o) - n \lambda) $ convergent en loi vers une distribution gaussienne $N_\mu$ de variance strictement positive. 
	\end{ithm}
	
	La stratégie repose sur l'étude d'un cocycle associé à la marche aléatoire, dont le comportement représente bien celui des déplacements $d(o, Z_n)$, et la preuve nécessite l'unicité de la mesure stationnaire ainsi que des estimations géométriques. Un résultat similaire est donné par Mathieu et Sisto (\cite[Theorem 13.4]{mathieu_sisto20}), mais le contexte est légèrement différent : les auteurs supposent que $G$ est un groupe de type fini acylindriquement hyperbolique, et que $\mu$ a un moment exponentiel fini. Par ailleurs, leur approche repose sur un contrôle d'inégalités de déviations, une méthode plus \og locale \fg{} que la nôtre. 
	\newline 
	
	Au cours de ce travail, I. Choi a publié une série de travaux dans lesquels plusieurs des lois limites précédentes apparaissent, \cite{choi22a} et \cite{choi22b}. L'hypothèse principale est toujours la présence d'une paire d'isométries contractantes indépendantes, mais l'approche est complètement différente. En effet, Choi s'inspire d'idées de Boulanger, Mathieu, Sisto \cite{mathieu_sisto20}, \cite{boulanger_mathieu_sisto21} et Gouëzel \cite{gouezel22}. Le principe est d'étudier des inégalités de déviations pour la marche aléatoire, et nécessite une technique de pivots développée par Gouëzel dans \cite{gouezel22}. 
	
	Notre approche repose au contraire sur des propriétés ergodiques au bord de la marche aléatoire, et sur la géométrie particulière des modèles hyperboliques $X_L$. Nous pensons que ces deux points de vue ont des intérêts spécifiques et complémentaires. 
	\newline 
	
	\subsection*{Mesures stationnaires dans un immeuble $\tilde{A}_2$}

	La seconde partie de cette thèse est consacrée à l'étude d'une famille particulièrement importante d'espaces $\cat$(0) en théorie géométrique des groupes : les immeubles affines. A moins qu'il s'agisse d'un arbre réel, un immeuble affine ne possède pas d'éléments contractants et les résultats précédents ne s'appliquent plus. Soit $(X,d)$ un immeuble simplicial épais de type $\tilde{A}_2$, et soit $G$ un groupe localement compact à base dénombrable. On se donne une action $G \curvearrowright X$ continue par isométries et non-élémentaire, c'est-à-dire sans sous-espaces plats invariants (possiblement dégénérés en des points). En dimension $\geq 3$, Tits a montré que tout immeuble affine simplicial est canoniquement associé à un groupe algébrique semi-simple sur un corps local. En dimension 2, ce n'est plus vrai, et la classe d'immeubles exotiques la plus étudiée est celle des immeubles $\tilde{A}_2$. Il est donc intéressant d'étudier les marches aléatoires dans ce contexte, car certains immeubles $\tilde{A}_2 $ sont exotiques, et donc échappent aux résultats de marches aléatoires dans les groupes linéaires, voir Section \ref{section immeuble affine}.
	\newline

	La notion de $G$-bord (ou de bord au sens fort) que nous utiliserons a été introduite dans \cite{bader_furman14} pour généraliser les propriétés de moyennabilité et d'ergodicité du bord de Poisson-Furstenberg associé à une mesure $\mu$. Grâce à des résultats de théorie des bords, notamment \cite{bader_duchesne_lecureux16}, et à la géométrie particulière des immeubles $\tilde{A}_2$, on montre que dans ce contexte, il existe une unique application de bord. 
	
	\begin{ithm}[{Existence et Unicité de l'application de bord}]
		Soit $G$ un groupe localement compact à base dénombrable et $G\curvearrowright X$ une action continue par isométries et non-élémentaire sur un immeuble épais, complet et localement fini de type $\tilde{A}_2$. Soit $(B, \nu)$ un $G$-bord. Alors il existe une unique application mesurable $G$-équivariante $\psi : B \rightarrow \ch(\bdinf)$ de $B$ vers les chambres de l'immeuble sphérique à l'infini $\bdinf$. 
	\end{ithm}
	
	Soit maintenant $\mu$ une mesure de probabilité admissible sur $G$, et $B$ le bord de Poisson-Furstenberg associé à $\mu$. Bader et Furman montrent qu'il s'agit d'un $G$-bord \cite{bader_furman14}. Comme on le verra plus en détails dans la Section \ref{section bord poisson}, il y a un lien fort entre les $G$-applications mesurables $B \rightarrow \bd X$ et les mesures $\mu$-stationnaires sur $\bd X$. En particulier, on prouve le résultat suivant. 
	
	\begin{ithm}[{Unicité de la mesure stationnaire sur $\bdinf$}]
		Soit $G$ un groupe localement compact à base dénombrable et $G\curvearrowright X$ une action continue par isométries et non-élémentaire sur un immeuble épais, complet et localement fini de type $\tilde{A}_2$. Soit $\mu$ une mesure de probabilité admissible sur $G$. Alors il existe une unique mesure $\mu$-stationnaire $\nu  \in \prob(\ch(\bdinf))$ sur les chambres de l'immeuble à l'infini $\bdinf$. 
	\end{ithm}
	
	Ce résultat est le point de départ d'une étude plus poussée de la marche aléatoire $(Z_n o)$ engendrée par la mesure $\mu$ sur l'immeuble $X$. En particulier, un projet en cours est de montrer un théorème de la limite centrale, à la manière de ce qui a été fait dans \cite{benoist_quint16CLTlineargroups} pour les groupes linéaires. 
	\newline
	
	Dans un travail en collaboration avec J. Lécureux et J. Schillewaert, l'unicité de la mesure stationnaire sur l'immeuble à l'infini et ses propriétés permettent de prouver des résultats de structure sur le groupe $G$. 
	
	\begin{ithm}[{Existence d'un élément hyperbolique}]
		Soit $G$ un groupe localement compact à base dénombrable et $G\curvearrowright X$ une action continue par isométries et non-élémentaire sur un immeuble épais, complet et localement fini de type $\tilde{A}_2$. Alors il existe un élément de $G$ qui agit comme une isométrie hyperbolique sur $X$. 
	\end{ithm}
	
	De plus, on peut montrer le résultat suivant. 
	
	\begin{ithm}[{Point fixe}]
		Soit $G$ un groupe de type fini et $G\curvearrowright X$ une action continue par isométries et non-élémentaire sur un immeuble épais, complet et localement fini de type $\tilde{A}_2$. Supposons que tous les éléments de $G$ soient elliptiques pour cette action. Alors $G$ fixe un point de $X$. 
	\end{ithm}
	
	Il est probable que les résultats de cette section restent vrais si on enlève l'hypothèse que l'immeuble $X$ est localement fini. En effet, la plupart des preuves restent valables dans le cadre d'un immeuble $\tilde{A}_2$ non-discret mais séparable et métriquement complet. Un projet en cours est de montrer les mêmes résultats pour des immeubles de type $\tilde{C}_2$ et $\tilde{G}_2$. Cela permettrait d'obtenir une extension des résultats d'Osajda et Przytycki, qui considèrent des complexes triangulaires munis d'une action presque libre \cite{osajda_przytycki21}.

	\section*{Structure du texte}
	
	Le texte est divisé en deux parties. La première partie présente le contexte général et introduit la plupart des notions du sujet. Le Chapitre \ref{chapter cat} décrit les espaces $\cat$(0), leurs propriétés et les grandes familles d'exemples. On y rappelle quelques résultats de structure et quelques conjectures importantes. On s'attache enfin à donner la construction des modèles hyperboliques $X_L$ telle qu'elle est faite dans \cite{petyt_spriano_zalloum22}. Dans le Chapitre \ref{chapter immeubles}, on présente la notion d'immeuble en insistant sur les immeubles affines. Forcément lacunaire, cette partie ne consacre que peu d'espace aux immeubles classiques associés aux groupes algébriques. Enfin, dans le chapitre \ref{chapter rw intro}, on présente les marches aléatoires et la théorie de bords d'après Furstenberg. On s'appuiera dans toute la suite sur les résultats énoncés dans cette section. 
	\newline
	
	La seconde partie présente les résultats de recherche, et contient deux chapitres indépendants. Le chapitre \ref{chapter rw cat} concerne les marches aléatoires sur les espaces $\cat$(0) en présence d'éléments contractants. On revient sur quelques éléments de théorie des espaces Gromov-hyperboliques, et on démontre les lois limites citées plus-haut. Ce chapitre est en fait composé de deux papiers soumis à la publication, quoique dans des formes plus faibles, \cite{LeBars22} et \cite{le-bars22b}. Dans le chapitre \ref{chapter rw immeuble}, on étudie les propriétés ergodiques d'une action sur un immeuble $\tilde{A}_2$, et on en déduit des résultats structurels sur le groupe agissant. Il fait l'objet de deux papiers en cours de préparation, dont un en collaboration avec J.~Lécureux et J.~Schillewaert, \cite{le-bars23immeuble} et \cite{le-bars_lecureux_schillewaert23}.

	\fancyhead[RO]{\rmfamily\nouppercase{\rightmark}}
	\fancyhead[LE]{\rmfamily\nouppercase{\leftmark}}
	
	\part{Présentation du contexte}

	\chapter{Espaces CAT(0)}\label{chapter cat}
	
	Ce chapitre est dédié à présenter la théorie des espaces $\cat$(0). L'un des objets de cette partie est de mettre en évidence des similarités entre la dynamique sur les espaces Gromov-hyperboliques et celle sur les espaces $\cat$(0), si on suppose la présence d'éléments contractants. Par souci d'économie, et car il existe de nombreuses références sur le sujet, on ne reviendra pas sur la théorie très riche des espaces Gromov-hyperboliques, et le lecteur ou la lectrice pourra consulter \cite{gromov87}, \cite{ghys_dlharpe90} et \cite[Part III.H]{bridson_haefliger99}. On commence dans la Section \ref{section generalite cat} par donner quelques généralités sur les espaces de courbure non-positive, et à introduire des familles importantes d'exemples, à savoir les complexes cubiques $\cat$(0) et les espaces symétriques de type non compact. Dans la Section \ref{section bord cat}, on construit quelques bordifications importantes possibles sur un espace $\cat$(0). La Section \ref{section isom cat} classifie les isométries sur un espace $\cat$(0), et présente la notion fondamental d'élément de rang 1. Enfin, on présente dans la Section \ref{section modèle hyp} les modèles hyperboliques construits par H.~Petyt, D.~Spriano et A.~Zalloum dans \cite{petyt_spriano_zalloum22}, et qui seront essentiels dans le chapitre \ref{chapter rw cat}. 
	\newline 
	
	\vspace{1cm}
	\minitoc
	
	\vspace{1cm}
	
	\section{Généralités}\label{section generalite cat}
	
	Dans cette section, on présente la théorie des espaces $\cat$(0), ainsi que quelques exemples fondamentaux. Nous ne couvrirons qu'une petite partie de ce sujet, et le lecteur intéressé trouvera plus d'informations dans les références classiques \cite{bridson_haefliger99} et \cite{ballman95}.
	
	\subsection{Définitions et premières propriétés}

	Soit $(X,d)$ un espace métrique. Soit $I$ un intervalle de $\R$, une application $\gamma : I \rightarrow X$ telle que pour tout $t,t' \in I$, $d(\gamma(t), \gamma(t')) = |t-t'|$ est appelée \textit{géodésique}. Dans toute la suite, on supposera $(X,d)$ \textit{géodésique}, c'est-à-dire que pour toute paire de points $x,y \in X$, il existe une géodésique $\gamma : [0, d(x,y)] \rightarrow X$ telle que $\gamma(0) = x$ et $\gamma(d(x,y)) = y$. On dira alors que $\gamma$ est un \textit{segment géodésique} dans $X$, qu'on notera $[x,y]$, bien qu'un tel segment puisse ne pas être unique. Un \textit{triangle géodésique} $\Delta(x,y,z)$ dans $X$ est la donnée de trois points $x,y,z \in X$ appelés \textit{sommets}, ainsi que d'un choix de trois segments géodésiques $[x,y], [x,z], [y, z]$ appelés \textit{côtés}. Un \textit{triangle de comparaison euclidien} par rapport à $\Delta(x,y,z)$ est un triangle géodésique $\overline{\Delta} := \Delta(\overline{x},\overline{y},\overline{z})$ dans $(\R^2, d_{E})$ tel que $d(x, y) = d_E(\overline{x}, \overline{y})$, $d(x, z) = d_E(\overline{x}, \overline{z})$, $d(z, y) = d_E(\overline{z}, \overline{y})$, où $d_E$ est la distance euclidienne sur $\R^2$. Un tel triangle de comparaison existe toujours et est unique à isométrie près. De plus, pour tout $p \in [x,y]$, il existe un unique $\overline{p} \in [\overline{x}, \overline{y}]$ tel que $d(x, p) = d_E(\overline{x}, \overline{p})$ et $d(p, y) = d_E(\overline{p}, \overline{y})$, voir figure \ref{fig tri comparaison}. Un tel point est un \textit{point de comparaison} pour $p$. 
	
	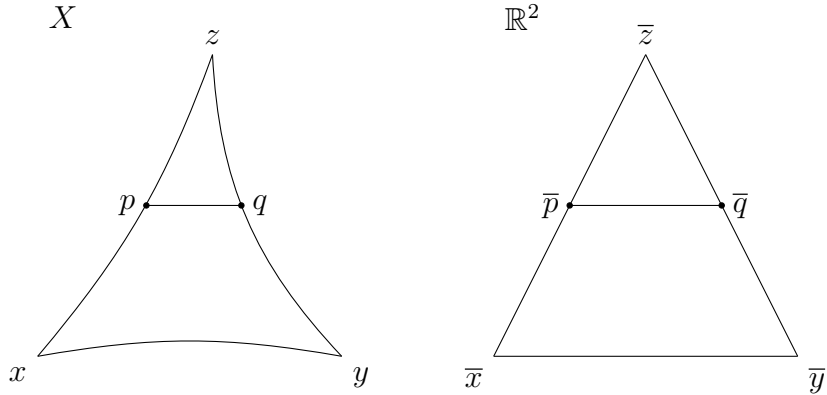
\begin{figure}[h!]
		\centering
		\begin{center}
			\begin{tikzpicture}[scale=1]
				\draw (0,0) to[bend left = 10] (4,0)  ;
				\draw (0,0) to[bend right = 10] (2.3,4)  ;
				\draw (4,0) to[bend left = 20] (2.3,4)  ;
				\draw (6,0) -- (10,0)  ;
				\draw (6,0) -- (8,4)  ;
				\draw (8,4) -- (10,0);
				\draw (1.43,2) -- (2.68, 2) ; 
				\draw (7,2) -- (9,2) ; 
				\draw (1.43,2) node[left]{$p$} ;
				\draw (2.68, 2) node[right]{$q$} ;
				\draw (7,2) node[left]{$\overline{p}$} ;
				\draw (9,2) node[right]{$\overline{q}$} ;
				\draw (0,0) node[below left]{$x$} ;
				\draw (4,0) node[below right]{$y$} ;
				\draw (2.3,4) node[above]{$z$} ;
				\draw (6,0) node[below left]{$\overline{x}$} ;
				\draw (10,0) node[below right]{$\overline{y}$} ;
				\draw (8,4) node[above]{$\overline{z}$} ;
				\filldraw[black] (9,2) circle(1pt);
				\filldraw[black] (7,2) circle(1pt);				
				\filldraw[black] (1.43,2) circle(1pt);
				\filldraw[black] (2.68,2) circle(1pt);
				\draw (0,4.5) node[right]{$X$}; 
				\draw (6,4.5) node[right]{$\R^2$};
				
			\end{tikzpicture}
		\end{center}
		\caption{Triangle et points de comparaison.}\label{fig tri comparaison}
	\end{figure}
	
	\begin{Def}
		Soit un espace métrique géodésique $(X,d)$. Un triangle géodésique $\Delta = \Delta(x,y,z)$ satisfait \emph{l'inégalité $\cat$(0)} si pour toute paire de points $p,q \in \Delta$, les points de comparaison $\overline{p}, \overline{q}$ dans un triangle de comparaison euclidien par rapport à $\Delta$ satisfont 
		\begin{eqnarray}
			d(p,q ) \leq d_E(\overline{p}, \overline{q}) \nonumber. 
		\end{eqnarray}
		On dira que $(X,d)$ est un \emph{espace $\cat$(0)} si tous ses triangles satisfont l'inégalité $\cat$(0). 
	\end{Def}
	
	Il existe une caractérisation équivalente des espaces $\cat$(0), plus pratique à utiliser dans les calculs, que nous donnons maintenant. 
	
	\begin{prop}[{\cite[Proposition I.5.1]{ballman95}}]
		Soit $(X,d)$ un espace métrique complet. Alors $X$ est $\cat$(0) si et seulement si pour tout triplet $x,y,z \in X$,  
		\begin{eqnarray}
			d(z,m)^2 \leq \frac{1}{2} (d(z,x)^2 + d(z,y)^2) - \frac{1}{4}d(x,y)^2, 
		\end{eqnarray}
	où $m$ est le milieu du segment $[x, y]$. 
	\end{prop}
	
	Soit $\kappa \in \mathbb{R}$. Il existe, à isométrie près, une seule surface riemannienne $M_\kappa$ qui est complète, simplement connexe et de courbure sectionnelle constante égale à $\kappa$, voir \cite[Chapter I.2]{bridson_haefliger99}. Soit $\Delta$ un triangle géodésique dans l'espace métrique $(X,d)$. On suppose que le périmètre de $\Delta$ est inférieur à $ 2 \pi/\sqrt{\kappa}$ si $\kappa > 0 $ et de périmètre quelconque sinon. Alors il existe un triangle de comparaison $\Delta_\kappa$ dans $M_\kappa$ respectant les longueurs entre les sommets, unique à isométrie près. Comme pour la définition des espaces $\cat$(0), on dira que $(X,d)$ est un espace $\cat(\kappa)$ si chaque tel triangle de $(X,d)$ est plus fin que son homologue dans l'espace $M_\kappa$, voir \cite[Section I.3]{ballman95}. 
	\newline
	
	Satisfaire l'inégalité $\cat$(0) est une propriété globale, et a des conséquences structurelles sur l'espace, voir \cite[Chapter II.1]{bridson_haefliger99}
	
	\begin{prop}
		Soit $(X,d)$ un espace $\cat$(0). Alors 
		\begin{enumerate}
			\item $X$ est uniquement géodésique : entre toute paire de points $x, y \in X$, il existe un unique segment géodésique qui relie $x$ à $ y$. 
			\item $X$ est contractile.  
		\end{enumerate}
	\end{prop}
	
	\begin{rem}
		Dans toute la suite, on supposera que les espaces $\cat$(0) sont complets. Dans le cas où $X$ n'est pas supposé complet a priori, le théorème de Hopf-Rinow \cite[Theorem I.2.4]{ballman95} donne une liste de conditions équivalentes pour que $X$ soit complet. Néanmoins, nous préciserons les hypothèses dans chaque théorème, et lorsque l'hypothèse de complétude ne sera pas nécessaire, nous l'omettrons. 
	\end{rem}

	Un espace métrique géodésique $(X,d)$ est dit localement $\cat$(0) si pour tout $x \in X$, il existe une boule $B(x,r)$ autour de $x$ de rayon non nul qui est $\cat$(0). Le théorème suivant montre que la notion locale peut s'étendre à tout l'espace si on suppose de plus que l'espace est simplement connexe \cite[Theorem II.4.1]{bridson_haefliger99}.
	
	\begin{thm}[Cartan-Hadamard]
		Un espace métrique géodésique complet $(X,d) $ est $\cat$(0) si et seulement si $X$ est localement $\cat$(0)  et simplement connexe. 
	\end{thm}
	
	\begin{ex}
		Voici quelques exemples d'espaces $\cat$(0) : 
		\begin{itemize}
			\item Les arbres simpliciaux ou réels $(T,d)$, avec la distance de longueur induite.
			\item Les espaces euclidiens $(\mathbb{E},d_E)$. Plus généralement, les espaces de Banach qui sont $\cat$(0) sont exactement les espaces de Hilbert $(H, \|. \|)$. 
			\item Soit $(X_1, d_1)$ et $(X_2, d_2)$ deux espaces $\cat$(0). Alors le produit $X = X_1 \times X_2$, muni de la distance $ d = \sqrt{d_1^2 + d_2^2}$ est $\cat$(0). Si les espaces $X_1$ et $X_2$ sont non bornés, on dira alors que $X$ est \textit{réductible}. 
			\item Les variétés riemanniennes complètes simplement connexes de courbure sectionnelle partout négative ou nulle, \cite[Theorem II.1.A.6]{bridson_haefliger99} sont des espaces $\cat$(0). De telles variétés sont dites \textit{de Hadamard}, et nous reviendrons plus précisément sur cette classe d'exemples. \label{var hadamard}	
		\end{itemize}
	\end{ex}
	
	Donnons quelques propriétés fondamentales satisfaites par les espaces $\cat$(0), et qui nous serviront par la suite.  
	
	\begin{prop}[Convexité, {\cite[Proposition II.2.2]{bridson_haefliger99}}]\label{prop convexite cat}
		Soit $(X,d)$ un espace $\cat$(0). Soit $\gamma_1, \gamma_2 : I \rightarrow \R$ deux géodésiques. Alors $t \in I \mapsto d(\gamma_1(t), \gamma_2(t))$ est une application convexe. 
	\end{prop}

	\begin{prop}[Projection, {\cite[Proposition II.2.4]{bridson_haefliger99}}]\label{prop projection cat}
		Soit $(X,d)$ un espace $\cat$(0). Soit $C\subseteq X$ un fermé convexe non vide, complet pour la métrique induite. Alors :
		\begin{enumerate}
			\item Pour tout $ x \in X$, il existe un unique point $\pi_C(x)$ tel que $d(x, \pi_C(x))= d(x, C) = \inf_{c \in C} d(x, c) $; 
			\item si $x' \in [x, \pi(x)]$, alors $\pi_C(x')= \pi_C(x)$; 
			\item $\pi_C $ est une rétraction de $X$ sur $C$. Plus précisément, l'application $H : [0,1] \times X \rightarrow X $ qui à $(t, x) $ associe le point à distance $td(x ,\pi_C(x))$ de $x$ sur le segment $[x, \pi_C (x)]$ est une homotopie de $\Id_X$ à $\pi_C$ ;
			\item $\pi_C$ n'augmente pas les distances : pour tout $x, y \in X$, $d(\pi_C(x), \pi_C(y )) \leq d(x,y)$. En particulier, $\pi_C$ est continue. 
		\end{enumerate}
	\end{prop}
	
	Soit $Y \subseteq X$ un sous-ensemble borné de $X$. Le \textit{rayon} (ou circumrayon) de $Y$ est défini comme 
	\begin{eqnarray}
		r_Y := \inf_{r>0} \{r \, | \, \text{ il existe }x \in X \text{ tel que } Y \subseteq B(x, r)\} \nonumber
	\end{eqnarray}
	
	\begin{prop}[Centre d'un ensemble borné, {\cite[Proposition II.2.2.7]{bridson_haefliger99}}]\label{prop bruhat tits centre}
		Soit $X$ un espace $\cat$(0) complet, et soit $Y \subseteq X$ un sous-ensemble borné et soit $r_Y$ son rayon. Alors il existe un unique $c_Y \in X$ tel que $Y \subseteq \overline{B}(c_Y, r_Y)$. On appelle $c_Y$ le centre (ou circumcentre) de $Y$. 
	\end{prop}
	
	Cette proposition implique le résultat fondamental suivant. 
	
	\begin{cor}[Théorème du point fixe de Bruhat-Tits]
		Soit $G$ un groupe d'isométries d'un espace $\cat$(0) complet. Si $G$ stabilise un sous-ensemble borné de $X$, alors il fixe un point de $X$. De plus, l'ensemble des points fixes de $G$ est convexe. 
	\end{cor}
	
	En particulier, si $G$ est groupe compact agissant continûment par isométries sur un espace $\cat$(0) complet, il admet un point fixe global. Terminons cette liste de propriétés par le théorème de la bande plate. Soit $X$ un espace $\cat$(0), et soit $\gamma_1, \gamma_2 : \mathbb{R} \rightarrow X$ deux droites géodésiques dans $X$. On dira que $\gamma_1 $ et $\gamma_2$ sont \textit{asymptotes} s'il existe $K \geq 0 $ tel que $d(\gamma_1(t), \gamma_2(t)) \leq K $ pour tout $t$. Par convexité de la métrique, il est équivalent de dire qu'il existe $K \geq 0 $ tel que $d(\gamma_1(t), \gamma_2(t)) = K $ pour tout $t$. 
	
	\begin{thm}[Théorème de la bande plate, {\cite[Theorem II.2.2.13]{bridson_haefliger99}}]
		Soit $X$ un espace $\cat$(0), et soit $\gamma_1, \gamma_2 : \mathbb{R} \rightarrow X$ deux droites géodésiques asymptotes dans $X$. Alors il existe $D \geq 0$ tel que l'enveloppe convexe de $\gamma_1(\R) \cup \gamma_2(\R) $ est isométrique à la bande euclidienne $\R \times [0,D] \subseteq \R^2$. 
	\end{thm}

	On conclut cette section en introduisant la notion de rang d'un espace $\cat$(0). Soit $(X, d)$ un espace $\cat$(0), un sous-espace $Y$ totalement géodésique de $X$ qui est isométrique à un espace euclidien $\mathbb{E}^n$ est appelé un \textit{plat de dimension $n$}.
	
	\begin{Def}
		On dit qu'un espace $X$ est de \textit{rang $\geq d$} si toute paire de points est reliée par une géodésique entièrement contenue dans un plat de dimension $d$ de $X$. Le \textit{rang} d'un espace $\cat$(0) est le minimum de la dimension sur ces plats :  
		\begin{eqnarray}
			\rang(X) := \max \{d \, | \,\textit{ toute géodésique est contenue dans un plat de dimension }d\}. \nonumber
		\end{eqnarray}
	\end{Def}
	
	Le rang d'un espace $\cat$(0) est un invariant qui apparait dans plusieurs conjectures importantes. On les discutera plus en détails dans la section \ref{conjectures}.

	\subsection{Des exemples importants}
	
	Dans cette section, on présente deux familles d'exemples très importants d'espaces $\cat$(0) : les complexes cubiques $\cat$(0) et les espaces symétriques de type non compact. Les premiers possèdent une structure combinatoire très riche, dont on s'inspirera pour trouver des modèles hyperboliques aux espaces $\cat$(0). Une présentation bien plus complète de ces espaces est faite dans \cite{sageev14}. De leur côté, les espaces symétriques ont, entre autres, une grande importance dans la théorie des groupes de Lie. 
	
	\subsubsection{Complexes cubiques $\cat$(0)}\label{section complexe cubique}
	
	Un \textit{complexe cubique} $X$ est un complexe cellulaire dont les cellules sont des cubes euclidiens. A priori, $X$ n'est pas un espace métrique, mais il existe une manière naturelle de construire une distance en supposant que les arêtes sont de longueur 1 et en utilisant la métrique de longueur associée. 
	
	On rappelle que si $(Y, \delta) $ est un espace métrique, la \textit{distance de longueur} $d$ associée à $\delta$ est définie comme 
	\begin{eqnarray}
		d(x, y) := \inf \big\{ \delta(\gamma) \, | \, \gamma \textit{ est une courbe rectifiable entre } x \textit{ et } y \big\}, \nonumber
	\end{eqnarray}
	où une courbe $\gamma : [0, T]$ est dite $\delta$-rectifiable si 
	
	\begin{eqnarray}
		\delta (\gamma) = \sup \big\{ \sum_i \delta(\gamma(t_i), \gamma(t_{i+1})) \, | \, 0 =  t_1 \leq t_2 \dots \leq t_n = T, \ n \in \mathbb{N} \big\} < \infty. \nonumber
	\end{eqnarray}
	
	Soit donc $ X$ un complexe cubique. En supposant que les arêtes sont de longueur 1 (ce qu'on fera toujours), on peut munir $X$ de la métrique de longueur $d$ induite par la distance euclidienne sur les cubes. L'espace $(X,d)$ est donc métrique, et on peut se demander s'il est $\cat$(0). Il existe en fait un argument combinatoire simple pour montrer qu'un complexe cubique donné est $\cat$(0).
	
	Le \textit{link} d'un sommet $v \in X$ est le complexe sphérique obtenu en dessinant une $\varepsilon$-sphère dans $X$ autour de $v$ pour $0< \varepsilon <1$. Le link d'un sommet dispose d'une structure simpliciale abstraite qui ne dépend pas de $0< \varepsilon <1$ ni de son caractère sphérique. On dit qu'un complexe simplicial $\Delta$ est \textit{de drapeau} si tout sous-ensemble de sommets $V \subseteq \Delta^{(1)}$ deux-à-deux adjacents engendre un simplexe.  
	
	\begin{thm}[{Critère de Gromov, \cite{gromov87}}]
		Un complexe cubique simplement connexe de dimension finie est $\cat$(0) si et seulement si le link de chacun de ses sommets est un complexe simplicial de drapeau. 
	\end{thm}
	
	En fait, un résultat récent montre que le théorème précédent est également vrai pour les complexes cubiques de dimension infinie : un complexe cubique dont le link de chaque sommet est de drapeau est un espace localement $\cat$(0) pour la métrique de longueur induite, \cite{leary13}. 
	
	\begin{ex}
		Un exemple parlant est le complexe cubique induit par le graphe de Cayley de $\mathbb{Z}/2\mathbb{Z}\ast \mathbb{Z}$ (pour les générateurs canoniques), dont le link de chaque sommet est donné par la figure \ref{fig link tree of flats}. On s'y réfère comme à \emph{l'arbre des plats}. D'après le critère de Gromov, $\mathbb{Z}/2\mathbb{Z} \ast \mathbb{Z}$ muni de la métrique de longueur induite est un espace $\cat$(0). 
		
		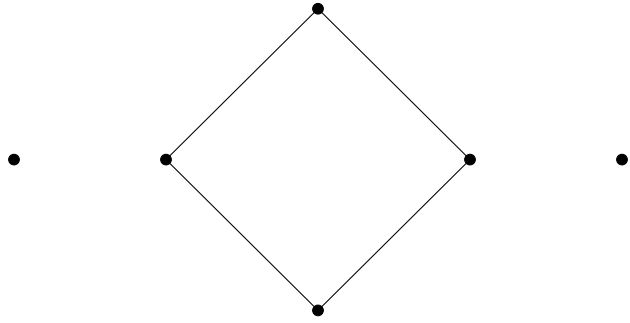
\begin{figure}[!]
			\centering
			\begin{center}
				\begin{tikzpicture}[scale=1]
					\draw (-2,0) -- (0,2)  ;
					\draw (0,2) -- (2,0)  ;
					\draw (2,0) -- (0,-2)  ;
					\draw (-2,0) -- (0,-2)  ;
					\filldraw[black] (-2,0)circle(2pt);
					\filldraw[black] (2,0)circle(2pt);
					\filldraw[black] (4,0)circle(2pt);
					\filldraw[black] (-4,0)circle(2pt);
					\filldraw[black] (0,2)circle(2pt);
					\filldraw[black] (0,-2)circle(2pt);
				\end{tikzpicture}
			\end{center}
			\caption{Link d'un sommet de $\mathbb{Z}^2 \ast \mathbb{Z}$.}\label{fig link tree of flats}
		\end{figure}
		
	\end{ex}
	
	Un intérêt particulier des complexes cubiques est l'existence d'une structure d'hyperplans sur $X$. On s'inspirera de ces hyperplans pour construire des modèles hyperboliques associés à des espaces $\cat$(0) quelconques. 
	
	\begin{Def}
		Soit $X$ un complexe cubique $\cat$(0), et soit un cube $C = [0,1]^n $ dans $X$. Un hypercube $M_i$ dans $C$ est donné par 
		\begin{eqnarray}
			M_i = \{ x \in C \, | \, x_i = \frac{1}{2}\}, \nonumber
		\end{eqnarray}
		où $i \in \{1, \dots, n\}$. 
	\end{Def}
	
	On dira que deux arêtes $e, e'$ sont équivalentes s'il existe une suite de $2$-cubes $C_0, \dots, C_{n-1}$ et des arêtes $e_0, \dots, e_n$ telles que $e_0 = e$, $e_n = e'$, et $e_i, e_{i+1}$ sont opposées dans le $2$-cube $C_i$. On notera $e \sim e'$ la relation d'équivalence, et $[e]$ la classe de $e$ pour $\sim$. Enfin, on dira qu'un hypercube $M$ est \textit{transverse} à une classe d'équivalence $[e]$ si l'intersection de $M$ avec le $1$-squelette de $X$ consiste en les milieux des arêtes $e' \in [e]$, et on notera alors $M \pitchfork [e]$. 
	
	\begin{Def}
		Un \textit{hyperplan} $H$ de $X$ est l'union de tous les hypercubes transverses à une classe d'arêtes $[e]$ : 
		\begin{eqnarray}
			H(e) = \underset{M \pitchfork [e]}{\bigcup} M. \nonumber
		\end{eqnarray}
	\end{Def}
	
	Les hyperplans vérifient notamment les propriétés suivantes. 
	
	\begin{prop}[{\cite[Theorem 1.1]{sageev14}}]\label{prop hyperplans ccc}
		Soit $X$ un complexe cubique $\cat$(0) de dimension finie, et soit $H$ un hyperplan de $X$. 
		\begin{enumerate}
			\item $H$ est un fermé convexe dans $X$;
			\item $H$ est lui-même doté d'une structure de complexe cubique $\cat$(0); 
			\item $X \setminus H$ a exactement deux composantes connexes, appelés demi-espaces. 
		\end{enumerate}
	\end{prop}
	
	Notons $\mathfrak{H}(X)$ l'ensemble des demi-espaces de $X$. Pour $v \in X$ un sommet, notons $U_v := \{ h \in \mathfrak{H} \, | \, v \in h\}$.

	Tout sommet $v $ est uniquement déterminé par $U_v$ : 
	\begin{eqnarray}
		\{v \} = \underset{h \in U_v}{\cap} h. \nonumber
	\end{eqnarray}
	On appelle cette correspondance la \textit{dualité de Roller-Sageev}. Cette structure combinatoire est très riche, et les complexes cubiques $\cat$(0) sont désormais bien étudiés. On sait en effet que les groupes agissant sur des complexes cubique $\cat$(0) vérifient des résultats puissants comme l'alternative de Tits ou la rigidité du rang, dont nous parlerons plus tard. 
	
	Notons également que du point de vue de la théorie géométrique des groupes, des classes très intéressantes de groupes agissent sur des complexes cubiques $\cat$(0) : les groupes d'Artin à angles droits, certains groupes de Coxeter et des groupes à petite annulation. On dit que ces groupes sont \textit{cubulés}.

	\subsubsection{Espaces symétriques de type non compacts}
	Les espaces symétriques de type non-compact sont des espaces $\cat$(0) qui ont une importance fondamentale dans la théorie des groupes de Lie semi-simples. On n'en fera qu'une présentation très rapide, et on renvoie aux références classiques \cite{eberlein96} et \cite{helgason01}. On utilisera également les notes de cours de F. Paulin \cite{paulin} et de B. Duchesne \cite{duchesne16}. 
	
	\begin{Def}
		Une variété Riemannienne connexe $M$ est un \textit{espace symétrique} si pour tout point $x \in M$, il existe une isométrie globale $\sigma_x$ qui fixe $x$ et dont la différentielle $T_x \sigma_x$ en $x$ est $-\Id_{T_x M}$. Un \textit{espace symétrique de type non compact} est un espace symétrique de courbure sectionnelle négative ou nulle, non isométrique à un produit $M' \times \mathbb{R}$. 
	\end{Def}
	
	\begin{ex}
		Toute variété de Hadamard $M$ à courbure sectionnelle constante est un espace symétrique. En effet, le revêtement universel d'une telle variété est (quitte à normaliser la constante de courbure) $\mathbb{E}^n $, $\mathbb{S}^n $ ou $\mathbb{H}^n$, \cite[Théorème 3.38]{paulin}. Dans chacun de ces cas, on trouve facilement des isométries involutives satisfaisant la condition de symétrie. En revanche, si cette courbure est positive ou nulle, l'espace symétrique $M$ n'est pas de type non compact. Un espace symétrique de type non compact est un espace $\cat$(0). 
	\end{ex}

	Soit $M$ un espace symétrique de type non compact. Notons $G = \iso_0(M)$ la composante connexe de l'identité dans $\iso(M)$, et $K = \stab(x)$ le stabilisateur d'un point $x \in M $. On peut montrer que le groupe $G$ est alors un groupe de Lie semi-simple de centre trivial et sans facteur compact, et que $K$ est un sous-groupe compact maximal de $G$ \cite[Chapter 4]{helgason62}. Le groupe $G$ agit sur $M$ de manière transitive, donc on dispose d'un $C^\infty$-difféomorphisme entre l'espace homogène $G/K$ et l'espace symétrique $M$. 
	
	Réciproquement, soit $G$ un groupe de Lie semi-simple de centre trivial et sans facteur compact, et soit $K$ un sous-groupe compact maximal, alors il existe une structure d'espace symétrique de type non compact sur l'espace homogène $G/K$. En fait, étant donné un tel groupe $G$, cette structure d'espace symétrique $M = G/K$ est canonique, au sens où deux telles structures sont multi-homothétiques, c'est-à-dire isométriques après homothéties sur les facteurs de de Rham, voir \cite[Théorème 4.29]{paulin} et \cite[Chap. X]{helgason01}. On a donc une correspondance géométrico-algébrique univoque entre groupes de Lie semi-simples sans facteur compact, de centre trivial, et espaces symétriques de type non-compact.

	La théorie de Lie permet de classifier les groupes de Lie réels semi-simples de centres triviaux et sans facteurs compacts, et par conséquent les espaces symétriques de type non compact. En particulier, tout espace symétrique de type non compact est (à multi-homothétie près) un produit riemannien d'espaces $G/K$, où $G$ est un groupe de Lie simple de centre trivial et sans facteur compact et $K$ un sous-groupe compact maximal. 
	\newline 
	
	Le \textit{rang réel} d'un groupe de Lie $G$ est la dimension d'un sous-espace de Cartan $\mathfrak{a}$ en tant qu'espace vectoriel réel, voir \cite[Section 4.2.1]{paulin}. On peut montrer que le rang d'un espace symétrique de type non compact $X$ en tant qu'espace $\cat$(0) coïncide avec le rang réel du groupe de Lie $\iso_0(X)$  \cite[Proposition 4.31]{paulin}. On dit qu'un espace symétrique est de \textit{rang supérieur} s'il est de rang $\geq 2$.

	\begin{ex}
		Voici quelques exemples importants d'espaces symétriques de type non compacts. 
		\begin{itemize}
			\item L'espace hyperbolique $\mathbb{H}^n_{\R}$ est symétrique de type non compact. En fait, les seuls espaces symétriques de type non compact et de rang 1 sont (à isométrie près) les espaces hyperboliques réels $\mathbb{H}^n_{\R}$, complexes $\mathbb{H}^n_{\mathbb{C}}$, quaternioniens $\mathbb{H}^n_{\mathbb{H}}$ et le plan hyperbolique des octaves de Cayley $\mathbb{H}^2_{\text{Ca}}$. 
			\item Considérons l'espace $\Ell_n$ les matrices symétriques définies positives de déterminant 1 sur $\R^n$, qui correspond aux ellipsoïdes de volume $1$ dans $\R^n $ : 
			\begin{eqnarray}
				\Ell_n := \{ A \in M_n(\R) \, | \, A \text{ est symétrique définie positive de déterminant 1}\} \nonumber. 
			\end{eqnarray}
			Alors $\Ell_n$ est un espace symétrique $\cat$(0), de rang $n-1$. Ici le groupe de Lie associé est $G = \iso_0 (\Ell_n) = \psl_n(\R)$ et $K = \SO(n) \cap \psl_n(\R)$.  
		\end{itemize}
	\end{ex}

	\section{Différentes notions de bords}\label{section bord cat}
	Dans cette section, on discute de la géométrie à l'infini des espaces $\cat$(0), en introduisant différentes bordifications. Chacune de ces bordifications a des avantages et des inconvénients, et on sera souvent amené à passer de l'une à l'autre selon les besoins.

	\subsection{Le bord visuel}
	
	\subsubsection{Bordification par les rayons}
	Soit $(X,d)$ un espace $\cat$(0). La première manière de construire un bord pour $X$ est de considérer les classes d'équivalence des rayons géodésiques. 
	
	\begin{Def}
		Soit $\gamma_1, \gamma_2 : [0, \infty[ \rightarrow X$ deux rayons géodésiques. On dit que $\gamma_1$ et $\gamma_2$ sont \textit{asymptotes} s'ils sont à distance de Hausdorff finie, c'est-à-dire qu'il existe $K \geq 0$ tel que $$\sup_t \{d(\gamma_1 (t),\gamma_2(t))\} \leq K.$$ Cette relation est une relation d'équivalence sur l'ensemble des rayons géodésiques sur $X$, invariante par isométrie. Le \textit{bord visuel} (ou conique) $\bd X $ est l'espace des classes de rayons géodésiques pour cette relation d'équivalence. On note $\overline{X} = X \cup \bd X$ la bordification visuelle correspondante. 
	\end{Def}
	
	Notons déjà qu'on peut recouvrir tout le bord à partir d'un unique point-base. 
	
	\begin{prop}[{\cite[Proposition 8.2]{bridson_haefliger99}}]\label{prop rayon geod}
		Soit $X$ un espace $\cat$(0) complet et $\gamma_x : [0, \infty[ \rightarrow X$ un rayon géodésique tel que $\gamma_x(0) = x$. Alors pour tout $y \in X $, il existe un unique rayon géodésique $\gamma_y$ tel que $\gamma_y(0) = y$, et asymptote au rayon $\gamma_x$.
	\end{prop}
	
	En particulier, pour tout $x \in X$, et pour tout $ \xi = [\gamma] \in \bd X$ une classe de rayons géodésiques, il existe un unique rayon géodésique $\gamma_x^\xi $ qui représente $\xi $ et tel que $\gamma_x^\xi(0) = x$. Dans la suite, pour $o \in X $ et $\xi \in \overline{X}$, on notera toujours $\gamma_o^\xi$ la géodésique issue de $o$ qui représente $\xi$ si $\xi  \in \bd X$ et telle que $\gamma_o^\xi (d(o, \xi))= \xi $ si $\xi \in X$. 
	
	Fixons un point-base $o \in X$. On peut mettre une topologie $\mathcal{T}(o)$ sur l'ensemble des rayons géodésiques issus de $o$, en prenant la topologie de la convergence uniforme sur les compacts. Une base pour la topologie $\mathcal{T}(o)$ est donnée par les ouverts~: 
	
	\begin{equation*}
		U(\gamma, r, \varepsilon) := \{ x \in \overline{X} \ | \ d(x, o) >r,\,  d(\gamma_o^x(r), \gamma(r)) < \varepsilon)\}, 
	\end{equation*}
	où $\gamma $ est un rayon géodésique issu de $o \in X$, $r, \varepsilon >0$, et où on admet que $d(x, o) >r$ est vérifié si $x \in \bd X$. Si $\gamma = \gamma_o^\xi$ représente la classe de $\xi \in \bd X$ et $\gamma(0) = o$, on notera alternativement 
	
	\begin{equation*}
		U(o, \xi, r, \varepsilon) := \{ x \in \overline{X} \ | \ d(x, o) >r, d(\gamma_o^x (r), \gamma(r)) < \varepsilon)\}.  
	\end{equation*}
	
	L'espace $\overline X$ muni de la topologie $\mathcal{T}(o)$ est métrisable. En effet, pour tout rayon géodésique $\gamma_1, \gamma_2$ issus de $o$, on considère la distance suivante : 
	
	\begin{eqnarray}
		d_o(\gamma_1, \gamma_2) = \inf \{ \varepsilon >0 \, | \, \underset{t \in [0, \frac{1}{\varepsilon}]}{\sup} d(\gamma_1 (t), \gamma_2 (t))< \varepsilon\} \nonumber. 
	\end{eqnarray}
	
	Pour $\gamma $ un rayon géodésique issu de $o \in X$, et pour tout $\varepsilon, r > 0$, prenons $r' < \varepsilon$ et $1/r' > r$. Alors par convexité de la métrique, pour tout rayon géodésique $\gamma'$ issu de $o $ qui vérifie $d_o (\gamma, \gamma') < r' $, on a $\gamma' \in U(\gamma, r, \varepsilon)$. Réciproquement, pour $r>0$ si $\varepsilon' < 1/r$ et $r' > 1/r$, alors pour tout $\xi \in U(\gamma, r, \varepsilon)$, on a $d_o (\gamma, \gamma_o^\xi) < r$. 
	
	La proposition suivante montre que l'on peut oublier la dépendance au point-base. 
	
	\begin{prop}[{\cite[Proposition II.8.8]{bridson_haefliger99}}]
		Pour $o, o' \in X$, les topologies $\mathcal{T}(o)$ et $\mathcal{T}(o')$ sont les mêmes. 
	\end{prop}
	
	On appelle cette topologie la \textit{topologie conique} sur $\overline{X}$. Il est clair avec cette topologie que toute isométrie de $X$ s'étend en un homéomorphisme du bord visuel $\bd X$, \cite[Corollary II.8.9]{bridson_haefliger99}.

	\subsubsection{Bordification par les fonctions de Busemann}
	
	Il existe une autre manière de recouvrir le bord visuel d'un espace $\cat$(0), en utilisant des fonctions de Busemann. Soit $(X, d)$ un espace métrique. Notons : 
	
	\begin{eqnarray}
		b : (o, x,y) \in X \times X \times X \mapsto b_x^{o}(y) = d(x, y) - d(x, o). \nonumber
	\end{eqnarray}
	
	Lorsqu'il n'y aura pas de confusion sur le point-base $o $, on notera plus simplement $b_x : X \rightarrow X$ la fonction de Busemann associée à $x \in X$. D'après l'inégalité triangulaire, 
	pour tout $o, x \in X$, $b_x^o $ est $1$-Lipschitzienne sur $X$. De plus, changer de point base ne fait changer la valeur des fonctions de Busemann que d'une constante : soit $o, o', x, y \in X$, 
	
	\begin{eqnarray}
		|b_x^{o'}(y) - b_x^o (y) | & = & | d(o', x ) - d(o, x) | \nonumber\\ 
		& \leq &  d(o, o'). \nonumber
	\end{eqnarray}
	Fixons $o \in X$. Pour $x, x', y \in X$, on a également $|b_x (y ) - b_{x'}(y)| \geq d(x,x')$. On a donc une application injective de $X$ vers les applications continues sur $X$ : 
	
	\begin{eqnarray}
		\iota : x \in X \mapsto b_x^o \in \mathcal{C}(X).  \nonumber
	\end{eqnarray}
	
	\begin{Def}
		Munissons $\mathcal{C}(X)$ de la topologie de la convergence uniforme sur les bornés. Alors la \textit{bordification de $X$ par les fonctions de Busemann} est la clôture $\overline{X}^B := \overline{\iota(X)} \subseteq \mathcal{C} (X)$ pour cette topologie, et $\partial^B X := \overline{X}^B \setminus X$ les fonctions de Busemann à l'infini. 
	\end{Def}
	
	Plus explicitement, on dira qu'une suite $(x_n)_n$ dans $X$ converge vers une fonction de Busemann à l'infini $ f$ si $d(o, x_n) \rightarrow \infty $ et $b_{x_n} \rightarrow f$ pour la topologie de la convergence sur les bornés. 
	
	La proposition suivante montre que la bordification par les fonctions de Busemann est la même que la bordification visuelle. 
	
	\begin{prop}[{\cite[Theorem II.8.13]{bridson_haefliger99}}]
		Soit $(x_n)$ une suite dans $X$ telle que $X$ telle que $d(o, x_n) \rightarrow \infty$. Alors $b_{x_n}^o $ converge vers $f \in \partial^B X$ si et seulement si il existe $\xi \in \bd X$ tel que $x_n \rightarrow \xi $ pour la bordification visuelle. De plus, si $\gamma^\xi_o$ est un rayon géodésique qui représente $\xi$ tel que $\gamma(0) = o$, on peut écrire 
		\begin{eqnarray}
			f(y) = \lim_{t \rightarrow \infty} d(\gamma^\xi_o(t), y) - t,  \nonumber
		\end{eqnarray}
		et on notera $f = b_\xi$. 
	\end{prop}
	
	Par conséquent, les bordifications $\overline{X}$ et $\overline{X}^B $ sont équivalentes. Dans toute la suite, on gardera la notation $\overline{X}$ et $\bd X$ de la bordification visuelle. 
	\newline 
	
	Lorsque $X$ est localement compact, l'espace $\bd X$ est compact, de même que $\overline{X}$ et $X$ est ouvert dense dans $\overline{X}$. En d'autres termes, la bordification visuelle est alors une \textit{compactification}. C'est une conséquence directe du théorème d'Arzelà-Ascoli sur les fonctions de Busemann $b_x$, qui sont $1$-Lipschitz et normalisées par $b_x(x) = 0 $. C'est en particulier le cas lorsqu'on suppose que l'espace est propre, c'est-à-dire où les boules fermées sont compactes. 
	\newline
	
	Il existe une autre bordification par les fonctions de Busemann si on change la topologie. Munissons maintenant $\mathcal{C}(X)$ de la topologie de la convergence uniforme sur les compacts. 
	
	\begin{Def}
		La \textit{bordification de $X$ par les horofonctions} (ou par les métriques fonctionnelles) est la clôture $\overline{X}^h:= \overline{\iota(X)} \subseteq \mathcal{C} (X)$ pour la topologie de la convergence uniforme sur les compacts, et $\partial^h X := \overline{X}^h \setminus X$ les horofonctions. 
	\end{Def}
	
	Puisqu'on utilise la topologie compacte-ouverte sur $\mathcal{C}(X)$, l'espace $\overline{X}^h$ est en fait toujours compact grâce au théorème d'Arzelà-Ascoli. En revanche, bien que l'injection $\iota : X \rightarrow \mathcal{C}(X)$ soit toujours continue, il est possible que $X$ ne soit pas ouvert dans $\overline{X}^h$, donc cette bordification n'est pas une compactification au sens strict. Il peut s'avérer très utile de travailler avec $\overline{X}^h $ car il est toujours compact, même lorsque l'espace $X$ n'est pas propre. C'est par exemple le cas dans \cite{karlsson_ledrappier11}, \cite{maher_tiozzo18}, \cite{fernos_lecureux_matheus18} et \cite{gouezel_karlsson20}. En revanche, lorsque l'espace $X$ est propre, la topologie de la convergence uniforme sur les bornée est la même que celle de la convergence uniforme sur les compacts, et donc les espaces $\overline{X}^h$  et $\overline{X}^B$ sont les mêmes.

	\subsection{Bord de Tits}
	
	Dans cette section, on présente la géométrie de Tits sur les espaces $\cat$(0). Le bord de Tits est en fait le bord visuel, muni d'une topologie plus fine qui permet de mieux détecter les plats de l'espace. Ce qui suit s'inspire de la présentation qui est faite dans \cite[Section II.4]{ballman95}, et dans \cite[Chapter II.9]{bridson_haefliger99}. 
	
	Soit $C_1, C_2 \in \R^\ast_+ \cup \{\infty\}$, $\sigma_1 : [0, C_1[ \rightarrow X$ et $ \sigma_2 : [0, C_2[ \rightarrow X$ deux géodésiques dans $X$ issus du même point $\sigma_1 (0 ) = \sigma_2(0) = x$. Pour tout $(s, t) \in ]0, C_1[ \times ]0,C_2[$, il existe un triangle de comparaison euclidien $\overline{\Delta}_{s,t}$ du triangle $\Delta_{s,t} := \Delta(x, \sigma_1 (s), \sigma_2(t))$ dans $X$ engendré par $(x, \sigma_1 (s), \sigma_2(t))$. Soit $\overline{\angle}_{\overline{x}} (\overline{\sigma_1} (s), \overline{\sigma_2} (t) )$ l'angle en $\overline{x}$ du triangle de comparaison $\overline{\Delta}_{s,t}$. Par l'inégalité $\cat(0) $,  $\overline{\angle}_{\overline{x}} (\overline{\sigma_1} (s), \overline{\sigma_2} (t) )$ est une fonction qui décroit avec $s,t$, donc on peut définir 
	\begin{equation*}
		\angle (\sigma_1, \sigma_2) = \lim_{s, t \rightarrow 0} \overline{\angle}_{\overline{x}} (\overline{\sigma_1} (s), \overline{\sigma_2} (t) ). 
	\end{equation*}
	
	Puisque $X$ est uniquement géodésique, pour tout triplet $x, y, z \in X$, il existe exactement un segment géodésique $\sigma_1$ (resp. $ \sigma_2$) de $x$ à $y$ (resp. de $x$ à $z$).  On définit donc l'angle $\angle_x (y,z) $ en $x$ entre $y$ et $z$  comme $ \angle (\sigma_1, \sigma_2)$. 
	Soit maintenant $\xi, \eta \in \bd X$, on peut étendre la notion d'angle aux points du bord. 
	D'après la proposition \ref{prop rayon geod}, il existe un unique rayon géodésique $\gamma_x^\xi : [0, \infty[ \rightarrow$ (resp. $\gamma_x^\eta$) issu de $x$ dans la classe de $\xi$ (resp. dans la classe de $\eta$). On définit alors : 
	\begin{equation*}
		\angle_x (\xi, \eta) =  \angle_x (\gamma_x^\xi, \gamma_x^\eta),
	\end{equation*}
	
	Intuitivement, si on veut mesurer la distance (angulaire) entre deux points à l'infini, on doit prendre en compte tous les points de l'espace. La \textit{métrique angulaire} $\angle : \bd X\times \bd X \rightarrow [0, \pi] $ est donc définie par 
	\begin{equation*}
		\angle (\xi, \eta) = \sup_{x \in X} \angle_x(\xi, \eta). 
	\end{equation*}
	
	\begin{rem}
		Par exemple, s'il existe une droite géodésique $\gamma$ qui relie deux points du bord $\xi, \eta \in \bd X$, alors $\angle(\xi, \eta) = \pi$, un maximum atteint en tout point de $\gamma$. 
	\end{rem}
	
	On liste ici quelques propriétés importantes de la métrique angulaire, qui sont empruntées à \cite[Proposition II. 9.5]{bridson_haefliger99}, \cite[Proposition II. 9.7]{bridson_haefliger99} et \cite[Corollary II.9.9]{bridson_haefliger99}.
	
	\begin{prop}\label{prop topo tits }
		Soit $X$ un espace $\cat$(0) complet. 
		\begin{enumerate}
			\item La fonction $(\xi, \eta) \in \bd X \times \bd X \rightarrow \angle (\xi, \eta) \in [0, \pi]$ est une métrique complète. Tout isométrie de $X$ s'étend en une isométrie de $\bd X$ pour la métrique angulaire. 
			\item \label{continuité topologie angle} L'application identité $\Id : (\bd X, \angle) \rightarrow \bd X$, où l'espace d'arrivée est muni de la topologie conique, est une application continue. 
			\item La fonction $(\xi, \eta) \in \bd X \times \bd X \rightarrow \angle (\xi, \eta) \in [0, \pi]$ est semi-continue inférieurement par rapport à la topologie conique : si $\xi_n \rightarrow \xi $ et $\eta_n \rightarrow \eta$ dans la topologie visuelle, $\liminf_{n\rightarrow \infty} \angle (\xi_n, \eta_n) \geq \angle (\xi, \eta)$. 
			\item \label{secteur plat angle} Soit $\xi, \eta \in \bd X$. Supposons qu'il existe $o \in X$ tel que $\angle(\xi, \eta) =\angle_o (\xi,\eta) < \pi$. Alors l'enveloppe convexe du triangle idéal déterminé par $\gamma^\xi_o$ et  $\gamma^\eta_o$ est isométrique à un secteur du plan euclidien borné par deux rayons qui se rencontrent à un angle $\angle(\xi, \eta)$. 
		\end{enumerate}
	\end{prop}
	
	\begin{ex}
		Pour l'espace hyperbolique $\mathbb{H}^n$, toute paire de points du bord $\xi, \eta $ est reliée par une droite géodésique. Donc $\angle$ donne la topologie discrète sur $\bd \mathbb{H}^n$. Cet exemple montre que l'application identité du \ref{continuité topologie angle} n'est en général pas un homéomorphisme. La proposition \ref{secteur plat angle} et l'exemple de $\mathbb{H}^n$ montre que la topologie angulaire donne des informations sur la présence de secteurs euclidiens dans $X$. 
	\end{ex}

	Dans la suite, il sera souvent utile de remplacer la métrique angulaire par sa distance de longueur associée. 
	\begin{Def}
		La \textit{métrique de Tits} $d_T : \bd X \times \bd X \rightarrow \mathbb{R} \cup \{\infty\}$ est la distance de longueur associée à $\angle (.,.)$. L'espace $\partial_T X := (\bd X, d_T) $ est le \textit{bord de Tits} de $X$. 
	\end{Def}
	
	Nous reviendrons dans l'article \cite{LeBars22} sur la métrique de Tits, mais donnons déjà quelques propriétés. 
	
	\begin{thm}[{\cite[Theorem II.4.11]{ballman95}}]\label{tits}
		Soit $X$ un espace $\cat$(0) complet. Alors $(\partial_T X, d_T)$ un espace $\cat $(1) complet. De plus, pour $\eta, \, \xi \in \bd X$ : 
		\begin{enumerate}
			\item s'il n'existe pas de géodésique dans  $X$ qui relie $\xi $ à $\eta$, alors $d_T(\xi, \eta) = \angle (\xi, \eta) \leq \pi$. 
			\item Réciproquement, si $\angle (\xi, \eta) < \pi$, alors il n'existe pas de géodésique dans $X$ qui relie $\xi $ à $\eta$ et il existe une unique $d_T$-géodésique de $\xi$ à $\eta $ dans $\partial_T X$. 
			\item S'il existe une géodésique $\gamma$ dans $X$  de $\xi $ à $ \eta$, alors $d_T(\xi, \eta) \geq \pi $, avec égalité si et seulement si $\gamma$ borde un demi-plat. 
			\item La métrique de Tits $d_T : \bd X \times \bd X \rightarrow \mathbb{R} \cup \{\infty\}$ est semi-continue inférieurement. 
		\end{enumerate}
	\end{thm}

	En conséquence de la proposition \ref{prop topo tits }, la topologie de Tits est plus fine que la topologie conique. En toute généralité, le bort de Tits $\partial_T X$ n'est pas compact, et sa structure peut être complexe. L'étude de la topologie de Tits d'un espace $\cat$(0) peut renseigner sur les décompositions de l'espace, ainsi que sur la connexité du bord visuel, voir \cite{papasoglu_swenson09}.

	\section{Isométries dans un espace $\cat$(0)}\label{section isom cat}
	
	Dans cette section, on discute des isométries d'un espace $\cat$(0) et en particulier des isométries de rang 1. On verra que celles-ci ont un comportement similaire aux isométries loxodromiques dans les espaces Gromov-hyperboliques.
	
	\subsection{Propriétés structurelles}

	Soit $X$ un espace $\cat$(0), et $g \in \iso(X)$ une isométrie. On définit la fonction de déplacement de $g$ par 
	\begin{eqnarray}
		\tau : 	x \in X \mapsto \tau_g(x) = d(gx, x). \nonumber
	\end{eqnarray}
	
	La \textit{longueur de translation} de $g$ est alors $|g| := \inf \{ \tau_g (x) \, | \, x \in X\}$.
	
	\begin{Def}
		On dit que $g$ est une isométrie \textit{semi-simple} si sa fonction de déplacement admet un minimum dans $X$, autrement dit si sa longueur de translation est atteinte. Si $g$ est semi-simple, on notera $\min(g)$ l'ensemble des $x \in X$ tels que $\tau_g(x) = |g| $. 
	\end{Def}
	
	Par convexité de la fonction distance dans les espaces $\cat$(0), l'ensemble $\min(g)$ d'une isométrie semi-simple est un fermé convexe non vide. 
	
	\begin{Def}
		Soit $X$ un espace $\cat$(0), et $g \in \iso(X)$ une isométrie. On dit que $g$ est 
		\begin{enumerate}
			\item \textit{elliptique} si $g$ est semi-simple et $|g| = 0$; 
			\item \textit{axiale} si $g$ est semi-simple et $|g| >0$; 
			\item \textit{parabolique} si la longueur de translation $|g|$ n'est pas atteinte dans $X$. 
		\end{enumerate}
	\end{Def}

	\begin{ex}
		\begin{itemize}
			\item Une isométrie $g$ d'un espace $\cat(0) $ complet est elliptique si et seulement si $g$  a une orbite bornée. En effet, $g$ fixe également le circumcentre de cette orbite. 
			\item Dans les espaces euclidiens et les arbres réels, toutes les isométries sont semi-simples. 
			\item On a vu que $\glnr$ agissait par isométries sur l'espace des matrices symétriques définies positives dans $\text{M}_n(\R)$. On peut montrer qu'une isométrie est semi-simple au sens précédent si et seulement si elle est semi-simple au sens classique, c'est-à-dire diagonalisable. 
		\end{itemize}
	\end{ex}
	
	Les isométries d'un espace $\cat$(0) laissent certains sous-espaces invariants, comme le montre la proposition suivante, voir \cite[Proposition II.3.3]{ballman95} et \cite[Theorem 6.8]{bridson_haefliger99}. 
	
	\begin{thm}
		Soit $X$ un espace $\cat$(0) et $g \in \iso(X)$ une isométrie. Alors 
		\begin{enumerate}
			\item $g$ est axiale si et seulement si il existe une droite géodésique $\gamma : \R \rightarrow X$ sur laquelle $g$ agit comme une translation de longueur $|g|$. Une telle droite est appelée $\emph{axe}$ de $g$. Tous les axes de $g$ sont parallèles. 
			\item Si $g$ est axiale, alors $\min(g)$ est de la forme $C \times \R$, sur lequel $g$ agit par $g(c, t) = (c, t + |g|)$. Si de plus $X$ est complet, alors $C$ est fermé convexe (non vide). 
			\item Si $g$ est une isométrie parabolique et $X$ est localement compact, il existe une fonction de Busemann $b_\xi $, $\xi \in \bd X$ qui est invariante par $g$. En d'autres termes, $g$ stabilise $\xi $ et les horosphères centrées en $\xi$. 
		\end{enumerate}
	\end{thm}
	
	\begin{rem}
		Une isométrie axiale a donc deux points fixes à l'infini $g^{+}, g^{-} \in \bd X$. L'un de ces points fixes est dit \textit{attractif} (resp. \textit{répulsif}) si pour tout $c \in C$, où $C$ est tel que dans le précédent théorème, on a $g^n(c, 0) \rightarrow g^+$ quand $n \rightarrow \infty$ (resp. $g^{-n}(c, 0) \rightarrow g^{-}$).  
	\end{rem}
	
	Dans une série de deux papiers \cite{caprace_monod1} et \cite{caprace_monod2}, Caprace et Monod s'intéressent la question suivante : étant donné un groupe agissant sur un espace $\cat$(0), quelle interaction existe-t-il entre les propriétés de l'action et la structure de l'espace ? Les auteurs démontrent de nombreux théorèmes de structures dans le cadre des actions propres et cocompactes sur des espaces propres, et on en donne seulement un aperçu. On dit qu'un espace métrique est géodésiquement complet si tout segment géodésique peut être étendu en une géodésique bi-infinie. Le premier résultat que l'on mentionne est une décomposition canonique de $X$ en facteurs irréductibles. 
	
	\begin{thm}[{\cite[Theorem 1.1]{caprace_monod2}}]
		Soit $X$ un espace $\cat$(0) propre et géodésiquement complet tel que $\iso(X)$ agit sans point fixe à l'infini. Alors $X$ admet une décomposition $\iso(X)$-équivariante de la forme 
		\begin{eqnarray}
			X = M \times \R^n \times Y, \nonumber
		\end{eqnarray}
		où $M$ est un espace symétrique de type non-compact, et tel que le groupe d'isomorphismes $\iso(Y)$ est totalement discontinu et composé d'isométries semi-simples. 
	\end{thm}  
	
	Si de plus l'action satisfait une certaine propriété à l'infini, alors on peut même caractériser l'espace $X$. 
	
	\begin{thm}[{\cite[Theorem 1.3]{caprace_monod2}}]
		Soit $X$ un espace $\cat$(0) propre et géodésiquement complet. Supposons que le stabilisateur $\stab(\xi) \leq \iso(X)$ de tout point $\xi \in \bd X$ agisse cocompactement sur $X$. Alors $X$ est isométrique à un produit d'espaces symétriques, d'immeubles euclidiens et d'arbres de Bass-Serre.
	\end{thm}
	
	Réciproquement, si l'espace $X$ satisfait certaines propriétés, alors on peut en déduire beaucoup d'informations sur le groupe agissant. On dit que l'action du groupe $G$ sur $X$ est \textit{minimale} s'il n'existe pas de sous-espace propre $X' \nsubseteq X$ qui soit convexe fermé et $G$-invariant. 
	
	\begin{thm}[{\cite[Theorem 5.7]{caprace_monod2}}]
		Soit $X \neq \R$ un espace $\cat$(0) propre irréductible, de bord de Tits $\partial_T X$ de dimension finie. Soit $G < \iso(X)$ un sous-groupe fermé agissant minimalement sur $X$ et sans point fixe à l'infini. Alors $G$ est soit totalement discontinu ou $G$ est un groupe de Lie et de centre trivial, tel que $G/G_0$ est compact (quasi-connexité). 
	\end{thm}

	\subsection{Isométries de rang 1 et contractantes}
	
	Dans cette section, on présente les isométrie de rang 1 et les dynamiques induites par ces éléments. On trouvera plus d'informations sur cette notion fondamentale dans \cite[Section III. 3]{ballman95}, et plus récemment dans \cite{caprace_fujiwara10} et dans \cite{bestvina_fujiwara09}. On rappelle qu'un demi-plat dans $X$ est l'image d'un demi-plan euclidien isométriquement plongé. 
	
	\begin{Def}
		On dit qu'une droite géodésique dans $X$ est de \textit{rang 1} si elle ne borde pas de demi-plat. Si $g$ est une isométrie axiale de $X$, on dit que $g$ est de rang 1 si aucun de ses axes ne borde de demi-plat. 
	\end{Def}
	
	Soit $X$ un espace $\cat(0) $ et soit $G$ un groupe agissant sur $X$ par isométries. Si l'action est géométrique, c'est-à-dire propre et cocompacte, on dira que $G$ est $\cat$(0), ou que le couple $(G, X)$ est $\cat$(0). On pourra dire que $G$ est un groupe de rang 1 s'il possède un élément $g \in G$ qui agit comme une isométrie de rang 1 sur $X$. En revanche, il faut garder à l'esprit qu'être de rang 1 est une propriété de l'action, et non du groupe. La question de savoir si on peut reconnaitre un groupe de rang 1 sans action donnée est assez complexe, notamment car la notion de bord d'un groupe $\cat$(0) n'est pas aussi satisfaisante que pour les groupes hyperboliques, pour lesquels le bord visuel est invariant par quasi-isométrie. Pour les groupes $\cat$(0), Croke et Kleiner montrent en effet dans \cite{croke_kleiner00} qu'il existe des groupes $\cat$(0) dont les bords visuels construits à partir de deux parties génératrices différentes ne sont pas homéomorphes entre eux. 
	
	On rappelle qu'un plat dans un espace $\cat$(0) est un espace euclidien (possiblement dégénéré en un point) isométriquement plongé. Par exemple, une droite géodésique est un plat de $X$ de dimension 1.
	
	\begin{Def}
		Soit $X$ un espace $\cat$(0) et soit $G$ un groupe. Une action $G \curvearrowright X$ par isométries sur  $X$ est \textit{non-élémentaire} si $G$ n'a pas de point fixe global $x \in \overline{X}$ et que $G$ ne stabilise aucune aucun plat de $X$. 
	\end{Def}
	
	On va montrer que dans le cas d'actions avec des éléments de rang 1, on peut donner un critère alternatif de non-élementarité. La définition qui suit est empruntée de \cite[Définition 2.1]{caprace_fujiwara10}. 
	
	\begin{Def}\label{def isom indé}
		Soit $g_1, \, g_2 \in G$ deux isométries axiales de $G$, et soit $o\in X$. Les éléments $g_1, g_2 \in G$ sont dits \textit{indépendants} si l'application 
		\begin{equation}
			\mathbb{Z} \times \mathbb{Z} \rightarrow [0, \infty) : (m,n) \mapsto d(g_1^m o, g_2^n o)
		\end{equation}
		est propre. 
	\end{Def}

	En particulier, les points fixes de deux isométries axiales indépendantes forment quatre points distincts de $\bd X$. En fait, lorsque l'action est proprement discontinue, c'est même une équivalence : deux isométries axiales sont indépendantes si et seulement si leurs points fixes sont mutuellement distincts.

	\begin{prop}[{\cite[Proposition 3.4]{caprace_fujiwara10}}]\label{non elem caprace fuj intro}
		Soit $X$ un espace $\cat $(0) propre et $G < \iso (X)$ un sous-groupe fermé qui contient un élément de rang 1 pour cette action. Alors exactement l'une des assertions suivantes est vraie : 
		\begin{enumerate}
			\item \label{alt elem intro} Le groupe $G$ a un point fixe dans $\bd X$, ou bien stabilise une droite géodésique. Dans les deux cas, il possède un sous-groupe d'indice au plus 2, d'abélianisé infini. De plus, si $X$ a un groupe d'isométries cocompact, alors $G$ est moyennable. 
			
			\item \label{alt non elem intro} Le groupe $G$ contient deux éléments de rang 1 indépendants. De plus, $G$ contient un sous-groupe libre non-abélien. 
		\end{enumerate}
	\end{prop}

	Par conséquent, l'action $G \curvearrowright X$ d'un groupe $G$ sur un espace $\cat (0)$  avec un élément de rang 1 est non-élémentaire si et seulement si l'alternative \ref{alt non elem intro} de la proposition précédente est vérifiée. 
	
	\begin{rem}
		Dans \cite[Theorem 1.1]{hamenstadt09}, Hamenstädt étend le résultat précédent en montrant que si l'alternative \ref{alt non elem intro} est vérifiée, il existe un sous-groupe libre à deux générateurs qui sont des isométries de rang 1. De plus, il est montré que l'ensemble limite $\Lambda G$, c'est-à-dire l'ensemble des points d'accumulation dans $\bd X$ d'une orbite $Go$, ne contient pas de point isolé. 
	\end{rem}
	
	On termine cette section en montrant en quoi les isométries de rang 1 ont un comportement similaire à celui des isométries loxodromiques dans un espace Gromov-hyperboliques. 
	
	\begin{Def}\label{def isom contract}
		Soit $(X,d)$ un espace $\cat$(0). Une géodésique $\gamma$ de $X$ est \textit{$C$-contractante}, avec $C >0$ si pour toute boule $B$ disjointe de $\gamma$, la projection $\pi_\gamma (B)$ de la boule $B$ sur $\gamma$ a un diamètre inférieur à $C$. Une isométrie axiale est dite \textit{$C$-contractante} pour $C>0$ si l'un de ses axes est $C$-contractant. 
	\end{Def} 
	
	Il est clair qu'une isométrie contractante ne peut posséder d'axe qui borde de demi-plat, donc est de rang 1. La réciproque est vraie pour les espaces propres.
	
	\begin{thm}[{\cite[Theorem 5.4]{bestvina_fujiwara09}}]\label{bestv fuj rank one contracting intro}
		Soit $X$ un espace $\cat (0)$ propre, $g \in G$ une isométrie axiale de $X$ et $\gamma$ un axe de $g$. Alors il existe $C >0$ tel que $\gamma$ est $C$-contractante si et seulement si $\gamma$ ne borde pas de demi-plat. 
	\end{thm}
	
	Dans les espaces Gromov-hyperboliques, toute isométrie loxodromique est contractante au sens précédent. Une seconde similarité avec ces isométries concerne la dynamique sur le bord visuel. Le résultat qui suit est une extension d'un théorème de Ballmann et Brin \cite{ballmann_brin95} sur les variétés de Hadamard, et nous l'utiliserons à bien des reprises par la suite. 
	
	\begin{thm}[{\cite[Lemma 4.4]{hamenstadt09}}] \label{thm contract dyn NS}
		Soit $X$ un espace $\cat$(0). Alors une isométrie axiale $g \in G$ est contractante si et seulement si $g$ agit avec dynamique Nord-Sud sur le bord par rapport à ses points fixes $g^+$ et $g^-$ : pour tout voisinage $V$ of $g^{-}$ et $U$ de $g^{+}$, il existe $k_0 \geq 0 $ tel que pour tout $k \geq k_0$, $g^{k} (\overline{X}  - V) \subseteq U$  et $g^{-k} (\overline{X} - U) \subseteq V$. 
	\end{thm}
	
	En vertu du théorème \ref{bestv fuj rank one contracting intro}, si l'espace ambiant est propre, les isométries de rang 1 vérifient donc une dynamique Nord-Sud sur le bord visuel.

	\subsection{Rigidité, conjectures}\label{conjectures}
	
	Dans cette section, on introduit quelques questions et conjectures dans le domaine, et qui ont connu des développements récents. Comme on le verra, ces questions sont liées à l'existence de plats, et donc au rang de l'espace $\cat$(0). 
	
	\subsubsection{Conjecture du rang}
	Soit $(X,d)$ un espace $\cat$(0). On rappelle que $X$ est géodésiquement complet si toute géodésique peut être étendue en une droite géodésique, et qu'une action géométrique d'un groupe $G$ sur $X$ est une action propre, cocompacte par isométries. Enfin, une variété est dite irréductible si elle ne se décompose pas en un produit métrique de variétés riemanniennes non triviales. Le théorème suivant est un résultat fondamental de W. Ballmann, \cite[Theorem C]{ballman95}. 
	
	\begin{thm}[Théorème de Rigidité du Rang pour les Variétés de Hadamard]
		Soit $X$ une variété de Hadamard irréductible, et soit $G$ un groupe discret qui agit géométriquement sur $X$. Alors soit $X$ est un espace symétrique de type non compact et de rang supérieur, soit $G$ contient une isométrie de rang 1. 
	\end{thm}
	
	Ce résultat a ouvert la voie d'une étude systématique de ce type de rigidité liée au rang de l'action, notamment avec les travaux de W. Ballmann, M. Brin, R. Spatzier, P. Eberlein et K. Burns (voir \cite{ballmann_brin95}, \cite{ballman_brin_eberlein}, \cite{burns_spatzier}, \cite{eberlein_heber} et \cite{ballman_burns_spatzier}). On peut formuler la conjecture de la rigidité du rang de la manière suivante. 
	
	\begin{conj}[Rigidité du rang]
		Soit $X$ un espace $\cat$(0) irréductible, localement compact et géodésiquement complet. Soit $G$ un groupe dénombrable qui agit géométriquement sur $X$. Alors soit $X$ est un espace symétrique de type non compact et de rang supérieur, soit $X$ est un immeuble affine de dimension $\geq2$, soit $G$ contient une isométrie de rang 1. 
	\end{conj}
	
	Cette conjecture a été prouvée dans certains cas particuliers, notamment : 
	\begin{itemize}
		\item si $X$ est un complexe cellulaire de dimension 2 ou 3, \cite{ballmann_brin95}, \cite{ballmann_brin00}; 
		\item si $X$ est un complexe cubique $\cat$(0) de dimension finie, \cite{caprace_sageev11}; 
		\item si $X$ est un complexe de Coxeter ou un immeuble \cite{caprace_fujiwara10}, moyennant quelques conditions supplémentaires sur l'action; 
		\item si $X$ ne contient pas de plat de dimension 3, \cite{stadler22}.
	\end{itemize}

	Selon cette conjecture, les actions de groupes sur des espaces $\cat$(0) qui contiennent des éléments de rang 1 sont \og génériques \fg{}, et une partie de mon travail de thèse a été de les étudier. 
	
	\subsubsection{Conjecture du plat périodique}
	
	Un problème ouvert et très général dans la théorie des espaces $\cat$(0) est le suivant : étant donné un groupe discret $G$ agissant géométriquement sur un espace $\cat$(0) propre, est-il vrai que $G$ est hyperbolique si et seulement si $G$ ne contient pas de sous-groupe isomorphe à $\mathbb{Z}^2$ ?
	
	\begin{thm}[{\cite[Theorem III.H.1.5]{bridson_haefliger99}}]
		Soit $G$ un groupe discret agissant géométriquement sur un espace $\cat$(0) propre. Alors $G$ est hyperbolique si et seulement si $X$ ne contient pas un plat de dimension 2. 
	\end{thm}
	
	Par conséquent, on peut reformuler la conjecture dans ces termes. Soit $G$ un groupe discret agissant géométriquement sur un espace $\cat$(0) propre, est-ce que l'existence d'un plat de dimension 2 dans $X$ implique l'existence d'un sous-groupe isomorphe à $\mathbb{Z}^2$ ? Plus généralement, est-ce que l'existence d'un plat de dimension n dans $X$ implique l'existence d'un sous-groupe isomorphe à $\mathbb{Z}^n$ ? 
	
	La réponse est vraie pour $n = 1$, \cite{swenson99}. Il y a eu des avancées plus récemment dans le cade général dans \cite{caprace_zadnik13}, mais la conjecture reste ouverte. 
	
	\subsubsection{Alternative de Tits}
	
	On termine cette section en discutant de la question de l'alternative de Tits sur un groupe. Un groupe localement compact est dit moyennable si toute action affine de $G$ sur un convexe compact contient un point fixe. On reviendra sur la moyennabilité dans la section \ref{section action moyennable}, mais rappelons simplement que cette notion est stable par extension et par sous-groupe fermé. Puisque les groupes compacts sont moyennables, tout groupe virtuellement résoluble est moyennable. En revanche, le groupe libre à deux générateurs $\mathbb{F}_2$ n'est pas moyennable. Le résultat fondamental suivant montre que pour les groupes linéaires, la question de la moyennabilité conduit à une dichotomie simple. 
	
	\begin{thm}[{Alternative de Tits, \cite{tits72}}]
		Soit $G < \gl (n, \R)$ un sous-groupe. Alors $G$ est virtuellement résoluble ou contient un sous-groupe libre non-abélien. 
	\end{thm}
	
	Il est naturel de chercher à savoir si des groupes $\cat$(0) plus généraux satisfont cette dichotomie. La question la plus générale possible est donc la suivante. Soit $G$ un groupe agissant proprement et cocompactement sur un espace $\cat$(0) : est-il vrai que $G$ est virtuellement résoluble ou possède un sous-groupe libre non-abélien ? 
	
	La réponse est positive dans les cas suivants : 
	\begin{itemize}
		\item si $X$ est un complexe cubique $\cat$(0) de dimension finie \cite{caprace_sageev11}.
		\item si $X$ est un complexe triangulaire $\cat$(0) \cite[Theorem A]{osajda_przytycki21}.
	\end{itemize}  
	
	Dans les cas précédents, la preuve se ramène à détecter la présence d'éléments de rang 1, afin de pouvoir utiliser l'argument de ping-pong qui conduit à la proposition \ref{non elem caprace fuj intro}.

	\section{Modèles hyperboliques}
	
	Dans cette section, on présente une structure combinatoire pour des espaces $\cat$(0) quelconques, qui mène à la construction de modèles hyperboliques. Ces modèles ont été construits par Petyt, Zalloum et Spriano dans \cite{petyt_spriano_zalloum22}, et s'inspirent du complexe des courbes dans le cadre des groupes modulaires de surfaces ou du graphe de contact dans les complexes cubiques $\cat$(0). Dans les deux cas, ces objets combinatoires s'avèrent être hyperboliques. En fait, les analogies entre ces théories sont multiples, et ont fait l'objet de plusieurs développements récents (\cite{petyt21}, \cite{hagen14}, \cite{genevois19}...). On les utilisera de manière essentielle sur ces objets dans le Chapitre \ref{chapter rw cat}. 
	
	\subsection{Graphe de contact sur un complexe cubique}
	
	Les modèles hyperboliques dont nous parlerons plus tard s'inspirent beaucoup du graphe de contact dans le cadre de complexes cubiques $\cat$(0), que nous rappelons maintenant. 
	
	Soit $X$ un complexe cubique $\cat$(0). 
	
	\begin{Def}
		Soit $H, K$ deux hyperplans de $X$, et soit $h, k \in \mathcal{H}(X)$ des demi-espaces associés respectivement à $H, K$. On notera par $h^\ast$ le demi-espace associé à $H$ qui correspond à changer l'orientation, donné par la proposition \ref{prop hyperplans ccc}. On dit que les hyperplans $H$ et $K$ sont 
		\begin{enumerate}
			\item \textit{fortement imbriqués} si à orientation près, $h \subsetneq k$ et si tout demi-espace $h'$ tel que $h \subseteq h' \subseteq k$ vérifie alors $h=h' $ ou $h'=k$. 
			\item \textit{transverses} si $h \cap k$, $h^\ast \cap k$, $h \cap k^\ast$ et $h^\ast \cap k^\ast$ sont toutes non vides. 
			\item \textit{fortement séparés} si $H $ et $K $ ne sont pas transverses et si aucun hyperplan n'est transverse à $H$ et $K $ simultanément. 
		\end{enumerate}
	\end{Def}
	
	La définition du graphe de contact est la suivante. 
	
	\begin{Def}
		Soit $X$ un complexe cubique $\cat$(0). Le \textit{graphe de contact} de $X$, que l'on note $\mathcal{C}X$, est le graphe dont les sommets sont les hyperplans de $X$ et tel que deux hyperplans sont reliés par une arête s'ils sont fortement imbriqués ou transverses. 
	\end{Def}
	
	Dans l'article fondamental \cite{masur_minsky99}, Masur et Minsky montrent que le complexe des courbes sur une surface orientée de type fini est $\delta$-hyperbolique. Un résultat analogue est vrai pour les complexes cubiques $\cat$(0). 
	
	\begin{thm}[{\cite[Theorem 4.1]{hagen14}}]
		Le graphe de contact $\mathcal{C}X$ associé à un complexe cubique $\cat$(0) est quasi-isométrique à un arbre. En particulier, il est Gromov-hyperbolique. 
	\end{thm}
	
	Plus récemment, A. Genevois a proposé dans \cite{genevois20} un raffinement de la notion d'espaces fortement séparés, dont s'inspireront Petyt, Zalloum et Spriano pour le cas d'un espace $\cat$(0) quelconque. 
	
	\begin{Def}
		Deux hyperplans disjoints $H, K$ sont $L$-fortement séparés, pour $L \geq 0$ si toute chaine d'hyperplans qui est transverse simultanément à $H$ et $K$ est de cardinal $\leq L$. 
	\end{Def}
	
	En utilisant des hyperplans fortement séparés, Genevois donne une caractérisation utile pour déterminer si une isométrie d'un complexe cubique $\cat$(0) est contractante. Soit $H_1, H_2$ des hyperplans de $X$, et $h_1, h_2$ des demi-espaces associés à $H_1$, $H_2$ respectivement. On dit que $g$ \textit{traverse} la paire $H_1,H_2$ si, quitte à changer $h_i $ par $h_i^\ast$, il existe $n \in \mathbb{Z} \setminus \{0\}$ tel que $g^n h_1 \subsetneq h_2 \subsetneq h_1$. 
	
	\begin{thm}[{\cite[Theorem 1.3]{genevois20}}]
		Soit $X$ un complexe cubique $\cat$(0) et $g$ une isométrie de $X$. Alors $g $ est contractante si et seulement $g$ traverse une paire d'hyperplans $H_1, H_2$ qui sont $L$-fortement séparés. 
	\end{thm}
	
	\subsubsection{Bord du graphe de contact}
	
	Dans la suite, on étudiera les relations entre les différents bords d'un espace $\cat$(0). Le résultat suivant montre que la définition du graphe de contact \og se comporte bien à l'infini \fg{} par rapport à la topologie conique. 
	
	\begin{thm}[{\cite[Theorem 3.7]{fernos_lecureux_matheus21}}]\label{bord contact}
		Soit $X$ un complexe cubique $\cat$(0). Il existe un plongement homéomorphe $\iso(X)$-équivariant $\iota : \partial \mathcal{C}X \rightarrow \bd X$ du bord de Gromov du graphe de contact vers le bord visuel (en tant qu'espace $\cat$(0)) de $X$. 
	\end{thm}

	\subsection{Modèles hyperboliques pour un espace $\cat$(0)}\label{section modèle hyp}

	Dans cette section, on présente quelques idées de \cite{petyt_spriano_zalloum22}. En particulier, on revoit la construction grâce à laquelle on peut attacher une famille de métriques $X_L = (X, d_L)_L$ sur un espace $\cat$(0) qui d'une part sont hyperboliques, et d'autre part capturent correctement le comportement des isométries contractantes de l'espace.

	\begin{Def}
		Soit $X$ un espace $\cat(0) $, et soit $\gamma: I \rightarrow X$ une géodésique. Soit $\pi_\gamma$ la projection sur (l'image de) $\gamma$ caractérisée par la proposition \ref{prop projection cat}. Soit $t \in I$ tel que $[t - \frac{1}{2} , t + \frac{1}{2}] \subseteq I$. Alors le \textit{rideau} (ou \textit{hyperplan}) dual à $\gamma$ en $t$ est
		\begin{eqnarray}
			h = h_{\gamma, t} = \pi^{-1}_\gamma (\gamma ([t - \frac{1}{2} , t + \frac{1}{2}])). \nonumber
		\end{eqnarray}
		Le \textit{pôle} du rideau $h:= h_{\gamma, t}$ est $ \gamma ([t - \frac{1}{2} , t + \frac{1}{2}]) $. En empruntant à la terminologie pour les complexes cubiques, on appellera $h^{-} =  \pi^{-1}_\gamma (\gamma ((-\infty, t- \frac{1}{2})\cap I))$ et $h^{+} =  \pi^{-1}_\gamma (\gamma ((t +\frac{1}{2}, + \infty )\cap I))$ les \textit{demi-espaces} déterminés par $h$.  	
	\end{Def}
	
	Il est clair par les propriétés de projection que $\{h^{-} , h , h^{+}\}$ forme une partition de $X$. Si $A \subseteq h^{-}$ et $B \subseteq h^{+}$ sont des sous-ensembles de $X$, on dira que $h$ \textit{sépare} $A$ et $B$.
	
	Dans la suite, on dénotera souvent un hyperplan par $h$, même s'il faut garder en tête que $h$ est caractérisé par une géodésique et par un pôle $h = h_{\gamma, t}$. 
	
	\begin{rem}\label{remark curtains thick intro}
		Par la proposition \ref{prop projection cat}, un rideau $h$ est fermé dans $X$, et puisque la projection n'augmente pas les distances, il est automatiquement épais : $d(h^{-}, h^{+})= 1$. 
	\end{rem} 
	
	En revanche, une différence avec les hyperplans dans les complexes cubiques est qu'un rideau peut ne pas être convexe : si $x, y \in h^{-}$, il peut arriver qu'il existe $z \in [x, y ]\cap h^{+}$. Il existe en revanche une notion plus faible de convexité.
	
	\begin{prop}[{\cite[Lemma 2.6]{papasoglu_swenson09}}]\label{star convexity}
		Soit $h$ un rideau dual à $\gamma$ et $P\subseteq \gamma$ son pôle. alors pour tout $x \in h$, le segment géodésique $[x, \pi_P (x)] $ est contenu dans $ h$. 
	\end{prop}
	
	\begin{Def}
		Une famille de rideaux $\{h_i\}$ est une \textit{chaine} si pour tout $i$, $h_i $ sépare $h_{i-1}$ et $h_{i+1}$. On dira qu'une chaine est \textit{duale} à une géodésique si tout hyperplan de la chaine est dual à cette géodésique. Pour $x \neq y \in X$, on définit maintenant
		\begin{eqnarray}
			d_\infty (x, y) = 1 + \max \{\, |c| \, : \, c \text{ est une chaine de rideaux qui sépare }x \text{ et } y \}. \nonumber
		\end{eqnarray}
	\end{Def}

	Puisque la projection n'augmente pas les distances (proposition \ref{prop projection cat}), tout rideau $h$ est épais, c'est-à-dire $d(h^{-}, h^{+}) = 1$. Par conséquent, il est immédiat que pour $x,y \in X$, $d_\infty (x, y ) \leq \lceil d(x,y ) \rceil$. Réciproquement, on peut montrer que $d$ et $d_\infty$ diffèrent d'au plus 1. 
	
	\begin{lem}[{\cite[Lemma 2.10]{petyt_spriano_zalloum22}}]
		Soit  $x, y \in X$. Alors il existe une chaine $c$, duale à $[x,y]$ telle que $d_\infty (x, y) = 1 + |c| $. 
	\end{lem}
	
	La définition suivante est l'analogue de la $L$-séparation pour les hyperplans dans les complexes cubiques $\cat$(0).

	\begin{Def}[$L$-separation]
		Soit $L \in \mathbb{N}$, on dit que deux rideaux $h_1, h_2$ sont $L$\textit{-séparés} si toute chaine transverse simultanément à $h_1$ et $h_2$ est de cardinal au plus $L$. Une chaine de rideaux $L$-séparés est appelée une $L$\textit{-chaine}. 
	\end{Def}
	
	Grâce à la $L$-séparation, on peut construire une nouvelle métrique sur $X$.  
	
	\begin{Def}
		Soit $x\neq y \in X$ deux points distincts dans $X$, et $L \in \mathbb{N}$, on définit 
		\begin{eqnarray}
			d_L(x,y) = 1 + \max\{|c| \, : \, c \text{ est une  } L\text{-chaine de rideaux qui sépare }x \text{ et } y\}, \nonumber
		\end{eqnarray}
		et on pose $d_L(x,x)= 0 $ pour tout $x \in X$. 
	\end{Def}

	Pour tout $L \geq 0$, $d_L$ définit une métrique sur $X$ \cite[Lemma 2.17]{petyt_spriano_zalloum22}. Dans la suite, on écrira $X_L = (X, d_L)$ l'espace métrique correspondant. 
	
	Le lemme qui suit est clef dans la preuve du théorème \ref{theorem hyperbolic models intro}, affirmant que les espaces modèles $X_L$ sont hyperboliques. Il sera utilisé plusieurs fois dans la suite, et montre que la $L$-séparation induit de bonnes propriétés de Morse. 
	
	\begin{lem}[{\cite[Lemma 2.14]{petyt_spriano_zalloum22}}]\label{lem bottleneck intro}
		Soit $A$, $B$ deux ensembles séparés par une L-chaine $\{h_1, h_2, h_3\}$ duale à une géodesique $\gamma= [x_1, y_1]$, telle que $x_1 \in A$ et $y_1 \in B$. Alors, pour tout $x_2 \in A$, $y_2 \in B$, si $p \in h_2 \cap [x_2, y_2]$, on a $d(p, \pi_\gamma(p))\leq 2L + 1$. 
	\end{lem}
	
	\begin{rem}
		Soit $x_2, x_3 \in A $ sont des points tels que dans le lemme \ref{lem bottleneck intro}. Prenons $z_2 \in [x_2, y_1 ]\cap h_2$ et $z_3 \in [x_3, y_1 ]\cap h_2$. Puisque le pôle de $h_2$ est de diamètre 1, l'inégalité triangulaire implique que $d(z_2, z_3) \leq 4L+3$. 
	\end{rem}
	
	Pour illustrer l'idée sous-jacente, rappelons le critère \og bottleneck \fg{} (ou du goulot) de Manning \cite{manning05}, voir la figure \ref{figure bottleneck}.
	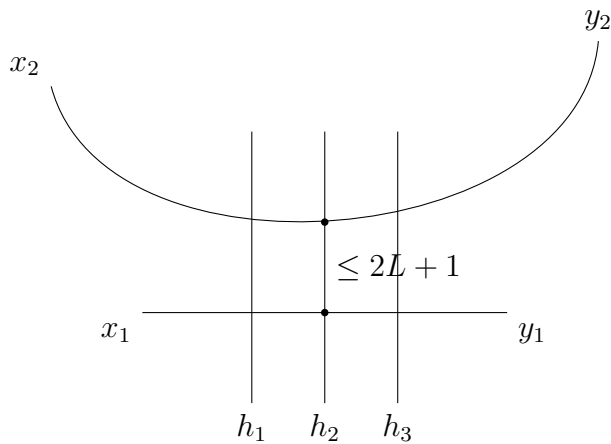
\begin{figure}
		\centering
		\begin{center}
			\begin{tikzpicture}[scale=1.2]
				\draw (0,0) -- (4,0)  ;
				\draw (2.8,2) -- (2.8,-1)  ;
				\draw (1.2,2) -- (1.2,-1)  ;
				\draw (2,2) -- (2,-1)  ;
				\draw (2,0.5) node[right]{$\leq 2L+1$} ;
				\draw (-1, 2.5) to[bend right = 80] (5, 3);
				\draw (0,0) node[below left]{$x_1$} ;
				\draw (2, -1) node[below]{$h_2$} ;
				\draw (1.2, -1) node[below]{$h_1$} ;
				\draw (2.8, -1) node[below]{$h_3$} ;
				\draw (-1, 2.5) node[above left]{$x_2$} ;
				\draw (4,0) node[below right]{$y_1$} ;	
				\draw (5,3) node[above]{$y_2$};
				\filldraw[black] (2,0) circle(1pt);
				\filldraw[black] (2,1) circle(1pt);
			\end{tikzpicture}
		\end{center}
		\caption{Illustration du lemme \og du goulot \fg{} \ref{lem bottleneck intro}.}\label{figure bottleneck}
	\end{figure}

	\begin{thm}[Critère  \og du goulot \fg{} de Manning]
		Soit $Y, d$ un espace métrique géodésique. Alors $Y$ est quasi-isométrique à un arbre (et donc hyperbolique) si et seulement s'il existe $\delta >0$ tel que, pour tout point $x, y\in Y$, il existe un point $m \in Y$ tel que $d(x,m) = d(m,y) = \frac{1}{2}d(x,y)$, et si tout chemin qui relie $x$ à $y$ contient un point à distance $\delta$ de $m$. 
	\end{thm}
	
	En référence à ce critère, on se réfèrera dans toute la suite au lemme \ref{lem bottleneck intro} comme le lemme \textit{du goulot}. 
	
	L'un des résultats fondamentaux de \cite{petyt_spriano_zalloum22} est l'hyperbolicité des modèles $(X, d_L)$. 
	
	\begin{thm}[{\cite[Theorem 3.9]{petyt_spriano_zalloum22}}] \label{theorem hyperbolic models intro}
		Soit $X$ un espace $\cat$(0), et soit $L$ un entier. Alors l'espace $(X, d_L)$ est un espace Gromov-hyperbolique presque géodésique, dont la constante d'hyperbolicité dépend uniquement de $L$. De plus, toute isométrie $g\in \iso(X)$ est également une isométrie de $(X, d_L)$. 
	\end{thm}
	
	On appellera donc $(X, d_L)$ un modèle hyperbolique pour l'espace $X$. L'utilisation de ces modèles est particulièrement intéressante grâce au théorème suivant, qui correspond à la situation des complexes cubiques $\cat$(0). On rappelle qu'une isométrie $g$ d'un espace hyperbolique $(X,d) $ est dite \textit{loxodromique} si pour tout $x \in X$, l'application orbitale
	\begin{eqnarray}
		n \in \mathbb{Z} \mapsto g^n x \in X \nonumber
	\end{eqnarray}
	est une quasi-isométrie. De manière équivalente, $g$ est loxodromique si c'est une isométrie semi-simple axiale de $X$. 
	
	\begin{thm}[{\cite[Theorem 4.9]{petyt_spriano_zalloum22}}] \label{thm contract loxodromic}
		Soit $X$ un espace $\cat$(0) et soit $g$ une isométrie semi-simple de $X$. Les assertions suivantes sont équivalentes : 
		\begin{enumerate}
			\item $g$ est une isométrie contractante de $X$;
			\item il existe $L \in \mathbb{N}$ tel que $g$ est une isométrie loxodromique de $X_L$.
			\item il existe $L$ tel que $g$ traverse deux hyperplans $L$-séparés. 
		\end{enumerate} 
	\end{thm} 
	
	Par la suite, cette correspondance nous permettra de voir qu'une action de groupe sur un espace $\cat$(0) avec des éléments contractants induit une action non-élémentaire sur un espace hyperbolique.

	\subsection{Actions acylindriquement hyperboliques}\label{section acylindrique cat}
	
	Beaucoup de résultats récents en théorie géométrique des groupes exploitent les bonnes propriétés qui apparaissent pour certaines actions sur des espaces hyperboliques. Parmi elles, l'existence d'une action non-élémentaire et acylindrique sur un espace hyperbolique a de nombreuses implications, et la recherche est très active dans ce domaine, \cite{bestvina_bromberg_fujiwara15}, \cite{sisto18}, \cite{osin2016}, \cite{dahmani_guirardel_osin17}. Dans \cite{bestvina_fujiwara09}, Bestvina et Fujiwara introduisent la notion d'élément faiblement proprement discontinu (WPD). Les éléments WPD sont des isométries d'un espace métrique qui satisfont une version faible d'acylindricité, et qui apparaitront plus tard quand nous parlerons de bord de Poisson-Furstenberg d'un espace hyperbolique. La définition qui suit est due à Bowditch \cite{bowditch08}.
	
	\begin{Def}
		L'action d'un groupe $G$ sur un espace métrique $(X,d)$ est \textit{acylindrique} si pour tout $\varepsilon >0$, il existe $R,N >0$ tels que pour tous $x,y\in X$ avec $d(x,y ) \geq R$, il existe au plus $N$ éléments de $G$ qui satisfont 
		\begin{eqnarray}
			d(x, gx) \leq \varepsilon \text{ et } d(y,gy) \leq \varepsilon \nonumber. 
		\end{eqnarray}
	\end{Def} 
	
	L'existence d'une action acylindrique sur un espace hyperbolique a de fortes conséquences de structure sur le groupe, comme le montre le résultat suivant. 
	
	\begin{thm}[{\cite[Theorem 1.1]{osin2016}}]
		Soit $G$ un groupe agissant de manière acylindrique sur un espace hyperbolique. Alors $G$ satisfait exactement l'une des trois propriétés suivantes. 
		\begin{enumerate}
			\item$G$ a des orbites bornées. 
			\item $G$ est virtuellement cyclique et a un élément loxodromique. 
			\item\label{thm acyl non eleme} $G$ contient une infinité d'éléments loxodromiques indépendants. 
		\end{enumerate}
	\end{thm}
	
	En d'autres termes, l'action de $G$ ne peut pas être parabolique ou quasi-parabolique. Remarquons que si l'action est acylindrique et non-élémentaire, alors seule la possibilité \ref{thm acyl non eleme} est permise. 
	
	Il faut comprendre un élément WPD comme une isométrie qui engendre une action acylindrique \og dans la direction de son quasi-axe\fg. 
	
	\begin{Def}
		Soit $G$ un groupe agissant sur un espace métrique $(X,d)$, et $g \in G$. Alors $g$ agit de manière \textit{faiblement proprement discontinue}, ou est un \textit{élément WPD}, si pour tout $\varepsilon >0$, et pour tout $x\in X$, il existe $M \in \mathbb{N}$ tel que 
		\begin{eqnarray}
			\# \{h \in G \; | \; d(hx, x) \leq \varepsilon \text{ et } d(hg^Mx, g^Mx)\leq \varepsilon\} < \infty. \nonumber
		\end{eqnarray}
	\end{Def}
	
	Il y a un lien entre action acylindrique et présence d'un élément loxodromique WPD. 
	
	\begin{thm}[{\cite[Theorem 1.2]{osin2016}}]
		Soit $G$ un groupe. Alors $G$ admet une action acylindrique non-élémentaire sur un espace hyperbolique si et seulement si $G$ n'est pas virtuellement cyclique et admet une action sur un espace hyperbolique tel qu'il existe un élément loxodromique $WPD$ pour cette action. 
	\end{thm}
	
	Si $G$ satisfait les conditions équivalentes du théorème précédent, on dit que $G$ est un groupe \textit{acylindriquement hyperbolique}. 
	
	\subsection{Bords des modèles hyperboliques}
	On conclut cette partie en discutant de la relation entre le bord des modèles hyperboliques, et la géométrie à l'infini de l'espace $\cat$(0) original. 
	
	\begin{Def}
		On dit qu'un rayon géodésique $\gamma : [0, \infty[ \rightarrow X$ \textit{traverse} un rideau $h $ s'il existe $t_0 \in [0, \infty[$ tel que  $h$ sépare $\gamma(0) $ et $\gamma ([t_0, \infty[)$. De même, on dit qu'une droite géodésique $\gamma: \mathbb{R} \rightarrow X$ traverse un hyperplan $h$ s'il existe $t_1, t_2 \in \mathbb{R}$ tels que $h$ sépare $\gamma (]-\infty, t_1])$ et $\gamma ([t_2, \infty[)$. On dit que $\gamma $ traverse une chaine de rideaux $c = \{h_i\}$ si elle traverse chaque rideau $h_i$ individuellement. 
	\end{Def}

	\begin{rem}
		On peut montrer que si deux rayons géodésiques issus du même point traversent la même $L$-chaine $c$, et que $c$ est infinie, alors les rayons sont asymptotique (et donc égaux). Cela provient notamment du lemme \og du goulot \fg, et de \cite[Lemma 2.21]{petyt_spriano_zalloum22} 
	\end{rem}
	
	Soit $o \in X$ un point-base, on définit $\mathcal{B}_L $ comme le sous-espace de $\bd X$ qui consiste en tous les rayons géodésiques $\gamma : [0, \infty[ \rightarrow X$ issus de $o$ et tels qu'il existe une $L$-chaine infinie qui est traversée par $\gamma$. Le résultat suivant est l'analogue du théorème \ref{bord contact} dans le cas des espaces $\cat$(0) propres. 
	
	\begin{thm}[{\cite[Theorem 8.1]{petyt_spriano_zalloum22}}] \label{equivariant embedding of boundaries}
		Soit $X$ un espace $\cat$(0) propre. Alors, pour tout $L \in \mathbb{N}^\ast$, l'application identité $\iota : X \longrightarrow X_L$ induit un homéomorphisme $\iso(X)$-équivariant $\partial_L : \mathcal{B}_L \longrightarrow \bdg X_L$. De plus, les points de $\mathcal{B}_L$ sont des points de visibilité de $X$. 
	\end{thm}
	
	Dans la Section \ref{section rw cat}, on étendra ce résultat à tout espace $\cat$(0) complet. 
	
	\chapter{Immeubles affines}\label{chapter immeubles}
	
	On fait dans ce chapitre une brève introduction à la théorie des immeubles, dans le but d'arriver rapidement aux immeubles affines, une classe particulièrement importante d'espaces $\cat$(0). La théorie des immeubles est très riche et les approches multiples, dont nous ne pourrons donner qu'un rapide aperçu. Du point de vue de la théorie des groupes, les immeubles affines sont en analogie avec les espaces symétriques, en ce qu'ils permettent d'étudier les groupes algébriques semi-simples sur des corps non-archimédiens. Nous adoptons ici l'approche des immeubles par les complexes de chambres, tels qu'il ont été originellement introduits par J.~Tits \cite{tits74}. 
	Une approche plus moderne est celle des immeubles en tant que $W$-espaces métriques, où $W$ est un groupe de Coxeter, et nous en donnerons un aperçu. Les références classiques pour cette section sont \cite{garrett97}, \cite{abramenko_brown08}, \cite{weiss09} et \cite{ronan09}. 
	\newline
	\vspace{2em}
	\minitoc
	
	\vspace{1cm}

	\section{Groupes de Coxeter}

	Cette section est consacrée à l'introduction des groupes de Coxeter. Les groupes et systèmes de Coxeter sont l'objet d'une littérature abondante. On en donne une introduction très brève afin d'aller le plus rapidement possible aux immeubles. Une référence très complète est \cite{davis15}.

	\subsection{Systèmes de Coxeter}
	
	\begin{Def}
		Un \textit{système de Coxeter} est un groupe $W$ associé à un système de générateurs $S$, et de présentation 
		\begin{eqnarray}
			W = \langle s \in S \, | \, (s t)^{m_{s,t}} = 1, \ s,t \in S \rangle \nonumber, 
		\end{eqnarray}
		où $m_{s,t} \in \mathbb{N}^\ast \cup \{\infty\}$, et $m_{s,s} = 1$ pour $s \in S$. On dira également que $(W,S)$ est un groupe de Coxeter, ou plus simplement $W$ si le système de générateurs est implicite. Un groupe \textit{diédral} est un système de Coxeter à deux générateurs. 
	\end{Def}
	
	Soit $(W,S)$ un système de Coxeter. Le diagramme de Coxeter associé à $(W, S)$ est le complexe simplicial de dimension 1 dont les sommets sont les éléments de $S$ et tel qu'il existe une arête entre deux sommets $s,t \in S$ si et seulement si $m_{s,t} >2$, auquel cas on étiquette l'arête entre $s$ et $t$ par le nombre $m_{s,t}$. Il est commun de ne pas étiqueter l'arête entre $s$ et $t$ si $m_{s,t}= 3$. Ces diagrammes donnent une manière schématique souvent pratique pour représenter un groupe de Coxeter. 
	
	\begin{ex}[Groupes diédraux finis]
		Soit $L$ une droite dans $\R^2$, et soit $L'$ une seconde droite qui coupe la première en un angle $\pi /m  \in[0, \pi / 2]$. Soit $s_L$ la réflexion par rapport à $L$, et $s_{L'}$ la réflexion par rapport à $L'$. Alors $s_L \circ s_{L'}$ est une rotation d'angle $2 \pi /m$. Le groupe engendré par $s_L, s_{L'}$ est noté $\mathbb{D}_m$. Ce sous-groupe de $O(2)$ est un groupe de Coxeter, pour la présentation 
		\begin{eqnarray}
			\mathbb{D}_m = \langle s,t \, | \, s^2= t^2= (s t)^{2m} = 1 \rangle.\nonumber
		\end{eqnarray}
		Son diagramme est représenté dans la figure \ref{fig diag Dm}
	\end{ex}
	
	\begin{ex}[Groupe diédral infini]
		Le groupe diédral infini $\mathbb{D}_\infty$ peut être construit de la manière suivante. Soit $r(t) = -t$ et $s(t) = 2 - t$ deux isométries affines de $\R$. Alors $\mathbb{D}_\infty$ est le groupe engendré par $s$ et $r$. 
	\end{ex}

	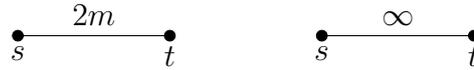
\begin{figure}[!]
		\centering
		\begin{center}
			\begin{tikzpicture}[scale=1]
				\draw (0,0) -- (2,0)  ;
				\draw (4,0) -- (6,0)  ;
				\draw (0,0) node[below]{$s$} ;
				\draw (2,0) node[below]{$t$} ;
				\draw (4,0) node[below]{$s$} ;
				\draw (6,0) node[below]{$t$} ;
				\draw (1,0) node[above]{$2m$} ;
				\draw (5,0) node[above]{$\infty$} ;
				\filldraw[black] (0,0)circle(2pt);
				\filldraw[black] (2,0)circle(2pt);
				\filldraw[black] (4,0)circle(2pt);
				\filldraw[black] (6,0)circle(2pt);
			\end{tikzpicture}
		\end{center}
		\caption{Diagrammes de Coxeter de $\mathbb{D}_m $ et $\mathbb{D}_\infty$}\label{fig diag Dm}
	\end{figure}
	
	Beaucoup de groupes de Coxeter peuvent être vus comme des groupes de pavages de l'espace euclidien $\mathbb{E}^n$, de la sphère $\mathbb{S}^n$ ou de l'espace hyperbolique $\mathbb{H}^n$, voir \cite[Theorem 6.4.3]{davis15} pour une construction à partir d'un polytope donné. Un système de Coxeter $(W, S)$ est dit \textit{sphérique} si c'est un groupe de pavage de la sphère, ou de manière équivalente, si le groupe $W$ est fini \cite[Theorem 6.12.9]{davis15}

	\subsection{Groupes de réflexions affines}\label{réal geom canonique cox}
	
	Dans cette section, on présente brièvement la notion de système de Coxeter affine ainsi que leur réalisation géométrique, qui permettra de les voir (ainsi que plus tard les immeubles) comme des espaces métriques $\cat$(0). 
	
	Soit $V$ un espace vectoriel euclidien de dimension finie. Une \textit{réflexion} $s $ sur $V$ est une isométrie qui stabilise un hyperplan affine $H_s$ de $V$. Réciproquement, à tout hyperplan $H$ on peut associer une réflexion $s_H$ qui stabilise $H$. Une telle réflexion $s$ est la composée d'une partie affine (qui correspond à une translation) et d'une partie linéaire, qu'on appelle la \textit{partie vectorielle} de $s$. 
	
	On dit que $W$ est un \textit{groupe de réflexions} sur $V$ si $W$ est engendré par des réflexions affines et tel que le groupe $\overline{W} < \gl(V)$ engendré par les parties vectorielles des éléments de $W$ est un groupe fini. Le groupe $W$ est dit \textit{linéaire} si $W = \overline{W}$. Un \textit{mur} de $W$ est un hyperplan fixé par une réflexion de $W$. L'ensemble des murs est stabilisé par l'action de $W$. Puisque $\overline{W}$ est fini, les murs vectoriels, c'est-à-dire les hyperplans (linéaires) associés aux parties vectorielles des réflexions de $W$ sont en nombre fini. 
	
	\begin{Def}
		Une composante connexe du complémentaire dans $V$ de la réunion des murs vectoriels est une \textit{chambre de Weyl vectorielle} (ouverte) de $V$ par rapport à $W$. C'est un cône polyédral convexe. Une facette d'une chambre de Weyl vectorielle est une \textit{facette vectorielle} de $V$. 
	\end{Def}
	
	Tout groupe de réflexions affine $(W, V)$ peut s'écrire $W = W_v\ltimes T$, où $T$ est un groupe de translations de $V$ engendré par des vecteurs orthogonaux aux murs de $\overline{W}$, et $W_v$ est un groupe de réflexion fini qui fixe $v\in V$ et isomorphe à $\overline{W}$. On appellera donc \textit{chambres de Weyl} (resp. \textit{murs, faces}) \textit{affines} les translatés des chambres de Weyl (resp. murs, faces) vectorielles par ce groupe de translations affines. 
	
	\begin{Def}
		Un groupe de réflexion $(W,V)$ est dit \textit{affine discret} (ou \textit{simplicial}) si le groupe de translations $T$ est un sous-groupe discret de $V$, et \textit{affine non-discret} (ou dense) si $T$ est un sous-groupe dense de $V$. On dit qu'il est \textit{irréductible} si $W$ agit de manière irréductible sur $V$, i.e. ne stabilise pas de sous-espace propre. 
	\end{Def}
	
	Un groupe de réflexions affine discret est la réalisation concrète d'un groupe de Coxeter $(W,S)$, comme on le montrera dans la section suivante. Un groupe de Coxeter obtenu de la sorte est dit \textit{affine}. Dans toute la suite, sauf mention contraire, un groupe de Coxeter affine est donc supposé discret. 
	\newline
	
	\begin{ex}
		On trouvera dans la figure \ref{figure diag classiques} des exemples classiques de systèmes de Coxeter irréductibles sphériques ($A_n, B_n, D_n$) et affines discrets ($\tilde{A}_n, \tilde{B}_n, \tilde{C}_n, \tilde{D}_n$). Mis à part un nombre fini d'exceptions, il s'agit de tous les systèmes de Coxeter irréductibles sphériques et affines discrets, voir \cite{davis15}. 
		
		\begin{figure}[!]
			\centering
			\begin{center}
				\begin{tikzpicture}[scale=0.7]
					\draw (0,0) -- (1,0) -- (2,0);
					\path (2,0) -- node[auto=false]{\ldots} (3,0);
					\draw (3, 0) -- (4,0);
					\draw (-1,0) node[left]{$A_n:$} ;
					\filldraw[black] (0,0)circle(2pt);
					\filldraw[black] (1,0)circle(2pt);
					\filldraw[black] (2,0)circle(2pt);
					\filldraw[black] (3,0)circle(2pt);
					\filldraw[black] (4,0)circle(2pt);
					
					\draw (7,0) -- (8,0) -- (9,0);
					\path (9,0) -- node[auto=false]{\ldots} (10,0);
					\draw (7,-1) -- (8,-1) -- (9,-1);
					\draw (7, 0) -- (7, -1);
					\draw (10, 0) -- (11,0) -- (11, -1) -- (10, -1); 
					\path (9,-1) -- node[auto=false]{\ldots} (10,-1);
					\draw (6.5,0) node[left]{$\tilde{A_n} :$} ;
					\filldraw[black] (7,0)circle(2pt);
					\filldraw[black] (8,0)circle(2pt);
					\filldraw[black] (9,0)circle(2pt);
					\filldraw[black] (10,0)circle(2pt);
					\filldraw[black] (11,0)circle(2pt);
					\filldraw[black] (11,-1)circle(2pt);
					\filldraw[black] (10, -1)circle(2pt);
					\filldraw[black] (9, -1)circle(2pt);
					\filldraw[black] (7,-1)circle(2pt);
					\filldraw[black] (8, -1)circle(2pt);
					
					\draw (0,-2) -- (1,-2) -- (2,-2);
					\path (2,-2) -- node[auto=false]{\ldots} (3,-2);
					\draw (3, -2) -- (4,-2);
					\draw (-1,-2) node[left]{$B_n: $ } ;
					\filldraw[black] (0,-2)circle(2pt);
					\filldraw[black] (1,-2)circle(2pt);
					\filldraw[black] (2,-2)circle(2pt);
					\filldraw[black] (3,-2)circle(2pt);
					\filldraw[black] (4,-2)circle(2pt);
					\draw (3.5, -2) node[above] {$4$} ;
					
					\draw (7,-2) -- (8,-2) -- (9,-2);
					\path (9,-2) -- node[auto=false]{\ldots} (10,-2);
					\draw (8,-2) -- (8, -3);
					\draw (10, -2) -- (11,-2) -- (12, -2); 
					\path (9,-1) -- node[auto=false]{\ldots} (10,-1);
					\draw (6.5,-2) node[left]{$\tilde{B_n} :$} ;
					\filldraw[black] (7,-2)circle(2pt);
					\filldraw[black] (8,-2)circle(2pt);
					\filldraw[black] (9,-2)circle(2pt);
					\filldraw[black] (10,-2)circle(2pt);
					\filldraw[black] (11,-2)circle(2pt);
					\filldraw[black] (8,-3)circle(2pt);
					\filldraw[black] (12,-2)circle(2pt);
					\draw (11.5, -2) node[above] {$4$} ;
					
					\draw (0,-4) -- (1,-4) -- (2, -4);
					\path (2, -4) -- node[auto=false]{\ldots} (3, -4);
					\draw (3, -4) -- (4, -4);
					\draw (1, -4) -- (1, -5);
					\draw (-1, -4) node[left]{$D_n:$} ;
					\filldraw[black] (0,-4)circle(2pt);
					\filldraw[black] (1,-4)circle(2pt);
					\filldraw[black] (2,-4)circle(2pt);
					\filldraw[black] (3,-4)circle(2pt);
					\filldraw[black] (4,-4)circle(2pt);
					\filldraw[black] (1,-5)circle(2pt);
					
					\draw (7,-4) -- (8,-4) -- (9, -4);
					\path (9, -4) -- node[auto=false]{\ldots} (10, -4);
					\draw (10,-4) -- (11,-4) -- (12, -4); 
					\filldraw[black] (7,-4)circle(2pt);
					\filldraw[black] (8,-4)circle(2pt);
					\filldraw[black] (9,-4)circle(2pt);
					\filldraw[black] (10,-4)circle(2pt);
					\filldraw[black] (11,-4)circle(2pt);
					\filldraw[black] (12,-4)circle(2pt);
					\draw (6.5, -4) node[left]{$\tilde{C_n}:$} ;
					\draw (11.5, -4) node[above] {$4$} ;
					\draw (7.5, -4) node[above] {$4$} ;
					
					\draw (7,-6) -- (8,-6) -- (9,-6);
					\path (9,-6) -- node[auto=false]{\ldots} (10,-6);
					\draw (8,-6) -- (8, -7);
					\draw (11,-6) -- (11, -7);
					\draw (10, -6) -- (11,-6) -- (12, -6); 
					\draw (6.5,-6) node[left]{$\tilde{D_n}:$} ;
					\filldraw[black] (7,-6)circle(2pt);
					\filldraw[black] (8,-6)circle(2pt);
					\filldraw[black] (9,-6)circle(2pt);
					\filldraw[black] (10,-6)circle(2pt);
					\filldraw[black] (11,-6)circle(2pt);
					\filldraw[black] (8,-7)circle(2pt);
					\filldraw[black] (11,-7)circle(2pt);
					\filldraw[black] (12,-6)circle(2pt);
				\end{tikzpicture}
			\end{center}
			\caption{Diagrammes de Coxeter de $A_n \;(n \geq 1);\;  B_n \; (n \geq 2) ; \, D_n \; (n \geq 4) ; \, \tilde{A}_n \; (n \geq 2) ; \; \tilde{B}_n \;(n \geq 3) ; \; \tilde{C}_n \; (n \geq 2) ; \, \tilde{D}_n \; (n \geq 4)$.}\label{figure diag classiques}
		\end{figure}
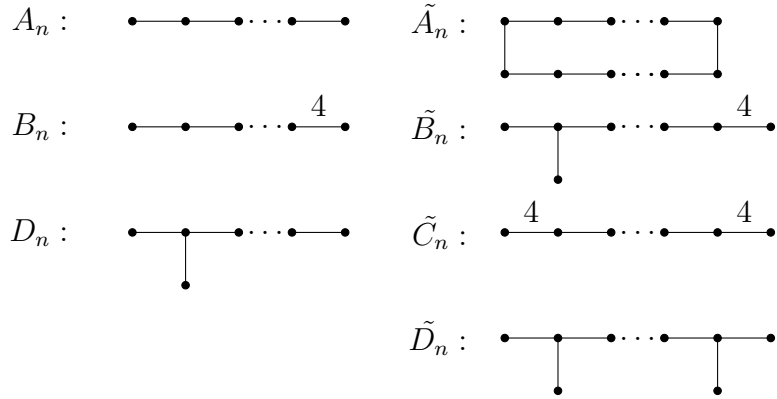
	\end{ex}

	\section{Complexes de chambres}
	
	Soit $V$ un ensemble, et $\Delta $ un sous-ensemble des parties finies de $V$. On dit que $\Delta$ est un \textit{complexe simplicial combinatoire} si c'est un ensemble partiellement ordonné muni d'une relation d'ordre $\subset$ (un poset), tel que pour tout $x \in \Delta$, si $y \subset x$, alors $y \in \Delta$. Les éléments de $\Delta$ sont les \textit{simplexes}, et ceux de $V$ sont les \textit{sommets}. La dimension du simplexe $x$ est $\# (x) - 1$, et dans toute la suite, les complexes simpliciaux seront tous de dimension finie, c'est-à-dire qu'il existe un maximum des dimensions des simplexes de $\Delta$. Si $y\subset x \in \Delta$, alors $y$ est une \textit{face} de $x$, et la codimension de $y$ dans $x$ est $\# (x \setminus y) $. Une face de codimension $1$ de $x$ est une \textit{cloison} de $x$. Un simplexe est dit \textit{maximal} s'il n'est pas lui même une face propre d'un autre simplexe de $\Delta$. On dit que deux simplexes $x$ et $y$ sont \textit{adjacents} s'ils possèdent une cloison commune. Une suite $x_0, \dots x_n$ de simplexes telle que $x_i$ et $x_{i+1}$ sont adjacents est une \textit{galerie} de $\Delta$. 
	
	\begin{Def}
		Un \textit{complexe de chambres} est un complexe simplicial tel que les simplexes maximaux sont tous de même dimension et tel que toute paire de simplexes maximaux est reliée par une galerie. Les simplexes maximaux sont appelés \textit{chambres}, et on notera l'ensemble des chambres $\ch(\Delta)$. On dit que le complexe de chambres est \textit{fin} si toute cloison est une face d'exactement deux chambres et \textit{épais} si toute cloison est une face d'au moins trois chambres. 
	\end{Def}
	
	Sur un complexe de chambres $\Delta$, on peut donner une distance entre toute paire de chambres $C,C'$ : 
	\begin{eqnarray}
		d_\Delta(C, C') = \min \{ |\Gamma|\, | \,  \Gamma \textit{ est une galerie de $C$ à $C'$}\} - 1 \nonumber. 
	\end{eqnarray}
	
	On appelle cette distance la \textit{distance de galerie}. On peut également étendre cette distance aux paires $(x, C') $, où $x$ est un simplexe et $C'$ est une chambre par 
	\begin{eqnarray}
		d_\Delta (x, C')= \min \{ d_\Delta(C,C') \, |\, C \text{ est une chambre qui contient }x\}. \nonumber
	\end{eqnarray}
	
	\subsection{Réalisation géométrique}
	
	Soit $\Delta$ un complexe simplicial sur un ensemble de sommets $V$.  
	\begin{Def}
		La \textit{réalisation géométrique} $|\Delta|$ de $\Delta$ est la famille des fonctions positives ou nulles $f : V \rightarrow \R_+$ telles que 
		\begin{eqnarray}
			\sum_{v \in  V} f(v) = 1 \nonumber,
		\end{eqnarray}
		et telles qu'il existe $x \in \Delta$ de sorte que $f(v ) \neq 0 $ implique $v \in x$. 
	\end{Def}
	
	Une topologie sur $|\Delta|$ est donnée par 
	\begin{eqnarray}
		d(f,g) =  \sup_{v \in V} | f(v ) - g(v)  |, \nonumber
	\end{eqnarray}
	de sorte que l'inclusion d'un simplexe $|x | \rightarrow |\Delta|$ est continue. En d'autres termes, pour un simplexe combinatoire $x = \{x_0, \dots, x_p\}$, la réalisation géométrique de $x$ est 
	\begin{eqnarray}
		|x | = \{ (t_0, \dots, t_p) \; | \; \sum t_i = 1, \; 0\leq t_i \leq 1\}, \nonumber
	\end{eqnarray}
	c'est-à-dire un simplexe géométrique plongé dans $\R^{p+1}$.

	\subsection{Complexe de Coxeter}
	
	Dans cette partie, on montre que pour tout système de Coxeter, il existe un complexe de chambres canonique associé. 
	
	Soit $(W, S)$ un système de Coxeter, et soit $\mathcal{P}$ le poset de tous les sous-ensembles de $W$, où l'ordre est donné par l'inclusion inverse : pour $x, y \in \mathcal{P} $, $x \geq y$ si $x \subseteq y$. Le \textit{complexe de Coxeter} $\Sigma(W, S)$ associé à $(W,S)$ est le complexe simplicial associé au sous-poset de $\mathcal{P}$ des ensembles de la forme $w \langle T \rangle $, pour $w \in W$ et $T$ un sous-ensemble de $S$. Les complexes maximaux de $\Sigma(W, S)$ sont donc les $w\langle \emptyset \rangle = \{w\}, w \in W$. 
	
	Un \textit{étiquetage} d'un complexe simplicial $\Delta$ est une application $\lambda : \Delta \rightarrow L$, où $(L, <)$ est un poset tel que $\lambda x < \lambda y $ si $x $ est une face de $y$. On dit alors que $\Delta$ est étiqueté, ou marqué. Le \textit{type} d'un simplexe $x \in \Delta$ est alors défini comme $\lambda x$. Le théorème suivant montre que les complexes de Coxeter sont des cas particuliers de complexes de chambres. 
	
	\begin{thm}[{\cite[Theorem 3.4]{garrett97}}]\label{lem unicite morphisme ch}
		Soit $(W, S)$ un système de Coxeter. Alors le complexe de Coxeter associé $\Sigma(W,S)$ est un complexe de chambres fin. De plus, $W$ agit sur  $\Sigma(W,S)$ par automorphismes, et transitivement sur l'ensemble des chambres. 
	\end{thm}

	Prenons $C, D \in \ch (\Delta)$ deux chambres de $\Delta$, et une cloison $F$ appartenant simultanément à $C$ et $D$. Cette cloison est caractérisée par la réflexion $s \in S$ qui la fixe, et telle que $C = sD$. On dit alors que \textit{le type} de $F$ est $s$, et que $C$ et $D$ sont $s$-adjacentes. On définit alors un étiquetage sur l'ensemble des cloisons du complexe $\Delta = \Sigma (W,S)$ déterminé par $S$. On peut vérifier (\cite[Lemma 3.68]{abramenko_brown08}) que cet étiquetage est préservé par l'action de $W$. Pour toute galerie $\{C = C_0,\dots , C_n=D\}$ entre deux chambres $C, D$ on peut donc associer un \textit{type} $(s_0, \dots, s_{n-1})$, de sorte que pour tout $i$, $C_i$ et $C_{i+1}$ sont $s_i$-adjacents. 
	On définit alors l'application suivante :
	\begin{eqnarray}
		\delta &:& \ch(\Delta) \times \ch (\Delta)  \rightarrow W \nonumber \\
		& & (C, D) \mapsto s_0\dots s_{n-1} \nonumber, 
	\end{eqnarray}
	où $(s_0\dots s_{n-1})$ est le type d'une galerie entre $C$ et $D$. Cette application est bien définie : elle ne dépend pas de la galerie entre $C$ et $D$, \cite[Section 3.5]{abramenko_brown08}. Cette application est la \textit{distance de Weyl} dans $\Delta$, on en dira plus dans la section \ref{section weyl}. 
	\newline 
	
	Soit $\Delta$ un complexe de chambres, $\lambda : \Delta \rightarrow I$ un étiquetage des cloisons de $\Delta$. Pour $J \subseteq I$, on dira que deux chambres $C,D$ sont \textit{$J$-équivalentes} s'il existe une galerie de type $(i_1, \dots, i_d)$, avec $i_k \in J$ pour tout $k$, qui relie $C$ à $D$. Les classes de chambres pour cette relation d'équivalences sont appelés des \textit{$J$-résidus}. Un \textit{résidu} dans $\Delta$ est un $J$-résidu pour $J \subseteq I$. 
	
	On termine cette section en introduisant la notion de projection dans un complexe de chambres. 
	
	\begin{thm}[{\cite[Theorem 3.22]{weiss03}}]\label{thm proj cox cpx}
		Soit $C$ une chambre dans un complexe de Coxeter $\Delta$, et soit $R$ un résidu de $\Delta$. Alors il existe une unique chambre $D$ dans $R$ qui minimise la distance de galerie entre $C$ et $R$. De plus, toute galerie minimale entre $C$ et une chambre $D' \in R$ passe par $D$ : 
		\begin{eqnarray}
			d_\Delta(C,D') = d_\Delta(C,D) + d_\Delta(D,D') \nonumber. 
		\end{eqnarray}
	\end{thm}
	On appelle la chambre $D$ la \textit{projection} de $C$ sur le résidu $R$, et on la note $\proj_R(C)$. 
	
	\subsection{Réalisation géométrique d'un complexe de Coxeter affine}

	Soit $(W, V)$ un groupe de réflexions affine discret sur un espace vectoriel euclidien $V$, et soit $\mathcal{H}$ l'ensemble des murs (affines) de $(W, V)$. L'ensemble de ces murs donne une réalisation géométrique du groupe de Coxeter associé $W$. En effet, $\mathcal{H}$ partitionne $V$ en un ensemble de cellules convexes qui forme un complexe cellulaire $\Sigma(W, V)$. Le vocabulaire introduit dans les sections précédentes reste le même : les cellules de dimension maximale sont appelés \textit{chambres}, et on définit de même les cloisons, facettes, etc. 
	
	Soit $C$ une chambre et $S$ l'ensemble des réflexions par rapports aux murs qui délimitent $C$. Alors $W$ agit transitivement sur $C$, $S$ engendre $W$, et $(W, S)$ est un groupe de Coxeter \cite[Section 10]{abramenko_brown08}. Le complexe simplicial $|\Sigma(W, S)|$ obtenu par ce pavage sur $V$ muni de sa distance euclidienne est alors une réalisation géométrique du groupe de Coxeter $(W, S)$ associé. 
	
	\begin{thm}[{\cite[Theorem 13.7]{garrett97}}]
		Soit $(W,S)$ et $(W', S')$ des systèmes de Coxeter affines, $|\Sigma (W,S)|$, $|\Sigma(W', S')|$ des réalisations géométriques euclidiennes de ces groupes, et supposons que les chambres dans ces réalisations géométriques ont diamètre 1. S'il existe un isomorphisme de complexes simpliciaux $\phi : \Sigma (W,S) \rightarrow \Sigma(W', S')$, alors  $|\Sigma (W,S)|$ et $|\Sigma(W', S')|$ sont canoniquement isométriques. 
	\end{thm}
	
	Il est donc naturel pour un groupe de Coxeter affine $(W,S)$ de parler de \textit{réalisation géométrique canonique} (euclidienne) $|\Sigma (W,S)|$ de $(W,S)$. 
	
	\section{Immeubles comme complexes de chambres}
	
	Dans cette section, on donne une première définition des immeubles simpliciaux. C'est la définition originale de Tits, et elle a l'avantage d'être immédiatement combinatoire et visuelle. Il est assez facile de la vérifier sur des exemples précis. 
	\subsection{Premières définitions}
	\begin{Def}
		Un \textit{immeuble} (épais) discret est un complexe de chambre $\Delta$ épais muni d'un ensemble $\mathcal{A}$ de sous-complexes de chambres appelés \textit{appartements} tel que tout appartement $A \in \mathcal{A}$ est un complexe de chambres fin et tel que 
		\begin{enumerate}
			\item Pour toute paire de simplexes $x, y \in \Delta$, il existe un appartement $A \in \mathcal{A}$ qui les contient. 
			\item Si deux appartements $A, A' \in \mathcal{A}$ contiennent tous deux un simplexe $x$ et une chambre $C$, alors il existe un isomorphisme de complexes de chambres $\phi : A \rightarrow A'$ qui fixe $x$ et $C$. \label{axiom isom appartements}
		\end{enumerate}
	\end{Def}
	
	\begin{rem}
		Il est souvent précisé dans les axiomes de la définition d'un immeuble que les appartements sont des complexes de Coxeter. Avec cette définition, on peut en fait prouver que c'est bien le cas, ce qui en fait une définition avec des hypothèses faibles \cite[Corollaire 4.3]{garrett97}. 
	\end{rem}
	
	L'axiome \ref{axiom isom appartements} peut être remplacé par la version suivante, plus forte et utile en pratique. 
	
	\begin{lem}[{\cite[Lemma 4.1]{garrett97}}]
		Soit $\Delta$ un immeuble épais avec un système d'appartements $\mathcal{A}$. Si deux appartements $A, A' \in \mathcal{A}$ contiennent tous deux une chambre $C$, alors il existe un isomorphisme de complexes de chambres $\phi : A \rightarrow A'$ qui fixe $A \cap A'$.
	\end{lem}
	
	On peut définir une distance de galerie $d_\Delta$ sur les chambres $\ch(\Delta)$ d'un immeuble. Soit $A $ un appartement de $\Delta$ et $d_A$ la distance de galerie sur $A$, vu comme complexe de chambres. En conséquence du Lemme précédent, pour toute paire de chambres $C,D \in \ch(\Delta)$ et tout appartement $A, A' \in \mathcal{A}$ qui contient $C, D$, on a  $d_A(C, D) = d_{A'} (C, D)$. On note cette valeur commune $d_\Delta (C, D)$, et on l'appelle \textit{distance de galerie} sur $\ch(\Delta)$. Plus généralement, soit $x $ un simplexe de l'immeuble $\Delta$, et soit $D$ une chambre, on peut définir la distance de galerie entre $x$ et $D$ par 
	\begin{eqnarray}
		d_\Delta (x, D) = \min \{ d_\Delta (C, D) \; | \; x \subseteq C, \; C \in \ch (\Delta)\}. \nonumber
	\end{eqnarray}
	
	La proposition suivante est fondamentale pour la suite, \cite[Proposition 4.2]{garrett97}
	
	\begin{prop}[Rétraction centrée sur une chambre]
		Soit $\Delta$ un immeuble muni d'un système d'appartements $\mathcal{A}$, et soit $A \in \mathcal{A}$. Pour toute chambre $C \in A$, il y a une rétraction  $\rho = \rho_{A, C} : \Delta \rightarrow A$ qui vérifie les propriétés suivantes : 
		\begin{enumerate}
			\item Pour tout appartement $A' \in \mathcal{A}$ qui contient $C$, $\rho_{|A'}$ est un isomorphisme de complexes de chambres tel que $\rho_{|A \cap A'} $ est l'identité. 
			\item La rétraction $\rho$ est l'unique morphisme de complexes de chambres $\Delta \rightarrow A$ qui fixe $C$ et tel que pour toute facette $x $ de $C$ et toute chambre $D$ de $\Delta$, 
			\begin{eqnarray}
				d_\Delta (x, D) = d_\Delta (x, \rho (D)). \nonumber
			\end{eqnarray}
		\end{enumerate}
	\end{prop}
	
	Le théorème suivant est dû à Tits, et énonce que les appartements d'un immeuble épais sont bien des complexes de Coxeter. 
	
	\begin{thm}[{\cite[Theorem 4.131]{abramenko_brown08}}]
		Soit $\Delta$ un immeuble épais, muni d'un système d'appartements $\mathcal{A}$. Alors il existe un système de Coxeter $(W,S)$ tel que tout appartement $A \in \mathcal{A}$ est isomorphe au complexe de Coxeter $\Sigma(W,S)$. 
	\end{thm}

	\begin{Def}
		Un immeuble épais est dit 
		\begin{itemize}
			\item \textit{sphérique} si ses appartements sont des complexes de Coxeter sphériques, c'est-à-dire finis. 
			\item \textit{affine discret} si le système de Coxeter $(W,S)$ associé à ses appartements provient d'un groupe de réflexions affine discret. 
		\end{itemize}
	\end{Def}
	
	En fait, on verra dans la section suivante qu'on peut définir des immeubles affines non discrets, au prix d'une modification des définitions. 
	
	\begin{ex}[Immeuble de type $A_n$]
		On présente ici une famille d'immeubles sphériques très étudiée, dont les complexes de Coxeter associés sont de type $A_n$. Soit $\mathbb{K}$ un corps et $V := \K^{n+1}$ un espace vectoriel sur $\K$ de dimension $n+1$. Définissons l'ensemble des drapeaux complets de sous-espaces vectoriels : 
		\begin{eqnarray}
			V_1 \subseteq V_2 \subseteq \dots V_n \subseteq V, \nonumber
		\end{eqnarray}
		où $V_i$ est de dimension $i$. On donne une structure de complexes de chambres sur cet ensemble en décrétant que deux drapeaux complets $V_1 \subseteq V_2 \subseteq \dots V_n \subseteq V$ et $V'_1 \subseteq V'_2 \subseteq \dots V'_n \subseteq V$ (les chambres) sont $i$-adjacents si et seulement si $V_i \neq V'_i $ et $V_j = V'_j$ pour tout $j \neq i $. Il est clair que le groupe $\gl(n+1, \K)$ agit de manière transitive sur les drapeaux. Soit $B< \gl(n+1, \K)$ le groupe des matrices triangulaires supérieures pour la base canonique sur $V$, et soit $G_i = \stab_G(V_i)$, pour $i = 1, \dots, n$. Remarquons que $B \subseteq G_i$ pour tout $i$. On peut définir l'immeuble $\Delta$ dont les chambres sont les classes $gB, g\in G$ de l'espace quotient $G/B$, et tel que deux chambres $gB, g'B$ sont $s$-équivalentes si et seulement si $gG_s = g'G_s$. Alors $\Delta$ est de type $A_n$, et le groupe de Weyl $W$ est le groupe de permutation $S_{n+1}$. 
	\end{ex}

	On peut montrer qu'un immeuble est sphérique si et seulement si il est de diamètre fini pour la distance de galerie. On dit que deux chambres $C, D \in \ch(\Delta)$ d'un immeuble sphérique sont \textit{opposées} si $d_\Delta (C, D) = \diam_\Delta(C,D)$. 
	
	\begin{prop}
		Soit $C, D \in \Delta$ deux chambres opposées d'un immeuble sphérique $\Delta$, et soit $A$ un appartement qui les contient toutes deux. Alors toute chambre $C' \in A$ appartient à une galerie minimale de $C$ à $D$. 
	\end{prop}

	\subsection{Distance de Weyl}\label{section weyl}
	Dans cette section, on définit la distance de Weyl, un outil très utile en pratique pour étudier les immeubles. 
	
	\begin{prop}[{\cite[Proposition 4.81]{abramenko_brown08}}]
		Soit $(W, S) $ un groupe de Coxeter, et soit $\Delta $ un immeuble de type $(W,S)$. Il existe une fonction $\delta : \ch(\Delta) \times \ch(\Delta) \rightarrow W$ avec les propriétés suivantes : 
		\begin{enumerate}
			\item Pour toute galerie minimale $\gamma = \{ C_0, \dots, C_n\}$ de type $(s_0, \dots , s_{n-1})$, l'élément $w = \delta(C_0, C_n) $ est représenté par $s_0\dots  s_{n-1}$. 
			\item Soit une paire de chambres $C,D \in \ch(\Delta)$, et soit $w : = \delta (C,D)$. Alors l'application qui à une galerie minimale $\gamma$ entre $C$ et $D$ associe son type $s(\gamma)$ donne une bijection entre les galeries minimales entre $C$ et $D$ et les mots réduits $s_0\dots  s_{n-1}, s_i \in S$ qui représentent $w$.  	
		\end{enumerate}
	\end{prop}
	
	\begin{Def}
		On appelle l'application $\delta : \ch(\Delta) \times \ch(\Delta) \rightarrow W$ la \textit{distance de Weyl} associée à $\Delta$. 
	\end{Def}

	\begin{prop}[{\cite[Proposition 4.84]{abramenko_brown08}}]
		La distance de Weyl $\delta : \ch(\Delta) \times \ch(\Delta) \rightarrow W$ a les propriétés suivantes : 
		\begin{enumerate}
			\item $\delta(C,D) = 1 $ si et seulement si $C = D$ ; 
			\item $\delta(C,D) = \delta(D,C)^{-1}$ ; 
			\item Si $\delta(C,C') = s \in S$ et $\delta(C', D) = w$, alors $\delta(C,D) = sw$ ou $w$. Si de plus $l(sw) = l(w) + 1$, alors $\delta(C, D )= sw$. 
			\item Si $\delta(C, D) = w$, alors pour tout $s \in S$, il existe une chambre $C'$ telle que $\delta(C', C) = s$ et $\delta(C', D) = sw$. Si de plus $l(sw ) = l(w) -1$, alors il existe une unique telle chambre $C'$. 
		\end{enumerate}
	\end{prop}

	En fait, la donnée d'un immeuble peut se retrouver à partir de la distance de Weyl. Soit $J \subseteq W$ une partie du groupe de Coxeter $W$. On appelle isométrie de $J$ dans les chambres du complexe de chambres $\Delta$ une application $f : J \rightarrow \ch(\Delta)$ telle que $\delta(f(w), f(w')) = w^{-1}w'$ pour tout $w,w' \in J$. 
	
	\begin{thm}[{\cite[Théorème 1]{tits86}}]\label{thm weyl dist ch cplx}
		Soit $(W,S)$ un système de Coxeter, et $\Delta$ un complexe de chambres. Alors $\Delta$ est un immeuble de type $(W,S)$ si et seulement s'il existe une application $\delta : \ch(\Delta) \times \ch(\Delta) \rightarrow W$ telle que pour tout mot réduit $w = s_1 \dots s_n$ (relativement à $S$), et toute paire de chambres $C,D$, il existe une galerie de type $(s_1, \dots, s_n)$ entre $C$ et $D$ si et seulement si $\delta(C, D) = w$. De plus, pour $J \subseteq W$ une partie du groupe de Coxeter $W$, toute isométrie $f : J \rightarrow W$ se prolonge en une isométrie globale $\tilde{f} : W \rightarrow \ch (\Delta)$, et les images de ces isométries forment un système d'appartements pour $\Delta$. 
	\end{thm}
	
	On appelle le système d'appartements défini dans le théorème \ref{thm weyl dist ch cplx} le système d'appartements complet de $\Delta$, car tout système d'appartements est contenu dans celui-ci. Dans la suite, on considère toujours un système d'appartements complet.

	On a vu grâce au théorème \ref{thm proj cox cpx} qu'il existait une application projection bien définie dans chaque appartement de l'immeuble $\Delta$. On peut étendre cette projection à tout l'immeuble : pour un simplexe $x $ et $C$ une chambre, soit $D$ la projection du simplexe $x$ sur la chambre $C$ dans n'importe quel appartement qui contient les deux. Alors par convexité des appartements, $D$ est indépendante du choix de l'appartement, et toute galerie minimale de $x$ à $C$ passe par $D$, \cite[Proposition 4.95]{abramenko_brown08}. On appelle $D$ la projection $\proj_C(x)$ de $x$ sur $C$. Cette projection peut être étendue à tout résidu de $\Delta$, \cite[Theorem 4.97]{abramenko_brown08}. 
	
	\section{BN-paires}
	
	Dans cette section, on explique comment construire de manière naturelle un immeuble sur lequel agit un groupe satisfaisant certaines conditions. Cette approche est fondamentale pour l'étude des groupes de Lie semi-simples sur des corps à valuation discrète. Cependant, nous nous intéresserons surtout aux immeubles exotiques, donc cette présentation sera assez courte. On renvoie à \cite[Chapter 6]{abramenko_brown08}, \cite[Chapter 16]{weiss09} pour une description plus complète.

	Soit $\Delta$ un immeuble épais muni d'un système d'appartements (complet) $\mathcal{A}$, $C$ une chambre et $A \in \mathcal{A}$. On rappelle que chaque appartement est isomorphe comme complexe de chambres à un complexe de Coxeter $(W,S)$, et que l'on peut étiqueter $\Delta$ grâce à $S$, et que cet étiquetage est essentiellement le seul sur $\Delta$. Soit $G$ un groupe d'automorphismes de $\Delta$ qui préserve l'étiquetage.
	\begin{Def}
		On dit que $G$ agit fortement transitivement sur $\Delta$ si $G$ agit transitivement sur les paires $(C,A) \in \ch (\Delta) \times \mathcal{A}$, où $C$ est une chambre de $A$. 
	\end{Def}
	On se donne un appartement $A_0$ et une chambre $C_0$ dans cet appartement. On rappelle que si on note $W $ le groupe d'automorphismes de $A_0$ qui préservent le type, et $S$ les réflexions par rapport aux facettes de $C_0$, alors $(W,S)$ est un système de Coxeter et $A_0$ est isomorphe au complexe de Coxeter associé. Notons 
	\begin{eqnarray}
		B & := & \stab_G(C_0) \nonumber \\
		N & := & \stab_G(A_0) \nonumber \\
		T & := & B \cap N \nonumber. 
	\end{eqnarray}
	
	On peut montrer en utilisant le Lemme d'unicité \ref{lem unicite morphisme ch} que $N/T$ est isomorphe à $W$, \cite[Lemma 5.2]{garrett97}. Pour $S' \subseteq S$, on note $F_{S'}$ la face de $C_0$ dont le stabilisateur dans $W$ est $W_{S'} := \langle S' \rangle $. On note enfin 
	
	\begin{eqnarray}
		P_{S'} := \stab_G (F_{S'}). \nonumber
	\end{eqnarray}
	
	Le groupe $P_{S'}$ est appelé sous-groupe parabolique standard de $G$. Pour $S' =\emptyset$, la face de $C_0$ dont le stabilisateur est l'identité est $C_0$ elle-même. Par conséquent, $P_\emptyset = B$. D'autre part, pour $S' = S$, $P_S = G$. 
	
	Le théorème suivant donne les propriétés fondamentales du système $(G,B,N,S)$. 
	
	\begin{thm}\label{thm bruhat-tits decomposition}
		Soit un système $(G,B,N,S)$ comme précédemment. Alors 
		\begin{enumerate}
			\item Pour tout sous-groupe parabolique standard, on a la décomposition de Bruhat-Tits 
			\begin{eqnarray}
				P_{S'} = \underset{ w \in \langle S' \rangle }{\bigsqcup} BwB
			\end{eqnarray} 
			\item Pour $s \in S, w\in W$, $BwB \, BsB = BwsB $ si $l_S(ws) = l(w) + 1$ et  $BwB \, BsB = BwsB \sqcup BwB $ sinon. 
			\item Pour $s \in S$, $BsB \nsubseteq B$. 
		\end{enumerate}
	\end{thm}
	
	Il est remarquable que ces conséquences sur la structure du groupe $G$ ne mentionnent pas l'immeuble $\Delta$. Réciproquement, soit $G$ un groupe, $B, N< G$ deux sous-groupes tels que $T = B \cap N \vartriangleleft N$ est normal dans $N$. On note $W := N/T$ le groupe quotient, et on se donne un ensemble $S $ de générateurs pour $W$. Si les conséquences du théorème \ref{thm bruhat-tits decomposition} sont vérifiées, on dira que $(G,B,N,S)$ est une \textit{$BN$-paire} pour $G$. 
	
	Soit $(G,B,N,S)$ une $BN$-paire pour $G$. Notons $\Delta(G, B)$ le poset des sous-groupes paraboliques standard $\underset{ w \in \langle S' \rangle }{\bigsqcup} BwB$, ordonnés par inclusion renversée. La notation $\Delta(G,B)$ a un sens car on peut montrer que l'immeuble dépend en fait uniquement de $G$ et de $B$, le système générateur étant uniquement déterminé, \cite[Theorem 6.56]{abramenko_brown08}. La proposition suivante est la réciproque du théorème \ref{thm bruhat-tits decomposition}. 
	
	\begin{thm}[{\cite[Theorem 6.56]{abramenko_brown08}}]
		Soit $(G,B,N,S)$ une $BN$-paire pour $G$. Alors $\Delta(G, B)$  est un immeuble épais sur lequel $G$ agit fortement transitivement. De plus, le stabilisateur d'une chambre est isomorphe à $B$, et $N$ stabilise un appartement. Réciproquement, si $G$ agit fortement transitivement sur un immeuble épais $\Delta$, alors $\Delta$ est canoniquement isomorphe à $\Delta(G,B)$. 
	\end{thm}
	
	En fait, on peut munir tout immeuble d'une métrique $\cat$(0) complète, ce qui étend un résultat de Moussong sur les complexes de Coxeter, \cite[Theorem 18.3.1]{davis15}. En appliquant le théorème du point-fixe de Bruhat-Tits et la théorie des décompositions pour les groupes ayant une $BN$-paire, on obtient donc le corollaire suivant. 
	\begin{cor}
		Soit $G$ un groupe muni d'une $BN$-paire, et soit $X$ l'immeuble associé sur lequel $G$ agit fortement transitivement. Soit $H \leq G$ un sous-groupe de $G$. Alors les assertions suivantes sont équivalentes. 
		\begin{enumerate}
			\item $H$ a une orbite bornée dans $X$. 
			\item $H$ fixe un point de $X$. 
			\item $H$ fixe un sommet de $X$. 
			\item $H$ est contenu dans un sous-groupe parabolique maximal, c'est-à-dire un sous groupe qui stabilise un sommet et qui est maximal pour l'inclusion. 
		\end{enumerate}
	\end{cor}

	\begin{ex}
		Considérons le groupe $G = \text{GL}(3, k)$ (ou le sous-groupe spécial linéaire $\text{SL}(3,k)$), pour $k$ un corps. Soit $\Delta$ le complexe de drapeaux associé au plan projectif $\text{P}^2(k)$. Il s'agit d'un immeuble de rang 2, où chaque sommet consiste en un sous-espace linéaire propre et non-trivial de $k^3$, et où une arête relie $L_1$ à $L_2$ si $L_1$ est un sous-espace linéaire de codimension 1 de $L_2$ (ou l'inverse), voir \cite[Section 6.1.2]{abramenko_brown08}. Le groupe $G$ agit fortement transitivement sur $\Delta$. On peut calculer que le groupe stabilisateur d'une chambre $B$ est isomorphe au groupe des matrices triangulaires supérieures :
		\begin{eqnarray}
			\begin{pmatrix}
				\ast & \ast & \ast \\
				0 & \ast & \ast \\
				0 & 0 & \ast
			\end{pmatrix} \nonumber.
		\end{eqnarray}
		Le groupe $N$ stabilisateur d'un appartement est le groupe monomial des matrices ayant exactement une entrée non nulle dans chaque colonne et dans chaque ligne. Enfin, le groupe $T$ consiste en les matrices diagonales, et donc $W$ s'identifie au groupe symétrique sur 3 lettres. Pour le groupe $\text{SL}(3,k)$, $B$, $N$ et $T$ sont seulement les intersections des groupes précédemment donnés avec $\text{SL}(3,k)$, tandis que $W$ reste le groupe symétrique sur 3 lettres. Tout ceci s'étend naturellement pour $n \geq 3$. 
		
	\end{ex}

	\section{Immeubles affines}\label{section immeuble affine}
	
	\subsection{Premières définitions}
	
	Dans cette section, on décrit plus précisément des immeubles lorsque les appartements sont construits à partir de groupes de réflexions affines non nécessairement discrets. Cette approche donne lieu à une définition plus générale que la définition simpliciale classique. Bien que nous ne travaillerons dans le Chapitre \ref{chapter rw immeuble} qu'avec des immeubles affines localement finis, nous pensons que la plupart des résultats présentés restent vrais dans le cadre plus général des immeubles non-discrets, moyennant quelques modifications dans les preuves. De plus, nous pensons que cette approche directement \og métrique\fg{} est plus naturelle pour notre sujet. Les références pour cette section sont \cite{parreau00} et \cite{rousseau09}. 
	
	Soit donc $(W, V)$ un groupe de réflexions affine (non nécessairement discret) sur l'espace vectoriel euclidien $V$. La définition qui suit est issue de \cite[Définition 1.1]{parreau00}
	\begin{Def}
		Soit $X$ un ensemble et $\mathcal{A}$ une famille d'injections de $V$ dans $X$, qu'on appelle \textit{appartements marqués}. On appelle \textit{appartement} l'image d'un appartement marqué. On dit que $(\Delta, \mathcal{A})$ est un \textit{immeuble affine} modelé sur $(W, V)$ si les axiomes suivant sont vérifiés : 
		\begin{enumerate}
			\item[(A1)] Le système d'appartement marqués est invariant par précomposition par $W$. 
			\item[(A2)] Pour tous $f,f' \in \mathcal{A}$, l'ensemble $Y :=(f')^{-1} (f(V))$ est un convexe fermé de $V$. De plus, il existe $w \in W$ tel que $((f')^{-1} \circ f)_{|Y} = w_{|Y}$. 
			\item[(A3)] Pour toute paire de points $x, y \in X$, il existe un appartement qui les contient tous deux, c'est-à-dire il existe $f \in \mathcal{A}$, $u,v \in V$, tels que $f(u ) = x$ et $f(v) =y$. 
		\end{enumerate}
		Pour toute paire de points $x,y \in X$, il existe un appartement marqué $f : V \rightarrow X$ tel que $f(a) = x$ et $f(b)= y$. On pose alors $d_X(x, y) = d_V(a,b)$. La fonction $d$ est une distance sur $X$ \cite[Proposition 1.3]{parreau00}, et ne dépend pas de l'appartement considéré. L'image d'une chambre de Weyl affine de $V$ par un appartement marqué est appelée \textit{chambre de Weyl} (ou \textit{secteur}) de $X$. 
		\begin{enumerate}
			\item[(A4)] Soit $S_1, S_2$ deux chambres de Weyl. Alors il existe des chambres de Weyl $S_1'\subseteq S_1, S_2' \subseteq S_2$ et un appartement $A \in \mathcal{A}$ qui contient $S'_1$ et $S'_2$. 
			\item[(A5)] Pour tout élément $x \in X$, et tout appartement $A$ qui contient $x$, il existe une rétraction $r$ de $X$ sur $A$ (c'est-à-dire une application $r : \Delta \rightarrow A$ telle que $r_{|A} = \Id_{|A}$) qui n'augmente pas la distance $d_\Delta$ et telle que $r^{-1}(x) = x$.
		\end{enumerate}
		Si ces conditions sont réunies, on dit que $X$ est un immeuble affine de \textit{type} $(\overline{W}, V)$, où $\overline{W}$ est le groupe fini engendré par les parties linéaires des éléments de $W$. 
	\end{Def}
	
	Remarquons que cette définition est immédiatement métrique, et non plus simpliciale. La paire $(X, \mathcal{A})$ est la réalisation géométrique d'un immeuble combinatoire affine au sens classique si et seulement si le groupe de réflexions affine $(W,V)$ sur lequel est modelé $X$ est discret et $\overline{W}$ est irréductible.

	\begin{rem}
		Dans toute la suite, on fait l'hypothèse que le système d'appartements marqués $\mathcal{A}$ est maximal, c'est-à-dire qu'il n'existe pas de système d'appartements $\mathcal{A}'$ qui contient strictement $\mathcal{A}$. De ce fait, toute géodésique et tout rayon géodésique est contenu dans un appartement, \cite[Proposition 2.18]{parreau00}.
	\end{rem}
	
	Les immeubles affines forment une classe particulièrement intéressante d'espaces $\cat$(0), comme le montre la proposition suivante. 
	
	\begin{prop}[{\cite[Proposition 2.10]{parreau00}}]
		Un immeuble affine au sens précédent muni de sa métrique $d$ est un espace $\cat$(0). Si l'immeuble $\Delta$ est simplicial, sa réalisation géométrique est de plus complète. 
	\end{prop}
	
	\begin{Def}
		Un automorphisme d'un immeuble affine $(X, \mathcal{A})$ est une bijection $ g : X \rightarrow X$ tel que le système des appartements marqués $\mathcal{A}$ est invariant par $g$, c'est-à-dire que pour tout $f \in \mathcal{A}$, $g \circ f $ et $g^{-1} \circ f$ sont dans $\mathcal{A}$. 
	\end{Def}
	
	Il est clair qu'un automorphisme d'un immeuble affine est une isométrie de cet immeuble.

	\subsection{Bords d'un immeuble affine}\label{section bdinf}
	
	Dans cette section, on discute de différentes bordifications possibles d'un immeuble affine. La diversité de ces constructions provient de ce qu'on peut les voir comme des espaces combinatoires ou comme des espaces métriques.
	
	\subsubsection{Immeuble sphérique à l'infini}

	Soit $(X, \mathcal{A})$ un immeuble affine épais modelé sur $(W, V)$, muni de sa métrique canonique $\cat$(0), et on suppose que $\mathcal{A}$ est un système d'appartements maximal. En tant que tel, on peut considérer la bordification conique $\overline{X}= X \cup \bd X$. L'immeuble sphérique à l'infini $\Delta$ est un immeuble construit à partir des rayons, dont $\bd X$ est une réalisation géométrique. Détaillons sa construction. 
	
	Tout appartement $A \in \mathcal{A}$ est isométrique à $V$. On rappelle qu'une facette vectorielle dans $V$ est une facette du cône polyédral formé par une chambre de Weyl vectorielle dans $V$. Les faces de codimension 1 des chambres de Weyl vectorielles seront appelées \textit{cloisons vectorielles}. L'ensemble des facettes vectorielles forme un complexe simplicial (sphérique) sur $V$. Les chambres de Weyl (resp. facettes, resp. cloisons) affines sont les translatés affines des chambres vectorielles, donc de la forme $v + \bold{F}$, où $v\in V $ et $\bold{F}$ est une chambre de Weyl (resp. facette, cloison) vectorielle. 
	
	\begin{Def}
		Soit $(X, \mathcal{A})$ un appartement modelé sur $(W,V)$. On appellera \textit{chambre de Weyl}, ou \textit{secteur} (resp. \textit{facette}, \textit{cloison}) l'image par un appartement marqué d'une chambre de Weyl (resp. facette, cloison) affine de $V$. Un point $x \in X$ est un sommet \textit{spécial} s'il existe $w \in W$, $f\in \mathcal{A}$ tel que $x = f(w \cdot0_V)$. 
	\end{Def}
	
	Une facette (vectorielle) $F$ de $X$ est donc de la forme $f(v + \bold{F})$, où $v\in V $, $f \in \mathcal{A}$ et $\bold{F}$ est une facette vectorielle de $V$. Le point $a = f(v)$ est appelé le \textit{sommet} de la facette. On dit qu'une facette $F' = f'(v' + \bold{F})$ est \textit{dominée} par une facette $F= f(v + \bold{F})$ si $F$ et $F'  $ ont même sommet $f(v) = f'(v')$ et $F'$ est incluse dans la facette fermée $\overline{F} = f(v + \overline{\bold{F}})$. La relation de domination fait de l'ensemble des facettes issues d'un point $x$ un complexe simplicial. 
	\newline
	
	Discutons maintenant de la notion de type dans l'immeuble $(X, \mathcal{A})$ modelé sur $(W, V)$. Soit $\bold{S}_0$ une chambre de Weyl vectorielle (ouverte) fondamentale dans $V$. Pour toute facette affine $F$ de $X$, il existe un appartement marqué $f$, $v \in V$ et $\bold{F}$ une facette vectorielle de $V$ tels que $F = f(v + \bold{F})$. Puisque $\bold{S}_0$ est un domaine fondamental strict pour l'action de $\overline{W} $ sur $V$, quitte à composer par un élément de $\overline{W}$, on peut supposer que $\bold{F}$ est une facette de la chambre vectorielle fermée $\overline{\bold{S}_0}$. On appelle alors $\bold{F} \subseteq \overline{\bold{S}_0}$ le \textit{type} $\theta(F)$ de $F$. Le type ne dépend pas de l'appartement marqué $f$ choisi. Il est clair que les automorphismes de $(X, \mathcal{A})$ préservent le type, et que le complexe simplicial défini par les facettes affines issues d'un point et la relation de domination (voir paragraphe précédent) est alors étiqueté par le type des facettes. 
	
	De la même manière, soit $[x,y]$ un segment géodésique dans un appartement $f$. Alors le \textit{type} de $[x,y] $ est l'unique vecteur $\theta[x,y]$ de la chambre de Weyl fermée fondamentale $\overline{\bold{S}_0}$ de $V$ telle que $f(v) = x$ et $f( v + \theta[x,y]) = y$ (quitte à composer par $\overline{W})$.
	\newline
	
	Soit $F$ une facette de $X$ issue de $x$, la \textit{facette de $F$ à l'infini} est l'ensemble $F_\infty$ des points  $\xi \in \bd X$ du bord visuel tels que $F$ contient le rayon géodésique $[x, \xi)$. Deux facettes sont dites \textit{asymptotes} si elles ont la même facette à l'infini. De manière équivalente, deux facettes sont asymptotes si elles sont à distance de Hausdorff finie. Dorénavant, une facette à l'infini $F_\infty$ est donc une classe d'équivalence de facettes de $X$ pour la relation \og être asymptote \fg{}. 
	
	On peut faire la construction réciproque de la manière suivante : soit $F_\infty$ une facette à l'infini, et $x \in X$, on note 
	\begin{eqnarray}
		Q(x,F_\infty) := \bigcup \{ [x, \xi), \, \xi \in F_\infty \}\nonumber. 
	\end{eqnarray}
	Alors $Q(x, F_\infty)$ est une facette affine de $X$, de sommet $x$ et de facette à l'infini $F_\infty$. 
	
	On peut montrer que deux chambres de Weyl $S,S'$ sont asymptotes si et seulement $S \cap S'$ contient un sous-secteur, c'est-à-dire une chambre de Weyl contenue dans $S$ et $S'$, \cite[Corollaire 1.6]{parreau00}.
	
	\begin{Def}
		Un \textit{simplexe idéal} est une facette à l'infini $F_\infty \subseteq \bd X$. 
	\end{Def}
	
	Justifions maintenant cette terminologie. Pour $x \in X$, il y a une bijection entre les simplexes idéaux de $\bd X$ et les facettes affines de sommet $x$, \cite[Lemma 11.75]{abramenko_brown08}. On peut donc transférer la relation de domination sur les simplexes idéaux pour en faire un poset. De manière équivalente, la relation de domination sur les facettes à l'infini peut être donnée directement de la manière suivante. Soit $F_\infty$ et $F'_\infty$ deux facettes à l'infini. On dit que $F_\infty$ domine $F'_\infty$ si pour toute paire de facettes affines $F \in F_\infty$, $F' \in F'_\infty$ dans les classes d'équivalence de $F_\infty$ et $F'_\infty$ respectivement, on a 
	\begin{eqnarray}
		\sup_{x \in F'} d(x, F) < \infty \nonumber.
	\end{eqnarray}
	
	\begin{Def}
		\textit{L'immeuble sphérique à l'infini} $\Delta^\infty$ associé à l'immeuble affine $(X,d)$ est le complexe simplicial formé par les simplexes idéaux de $X $, pour la relation d'ordre partiel donnée par la relation de domination précédente. 
	\end{Def}
	
	Le type d'une facette idéale est le type d'une facette affine dans sa classe d'équivalence, ce qui donne un étiquetage bien défini sur $\bdinf$. Le théorème suivant justifie la terminologie d'immeuble sphérique, voir \cite[Theorem 11.79]{abramenko_brown08}, \cite[Corollaire 2.19]{parreau00}. 
	
	\begin{thm}
		Le poset $\Delta^\infty$ est un immeuble sphérique à l'infini. Ses appartements sont les bords des appartements de $X$. Il est étiqueté par le type des facettes. Le bord $\bd X$ muni de la métrique de Tits est une réalisation géométrique de $\Delta^\infty$. 
	\end{thm}
	
	L'ensemble des chambres à l'infini $\ch (\bdinf)$ peut être muni d'une topologie compacte et totalement discontinue. Pour tout sommet $x \in X$, et toute chambre à l'infini $C_\infty \in \ch(\bdinf)$, il existe une chambre de Weyl $S$ issue de $x$ qui représente $C_\infty$ \cite[Corollaire 1.10]{parreau00}. Pour $x,y \in X$, une topologie de $\ch(\bdinf)$ est donnée par les ouverts suivants :  
	
	\begin{eqnarray}
		\Omega_x(y) := \{ C_\infty \in \ch(\bdinf) \, | \, y \in Q (x, C_\infty) \}. 
	\end{eqnarray}
	
	Avec cette topologie, tout automorphisme $g : X \rightarrow X$ induit un automorphisme continu de l'immeuble sphérique à l'infini $\Delta^\infty$, \cite[Proposition 2.6]{parreau00}. 
	\newline 
	
	On termine cette section par une propriété importante vérifiée par les chambres de Weyl. Puisque l'immeuble à l'infini est sphérique, il est de diamètre fini pour la distance de galerie associée et deux chambres à l'infini sont opposées si elles sont à distance maximale l'une de l'autre. Deux chambres de Weyl affines $S$ et $S'$ de $X$ sont dites opposées si elles ont même sommet $x \in X$ et qu'elles représentent deux chambres opposées à l'infini. De manière équivalente, elles sont opposées s'il existe $y ,z$ respectivement dans l'intérieur de $S, S'$ tels que $x \in [y,z]$. 
	
	\begin{prop}
		Soit $X$ un immeuble affine et $\Delta^\infty$ son immeuble sphérique à l'infini. Soit $S_\infty$ et $ S'_\infty$ deux chambres idéales de $\Delta^\infty$ qui sont opposées. Alors il existe un unique appartement contenant deux chambres de Weyl $S $ et $S'$ qui sont opposées et dans les classes d'équivalence de $S_\infty$ et $ S'_\infty$ respectivement. 
	\end{prop}

	\subsubsection{Immeubles affines à l'infini}\label{section panel trees}
	
	Soit $X$ un immeuble affine modelé sur $(W,V)$ muni d'un système d'appartements (maximal) $\mathcal{A}$, et $\bdinf$ son immeuble sphérique à l'infini. Dans ce paragraphe, pour simplifier la présentation, on suppose que $X$ est la réalisation géométrique d'un immeuble affine discret $\Delta$, de sorte qu'il est possible de parler de faces dans $\Delta$, vu comme complexe de chambres. Tout ce qui suit peut être généralisé au cas d'immeubles non discrets, en substituant les faces par des filtres de sous-ensembles de $X$, voir \cite{rousseau11}. Les références pour cette partie sont \cite{charignon09}, \cite{remy_trojan21} et dans le cas plus général des immeubles denses \cite{ciobotaru_muhlerr_rousseau_20}, \cite{rousseau11}.
	
	Soit $ F = x + F_\infty$ un facette affine issue d'un sommet $x$ de $\Delta$, de direction $F_\infty$. Une sous-facette $ F'$ (c'est-à-dire une facette affine contenue dans $ F$) est dite \textit{pleine} dans $ F $ si $ F$ et $ F'$ ont la même facette à l'infini. Suivant \cite[Section 4.2]{remy_trojan21}, on peut faire la même construction en considérant non plus seulement des facettes issues de points, mais de n'importe quelle face de $\Delta$. Pour une partie $Z \subseteq A$ incluse dans un appartement, la \textit{clôture convexe} $\conv_A(Z)$ de $Z$ est l'intersection de tous les demi-appartements (associés aux réflexions affines de $(W,V)$) de $A$ qui contiennent $Z$. En fait, la clôture convexe de l'ensemble $Z$ ne dépend pas de l'appartement dans lequel il est inclus, on peut donc la noter $\conv(Z)$. 
	
	\begin{Def}
		Une \textit{cheminée} dans un appartement $A \in \mathcal{A}$ est un ensemble de la forme $R= R(C, F_\infty) = \conv(C + F_\infty)$, où $C$ est une face de $A$, $F_\infty$ est une facette de l'immeuble à l'infini (ou la direction d'une facette affine). Le simplexe idéal  $F_\infty$ est la \textit{direction} de la cheminée $R$. Une sous-cheminée $R' \subseteq R$ est \textit{pleine} dans $R'$ si $ R$ et $R'$ ont la même direction et que les cheminées engendrent les mêmes espaces affines dans tout appartement qui les contient. 
	\end{Def}
	
	Bien que cela ne soit pas immédiat dans la définition, la direction d'une cheminée $R$ est bien définie, \cite[Remarque 1.12]{rousseau11}. Enfin, on appellera \textit{germe à l'infini} $\germ_\infty(R)$ d'une cheminée $R$ 
	\begin{eqnarray}
		\germ_\infty(R) := \{ R' \, | \, R \cap R' \text{ contient une cheminée qui est pleine dans $R$ et $R'$} \}. \nonumber
	\end{eqnarray}
	
	Soit une facette $F_\infty\subseteq \bd X$ de l'immeuble sphérique à l'infini $\bdinf$. On appelle $X(F_\infty)$ l'ensemble 
	\begin{eqnarray}
		X(F_\infty) = \{\germ_\infty(F') \, | \, F'\text{ est une facette affine telle que }\bd F' \subseteq F\}. \nonumber
	\end{eqnarray}
	
	Pour toute facette $F_\infty$ à l'infini, $X(F_\infty)$ admet une structure d'immeuble affine \cite[Proposition 8.1.5]{charignon09}, et on dispose d'un morphisme d'immeubles 
	\begin{eqnarray}
		\pi_{F_\infty} : X &\rightarrow &X(F_\infty) \nonumber \\
		x & \mapsto & \germ_\infty(x + F_\infty). \nonumber
	\end{eqnarray}
	On appelle $X(F_\infty)$ \textit{l'immeuble affine à l'infini} de $X$ associé à $F_\infty$. 
	\newline 
	
	Lorsque $F_\infty$ est une cloison (une facette de codimension 1) de l'immeuble sphérique à l'infini $\bdinf$ de $\Delta$, on notera plus souvent $X(F_\infty)= T_{F_\infty}$. L'espace $T_{F_\infty}$ dispose d'une structure d'arbre : c'est \textit{l'arbre-cloison} associé au sommet $F_\infty$. 
	\newline 
	
	Dans le cas particulier où l'immeuble est de dimension 2, les cloisons affines sont des rayons géodésiques. Soit $u$ une cloison de l'immeuble sphérique à l'infini $\bdinf$, et $\gamma $ un rayon géodésique issu d'un sommet $x$ dans la classe d'équivalence de $u$. Alors $\gamma$ est une cheminée de la forme $R(x, u)$. Le germe de $\gamma$ est un élément $c \in T_u$ de l'arbre-cloison sur $u$. Par définition, un autre rayon géodésique $\gamma'$ appartient à $\germ_\infty(\gamma)$ si $\gamma \cap \gamma'$ contient un sous-rayon géodésique. On dit alors que $\gamma$ et $\gamma'$ sont \textit{fortement asymptotiques}. Une distance entre deux germes à l'infini dans $T_u$ est donnée par 
	
	\begin{eqnarray}
		d_{T_u}(\germ_\infty(\gamma), \germ_\infty(\gamma')) := \underset{s}{\inf} \underset{t \rightarrow \infty}{\lim} d(\gamma(t+s), \gamma'(t)). \nonumber
	\end{eqnarray}
	
	Enfin, soit $F_\infty \subseteq \Delta^\infty$ une facette à l'infini, et soit $X(F_\infty)$ l'immeuble affine à l'infini précédemment construit. Il existe une application canonique entre le bord de $X(F_\infty)$ et le résidu de $F_\infty$ dans $\bdinf$. En effet, soit $S,S'$ deux chambres de Weyl tel que $\bd S$ et $\bd S'$ contiennent $F_\infty$ comme facette, et supposons que $S$ et $S'$ soient équivalentes. Alors $\pi_{F_{\infty}}(S) $ et $\pi_{F_\infty}(S')$ contiennent un sous-secteur commun, et sont donc équivalentes dans l'immeuble affine $X(F_\infty)$. En fait, si on munit l'immeuble à l'infini $(X(F_\infty))^\infty$ de $X(F_\infty)$ de la topologie classique des immeubles sphériques à l'infini (voir Section \ref{section bdinf}), cette application est un homéomorphisme.
	
	\begin{prop}[{\cite[Lemma 4.2]{remy_trojan21}}]\label{prop homeo panel tree ch}
		Soit $X$ un immeuble affine discret. Soit $F_\infty \subseteq \Delta^\infty$ une facette à l'infini, et $X(F_\infty)$ l'immeuble affine à l'infini associé à $F_\infty$. Alors la bijection décrite ci-dessus est un homéomorphisme équivariant entre le résidu de $F_\infty$ dans $\bdinf$ et l'immeuble sphérique à l'infini $(X(F_\infty))^\infty$ de $X(F_\infty)$, qui à une chambre à l'infini (vue comme classe d'équivalence de secteurs de $X$) associe une classe d'équivalence de secteurs dans $(X(F_\infty))^\infty$. 
		\begin{eqnarray}
			\phi_F : \res(F_\infty) & \rightarrow & X(F_\infty)^\infty \nonumber \\
			\bigr[S\bigr] & \mapsto & \bigr[\pi_{F_\infty} (S)\bigr] \nonumber
		\end{eqnarray}
	\end{prop}
	
	\begin{rem}
		Dans le cadre des immeubles non-discrets, tout ce qui précède a un sens sous réserve de redéfinir les facettes et les sommets comme des filtres de sous-ensembles de l'immeuble, voir \cite[Section 5]{rousseau09}. Les immeubles affines à l'infini sont seront alors également non-discrets, voir \cite{ciobotaru_muhlerr_rousseau_20}. 
	\end{rem}
	
	\subsection{Isométries d'un immeuble affine}
	On clôt ce chapitre en discutant des isométries d'un immeuble affine. Soit $(X, \mathcal{A})$ un immeuble affine épais (pas nécessairement discret) modelé sur $(W, V)$, et soit $\bold{S}_0$ une chambre de Weyl vectorielle fondamentale de $V$. 
	\begin{Def}\label{def regular geod immeuble}
		Une géodésique $\gamma$ dans $X$ est dite \textit{régulière} si pour $x \neq y \in \gamma$, le type $\theta[x,y]$ de $[x,y] $ est dans l'intérieur de la chambre de Weyl vectorielle fermée $\overline{\bold{S}_0}$. Autrement, $\theta[x,y]$ appartient au bord de $\bold{S}_0$ et on dit que $\gamma$ est \textit{singulière}.  
	\end{Def}
	
	\begin{prop}[{\cite[Proposition 2.26]{parreau00}}]
		Supposons que l'immeuble à l'infini $\bdinf$ de $X$ est épais et que le système d'appartements est maximal. Alors une géodésique $\gamma$ dans $X$ est régulière si et seulement si elle est contenue dans un unique appartement de $X$. 
	\end{prop}
	
	De plus, la proposition suivante montre que les isométries de $X$ agissent comme des automorphismes de l'immeuble à l'infini. 
	
	\begin{prop}[{\cite[Proposition 2.27]{parreau00}}]\label{prop isom autom immeuble}
		Supposons que l'immeuble à l'infini $\bdinf$ est épais. Alors toute isométrie de $X$ préserve les chambres de Weyl. 
	\end{prop}

	On clôt cette section avec un résultat de classification sur les isométries de $X$.

	\begin{thm}[{\cite[Théorème 4.1]{parreau00}}]\label{thm isom ss immeuble affine}
		Soit $(X, \mathcal{A})$ un immeuble affine complet. Alors toute isométrie de $X$ qui préserve les chambres de Weyl est semi-simple (i.e. elliptique ou axiale). 
	\end{thm}
	
	\begin{rem}
		Dans le cadre d'un système d'appartements maximal, et si l'immeuble à l'infini $\bdinf$ est complet, alors toute isométrie de $X$ préserve les chambres de Weyl \cite[Proposition 2.27]{parreau00}. En particulier, tout automorphisme de $X$ est semi-simple. 
	\end{rem}
	
	Il y a plusieurs résultats récents de structure et de points fixes sur les groupes agissant par isométries sur des immeubles affines. Par exemple, J. Schillewaert, K. Struyve et A. Thomas montrent dans \cite[Theorem A]{schillewaert_struyve_thomas22} que si $G$ agit sur un immeuble affine complet de type $\tilde{A}_2$ ou $\tilde{C}_2$ tel que chaque élément $g \in G$ est elliptique, alors $G$ a un point fixe global. Osajda et Przytycki montrent dans  \cite{osajda_przytycki21} un résultat similaire, quoique vrai pour tout complexe triangulaire $\cat$(0).

	\chapter{Marches aléatoires et théorie des bords}\label{chapter rw intro}

	Dans ce chapitre, on introduit le vocabulaire des marches aléatoires et certaines techniques de théorie des bords. La Section \ref{section rw generalite} revient sur les marches aléatoires et la notion, fondamentale, de mesure stationnaire pour un système dynamique mesuré. Dans la Section \ref{section bord poisson}, on présente le bord de Poisson-Furstenberg, qui est l'objet le plus important de ce chapitre. On verra qu'il représente les propriétés asymptotiques de la marche aléatoire sur $G$, et qu'il dispose d'une action naturelle par $G$ qui est ergodique et Zimmer-moyennable. Dans la Section \ref{section bords ergodiques}, on introduit d'après Bader et Furman \cite{bader_furman14} la notion de $G$-bord, dont le bord de Poisson-Furstenberg est le modèle privilégié. Enfin, on fait dans la Section \ref{section rw courbure non positive} un bref résumé des résultats récents pour les marches aléatoires dans les espaces hyperboliques et $\cat$(0). 
	\newline
	
	\vspace{1cm}
	\minitoc
	\vspace{1cm}

	\section{Généralités sur les marches aléatoires}\label{section rw generalite}
	
	Dans cette section, on présente la théorie générale des marches aléatoires, ainsi que les premiers résultats de stationnarité, d'ergodicité et les questions qui nous intéresseront. 
	
	\subsection{Premières définitions}	
	Soit $G$ un groupe localement compact et à base dénombrable (ce sont des hypothèses qu'on fera toujours sur $G$), et soit $\mu \in \prob(G)$ une mesure de probabilité sur $G$. Soit $\Omega = G^\mathbb{N}$, muni de la mesure $\mathbb{P} = \mu^{\otimes \mathbb{N}}$. On appellera $\Omega$ \textit{l'espace des incréments} : un élément $\omega = (\omega_0, \omega_1, \dots)$ est une suite de variables aléatoires $\omega_i$ indépendantes et distribuées selon la loi $\mu$. La marche aléatoire sur $G$ associée à $\mu$ est la suite de variables aléatoires $(Z_n(\omega))_{n\in \mathbb{N}}$, où $Z_0 = \Id $ et $Z_n$ est donnée par
	\begin{eqnarray}
		Z_n (\omega) = \omega_1 \omega_2 \dots \omega_n, \nonumber
	\end{eqnarray}
	le produit des $n$ premiers incréments de $\omega = (\omega_i)_{i\in \mathbb{N}}$. En pratique, on pourra omettre d'écrire la dépendance explicite en $\omega $, et simplement écrire $Z_n := Z_n(\omega)$. 
	
	Supposons maintenant que $G$ agisse par isométries sur un espace métrique séparable $(X,d)$, et soit $o \in X$ un point-base. On peut considérer la marche aléatoire $(Z_n o )_{n\in \mathbb{N}}$ dans l'espace $X$. Dans une telle situation, les questions que l'on se posera seront principalement : 
	\begin{enumerate}
		\item que peut-on dire du comportement asymptotique de $(Z_n o )_{n\in \mathbb{N}}$ ? 
		\item en fonction de ce comportement, que peut-on en déduire sur le groupe $G$ ? 
	\end{enumerate}
	
	Le but de ce chapitre est d'apporter du contexte et de la précision à ces deux questions. 
	\newline 
	
	Commençons par quelques définitions. Soit $(X,d)$ un espace métrique séparable muni de sa tribu borélienne, et soit $G \curvearrowright X$ une action continue. Il sera souvent utile de munir $X$ d'une mesure de probabilité borélienne $\nu \in \prob(X)$. Dans toute la suite, on fera l'hypothèse que $X$ est un \textit{espace borélien standard}, c'est-à-dire qu'il existe un isomorphisme borélien de $X$ vers un sous-ensemble borélien dans un espace polonais. En d'autres termes, un espace borélien standard est soit dénombrable (potentiellement fini), soit mesurablement isomorphe au segment $[0,1]$. Dans ce cas, $\prob (X)$ désigne l'espace des mesures de probabilités et boréliennes sur $X$. L'espace $\prob (X)$ est muni de la topologie faible-$\ast$ (métrisable) : une suite de mesures $\mu_n $ converge faiblement-$\ast$ vers $\mu$ si pour toute fonction continue bornée $f$ sur $X$, 
	\begin{eqnarray}
		\int_X f(x) d\mu_n(x) \underset{n \rightarrow \infty}{\longrightarrow} \int_X f(x) d\mu(x). \nonumber
	\end{eqnarray}
	
	On rappelle que si $f : X \rightarrow X$ est une application borélienne, alors $f_\ast\nu$ est une mesure de probabilité sur $X$ donnée par $f_\ast \nu (A) := \nu(f^{-1}A) $ pour tout borélien $A$. Puisque $G$ est localement compact, on peut considérer la mesure de Haar invariante à gauche $m_G$ sur $G$. On dit que l'action de $G$ sur $(X, \nu)$ est \textit{non-singulière} si pour tout $g \in G$, les mesures $\nu$ et $g_\ast \nu$ sont équivalentes sur $X$, c'est-à-dire ont les mêmes boréliens négligeables. On dit alors que $(X, \nu)$ est un \textit{$G$-espace}. En particulier, lorsque $\nu$ est $G$-invariante, c'est-à-dire lorsque pour tout $g \in G$, $g_\ast \nu = \nu $, alors l'action est non-singulière. On dit alors que l'action de $G$ sur $(X,d, \nu)$ \textit{préserve la mesure de probabilité}, ou plus simplement que l'action est \textit{pmp}.

	Soit $G \curvearrowright(X,d, \nu)$ une action continue par isométries, non-singulière. On dit qu'une fonction mesurable $f : X \rightarrow \mathbb{C}$ est \textit{$G$-invariante} (au sens mesurable) si pour tout $g \in G$, pour $\nu$-presque tout $x \in X$, $f(gx )= f(x)$. De même, on dira qu'un ensemble mesurable $Y \subseteq X$ est \textit{$G$-invariant} si la fonction indicatrice $\mathds{1}_Y$ est $G$-invariante. Enfin, soit $Z$ un espace métrique sur lequel $G$  agit de manière borélienne. Une application mesurable $f : X \rightarrow Z$ est dite \textit{$G$-équivariante} si pour tout $g \in G$ et pour $\nu$-presque tout $x \in X$, $f(gx)= g f(x)$ : on dit alors que $f$ est une $G$-application. Le lemme suivant permet de toujours se ramener à des fonctions strictement invariantes. 
	
	\begin{lem}\label{lemme stricte invariance}
		Soit $G \curvearrowright(X,d, \nu)$ une action non-singulière. Alors les assertions suivantes sont vraies. 
		\begin{enumerate}
			\item \label{lemme ensemble}Pour tout ensemble $G$-invariant $Y \subseteq X$, il existe $Z \subseteq X$ tel que pour tout $g\in G$, $x \in Z$, $g Z = Z $ et $\nu (Z \triangle Y)= 0$. 
			\item \label{lemme stricte invariance fonctions}Pour toute fonction mesurable $G$-invariante $f : X \rightarrow \mathbb{C}$, il existe une application mesurable $F: X \rightarrow \mathbb{C}$ telle que $\nu(\{f \neq F\}) = 0$ et pour tout $g \in G$, pour tout $x \in X$, $F(gx ) = F(x)$. 
			\item Pour toute application mesurable $G$-équivariante $f : X \rightarrow Z$, il existe un ensemble de mesure pleine $X_0 \subseteq X$ et application $F  : X_0 \rightarrow Z$ telle que $\nu(\{F \neq f\}) = 0$ et pour tout $g \in G$, pour tout $x \in X_0$, $F(gx) = gF(x)$. 
		\end{enumerate}
	\end{lem}
	
	\begin{proof}
		On ne montre que \ref{lemme stricte invariance fonctions} car les situations sont analogues. Soit $f : X \rightarrow \mathbb{C}$ une application mesurable $G$-équivariante. Fixons $m_G$ la mesure de Haar invariante à gauche sur $G$. Pour presque tout $x \in X$, on considère l'application $f_x(g) := f(g x) \in \mathbb{C}$. On définit alors
		\begin{eqnarray}
			X_0 := \{ x\in X \, | \, f_x \textit{ est $m_G$-essentiellement constante}\}. \nonumber
		\end{eqnarray}
		Alors $X_0$ est mesurable, de mesure pleine par hypothèse. Notons $F (x)$ la valeur essentielle de $f_x$ pour tout $x \in X_0$, et posons $F(x)= 0$ pour $x \in X \setminus X_0$. Alors $F$ est mesurable par Fubini, et $\nu(\{f \neq F\}) = 0$. Enfin, $X_0$ est strictement $G$-invariant : pour tout $x \in X_0$ et pour tout $g \in G$, $gx \in X_0$. Donc $F$ est strictement $G$-invariante. 
	\end{proof}

	\subsection{Mesure stationnaire}\label{section mesure stat intro}
	Soit $G$ un groupe agissant de manière non-singulière sur un espace probabilisé $(Y, \nu)$, et soit $\mu \in \prob(G)$ une mesure de probabilité sur $G$. La convolution de $\nu$ par $\mu $ est la mesure de probabilité sur $X$ donnée par l'image de $\mu \otimes \nu$ par l'action  $G \times Y \rightarrow  Y $. En d'autres termes, pour $f $ une application mesurable bornée sur $Y$, 
	
	\begin{equation*}
		\int_Y f(y)d(\mu \ast \nu)(y) = \int_G\int_Y f(g \cdot y ) d\mu(g) d\nu (y).
	\end{equation*}
	
	Si $G$ est dénombrable, alors pour tout ensemble mesurable $A$ dans $Y$, 
	
	\begin{equation*}
		\mu \ast \nu (A) = \sum_{g \in G} \mu(g) \nu (g^{-1}A). 
	\end{equation*}
	
	On écrira $\mu_m = \mu^{\ast m}$ la $m$-ème puissance de convolution de $\mu$, où $G$ agit sur lui-même par action à gauche $(g,h) \mapsto  gh$. 
	
	Souvent, il est trop restrictif de s'attendre à une action pmp sur l'espace $Y$. On considèrera surtout des mesures qui sont invariantes \og sous l'action de la marche aléatoire \fg. C'est le sens de la définition suivante. 
	\begin{Def}
		Une mesure de probabilité $\nu \in \prob(Y) $ est \textit{$\mu$-stationnaire} (ou harmonique) si $\mu \ast \nu = \nu $. Si $(Y,\nu)$ est un $G$-espace et $\nu$ est $\mu$-stationnaire, on dit que $(Y,\nu)$ est un \textit{$(G,\mu)$-espace}.  
	\end{Def}
	
	Il est évident que si $(G,\mu)\curvearrowright (Y,\nu)$ est pmp, alors elle est stationnaire. La réciproque est fausse en général, comme le montre l'exemple suivant 
	\begin{ex}[{\cite[Example 8.2]{furstenberg73}}]
		Soit $G = \psl_2(\mathbb{R})$, $X = P^1 (\mathbb{R})$, et $\mu = \frac{1}{2} (\delta_g + \delta_h)$, où 
		\begin{eqnarray}
			g = \begin{pmatrix}
				1 & 2 \\
				0 & 1
			\end{pmatrix}, \ h = \begin{pmatrix}
				1 & 0\\
				2 & 1
			\end{pmatrix} \nonumber
		\end{eqnarray}
		Alors $X$ admet une unique mesure stationnaire, qui n'est pas $G$-invariante. 
	\end{ex}
	
	En revanche, lorsque $(Y,d)$ est un espace métrique compact, alors toute action $(G,\mu)\curvearrowright (Y,\nu)$ admet une mesure stationnaire. Rappelons que dans ce cas, l'espace $\prob(Y)$ des mesures de probabilités est un convexe compact dans le dual $C(Y)$ des fonctions continues sur $Y$, pour la topologie faible-$\ast$.  
	
	\begin{prop}\label{prop existence stationnaire compact}
		Soit $(Y,d)$ un espace métrique compact et soit $G$ un groupe localement compact à base dénombrable agissant continûment sur $Y$ par homéomorphismes. Soit $\mu \in \prob(G) $ une mesure de probabilité sur $G$. Alors il existe une mesure de probabilité $\nu \in \prob(Y)$ sur $Y$ qui est $\mu$-stationnaire. 
	\end{prop}
	\begin{proof}
		Soit $\nu$ une mesure de probabilité sur $(Y,d)$. On considère la suite de mesures de probabilités 
		\begin{eqnarray}
			\nu_n = \frac{1}{n} (\nu + \mu \ast \nu + \dots + \mu_{n-1}\ast \nu) \nonumber,
		\end{eqnarray}
		où on rappelle que $\mu_n = \mu^{\ast n}$ est la $n$-ème puissance de convolution de $\mu$. L'espace $\prob(Y)$ étant faiblement-$\ast $ compact, tout point limite de la suite $(\nu_n)$ est une mesure $\mu$-stationnaire par le théorème de Markov-Kakutani. 
	\end{proof}
	
	Le \textit{support} d'une mesure $\nu$ sur un espace topologique $Y$ est le plus petit fermé $C$ tel que $\nu(Y \setminus C)= 0$. En d'autres termes, $y \in \supp(m)$ si et seulement si pour tout ouvert $U $ contenant $y $, $m(U) >0$. On dira que $\mu \in \prob(G)$ est \textit{admissible} si son support engendre $G$ comme semi-groupe et si elle est étale, i.e. il existe $n$ tel $\mu^{\ast n}$ est absolument continue par rapport à la mesure de Haar.

	\begin{Def}
		Soit $G$ un groupe localement compact et $\mu \in \prob(G)$ une probabilité sur $G$. On définit \textit{l'opérateur de Markov} $P_\mu $ associé à $\mu$ comme 
		\begin{eqnarray}\label{convergence harmonique}
			P_\mu(f) : g \in G \mapsto \int_G f(gh)d\mu(h), 
		\end{eqnarray}
		pour toute fonction $f \in L^1 (G, \mu)$. Une fonction $f $ sur $G$ est \textit{$\mu$-harmonique} si $f$ est $P_\mu$-invariante, i.e. $P_\mu (f) = f$.  
	\end{Def}
	
	La proposition suivante montre le lien entre marches aléatoires et fonctions harmoniques sur $G$. 
	
	\begin{prop}\label{prop cv fonction harm}
		Soit $G$ un groupe localement compact à base dénombrable et $\mu \in \prob(G)$ une probabilité sur $G$. Soit $f $ une fonction $\mu$-harmonique bornée sur $G$, et soit $(Z_n)$ la marche aléatoire sur $G$ associée à $\mu $. Alors pour $\mu$-presque tout $g \in G$ et $\mathbb{P}$-presque tout $\omega \in\Omega$, la limite
		\begin{eqnarray}
			\tilde{f}^g(\omega) = \lim_{n \rightarrow \infty} f(g Z_n(\omega)) \nonumber
		\end{eqnarray} 
		existe. De plus, pour  $\mu$-presque tout $g$, l'application $\tilde{f}^g : \omega \in \Omega\mapsto \tilde{f}^g( \omega)$ est une application bornée mesurable. 
	\end{prop}
	\begin{proof}
		Soit $f$ comme dans l'énoncé, et $g \in G$. Soit $\Omega_n$ la tribu engendrée par les variables aléatoires $X_1, \dots ,  X_n $ indépendantes et distribuées selon $\mu$, où $X_n : \omega \in\Omega \mapsto  \omega_n$. Alors 
		\begin{eqnarray}
			\mathbb{E} [f(Z_{n+1}(\omega)) \,| \, \Omega_n ] &=& \int_G f( Z_n(\omega) \omega_{n+1} ) d\mu(\omega_{n+1}) \nonumber \\ 
			& = & f(Z_n(\omega)  ) \text{ par $\mu$-harmonicité de }f \nonumber. 
		\end{eqnarray}
		En particulier, la suite d'applications $f_n : \omega \mapsto f(Z_n(\omega))$ forme une martingale bornée par rapport à la suite croissante de tribus $\Omega_n$. Par le théorème de convergence des martingales de Doob, $f_n$ converge presque sûrement et point par point vers une fonction bornée $\tilde{f} \in L^\infty(\Omega, \mathbb{P})$. Pour la démonstration des limites de l'équation \eqref{convergence harmonique} pour $ \mu$-presque tout $g$ , on renvoie à \cite[Lemme 3.2]{benoist_quint}.
	\end{proof}
	
	L'observation suivante est absolument fondamentale, et remonte à Furstenberg, \cite[Lemma 1.33]{furstenberg73}. Elle marque le début de la théorie des bords. L'idée est d'utiliser le théorème de convergence des martingales pour associer à toute mesure stationnaire sur $X$ une famille mesurable de probabilités sur $X$. 
	
	\begin{thm}[{\cite[Lemme 3.2]{benoist_quint11}}]\label{thm mesures limites}
		Soit $G$ un groupe localement compact à base dénombrable, $\mu$ une probabilité borélienne sur $G$ et $(Y, \nu)$ un espace borélien standard muni d'une action borélienne de $G$. On suppose de plus que $\nu$ est $\mu$-stationnaire. Soit $(Z_n)$ la marche aléatoire sur $G$ associée à $\mu$. Alors il existe une application mesurable $\omega \in\Omega \mapsto \nu_\omega \in \prob(Y)$ telle que pour $\mathbb{P}$-presque tout $\omega \in \Omega$ et pour  $\mu$-presque tout $g \in G$, 
		\begin{eqnarray}
			Z_n (\omega)g_\ast\nu \rightarrow \nu_\omega \nonumber
		\end{eqnarray}
		pour la topologie faible-$\ast$. De plus, on a la décomposition 
		\begin{eqnarray}
			\nu = \int_{\Omega} \nu_\omega d \mathbb{P}(\omega) \nonumber. 
		\end{eqnarray}
		Enfin, pour $\mathbb{P}$-presque tout $\omega = (\omega_0, \omega_1, \dots)$, 
		\begin{eqnarray}\label{lem shift mesures cond}
			\nu_\omega = {\omega_0}_\ast \nu_{S\omega}. 
		\end{eqnarray}
		De plus, l'application $\omega \rightarrow \nu_\omega$ est l'unique application mesurable vérifiant ces propriétés. 
	\end{thm}
	
	Les probabilités $\nu_\omega$ sont appelées \textit{probabilités conditionnelles} pour $\nu$. Le théorème précédent apparait dans \cite{furstenberg73} pour $(Y, \nu)$ un $(G, \mu)$-espace compact.

	\begin{proof}
		Puisque $(Y,\nu)$ est compact métrisable, l'ensemble $C(Y)$ est séparable. Soit $A \subseteq C(X)$ un sous-ensemble dense dénombrable et soit $\Omega_0 \subseteq \Omega$ tel que pour tout $f \in A$, pour tout $\omega \in \Omega_0$, la conclusion de la proposition \ref{prop cv fonction harm} est vraie : $Z_n(\omega)_\ast \nu (f) $ converge vers $\tilde{f}(\omega)$. Pour tout $\omega \in \Omega_0$ et pour tout $g \in G$, on a donc une application 
		\begin{eqnarray}
			f \in A & \rightarrow & \tilde{f}(\omega) \in \mathbb{C} \nonumber
		\end{eqnarray}
		qui s'étend en une unique fonctionnelle positive linéaire de norme 1
		\begin{eqnarray}
			f \in C(Y) & \rightarrow & \tilde{f}(\omega) \in \mathbb{C}. \nonumber
		\end{eqnarray}
		Par le théorème de représentation de Riesz, cette fonctionnelle est représentée par une mesure de probabilité $\nu_\omega \in \prob(Y)$, et donc $Z_n (\omega)g_\ast\nu \rightarrow \nu_\omega $. L'application $w \in \Omega \mapsto \nu_\omega$ est mesurable et définie $\mathbb{P}$-presque partout.
		
		Pour montrer la décomposition, pour toute fonction $f \in C(X)$, le théorème de convergence dominée donne  
		\begin{eqnarray}
			\int_\Omega \nu_\omega(f) d\mathbb{P}(\omega) & = & \int_\Omega \tilde{f}(\omega) d\mathbb{P}(\omega) \nonumber \\
			& = & \int_\Omega \lim_n Z_n(\omega)_\ast \nu (f) d\mathbb{P}(\omega) \nonumber \\
			& = & \lim_n \int_\Omega \int_X f (Z_n(\omega)x) d\nu (x) d\mathbb{P}(\omega) \nonumber \\
			& = & \int_X f(x) d\nu(x) \text{ par $\mu$-stationnarité.}\nonumber
		\end{eqnarray}
	\end{proof}
	L'observation suivante montre qu'il existe une correspondance bi-univoque entre les mesures stationnaires sur $Y$ et les applications de bord $\omega \rightarrow \nu_\omega \in \prob(Y)$ vérifiant la propriété \eqref{lem shift mesures cond}. 
	 
	\begin{cor}\label{cor application de bord mesure stat}
		Soit $\phi : \Omega \rightarrow \prob(Y)$ une application mesurable telle que pour $\mathbb{P}$-presque tout $\omega = (\omega_0, \omega_1, \dots)$, 
		\begin{eqnarray}
			\phi(\omega) = {\omega_0}_\ast \phi(S\omega) \nonumber. 
		\end{eqnarray}
		Alors la mesure $\nu = \int_\Omega \phi(\omega) d\mathbb{P}(\omega)$ est $\mu$-stationnaire. 
	\end{cor}
	
	\begin{proof}
		On a en effet 
		\begin{eqnarray}
			\mu \ast \nu & = & \int_G \int_\Omega g_\ast\phi(\omega) d\mu(g)d\mathbb{P}(\omega) \nonumber \\
			& = &  \int_G \int_\Omega {\omega_0}_\ast\phi(T\omega) d\mathbb{P}(\omega) \nonumber\\
			& =& \int_\Omega \phi(\omega) d\mathbb{P}(\omega) = \nu \nonumber. 
		\end{eqnarray}
	\end{proof}

	\section{Bord de Poisson-Furstenberg}\label{section bord poisson}
	
	Dans cette section, on construit le bord de Poisson-Furstenberg associé à une mesure de probabilité $\mu$ sur un groupe $G$. Il existe plusieurs constructions équivalentes, qui sont résumées dans \cite{kaimanovich_vershik}. 
	
	\subsection{Construction et propriétés}
	
	Notons $\Har(G, \mu)$ l'ensemble des fonctions bornées $\mu$-harmoniques sur $G$, c'est-à-dire les fonctions bornées mesurables et $P_\mu$ invariantes. Soit $(A, \eta)$ un $(G,\mu)$-espace. Pour toute fonction $f \in L^\infty (A, \eta)$, on peut construire une application $\mu$-harmonique par 
	\begin{eqnarray}
		\phi_\mu (f) : g \in G \mapsto \int_A f(gx) d\eta(x) \nonumber. 
	\end{eqnarray}
	En effet, par Fubini,
	\begin{eqnarray}
		P_\mu(\phi_\mu(f)) (g) & =& \int_G \int_X f(gg'x) d\eta(x) d \mu(g') \nonumber\\
		&  =& \int_X f(gx) d\eta(x) \text{ car $\eta$ est $\mu$-stationnaire}\nonumber \\
		& =& \phi_\mu(f)(g). \nonumber
	\end{eqnarray}
	Donc $\phi_\mu : L^\infty(A, \eta) \rightarrow \Har(G, \mu) $ est une contraction positive linéaire $G$-équivariante. On appelle cette application \textit{la transformée de Poisson} associée à $\mu$. Le \textit{bord de Poisson-Furstenberg} est l'unique $(G,\mu)$-espace (à isomorphisme mesurable près) pour lequel $\phi_\mu$ est une isométrie surjective.  
	
	Soit donc $\mu \in \prob(G) $ une probabilité sur un groupe localement compact à base dénombrable $G$. Dans ce paragraphe, on suppose pour simplifier l'argument que $\mu$ est équivalente à la mesure de Haar sur $G$, mais tout ce qui suit reste valable si on suppose seulement $\mu$ \emph{étalée} (c'est-à-dire telle que $\mu^n$ est non-singulière par rapport à Haar pour un certain $n$) et \emph{non-dégénérée}, i.e. telle que tout sous-groupe fermé $S$ tel que $\mu(S) = 1$ est $G$, voir \cite{kaimanovich03}. On rappelle que $(\Omega,\mathbb{P}) = (G^\mathbb{N}, \mu^{\otimes \mathbb{N}})$ est l'espace des incréments sur $G$, munie de la tribu borélienne produit. On se donne l'action suivante de $G$ sur $\Omega$ : pour $g \in G$, pour $\omega \in \Omega$, 
	\begin{eqnarray}
		g. (\omega_1, \omega_2, \dots) & =& (g\omega_1, \omega_2, \dots). \nonumber
	\end{eqnarray}
	Cette action est non-singulière car $g_\ast \mathbb{P} = g_\ast \mu \otimes \bigotimes_{n=1}^\infty \mu$ et $g_\ast \mu \sim \mu$ car $\mu \sim \haar_G$. De même, on considère l'action de $\mathbb{N}$ sur $\Omega$ par
	\begin{eqnarray}
		T : \omega  = (\omega_1, \omega_2, \dots) \in \Omega & \mapsto & (\omega_1\omega_2, \omega_3, \dots) \in\Omega \nonumber
	\end{eqnarray}
	L'action de $T$ sur $\Omega$ est également non-singulière. Il est clair que l'action de $G$ et celle de $\mathbb{N}$ commutent. Considérons maintenant l'application 
	\begin{eqnarray}
		\alpha : G \times \Omega &\rightarrow &G^\mathbb{N} \nonumber \\
		(g, \omega) & \mapsto & (g Z_n(\omega))_{n\in \mathbb{N}} \nonumber 
	\end{eqnarray}
	Munissons $G \times \Omega $ de la mesure $(\haar_G \otimes \mathbb{P})$, et soit $\tilde{m}$ la mesure image $\alpha_\ast (\haar \otimes \mathbb{P})$ sur $G^\mathbb{N} $. Quitte à prendre une mesure dans la classe d'équivalence de $\tilde{m}$, on peut la supposer finie. On note l'espace image $\tilde{\Omega} := G^\mathbb{N} $ muni de $\tilde{m}$ et on l'appelle \textit{l'espace des trajectoires} de la marche aléatoire : ses éléments sont les marches aléatoires qui partent d'un point $g_0 \in G$, distribué selon $\haar_G$. L'espace $(\tilde{\Omega}, \tilde{m} ) $ est un espace de Lebesgue standard, sur lequel $\mathbb{N}$ agit par la translation de Bernoulli $S : \tilde{\Omega} \rightarrow \tilde{\Omega}$ définie par  
	\begin{eqnarray}
		S (Z_n)_{n\in \mathbb{N}} = (Z_{n+1})_{n\in \mathbb{N}} \nonumber. 
	\end{eqnarray}
	
	Cette action est non-singulière (c'est l'image par $\alpha$ de l'action naturelle de $T$ sur $\Omega$). De même, $G$ agit sur $\tilde{\Omega}$ par $g \cdot (Z_n)_{n \in \mathbb{N}} = (gZ_n)$, et les actions de $G$ et $S$ commutent. Le bord de Poisson-Furstenberg $(B(G, \mu), \nu)$ est l'espace des composants ergodiques de l'action de $\mathbb{N}$ sur $\tilde{\Omega}$. Plus précisément, considérons la relation d'équivalence $\sim$ sur $\tilde{\Omega}$ 
	\begin{eqnarray}
		(Z_n)_{n \in \mathbb{N}} \sim (Z'_n)_{n \in \mathbb{N}} \text{ s'il existe $m,k_0 \in \mathbb{N} $ tels que } S^m\cdot (Z_n)_{n \geq k_0} = (Z'_n)_{n\geq k_0}  \nonumber.
	\end{eqnarray} 
	
	On considère la tribu $\mathcal{A}_S$ définie de la manière suivante. Un ensemble mesurable $A \subseteq \tilde{\Omega}$ est dit $S$-stationnaire (à ensemble négligeable près) si pour $\tilde{m}$-presque toute trajectoire $ x \in A$, $A$ contient également tous les éléments de la classe de $x$ pour la relation d'équivalence $\sim$. La tribu $\mathcal{A}_S $ est la tribu engendrée par les ensembles $S$-stationnaires (à mesure nulle près). Cette tribu forme une $\sigma$-algèbre de Boole standard. On peut donc appliquer le théorème de réalisation de Mackey \cite[Section 3]{mackey1962} : il existe un espace probabilisé standard $(B = B(G, \mu), \nu)$ munie d'une action non-singulière de $G$ et qui réalise l'espace quotient de $\tilde{\Omega}$ par la partition mesurable des ensembles stationnaires $\mathcal{A}_S$. Cet espace est de plus essentiellement unique \cite[Section 5]{mackey1962}. De plus, il existe une application mesurable canonique 
	\begin{eqnarray}
		\bnd : \tilde{\Omega} \rightarrow B \nonumber
	\end{eqnarray}
	telle que $\nu = \bnd_\ast \tilde{m}$ est $\mu$-stationnaire.  
	
	\begin{Def}
		L'espace probabilisé $B(G, \mu)$ est le \textit{bord de Poisson-Furstenberg} associé à $(G,\mu)$. 
	\end{Def}
	
	Le bord de Poisson décrit en quelque sorte le comportement asymptotique de la marche aléatoire associée à $\mu$ sur $G$. 
	
	Le théorème suivant vient de \cite{furstenberg63}. Il montre que le bord de Poisson représente les fonctions $\mu$-harmoniques sur $G$. On peut en trouver une démonstration dans \cite[Theorem 3.16]{houdayer21}.
	\begin{thm}
		La transformée de Poisson associée à $\mu$ 
		\begin{eqnarray}
			\phi_\mu : L^\infty (B, \nu) &\rightarrow &\Har(G, \mu) \nonumber \\
			f & \mapsto & \phi_\mu(f) : g \in G \mapsto \int_B f(gx) d\nu(x) \nonumber
		\end{eqnarray}
		est une isométrie linéaire $G$-équivariante et surjective. 
	\end{thm}
	
	Donnons une première propriété d'ergodicité du bord de Poisson. On verra par $B(G, \mu)$ satisfait des conditions d'ergodicité beaucoup plus fortes, notamment grâce aux travaux de Kaimanovich \cite{kaimanovich03}, et Bader-Furman \cite{bader_furman14}. On rappelle que l'action non-singulière d'un groupe $G$ sur un espace mesurable $(S,\eta)$ est \textit{ergodique} si pour toute fonction mesurable $G$-invariante $f : S \rightarrow \mathbb{C}$ est $\eta$-presque partout constante. De même, l'action non-singulière d'un groupe $G$ sur $(S,\eta)$ est \textit{doublement ergodique} si son action diagonale sur $(S \times S, \eta \times \eta)$ est ergodique.
	
	\begin{cor}
		L'action de $G$ sur son bord de Poisson-Furstenberg $(B(G,\mu), \nu)$ est ergodique. 
	\end{cor}
	\begin{proof}
		Soit $F \subseteq B$ un ensemble mesurable $G$-invariant. Alors si on note $\mathds{1}_F$ la fonction caractéristique de $F$, $\phi_\mu(\mathds{1}_F)$ est une application $\mu$-harmonique constante égale à $\nu(F)$. Par injectivité de l'application $\phi_\mu$, $\mathds{1}_F = \nu(F) \mathds{1}_B$, ce qui implique que $\nu (F) = 0 $ ou $\nu(F) =1$. 
	\end{proof}
	
	Le théorème suivant donne la propriété universelle que satisfait le bord de Poisson parmi les $(G, \mu)$-espaces. Une application mesurable $\pi : (S, \eta) \rightarrow (A, \beta)$ est un \textit{facteur} si $\pi_\ast \eta = \beta$. 
	
	\begin{thm}[{\cite[Theorem 3.1]{furstenberg63}}]\label{thm furst mu bord}
		Soit $(S, \eta)$ un $(G, \mu)$-espace compact et métrisable. Alors les propositions suivantes sont équivalentes. 
		\begin{enumerate}
			\item[i)] Les mesures limites $\eta_\omega = \lim Z_n(\omega)_\ast \eta$ données par le théorème \ref{thm mesures limites} sont $\mathbb{P}$-presque sûrement des mesures de Dirac $\delta_{\phi(\omega)}$, pour $\phi : \Omega \rightarrow S$ une application mesurable $G$-équivariante. 
			\item[ii)] Il existe un facteur mesurable $G$-équivariant $\pi : (B(G, \mu), \nu) \rightarrow (S, \eta)$
		\end{enumerate}
	\end{thm} 
	
	Si $(S, \eta)$ vérifie l'une des conditions équivalentes précédentes, on dit que c'est un \textit{$\mu$-bord} de $G$. 
	
	On termine cette section en construisant une première application de bord. Ce type de résultat fait référence aux théorème \ref{thm mesures limites} et Corollaire \ref{cor application de bord mesure stat}, et sera étendu dans la section suivante. 
	\begin{thm}[{\cite[Theorem 3.2]{furstenberg63}}]\label{thm bdry map = stat measure}
		Soit $(S, \eta)$ un $(G, \mu)$-espace compact métrisable. Alors il existe essentiellement une unique application mesurable $G$-équivariante $\phi : (B(G, \mu), \nu) \rightarrow \prob(S)$, qui vérifie la décomposition
		\begin{eqnarray}
			\eta = \int_B \phi(b) d\nu(b) \nonumber. 
		\end{eqnarray}
		Réciproquement, pour toute application mesurable $G$-équivariante $\phi' : (B(G, \mu), \nu) \rightarrow \prob(S)$, la mesure de probabilité sur $Y$
		\begin{eqnarray}
			\eta' = \int_B \phi'(b) d\nu(b) \nonumber. 
		\end{eqnarray}
		est $\mu$-stationnaire. 
	\end{thm}
	
	On conclut cette section par un lemme classique sur les mesures atomiques. On rappelle qu'une mesure $\nu \in \prob(Y)$ est \textit{atomique} s'il existe $\xi \in Y$ tel que $\nu(\{\xi\}) > 0$. On dit alors que $\xi$ est un \textit{atome} de $\nu$. 
	
	\begin{lem}\label{lem non atomic}
		Soit $G$ un groupe dénombrable discret et $\mu$ une mesure admissible sur $G$. Soit $Y$ un $G$-espace sans orbite finie et soit $\nu$ une mesure $\mu$-stationnaire sur $Y$. Alors $\nu$ est non-atomique. 
	\end{lem}
	\begin{proof}
		Supposons qu'il existe un atome pour $\nu$. Soit $m :=  \max\{ \nu(x) : x \in Y\}$ la masse maximale d'un atome est $Y_m = \{ x \in \overline{X} : \nu (x) = m\}$. L'ensemble $Y_m $ est non-vide par hypothèse, et fini car $\nu(Y) = 1$.  
		Soit $x \in Y_m$. Puisque $\nu $ est $\mu$-stationnaire,  $\mu \ast \nu (x ) = \nu (x)$. En d'autres termes, 
		\begin{equation}
			\sum_{g \in G} \mu(g) \nu(g^{-1} x ) = m \nonumber. 
		\end{equation}
		Ainsi, pour tout $g \in \supp(\mu)$, $\nu(g^{-1} x) = m$. Puisque $\mu$ est admissible, son support engendre $G$ comme semi-groupe, donc $Y_m $ est $G$-invariant, fini et non-vide, ce qui contredit l'hypothèse que $G$ n'a pas d'orbite finie.
	\end{proof}
	
	\subsection{Actions moyennables}\label{section action moyennable}
	Dans cette section, on revient sur la notion d'action moyennable au sens de Zimmer \cite{zimmer84}, qui se définit comme une propriété de points fixes pour les actions affines. On rappelle la caractérisation suivante de la moyennabilité. Soit $G$ un groupe localement compact à base dénombrable, et soit $E$ un espace de Banach séparable. Soit $G \curvearrowright E$ une action continue (pour la topologie forte des opérateurs) par isométries affines. On munit son dual topologique $E^\ast $ de la topologie faible-$\ast$, et de l'action naturelle de $G$. Si $C \subseteq E^\ast_1$ est un ensemble compact convexe $G$-invariant dans la boule unité de $E^\ast$, on dit que l'action de $G$ sur $C$ est une \textit{action affine continue}. En particulier, si $G$ agit continûment par isométries sur un compact métrisable $(Y,d)$, alors l'action naturelle de $G$ sur $\prob(Y)$ est une action continue affine. 
	\begin{Def}
		Un groupe $G$ est \textit{moyennable} si toute $G$-action affine continue sur un compact convexe non vide $C$ admet un $G$-point fixe dans $C$. 
	\end{Def} 
	
	En se fondant sur cette caractérisation, on donne la définition d'une action moyennable. 
	Soit $(S, \eta)$ un $G$-espace, et soit $E$ un espace de Banach séparable. On se donne un cocycle $\alpha : G \times S \rightarrow \iso(E)$, à valeurs dans les isométries affines de $E$, ce qui signifie que pour tout $g,h \in G$, pour $\eta$-presque tout $s \in S$, 
	\begin{eqnarray}
		\alpha(gh, s) = \alpha(g,hs)\alpha(h, s) \nonumber
	\end{eqnarray}
	Soit maintenant $\{A_s\}_{s\in S}$ une famille de compacts convexes de $E^\ast_1$ telle que $\{(s, A_s)\}_{s\in S}$ est un borélien de $S \times E^\ast_1$. On suppose enfin que pour tout $(g, s) \in G \times S$, $\alpha(g,s)A_s \subseteq A_{gs}$. On dira que $\mathbb{S} =\{A_s\}_{s\in S}$ est un \textit{$G$-fibré affine} au-dessus de $S$ (pour l'action de $\alpha$). Une \textit{section mesurable $G$-équivariante} $\phi : S \rightarrow \mathbb{S}$ est une application mesurable qui satisfait la propriété naturelle : pour tout $g \in G$, presque tout $s \in S$, 
	\begin{eqnarray}
		\phi(gs)  = \alpha (g,s) \phi(s) \nonumber
	\end{eqnarray}
	\begin{Def}
		L'action de $G$ sur $(S,\eta)$  est \textit{moyennable} si pour tout $G$-fibré affine $\mathbb{S}$ au-dessus de $(S, \eta)$, il existe une section mesurable $G$-équivariante $S\rightarrow \mathbb{S}$. 
	\end{Def}
	
	En pratique, il faut comprendre la définition précédente comme une généralisation de la propriété des applications de bord, dont on donne maintenant la définition. 
	
	\begin{Def}
		L'action de $G$ sur $(A,\eta)$ a la \textit{propriété des applications de bords} si pour toute action continue affine de $G$ sur un compact convexe $C$, il existe une application mesurable $G$-équivariante $S \rightarrow C$. 
	\end{Def}
	
	\begin{rem}
		Dans la définition précédente, on se ramène en fait aux fibrés affines de la forme $S \times C$, où $C$ est un $G$-espace affine. 
	\end{rem}
	
	Si $G$ est un groupe moyennable, alors toute action non-singulière $G \curvearrowright (S, \eta)$ est moyennable. Réciproquement, un groupe est moyennable si et seulement si son action sur un point est moyennable. La propriété suivante est fondamentale pour la suite. 
	
	\begin{prop}[Application de bord]\label{prop application de bord}
		Soit $(S, \eta)$ un $G$-espace de Lebesgue tel que $G \curvearrowright S$ est moyennable, et soit $(Y, d)$ un $G$-espace compact métrisable. Alors il existe une application mesurable $G$-équivariante $(S, \eta) \rightarrow \prob(Y)$. 
	\end{prop}
	
	\begin{proof}
		Soit $(S, \eta) $, $(Y, d)$ comme dans l'énoncé. On dispose d'un fibré affine (trivial) $S \times \prob(Y)$ au-dessus de $S$, munie de l'action (mesurable) diagonale de $G$. Par moyennabilité de l'action, il existe une section $G$-équivariante mesurable $S \rightarrow S \times \prob(Y) $, qui fournit l'application de bord souhaitée. 
	\end{proof}
	
	L'exemple fondamental d'une action moyennable est fournie par le bord de Poisson-Furstenberg d'un groupe $G$.  
	
	\begin{thm}[{\cite[Theorem 3.3]{zimmer78}}]
		Soit $G$ un groupe localement compact à base dénombrable, et soit $\mu$ une mesure admissible sur $G$. Alors l'action de $G$ sur son bord de Poisson $(B(G, \mu), \nu)$ est moyennable. 
	\end{thm}
	
	On conclut cette section en donnant une caractérisation importante des groupes moyennables, démontrée indépendamment par Kaimanovich-Vershik \cite{kaimanovich_vershik} et par Rosenblatt \cite{rosenblatt81}. 
	
	\begin{thm}
		Soit $G$ un groupe localement compact à base dénombrable. Alors $G$ est moyennable si et seulement si il existe une mesure de probabilité admissible $\mu \in \prob(G)$ pour laquelle le bord de Poisson $B(G, \mu)$ est trivial. En d'autres termes, $G$ est moyennable si et seulement si il existe une mesure de probabilité admissible $\mu \in \prob(G)$ pour laquelle les seules fonctions $\mu$-harmoniques sur $G$ sont constantes. 
	\end{thm}
	
	L'étude du lien entre les propriétés asymptotiques de groupe et la trivialité du bord de Poisson sont subtiles et riches. Par exemple, il existe des analogies fortes entre non-trivialité du bord et croissance exponentielle dans le groupe pour la distance des mots, voir par exemple \cite{kaimanovich_vershik} et plus récemment \cite{erschler_bartholdi17}.

	\section{Propriétés ergodiques des bords}\label{section bords ergodiques}
	
	Dans cette section, on décrit les propriétés ergodiques du bord de Poisson-Furstenberg. On verra que ces propriétés imposent des conditions de rigidité très fortes. Ce qui suit provient de résultats de Bader et Furman \cite{bader_furman14}, qui étendent des résultats et idées de Kaimanovich \cite{kaimanovich03}. L'article \cite[Section 4]{duchesne_lecureux_pozetti23}  et les notes de C. Houdayer \cite[Section 2]{houdayer23} offrent également une introduction détaillée à ces notions. 
	
	\subsection{Ergodicité métrique relative}
	Commençons cette section par une définition, qui renforce la notion classique d'ergodicité. Dans toute la suite, $G$ désigne un groupe localement compact, à base dénombrable et les espaces mesurables que l'on considère sont des espaces de Lebesgue standard. On rappelle qu'un $G$-espace mesurable est un espace de Lebesgue standard $(S, \eta)$ munie d'une action non-singulière de $G$. 
	
	\begin{Def}
		Soit $(S, \eta)$ un $G$-espace, où $\eta \in \prob(S)$ est une probabilité sur $S$. On dit que l'action de $G$ sur $S$ est \textit{isométriquement ergodique} si pour tout espace métrique séparable $(Y, d)$ muni d'une action par isométries de $G$, et pour toute application mesurable $G$-équivariante $f: S \rightarrow Y$, alors $f$ est essentiellement constante. En d'autres termes, il n'existe pas de $G$-application mesurable non triviale  
		\[
		\begin{tikzcd}
			f : (S, \eta ) \arrow[r, "G"]  & (Y, d). 
		\end{tikzcd}
		\]
		Si l'action diagonale de $G$ sur $(S \times S, \eta \times \eta)$ est isométriquement ergodique, on dit que $G \curvearrowright S$ est \textit{doublement isométriquement ergodique}. 
	\end{Def}
	
	Soit $\Map_G(S, Y)$ l'espace des classes d'équivalences des applications mesurables $G$-équivariantes $f : (S, \eta) \rightarrow (Y, d) $, avec identification $f \sim f'$ si $f$ et $f'$ sont presque partout égales. L'action de $G$ sur $(S, \eta)$ est donc isométriquement ergodique si pour tout espace métrique séparable $(Y, d)$ sur lequel $G$ agit par isométries, 
	\begin{eqnarray}
		\Map_G(S, Y) = \Map_G(\{\ast\}, Y). \nonumber
	\end{eqnarray}
	
	L'ergodicité isométrique est une propriété plus forte que l'ergodicité classique. En effet, soit $G \curvearrowright (S, \eta)$ une action isométriquement ergodique, et soit $f : S \rightarrow [0,1]$ une application $G$-invariante. Prenons l'action triviale de $G$ sur $[0,1]$, qui est isométrique, et qui fait de $f$ une application $G$-équivariante, donc constante par ergodicité isométrique. En revanche, la propriété d'isométrie ergodique est plus faible que celle de double ergodicité. 
	
	\begin{prop}
		Soit $G \curvearrowright (S, \eta)$ une action doublement ergodique. Alors $G~\curvearrowright~(S, \eta)$ est isométriquement ergodique. 
	\end{prop}
	\begin{proof}
		Soit $G \curvearrowright (Y,d)$ une action isométrique sur un espace métrique séparable, et soit $f: S \rightarrow Y$ une application $G$ équivariante. Soit l'application $G$-invariante 
		
		\begin{eqnarray}
			\phi : S \times S &\longrightarrow  & \mathbb{R} \nonumber \\
			(a,b) & \longmapsto & d(f(a),f(b)), \nonumber
		\end{eqnarray}
		qui est essentiellement constante par double ergodicité. Supposons que $\phi = \alpha > 0$ presque sûrement. Alors il existe $y \in Y$ tel que $U:= f^{-1}(B(y, \alpha/2))$ vérifie $\eta(U) > 0$. Dans ce cas, pour tout $(a,b) \in  U \times U $, $d(f(a), f(b)) < \alpha$ par inégalité triangulaire, une contradiction. En conclusion, par le théorème de Fubini, on peut choisir $a \in S$ tel que pour $\eta$-presque tout $b \in S$, $f(a ) = f(b)$, et on a prouvé l'ergodicité isométrique. 
	\end{proof}
	
	Le théorème suivant est dû à Kaimanovich. 
	
	\begin{thm}[{\cite[Theorem 3]{kaimanovich03}}]\label{thm double ergod}
		Soit $G$ un groupe localement compact à base dénombrable, et soit $\mu$ une mesure de probabilité symétrique admissible sur $G$. Alors l'action de $G$ sur le bord de Poisson-Furstenberg $(B(G, \mu), \nu) $ est doublement ergodique.
	\end{thm}
	
	En réalité, le bord de Poisson-Furstenberg satisfait des propriétés ergodiques plus fortes que la double ergodicité. C'est ce qu'on rend explicite maintenant. 
	
	Soit $\mathbb{X}$ et $A$ deux $G$-espaces de Lebesgue standard, et soit $p : \mathbb{X} \rightarrow A$ une application mesurable $G$-équivariante. On dénote par $\mathbb{X} \times_p \mathbb{X} $ le \textit{produit fibré au-dessus de $p $}, défini par $\mathbb{X} \times_p \mathbb{X} := \{(x, y ) \in \mathbb{X} \times \mathbb{X} \, | \, p(x ) = p(y)\} $. Le produit fibré $\mathbb{X} \times_p \mathbb{X}$ est naturellement muni d'une structure d'espace borélien issue de $\mathbb{X} \times \mathbb{X}$. Les mesurables non vides $p^{ -1}(a)$, $a\in A$ sont appelés les fibres de $p$, et on les notera $\mathbb{X}_a \subseteq \mathbb{X}$. 
	\begin{Def}
		On dit que $p  : \mathbb{X} \rightarrow A$ admet une \textit{action fibrée isométrique de $G$} s'il existe une application borélienne $G$-invariante $d : \mathbb{X}\times_p \mathbb{X} \rightarrow \R $ telle que pour toute fibre $\mathbb{X}_a \subseteq \mathbb{X}$, $d_a $ est une distance séparable sur $\mathbb{X}_a$. On appellera $d$ la métrique sur le fibré $\mathbb{X} \times_p \mathbb{X}$. 
	\end{Def}

	Par définition, l'action d'un élément $g \in G$ sur le fibré $\mathbb{X} \times_p \mathbb{X}$ induit des isométries $g : \mathbb{X}_a \rightarrow \mathbb{X}_{ga}$, qui vérifient donc $d_{ga}(gx, gy) = d_{a}(x,y)$. 
	
	\begin{Def}
		Une application mesurable $\pi : S \rightarrow Q $ entre $G$-espaces de Lebesgue est dite \textit{relativement isométriquement ergodique} si pour toute $G$-action fibrée isométrique sur une $G$-application mesurable $p : \mathbb{X} \rightarrow A$ telle qu'il existe des applications $G$-équivariantes $f : S \rightarrow \mathbb{X}$ et $f_0 : Q\rightarrow  A$, telles que $p \circ f = f_0 \circ \pi$, il existe une $G$-application mesurable $\phi : Q  \rightarrow \mathbb{X} $ telle que presque sûrement, $f = \phi \circ \pi$ et $f_0 = p \circ \phi$. En d'autres termes, si le diagramme suivant commute 
		\[
		\begin{tikzcd}
			S \arrow[d, " \pi" ] \arrow[r,  " f"] & \mathbb{X} \arrow[d, " p" ] \\
			Q \arrow[r,  " f_0"] & A
		\end{tikzcd}
		\]

		il y a une application mesurable $G$-équivariante $\phi : Q  \rightarrow \mathbb{X} $ telle que le diagramme suivant commute.
		\[
		\begin{tikzcd}
			S \arrow[d, " \pi" ] \arrow[r,  " f"] & \mathbb{X} \arrow[d, " p" ] \\
			Q \arrow[r,  " f_0"] \arrow[ur, dashed, "\phi"] & A
		\end{tikzcd}
		\]
		Dans la suite, on dira que $\pi $ satisfait la propriété d'ergodicité isométrique relative (RIE). 
	\end{Def}
	
	\begin{rem}
		Soit $(S, \eta)$ un $G$-espace de Lebesgue tel que la projection sur le premier facteur $\pi_1 : S \times S \rightarrow S$  est relativement isométriquement ergodique, alors l'action de $G$ sur $S$ est isométriquement ergodique. En effet, soit $f : S \rightarrow (Y,d) $ une application mesurable $G$-équivariante vers un espace métrique séparable $(Y,d)$ sur lequel $G$ agit par isométries. Alors soit $\tilde{f} : (b,b') \in B \times B \mapsto f(b') \in Y$ est une $G$-application mesurable, et on peut construire un fibré trivial 
		\[
		\begin{tikzcd}
			S \times S \arrow[d, " \pi_1" ] \arrow[r,  " f"] & Y \arrow[d, " p" ] \\
			S\arrow[r] & \{\ast\}
		\end{tikzcd}
		\]
		Par RIE, il existe une application mesurable $G$-équivariante $\phi : S \rightarrow Y$ telle que pour presque tout $(b,b')$, $\phi(b) = f(b')$. Par Fubini, cela revient à dire que $f$ est essentiellement constante. 
	\end{rem}
	
	La définition qui suit est fondamentale. 
	
	\begin{Def}
		Soit $G$ un groupe localement compact à base dénombrable et soit $(S_-, S_+) $ deux $G$-espace de Lebesgue. On dit que $(S_-, S_+) $ forme une \textit{$G$-paire de bords} si
		\begin{enumerate}
			\item les actions de $G$ sur $S_-$ et sur $S_+$ sont moyennables ;
			\item les projections sur chaque facteur $\pi_- : S_- \times S_+ \rightarrow S_{-}$ et  $\pi_+ : S_- \times S_+ \rightarrow S_{+}$ sont relativement isométriquement ergodiques. 
		\end{enumerate} 
		Un $G$-espace $(S, \eta)$ est un \textit{$G$-bord} si $(S, S)$ est une $G$-paire de bords. 
	\end{Def}
	
	Soit $\mu \in \prob(G)$, on dénote par $\mui$ la mesure de probabilité donnée par $\mui  = \iota_\ast \mu$, où $\iota (g) = g^{-1} $. Tout groupe $G$ localement compact à base dénombrable admet un $G$-bord, comme le montre le théorème suivant. 
	
	\begin{thm}[{\cite[Theorem 2.7]{bader_furman14}}]\label{thm paire de bords}
		Soit $G$ un groupe localement compact à base dénombrable, et soit $\mu \in \prob(G)$ une mesure de probabilité admissible sur $G$. Notons $(B_+,\nu_+)$ et $(B_-,\nu_-)$ les bords de Poisson-Furstenberg associés à $\mu$ et $\mui$ respectivement. Alors $(B_+, B_-)$ est une $G$-paire de bords. Si on suppose de plus $\mu$ symétrique, alors $B(G, \mu)$ est un $G$-bord.
	\end{thm}
	
	\subsection{Rigidité et bords}
	
	La théorie des $G$-bords permet de prouver des résultats de rigidité pour le groupe $G$, parmi lesquels l'existence d'applications de Furstenberg. On rappelle qu'une action de groupe $G\curvearrowright (X,d)$ sur un espace $\cat$(0) est élémentaire s'il existe un plat invariant (possiblement dégénéré en un point de $X$) ou un point fixe à l'infini $\xi \in \bd X$. Le théorème suivant a été démontré par Bader, Duchesne et Lécureux, et montre l'existence d'applications de bords dans le cadre d'actions $\cat$(0) non-élémentaires. 
	
	\begin{thm}[{\cite[Theorem 1.1]{bader_duchesne_lecureux16}}]\label{thm furst map}
		Soit $X$ un espace $\cat $(0) de dimension télescopique finie et soit $G$ un groupe localement compact agissant par isométries et non-élémentairement sur $X$. Soit $(B, \nu)$ un $G$-bord, alors il existe une application mesurable $G$-équivariante $B \rightarrow \bd X$. 
	\end{thm}
	
	On utilisera ce théorème dans le Chapitre \ref{chapter rw immeuble} afin de construire des mesures stationnaires sur le bord dans le contexte d'une action de groupe sur un immeuble affine. 
	
	Dans la Section \ref{section rw hyp}, on construira une application de bord $B \rightarrow \bdg X$ dans le cadre d'une action de groupe $G \curvearrowright X$ sur un espace Gromov-hyperbolique. On prouvera de plus qu'une telle application est essentiellement unique.

	\section{Marches aléatoires en courbure non-positive}\label{section rw courbure non positive}
	
	Lors de l'étude de marches aléatoires en présence d'éléments contractants sur des espaces $\cat (0)$, on sera souvent amené à identifier et profiter de comportements hyperboliques. Dans cette section, on revient sur quelques résultats importants de lois limites dans des espaces hyperboliques ou à courbure non-positive. La théorie est très riche, et nous ne prétendons pas en faire un inventaire exhaustif. 
	
	\subsection{Convergence de la marche aléatoire}
	
	Soit $G$ un groupe dénombrable agissant par isométries sur un espace Gromov-hyperbolique séparable. On suppose que l'action est non-élémentaire, c'est-à-dire qu'il existe une paire d'isométries loxodromiques $g_1, g_2 \in G$. Soit $\mu$ une mesure admissible sur $G$. Dans cette section, on résume quelques résultats de lois limites sur la marche aléatoire $(Z_n o )_{n \in \mathbb{N}}$ dans ce contexte.

	Le premier résultat que l'on mentionne est que la marche aléatoire converge vers un point du bord. Notons que cette convergence était connue dans le contexte d'un espace hyperbolique propre, alors qu'ici Maher et Tiozzo montrent la convergence en toute généralité, et sans hypothèse de premier moment sur la mesure $\mu$. 
	
	\begin{thm}[{\cite[Theorem 1.1]{maher_tiozzo18}}] \label{thm maher tiozzo conv}
		Soit $G$ un groupe dénombrable et $X$ un espace hyperbolique séparable sur lequel $G$ agit par isométries, de manière non-élémentaire. Soit $\mu$ une probabilité admissible sur $G$, et $o \in Y$ un point-base. Alors la marche aléatoire $(Z_n (\omega)o )_n $ induite par $\mu$ converge presque sûrement vers $z^{+} (\omega)\in \bd X$, et la mesure limite $\nu$ est l'unique mesure $\mu$-stationnaire sur le bord $\bd X$. 
	\end{thm}

	\subsection{Vitesse de fuite}\label{section drift}
	Puisque d'après le théorème \ref{thm maher tiozzo conv}, on sait que la marche aléatoire converge vers le bord, la question suivante est de déterminer la vitesse moyenne à laquelle cette convergence a lieu. En d'autres termes, on cherche à estimer la limite de $\frac{1}{n}d(Z_n o, o)$ lorsque $n \rightarrow \infty$. A priori, cette limite n'est pas définie, mais dans l'hypothèse où $\mu $ a un premier moment fini $\int d(go, o) d\mu(g) < \infty$, elle existe. 
	
	Soit $(\Omega, \mathbb{P})$ un espace probabilisé, et $T : \Omega \rightarrow \Omega$ une application mesurable pmp et ergodique. On se donne une fonction $a : \mathbb{N} \times \Omega \rightarrow \R $ qui vérifie la propriété de cocycle sous-additif : pour tout $n,m \in \mathbb{N}$, pour presque tout $\omega \in \Omega$, 
	\begin{eqnarray}
		a(n+m , \omega) \leq a(n, w) + a(m, T^n \omega).  \nonumber
	\end{eqnarray}
	On dit que $a$ est \textit{intégrable} si $\int_\Omega |a(1, \omega)| d\mathbb{P}(\omega) < \infty$. On définit alors
	\begin{eqnarray}
		\lambda := \inf_n \frac{1}{n} \int_\Omega a(n, \omega) d\mathbb{P}(\omega) \in [- \infty, \infty)\nonumber.
	\end{eqnarray}
	\begin{thm}[{Théorème sous-additif de Kingman, \cite{kingman68}}]\label{thm kingman}
		
		Pour $\mathbb{P}$-presque tout $\omega \in \Omega$, 
		\begin{eqnarray}
			\frac{1}{n}a(n,\omega) \underset{n\rightarrow\infty}{\longrightarrow} \lambda \nonumber. 
		\end{eqnarray}
		De plus, si $ \lambda > - \infty$, la convergence a lieu dans $L^1 (\Omega, \mathbb{P})$. 
	\end{thm}
	
	\begin{rem}
		Ce théorème a été grandement amélioré dans \cite[Theorem 1.1]{gouezel_karlsson20}, dans lequel les auteurs montrent que pour presque tout $\omega$, il existe une suite d'entiers $n_i(\omega) \rightarrow \infty$ tels que le cocycle $a(l, \omega)$ est proche (en un sens précis) d'être additif pour tout $l \leq n_i$. Néanmoins, nous n'utiliserons pas ce résultat. 
	\end{rem}
	Soit maintenant $G$ un groupe localement compact agissant par isométries sur un espace métrique $(X,d)$, et soit $\mu$ une probabilité admissible sur $G$. Soit $\Omega$ l'espace des incréments sur $G$ induits par $\mu$, et soit $o \in X$ un point-base. On définit l'application $a : \mathbb{N} \times \Omega \rightarrow \R $ par 
	\begin{eqnarray}
		a(n, \omega) = d(Z_n(\omega) o , o) \nonumber. 
	\end{eqnarray}
	Un rapide calcul montre que $a$ est un cocycle sous-additif. 
	\begin{Def}
		On dit que $\mu$ a un \textit{premier moment fini} si le cocycle $a$ est intégrable, i.e. 
		\begin{eqnarray}
			\int_G d(go, o) d\mu(g) < \infty \nonumber.
		\end{eqnarray}
	\end{Def}
	
	D'après le théorème sous-additif de Kingman, si $\mu$ a un premier moment fini, alors presque sûrement 
	\begin{eqnarray}
		\frac{1}{n}d(Z_n(\omega) o , o) \underset{n\rightarrow\infty}{\longrightarrow} \lambda:= \inf_n \frac{1}{n} \int d(Z_n(\omega) o, o) d\mathbb{P}(\omega) \in [0, \infty)\nonumber. 
	\end{eqnarray}
	On appelle la limite $\lambda$ la \textit{vitesse de fuite} (ou le \textit{drift}) de la marche aléatoire $(Z_n)$. Elle ne dépend pas du point-base $o \in X$. 
	Si $\mu $ a un premier moment infini, on posera $\lambda:= \infty$. 
	Dans le contexte d'une marche aléatoire dans un espace hyperbolique, Gouëzel montre que la vitesse de fuite est presque sûrement positive.
	
	\begin{thm}[{\cite[Theorem 1.1]{gouezel22}}]\label{thm gouezel drift}
		Soit $G$ un groupe dénombrable et $X$ un espace hyperbolique séparable sur lequel $G$ agit par isométries, de manière non-élémentaire. Soit $\mu$ une probabilité admissible sur $G$, et $o \in Y$ un point-base. Alors la vitesse de fuite $\lambda$ de la marche aléatoire est presque sûrement strictement positive. 
		
		De plus, pour tout $\lambda' < \lambda$, il existe $\kappa >0$ tel que 
		\begin{eqnarray}
			\mathbb{P}\big( \omega \in \Omega \, : \, d_Y(Z_n(\omega) o, o) \leq \lambda'n \big)< e^{-\kappa n }.
		\end{eqnarray}
	\end{thm}
	
	Il est à noter ici que ce théorème reste vrai si on suppose seulement $G$ localement compact, et sans l'hypothèse que $X$ soit géodésique. Ce résultat était connu de Maher et Tiozzo si $\mu$ a un premier moment fini \cite[Theorems 1.2 and 1.3]{maher_tiozzo18}. Pour prouver ce résultat, Gouëzel utilise une technique de pivots très puissante, et qui a été utilisée par la suite dans \cite{choi22a} pour prouver de nombreuses lois limites de marches aléatoires en présence d'isométries contractantes dans divers espaces métriques. 
	
	Ce théorème doit se comprendre comme une loi des grands nombres dans un contexte non-commutatif : presque sûrement, les variables aléatoires $\{d(o, Z_no )\} $ vérifient $$d(o, Z_no ) \sim n \lambda.$$ 
	
	\subsubsection{Dérive et convergence dans les espaces $\cat$(0)}
	Dans le cadre d'espaces $\cat (0)$, un résultat fondamental de convergence est dû à Karlsson et Margulis \cite{karlsson_margulis}. Soit $(X,d)$ un espace $\cat$(0) et $G$ un groupe agissant par isométries sur $X$. On se donne un espace probabilisé $(\Omega, \mathbb{P})$ et $T : \Omega \rightarrow \Omega$ une application ergodique pmp. Soit $f : \Omega \rightarrow G$ une application mesurable. Le cocycle ergodique $u : \mathbb{N} \times \Omega \rightarrow G$ est défini par 
	\begin{eqnarray}
		u(n, \omega ) = f(\omega) f(T\omega) \dots f(T^{n-1}\omega) \nonumber. 
	\end{eqnarray}
	On dit que le cocycle $u$ est \textit{intégrable} si 
	\begin{eqnarray}
		\int_\Omega d(f(\omega)o, o) d\mathbb{P}(\omega) < \infty. \nonumber
	\end{eqnarray}
	Si $u$ est intégrable, alors par le théorème de Kingman, la limite $\lim \frac{1}{n} d(u(n, \omega)o, o)$ existe presque sûrement : il s'agit de la vitesse de fuite du cocycle. 
	\begin{thm}[{\cite[Theorem 2.1]{karlsson_margulis}}]
		Soit $(\Omega, \mathbb{P}, T)$, $f : \Omega \rightarrow G$ et $(X,d)$ comme dans la situation précédente. On suppose que le cocycle $u(n, \omega)$ est intégrable, et que sa vitesse de fuite $\lambda$ est strictement positive. Alors $\mathbb{P}$ presque tout $\omega$, il existe un unique rayon géodésique $\gamma^\omega$dans $X$ issu de $o $ tel que 
		\begin{eqnarray}
			\lim_{n \rightarrow \infty} \frac{1}{n} d(u(n, \omega) o ,  \gamma^\omega (\lambda n)) = 0. \nonumber
		\end{eqnarray} 
	\end{thm}
	Ce théorème implique en particulier la convergence du cocycle ergodique vers le bord visuel de $X$. Il permet en plus d'affirmer qu'il existe un rayon qui approxime sous-linéairement $u(n, \omega)$. En fait, ce théorème reste vrai dans une classe d'espaces qui contient les espaces $\cat$(0) (tout espace uniformément convexe complet satisfaisant l'inégalité de Busemann), et si $G$ est seulement un semigroupe d'applications non expansives, i.e. tel que pour tout $g \in G$, $x,y \in X$, 
	\begin{eqnarray}
		d(gx, gy) \leq d(x, y) \nonumber. 
	\end{eqnarray}
	
	La condition de positivité de la vitesse de fuite est ici essentielle. Néanmoins, H.~Izeki a récemment démontré dans \cite{izeki22} que sous une certaines conditions de moments, la dérive ne peut être nulle que si l'action du groupe $G$ est élémentaire. On rappelle qu'un plat dans un espace $\cat (0)$ est un sous-ensemble convexe de $X$ qui est isométrique à $\R^n $, $n \geq 0$. 
	
	\begin{thm}[{\cite[Theorem A]{izeki22}}]
		Soit $(X,d)$ un espace $\cat$(0) propre ou de dimension télescopique finie, et $G$ un groupe de type fini agissant sur $X$ par isométries sans point fixe dans $\bd X$. Soit $\mu \in \prob(G)$ une mesure de probabilité sur $G$ de second moment fini, et soit $\lambda := \lim \frac{1}{n} d(Z_n o, o)$ la vitesse de fuite de la marche aléatoire induite par $\mu $ sur $X$. Si $\lambda = 0$, alors il existe un plat $G$-invariant dans $X$. 
	\end{thm}

	\subsection{Théorème de la limite centrale}
	
	Soit $G$ un groupe localement compact agissant par isométries sur un espace métrique $(X,d)$, $\mu$ une probabilité admissible sur $G$ avec un premier moment fini, et $o \in Y$ un point-base. D'après la section précédente, la marche aléatoire $(Z_n o )_n$ vérifie une loi des grands nombres. Dans \cite{benoist_quint16}, Benoist et Quint montrent que si $X$ est un espace hyperbolique propre, alors sous des conditions naturelles, $(Z_n o )_n$ satisfait un théorème central limite. On dit que $\mu \in \prob(G)$ a un second moment fini si 
	\begin{eqnarray}
		\int_G d(go, o)^2 d\mu(g) < \infty \nonumber 
	\end{eqnarray}
	
	\begin{thm}[{\cite[Theorem 1.1]{benoist_quint16}}]
		Soit $G$ un groupe dénombrable et $X$ un espace hyperbolique propre séparable sur lequel $G$ agit par isométries, de manière non-élémentaire. Soit $\mu$ une probabilité admissible sur $G$ de second moment fini, et $o \in Y$ un point-base. Soit $\lambda $ la vitesse de fuite de $(Z_n)_n$ dans $X$ (strictement positive grâce au théorème \ref{thm gouezel drift}). 
		
		Alors les variables aléatoires $\frac{1}{\sqrt{n}}(d(Z_n o, o) - n \lambda) $ convergent en loi vers une distribution gaussienne non-dégénérée $N_\mu$. 
	\end{thm}
	
	Ce résultat était connu de Björklund \cite{bjorklund09} si $\mu$ a un moment exponentiel fini. Ici, l'hypothèse de moment est considérablement affaiblie. La stratégie de Benoist et Quint repose sur l'étude d'un certain cocycle (le cocycle de Busemann), dont les propriétés asymptotiques sont comparables à celles des variables aléatoires $d(Z_n o, o)$. En améliorant un résultat de Brown \cite{brown71}, les auteurs montrent que sous certaines conditions, le cocycle satisfait un théorème central limite, \cite[Theorem 3.4]{benoist_quint16CLTlineargroups}. La clef est évidemment de montrer que les conditions sont réunies.  
	\newline
	
	Récemment, I. Choi a adapté les techniques de pivots de Gouëzel \cite{gouezel22} pour étudier de nombreuses lois limites dans des contextes divers. Les hypothèses étaient les suivantes : 
	
	\begin{enumerate}
		\item $(X,d) $ est un espace métrique géodésique; 
		\item $G$ est un groupe dénombrable d'isométries de $X$;
		\item $G$ contient deux isométries contractantes indépendantes, au sens de la Définition \ref{def isom indé}. 
	\end{enumerate}
	
	Parmi d'autres lois limites, Choi démontre dans ce contexte un théorème de la limite centrale et une loi du logarithme itéré. 
	
	\begin{thm}[{\cite[Theorem C]{choi22a}}]
		Soit $(X,G)$ satisfaisant les hypothèses précédentes, et soit $(Z_n)$ la marche aléatoire induite par une mesure de probabilité admissible sur $G$. Supposons que $\mu$ a un second moment fini et notons $o \in X$ un point-base, et $\lambda$ la dérive de $(Z_n) $ sur $X$. Alors les variables aléatoires $\frac{1}{\sqrt{n}}(d(Z_n o, o) - n \lambda) $ convergent en loi vers une distribution gaussienne non-dégénérée $N(0,\sigma(\mu)^2)$, où $\sigma(\mu)^2$ est la variance de la loi gaussienne limite. De plus, $(Z_n)_n$ satisfait la loi des logarithmes itérés : presque sûrement, 
		\begin{eqnarray}
			\limsup_{n \rightarrow \infty} \pm \frac{d(Z_n o, o) - n \lambda}{\sqrt{2 n \log \log n }} = \sigma(\mu) \nonumber. 
		\end{eqnarray} 
	\end{thm}
	
	La stratégie de preuve s'inspire beaucoup des idées de \cite{boulanger_mathieu_sisto21} et \cite{mathieu_sisto20}. Elle consiste à étudier des inégalités de déviations locales, et à appliquer des techniques de pivots de Gouëzel. Dans la Section \ref{section clt cat}, on montre un résultat similaire pour une action sur un espace $\cat$(0) complet, en adoptant une approche entièrement différente.

	\subsection{Mesure limite}
	
	Soit $G$ un groupe localement compact agissant par isométries sur un espace métrique $(X,d)$, $\mu$ une probabilité admissible sur $G$, et $o \in Y$ un point-base. Si la marche aléatoire $(Z_n o )_n$ induite par $\mu$ converge presque sûrement vers le bord de $X$, on peut munir $\bd X$ de cette distribution limite $\nu$ : pour un borélien $F \subseteq \bd X$ 
	\begin{eqnarray}
		\nu (F) = \mathbb{P}\big(\{\omega \in \Omega \, | \, Z_n(\omega ) o \underset{n}{\rightarrow} \xi \in F \}\big) \nonumber. 
	\end{eqnarray} 
	Une autre manière de voir cette mesure est la suivante. Pour presque tout $\omega \in \Omega$, notons $\tilde{\phi}(\omega) = \lim_n Z_n(\omega) o  \in \bd X$. Par convergence de la marche aléatoire, cette fonction est mesurable, définie presque partout et $G$-équivariante. De plus, par construction du bord de Poisson-Furstenberg, la limite de la marche aléatoire est $T$-invariante, et on obtient une application mesurable $G$-équivariante du bord de Poisson-Furstenberg $(B = B(G, \mu), \nu_B)$ vers le bord visuel $\bd X $
	\begin{eqnarray}
		\phi : B &\longrightarrow& \bd X \nonumber \\
		b &\longmapsto& \lim_n Z_n(b) o \nonumber.
	\end{eqnarray}
	La \textit{mesure limite} $\nu$ est donc définie par 
	\begin{eqnarray}
		\nu = \int_B \delta_{\phi(b)} d\nu_B(b) \nonumber. 
	\end{eqnarray}
	
	Cette mesure est $\mu$-stationnaire car issue d'une application $G$-équivariante. En particulier, $(\bd X, \nu)$ est un $\mu$-bord, et en vue du théorème \ref{thm furst mu bord}, on peut se demander si $(\bd X, \nu)$ est isomorphe au bord de Poisson-Furstenberg de $(G, \mu)$. Dans cette section, on revient sur ce type de résultats concernant cette mesure limite, ainsi que sur la description de son support. 
	\newline
	
	Le premier résultat concerne le support de la mesure limite. Soit $G$ un groupe d'isométries non-élémentaire d'un espace $\cat$(0) propre et soit $\mu \in \prob(G)$ une mesure admissible. En présence d'éléments de rang 1 de nombreux résultats montrent que les points limites ne couvrent pas uniformément le bord de l'espace, seulement les \og secteurs hyperboliques \fg. Plus précisément, Gekhtman, Qing et Rafi montrent que dans dans certaines situations, la marche aléatoire converge presque sûrement vers un point du bord de Morse sous-linéaire. On dit que l'action de $G$ sur $X$ est \textit{tempérée} si le nombre de points dans une orbite de $G$ qui appartient à une boule donnée croît au plus exponentiellement avec son rayon.  
	
	\begin{thm}[{\cite[Theorem 1.2]{gekhtman_qing_rafi22}}]
		Soit $G$ un groupe de type fini agissant par isométries, proprement discontinument sur un espace $\cat$(0) propre et géodésiquement complet $(X,d)$, $\mu$ une probabilité de support fini et non-élémentaire sur $G$, et $o \in Y$ un point-base. On suppose de plus que l'action est tempérée et que $G$ possède un élément de rang 1. Alors la mesure limite $\nu$ de la marche aléatoire $(Z_n o)$ induite par $\mu$ donne une mesure pleine au bord de Morse sous-linéaire. 
	\end{thm}
	
	Plus récemment, Chawla, Forghani, Frisch et Tiozzo ont démontré dans des conditions très générales que le bord visuel muni de la mesure limite était bien une réalisation topologique du bord de Poisson-Furstenberg $B(G, \mu)$. Ce résultat fait suite à des théorèmes similaires dans des cadres moins généraux, voir \cite{kaimanovich00}, \cite{maher_tiozzo18}, \cite{karlsson_margulis}. Ce qui est particulièrement fort dans le théorème suivant est l'absence de conditions de moment. On rappelle que \textit{l'entropie} de $\mu$ sur un groupe dénombrable $G$ est définie par 
	\begin{eqnarray}
		H(\mu) := -\sum_G \log(\mu(g)) \mu(g) \nonumber. 
	\end{eqnarray}
	
	On renvoie à la Section \ref{section acylindrique cat} pour la définition d'un élément WPD. 
	\begin{thm}[{\cite{chawla_forghani_frisch_tiozzo22}}]
		Soit $G$ un groupe dénombrable agissant de manière non-élémentaire sur un espace géodésique hyperbolique $X$ avec un élément WPD. Soit $\mu \in \prob(G)$ une mesure de probabilité admissible sur $G$ d'entropie finie, et $o \in X$. Alors la mesure limite  $\nu$ de la marche aléatoire $(Z_n o ) $ sur $X$ fait de $(\bd X, \nu)$ une réalisation topologique du bord de Poisson-Furstenberg $B(G, \mu)$. 
	\end{thm}

	\part{Résultats de recherche}
	\selectlanguage{english}
	\mtcselectlanguage{english}
	\chapter{Random walks in complete $\cat$(0) spaces}\label{chapter rw cat}

	This chapter is dedicated to the study of a random walk in a separable and complete $\cat$(0) space in presence of contracting isometries. Many of the results presented here appeared in \cite{LeBars22} and \cite{le-bars22b}, although the $\cat$(0) space was assumed proper. 
	
	In Section \ref{section prelim hyp}, we review some definitions and basic properties of hyperbolic spaces in the sense of Gromov. Although these notions are probably well-known, the spaces we consider are not assumed proper nor geodesic, and it seemed useful to present them in the greatest generality. In Section \ref{section rw hyp}, we use the theory of boundary pairs and their ergodic properties to derive the convergence of the random walk to the Gromov boundary in a hyperbolic space. This result was already known to Maher and Tiozzo \cite{maher_tiozzo18} and to Gouëzel \cite{gouezel22}, but the techniques presented rely on boundary theory \cite{bader_caprace_furman_sisto22}. These could prove useful if we study the more general ergodic cocycles instead of random walks. Section \ref{section rw cat} is dedicated to studying a random walk in a complete $ \cat$(0) space. With the help of the hyperbolic models presented in Section \ref{section modèle hyp}, we are able to derive many interesting limit laws: convergence to the boundary, uniqueness of the stationary measure, positive drift and description of the limit points. In Section \ref{section clt cat}, we prove that under a finite second moment assumption, the random walk satisfies a central limit theorem. 
	\newline
	
	\minitoc
	\vspace{1cm}
	
	\section{Preliminaries on hyperbolic spaces}\label{section prelim hyp}
	In this section we recall some basic facts about hyperbolic spaces. All what follows is standard, but since most references deal about geodesic, and sometimes even proper, hyperbolic spaces, we will present this more general setting. References for geodesic hyperbolic spaces include \cite{ghys_dlharpe90}, \cite[Chapter III.H]{bridson_haefliger99}, \cite[Chapter 11]{drutu_kapovich18} and \cite{bonk_schramm00}, the last two dealing also with non-geodesic spaces. 
	
	\subsection{First Definitions} 
	
	Let $(X,d)$ be a metric space. Given $x,y,o \in X$, the Gromov product $x$ and $y$ with respect to $o $ is given by 
	\begin{eqnarray}
		(x|y)_o = \frac{1}{2} (d(w, o ) + d(y, o ) - d(x,y)) \nonumber. 
	\end{eqnarray}
	Note that if we take another basepoint $o' $, we have 
	\begin{eqnarray}
		|(x|y)_o - (x|y)_{o'} | \leq d(o, o') \nonumber. 
	\end{eqnarray}
	
	The following definition is due to Gromov \cite{gromov87}. 
	\begin{eDef}
		We say that the metric space $(X,d)$ is \textit{$\delta$-hyperbolic} for $ 0\leq \delta < \infty$, if for all $x,y,z, o \in X$, 
		\begin{eqnarray}
			(x|y)_o \geq \min\{(x|z)_o, (z|y)_o\} - \delta. \nonumber
		\end{eqnarray}
		If $X$ is $\delta$-hyperbolic for some $\delta$, we say that it is \textit{Gromov-hyperbolic}. 
	\end{eDef}
	
	Let $(X,d)$ be a Gromov-hyperbolic space. We say that a sequence of points $(x_i)_{i \in \mathbb{N}}$ \textit{converges at infinity } if 
	\begin{eqnarray}
		\lim_{i,j \rightarrow \infty} (x_i|x_j)_o =  \infty \nonumber, 
	\end{eqnarray}
	and that two sequences that converge at infinity are equivalent if 
	\begin{eqnarray}
		\lim_{i \rightarrow \infty} (x_i|y_i)_o =  \infty \nonumber. 
	\end{eqnarray}
	
	This is an equivalence relation, and this definition does not depend on the basepoint. The Gromov boundary $\bdg X$ of $X$ is then defined as the set of equivalence classes of sequences $(x_i)$ that converge at infinity, for this equivalence relation. We denote by $\XG= X \cup \bdg X$ this bordification.

	We extend the Gromov product to points of the boundary in the following way. For $\xi, \eta \in \bdg X$, 
	\begin{eqnarray}
		(\xi| \eta)_o := \sup \{ \liminf_{i\rightarrow \infty} (x_i| y_i)_o \; : \; \{x_i \}\in \xi, \{y_i \}\in \eta\}. \nonumber
	\end{eqnarray}
	
	Taking the supremum and $\liminf$ in this Definition is not the only choice. However, all four possible choices only differ by $2 \delta$. Note that $(\xi|\eta)_o = \infty$ if and only if $\xi = \eta \in \bdg X$.
	
	As a consequence of the definitions, if $(X,d)$ is a $\delta$-hyperbolic space, there exists a constant $C= C(\delta)$ depending only on $\delta$ such that for all $x, y , z \in X \cup \bdg X$, 
	
	\begin{eqnarray}
		(x|y)_o \geq \min\{(x|z)_o, (z|y)_o\} - C. \nonumber
	\end{eqnarray}
	
	The topology on $\bdg X $ is given by the basis of open sets 
	\begin{eqnarray}
		U_{\xi, R} = \{ x \in \XG \; : \; (\xi| x)_o > R\} \nonumber, 
	\end{eqnarray}
	so that $x_n \rightarrow \xi \in \bdg X$ if and only if $(\xi|x_n)_o \rightarrow \infty$. Moreover, if $(x_n)$ is a sequence in $X$, $(x_n)$ converges to $\xi \in \bdg X$ in this sense if and only if $(x_n) $ converges at infinity and the equivalence class $[(x_i)]$ of $(x_i)$ is $ \xi$, \cite[Lemma 11.101]{drutu_kapovich18}. For $Y\subseteq X$ a subset, we will write $\overline{Y}^{\text{Grom}}$ the closure of $Y$ in $\XG$ with this topology. 
	\newline 
	
	\subsection{Coarse geometry}
	The hyperbolic spaces we are going to consider are not geodesic, so let's introduce some terminology in coarse geometry. 
	
	Let $f : X \rightarrow Y$ be a map between metric spaces, and let $Q \geq 1, C \geq 0$. We say that $f$ is a \textit{$(Q,C)$-quasi-isometric embedding} if for all $x, y \in X$, 
	\begin{eqnarray}
		\frac{1}{Q} d_X(x, y) - C \leq d_Y(f(x) , f(y)) \leq Q d_X(x,y) + C \nonumber. 
	\end{eqnarray}
	If moreover, there exists $A>0$ such that every point $y \in Y$ is at distance less than $A$ from $f(X)$, we say that $f$ is a \textit{$(Q,C)$-quasi-isometry}. In the particular case where $I \subseteq \R$ is an interval, and $f : I \rightarrow (X,d)$ is a $(Q,C)$-quasi-isometric embedding, we say that $f$ (or $f(I)$) is a \textit{quasigeodesic path}. If $I = [a, b]$ is compact, resp. $I = [a, \infty)$, resp. $I = \R$, we say that $f$ is a \textit{quasigeodesic segment}, resp. \textit{quasigeodesic ray}, resp. \textit{quasigeodesic line}. 
	
	Last, we say that a space is \textit{$(Q,C)$-quasigeodesic} if for every $x,y \in X$, there is a $(Q,C)$-quasigeodesic segment $\gamma : [a, b] \rightarrow X$ that joins $x$ to $y$, i.e. such that $\gamma(a) = x$ and $\gamma(b)= y  $. 
	
	\begin{eDef}
		A metric space $(X,d)$ is \textit{$\alpha$-almost geodesic} if for every $ x, y \in X$, and every $t \in [0, d(x,y)]$, there is $z \in X $ such that 
		\begin{eqnarray}
			|d(x,z) - t | \leq \alpha \text{ and } |d(y, z ) - (d(x,y) - t)|\leq \alpha \nonumber.
		\end{eqnarray}
	\end{eDef}
	
	In almost geodesic hyperbolic spaces, we have the following. 
	
	\begin{eprop}[{\cite[Proposition 5.2]{bonk_schramm00}}]\label{prop exist qgeod hyp}
		Let $(X,d)$ be a $\delta$-hyperbolic, $\alpha$-almost geodesic metric space. Then there exists a universal constant $C= C(\delta, \alpha)$ such that:
		\begin{enumerate}
			\item \label{rough geodesic} for all $ x, y \in X$, there exists a $(1, C)$ quasigeodesic segment $\gamma : [a, b] \rightarrow X$ such that $\gamma(a) = x$ and $\gamma(b) = y$. 
			\item for all $ x\in X, y \in \bdg X$, there exists a $(1, C)$ quasigeodesic ray $\gamma : [0, \infty) \rightarrow X$ such that $\gamma(0) = x$ and $\lim_{t \rightarrow \infty} \gamma(t) = y$. 
			\item for all $ x, y\in \bdg X$, $x\neq y$, there exists a $(1, C)$ quasigeodesic line $\gamma : \R \rightarrow X$ such that $\lim_{t \rightarrow -\infty} \gamma(t)= x$ and $\lim_{t \rightarrow \infty} \gamma(t) = y$. 
		\end{enumerate}
	\end{eprop}
	
	A metric space satisfying the conclusion of \ref{rough geodesic} is said to be \textit{$C$-roughly geodesic}. The spaces we are going to deal with are $\delta$-hyperbolic and $\alpha$-almost geodesic, for some $\delta, \alpha \geq 0$. In particular, they will be quasigeodesic. 
	\newline
	
	In a geodesic $\delta$-hyperbolic space, there is a universal constant $C = C(\delta, Q, C)$ such that if $\gamma$ is a $(Q,C)$-quasigeodesic segment, then any geodesic segment between $x = \gamma(a)$ and $y = \gamma(b) $ lies in the $C$-neighbourhood of $\gamma$. This is called \textit{geodesic stability}. It turns out this results remains true for general hyperbolic spaces, thanks to the following theorem. 
	
	\begin{ethm}[{\cite[Theorem 4.1]{bonk_schramm00}}]\label{thm isom embed geod hyp space}
		Let $(X,d)$ be a $\delta$-hyperbolic space. Then there is an isometric embedding $ \iota : X \rightarrow Y $ of $X$ into a complete $\delta$-hyperbolic geodesic metric space.
	\end{ethm}
	
	The analogue of geodesic stability in the general context is then the following. 
	
	\begin{ecor}[{\cite[Proposition 5.4]{bonk_schramm00}}]\label{cor stab qgeod hyp}
		Let $(X,d)$ be a $\delta$-hyperbolic metric space and $\gamma_1, \gamma_2$ be $(Q,C)$-quasigeodesic paths in $X$ with the same endpoints. Then there exists $C' = C'(\delta, Q, C)$ such that $d_\text{Haus}(\gamma_1, \gamma_2) \leq C'$. 
	\end{ecor}

	We define the $R$-shadow of a point $x $ seen from $o \in X$ by 
	\begin{eqnarray}
		S_o (x, R) := \{ y \in X \; : \; (y|o)_x \leq R\}. \nonumber
	\end{eqnarray}
	
	There are several equivalent definitions of a shadow. For instance, Maher and Tiozzo define a shadow by 
	\begin{eqnarray}
		S'_o(x, R) = \{y \in X \; : \; (y|x)_o \geq d(x,o ) - R\}, \nonumber
	\end{eqnarray}
	but a quick computation shows that $y \in S_o (x, 2R) $ if and only if $y \in S'_o(x, R)$. Rewriting the inequality of the Gromov product gives that for all $y, z \in \overline{S_o (x, 2R)}^{\text{Grom}}$, we have 
	\begin{eqnarray}
		(y|z)_o \geq d(o, x) - R + O(\delta) \nonumber, 
	\end{eqnarray}
	where $O(\delta)$ is a universal constant depending only on $\delta$. 
	
	A very useful fact about shadows in geodesic hyperbolic metric spaces is that they are weakly convex, in the following sense: there exists a constant $C= C(\delta) $ such that for every shadow $S = S_o (x, R)$, if $y, z \in S$, then any geodesic segment $[y,z]$ between them belongs to $S_o(x, R+ C)$, \cite[Corollary 3.21]{maher_tiozzo18}. Thanks to Theorem \ref{thm isom embed geod hyp space}, one can extend this result to arbitrary hyperbolic spaces. 
	
	\begin{ecor}\label{cor conv shadows}
		Let $(X,d)$ be a $\delta$-hyperbolic space. Then for every $Q\geq 1, C\geq 0$, there exists a constant $C = C(\delta, Q, C)$ such that:
		\begin{enumerate}
			\item if $y, z$ belong to the shadow $S_o (x, R)$, then any $(Q,C)$-quasigeodesic segment between $y$ and $z$ belongs to $S_o (x, R+C)$. 
			\item if $y, z$  do not belong to the shadow $S_o (x, R)$, then any $(Q,C)$-quasigeodesic segment between $y$ and $z$ does not intersect $S_o (x, R-C)$.
		\end{enumerate}
	\end{ecor}
	
	\begin{proof}
		This result is known for geodesic hyperbolic spaces, \cite[Corollary 3.21]{maher_tiozzo18}. In the general case, due to Theorem \ref{thm isom embed geod hyp space}, there is an isometric embedding $f : (X, d) \rightarrow (Y, d_Y)$. Check that $y \in S_o(x, R) $ if and only if $f(y)$ belongs to the shadow $S' = S^Y_{f(o)}(f(x), R) $ in $Y$. Applying \cite[Corollary 3.21]{maher_tiozzo18} on the shadow $S'$, we obtain the result. 
	\end{proof}

	\subsection{Horofunction boundary}
	Let $(X, d)$ be a separable metric space, and $o \in X$. 
	
	Let $\R^X$ be the set of all functions on $X$ with values in $\R$, and endow it with the topology of uniform convergence on compact sets. Let $\mathds{1}$ be the constant function equal to $1$ and consider the quotient topological vector subspace $\R^X / \R \cdot\mathds{1}$, with the quotient topology. 
	
	Denote by 
	\begin{eqnarray}
		\Lip^1_{o} (X) := \{ f : X \rightarrow \R \; : \; f \text{ is $1$-Lipschitz and }f(o) = 0\} \nonumber.
	\end{eqnarray}
	
	With the topology induced from $\R^X$, $\Lip^1_{o} (X)$ is compact, Hausdorff and second countable, \cite[Proposition 3.1]{maher_tiozzo18}.

	\begin{eDef}
		For $o \in X$ a basepoint, the horofunction map is 
		\begin{eqnarray}
			b^o : &&X \rightarrow \R^X \nonumber \\
			&& x \mapsto b_x^o : y \mapsto d(x,y) - d(o, x). \nonumber
		\end{eqnarray}
	\end{eDef}
	
	A quick computation shows that the target space is actually $\Lip_o^1(X)$, and that $b^o $ is continuous injective. 
	
	\begin{eDef}
		The \textit{horoclosure} $\Xhc$ of $X$ (with basepoint $o \in X$) is the closure of $b^o (X)$ in $\R^X / \R \mathds{1}$. It is clear that $\Xhc$ does not depend on $o  \in X$. The set $\Xh$ is the preimage of $\Xhc$ by the natural quotient $\R^X \rightarrow \R^X / \R \mathds{1}$. We will call \textit{horofunctions} the elements of $\Xhc$  (and by abuse of notations, of $\Xh$ too). We call $\partial^h X = \Xhc \setminus b^o (X)$ the \textit{horoboundary} of $X$. 
	\end{eDef} 
	
	\begin{erem}
		The space $\Xhc$ corresponds to considering the closure of $X$ when we forget the basepoint, while $\Xh$ corresponds to the closure of all horofunction maps with every possible basepoint. 
	\end{erem}
	
	For $o \in X$, we shall denote the fibre $\Xh_o = \{ h \in \Xh \; : \; h(o)= 0\}$. The following lemma follows from the definitions. 
	\begin{elem}[{\cite[Lemma 2.4]{bader_caprace_furman_sisto22}}]
		The space $\Xhc$ is a compact metrizable space. For a fixed $o \in X$, the quotient map $\Xh_o \rightarrow \Xhc$ is a homeomorphism. 
	\end{elem}
	
	We denote by $\Xhbc := \{h \in \Xhc \; | \; \inf (h) > -\infty\}$ and $\Xhuc := \{h \in \Xhc \; | \; \inf (h) = -\infty\}$. This decomposition is constant on the fibres in $\Xh$ so it gives a decomposition $\Xh = \Xhb \cup \Xhu$. Moreover, $b^o (X) \subseteq \Xhbc$, hence $\Xhbc$ is dense in $\Xhc$.

	\begin{elem}[{\cite[Lemma 2.5]{bader_caprace_furman_sisto22}}]
		The decompositions $\Xh = \Xhu \cup \Xhb$  and $\Xhc = \Xhbc \cup \Xhuc$ are measurable and $\iso(X) $-equivariant. 
	\end{elem}

	For the rest of this section $(X, d)$ is a $\delta$-hyperbolic, $\alpha$-almost geodesic space. 
	
	The next Lemma is folklore, but most of the proofs available concern geodesic spaces so we provide one for self-containment. 
	
	\begin{elem}\label{lem horobdry gromov bdry}
		Let $(x_n)$ be a sequence of points in $X$ which converge in $\Xhc$ to the horofunction $\hat{h}$. Then $(x_n)$ converges to a point $\xi $ of the Gromov boundary $\bdg X$ if and only if $\hat{h} \in \Xhuc$. In this case, if $(x'_n)$ is another sequence such that $\lim x'_n = \hat{h}$, then $(x_n| x'_n )_o \rightarrow \infty$. 
	\end{elem}
	
	\begin{proof}
		Let $(x_n)$, $\hat{h}$ as in the Lemma, and denote by $h$ the lift of $\hat{h}$ in $\Xh_o$. We must show that $h \in \Xhu$. Assume that $(x_n)$ converges to infinity, that is $\underset{{n,m}}{\lim} (x_n|x_m)_o = \infty$. Fix $A > 0 $ and $N $ such that for $n, m > N$, $(x_n| x_m )_o \geq A$. For $m > N$, $(x_m|x_m)_o \geq A$ hence $d(x_m, o) \geq A$. By $\alpha$-almost geodesicity, take $z$ such that 
		\begin{eqnarray}
			|d(o,z)- A| \leq \alpha \text{ and } |d(z,x_m ) - (d(o,x_m) - A)|\leq \alpha \nonumber
		\end{eqnarray}
		By hyperbolicity, for all $n \geq N$, 
		\begin{eqnarray}
			(x_n| z)_o \geq \min \{ (x_n|x_m)_o , (x_m| z)_o\} - \delta. \nonumber
		\end{eqnarray}
		But by definition, $d(x_n, z) - d(x_n, o)= d(o, z) - 2 (z|x_n)_o $, hence
		\begin{eqnarray}
			h(z) &=& \lim_n (d(x_n, z) - d(x_n, o) ) \nonumber \\
			&\leq & d(o, z) + 2 \delta - 2(x_m| z)_o \nonumber \\
			& \leq & 2 \delta + \alpha - A \nonumber. 
		\end{eqnarray}
		As $A$ was arbitrary, $h \in \Xhu_o$.
		
		The converse follows from \cite[Lemma 3.8]{maher_tiozzo18}: for any $h' \in \Xh_o$, and any $x, y \in X$, 
		\begin{eqnarray}
			\min\{- h'(x), -h'(y) \} \leq (x|y)_o + O(\delta). 
		\end{eqnarray}
		Using this inequality, one deduces the last part of the proof. 
	\end{proof}
	
	We can then define the map $\pi : \Xhuc \rightarrow \bdg X$ so that if $x_n \rightarrow h \in \Xhuc$,  $\pi (\lim_n x_n) = \xi  \in \bdg X$, where $\xi$ is the equivalence class of $(x_n)$. By \cite[Lemma 2.11]{bader_caprace_furman_sisto22}, this map is continuous and $\iso(X)$-equivariant. 
	\newline
	
	Denote by $\bdd(X)$ the set of closed non-empty bounded subsets of $X$, and endow it with the Hausdorff metric. According to \cite[Lemma 2.6]{bader_caprace_furman_sisto22}, the Borel $\sigma$-algebra on $\bdd(X)$ is generated by the collection
	\begin{eqnarray}
		\{ K \in \bdd(X) \; | \;  \text{there exists $U$ open in $X$ such that }K \subseteq U\}. 
	\end{eqnarray}
	
	We consider the infimum function on $\Xhb$, defined by:
	\begin{eqnarray}
		\inf : h \in \Xhb \mapsto \inf\{h(x) \; | \; x \in X\} \nonumber. 
	\end{eqnarray}
	
	It is clearly $\iso(X)$-equivariant. It is moreover measurable. Indeed, fix a dense countable subset $X_0 \subseteq X$. By continuity of the functions in $\Xh$, $\inf(h) = \inf\{h(x) \; | \; x \in X_0\}$. 
	\newline
	
	The next Lemma is proven in \cite[Lemma 2.7]{bader_caprace_furman_sisto22} only in the geodesic case, but extends to almost geodesic hyperbolic spaces. 
	
	\begin{elem}\label{lemma bdd subset}
		Let $(X,d)$ be a $\delta$-hyperbolic, $\alpha$-almost geodesic separable metric space. For every $h \in \Xhb$, the set 
		\begin{eqnarray}
			\tilde{I} (h) := \overline{\{x \in X \; | \; h(x ) < \inf(h) +1 \}}\nonumber
		\end{eqnarray}
		is bounded. This gives a measurable, $\iso(X)$-equivariant map $\tilde{I} : \Xhb \rightarrow \bdd(X)$, and this map factors measurably and $\iso(X)$-equivariantly via $\Xhb \rightarrow \Xhbc$, defining a measurable $\iso(X)$-map $I :  \Xhbc \rightarrow \bdd(X)$. 
	\end{elem}
	
	\begin{proof}
		Let $h \in \Xhb$, and assume without loss of generality that $\inf(h) = 0 $. Consider $x, x' \in \tilde{I}(h) $ and define $A :=d(x,x')$. As $X$ is $\alpha$-almost geodesic, take $z \in X$ such that 
		\begin{eqnarray}
			|d(x,z)- A/2| \leq \alpha \text{ and } |d(x', z ) - A/2 |\leq \alpha \nonumber.
		\end{eqnarray}
		Let $\hat{h}$ be the image of $h $ in $\Xhc$. By definition, there is a sequence $\{y_n\}$ in $X$ such that $\hat{h}$ is the image of the pointwise limit of elements of the form
		\begin{eqnarray}
			b_{y_n} : x \in X \mapsto d(x,y_n) \nonumber
		\end{eqnarray}
		in the quotient space $\R^X / \R \mathds{1}$. In particular, there exists $y \in X$ such that 
		
		\[  \left\{ \begin{array}{ll}
			|d(y,z) - h(z) - (d(y, x) - h(x))| < 1, \text{      and } \\
			|d(y,z) - h(z) - (d(y, x') - h(x'))| < 1.\end{array} \right. \] 
		
		Rewriting the previous system, this yields 
		
		\[  \left\{ \begin{array}{ll}
			h(z) < 1 + d(y , z ) + h(x) - d(y, x) \\
			h(z) < 1 + d(y , z ) + h(x') - d(y, x').\end{array} \right. \] 
		Since $h(x), h(x') < 1$, we obtain  
		\[  \left\{ \begin{array}{ll}
			h(z) < 2 + d(y , z ) - d(y, x) \\
			h(z) < 2 + d(y , z ) - d(y, x').\end{array} \right. \] 
		
		Fix such a $y \in X$. By Gromov hyperbolicity:
		\begin{eqnarray}
			(x,x')_y \geq \min \{(x,z)_y, (x', z)_y\} - \delta \nonumber. 
		\end{eqnarray}
		Rewriting this, we have two possibilities. If $(x,z)_y \geq (x', z)_y$, 
		\begin{eqnarray}
			&&d(x,y) + d(x', y) - d(x, x') \geq d(x', y) + d(z, y) - d(x', z) -2 \delta \nonumber \\
			\Leftrightarrow && d(z, y ) \leq d(x,y) + d(x', z) - d(x,x') + 2 \delta. \label{eq geq}
		\end{eqnarray}
		If $(x,z)_y \leq (x', z)_y$, the same computation gives
		\begin{eqnarray}
			d(z, y ) \leq d(x',y) + d(x, z) - d(x,x') + 2 \delta. \label{eq leq}
		\end{eqnarray}

		Now if $(x,z)_y \geq (x', z)_y$, we have in particular by equation \eqref{eq geq}:
		\begin{eqnarray}
			h(z) &<& 2 + d(y , z ) - d(y, x) \nonumber\\
			& \leq& 2 + d(x,y) + d(x', z) - d(x,x') + 2 \delta - d(x,y) \nonumber \\
			&\leq &2 + 2 \delta + d(x', z ) - d(x,x') \label{ineq geq}. 
		\end{eqnarray}
		If $(x,z)_y \leq (x', z)_y$, equation \eqref{eq leq} implies 
		\begin{eqnarray}
			h(z) < 2 + 2 \delta + d(x, z) - d(x,x') \nonumber
		\end{eqnarray}
		
		But by definition of $z$, 
		\begin{eqnarray}
			d(x, z) - d(x,x') \leq \alpha - A/2  \text{ and } d(x', z) - d(x,x') \leq  \alpha - A/2 \nonumber
		\end{eqnarray}
		Hence, taking $A > 2\alpha + 4 + 4 \delta $ gives $h(z) < 0 $, a contradiction with the fact that $\inf(h) = 0$. Then $A$ is (uniformly) bounded. 
		
		Now for the rest of the proof, the demonstration of \cite[Lemma 2.7]{bader_caprace_furman_sisto22} does not use any geodesicity assumption, hence we can proceed in the exact same way. 
	\end{proof}
	
	\subsection{Constructing subsequences}
	We end this section by presenting a useful way of constructing subsequences that converge to a horofunction $h \in \Xhuc$. These ideas come from Maher and Tiozzo, \cite{maher_tiozzo18}. Let $S = S_o(x, R)$ be a shadow. A straightforward computation shows that 
	\begin{eqnarray}
		(y|o)_x \leq R & \Leftrightarrow & b_y^o (x ) \leq 2R - d(o, x)  \nonumber. 
	\end{eqnarray}
	We call the quantity $2R - d(o, x) $ the \emph{depth} $\dep(S)$ of the shadow $S$. This computation shows that $y \in S$ if and only if $b_y^o (x) \leq \dep(S)$. As a consequence, the horoclosure of $S$ in $\Xhc$ is given by 
	\begin{eqnarray}
		\overline{S} = \{h \in \Xhc \; | \; h(x)\leq \dep (S)\}. \nonumber
	\end{eqnarray}
	
	Similarly, we define the open shadow $S^\circ_o (x, R) $ by:
	
	\begin{eqnarray}
		S^\circ_o (x, R)  = \{h \in \Xhc \; | \; h(x)< \dep (S)\}, \nonumber
	\end{eqnarray}
	which is an open set in $\Xhc $, contained in the interior of $\widehat{S}$. The following Lemma uses the separability of $X$. 
	\begin{elem}[{\cite[Lemma 4.8]{maher_tiozzo18}}]
		For all $M < 0$, the set $\Xhuc$ is contained in a countable collection of open shadows of depth $\leq M$.  
	\end{elem}
	
	\begin{eDef}
		A \emph{descending shadow sequence} is a sequence $\mathcal{S} = (\mathcal{S}_M)_{M \in \mathbb{N}}$, such that for all $M$, $\mathcal{S}_M$ is a finite collection of shadows of depth $\leq - M $. 
	\end{eDef}
	
	Given a descending shadow sequence, we say that $C $ is a \emph{cylinder of size $M$} if 
	\begin{eqnarray}
		C = S^\circ_1 \cap S^\circ_2 \cap \dots \cap S^\circ_M, \nonumber
	\end{eqnarray}
	where, for every $i = 1, \dots, M$, $S_i \in \mathcal{S}_i$, i.e. has depth $\leq -i$. Note that $C$ is open as the intersection of a finite number of open sets. Given $M \in \mathbb{N}$, define 
	\begin{eqnarray}
		\Sigma_M := \bigcup \, \{C \; | \; C \text{ is a cylinder of size M}\}.  \nonumber
	\end{eqnarray}
	It is clear that $\Sigma_M$ is open, and that $\Sigma_{M+1} \subseteq \Sigma_M$. Last, observe that if $h \in \Sigma_M $, then $\inf(h) \leq -M $, hence 
	\begin{eqnarray}
		\underset{M}{\bigcap} \, \Sigma_M = \Xhuc \nonumber. 
	\end{eqnarray}
	
	The following Lemma will help us finding sequences that converge to a horofunction of the horoboundary. 
	
	\begin{elem}[{\cite[Lemma 4.9]{maher_tiozzo18}}]\label{lem shadow subsequence}
		Let $(y_n)$ be a sequence of points in $X$, and let $\mathcal{S} = (\mathcal{S}_M)_{M \in \mathbb{N}}$ be a descending shadow sequence. If $(b^o_{y_n})_n $ intersects $\Sigma_M$ for every $M$, then there exists a subsequence $n_k$ such that $(b_{y_{n_k}})_k$ converges to a horofunction in $\Xhuc$. 
	\end{elem}

	\section{Random walks on Gromov hyperbolic spaces}\label{section rw hyp}
	In this section, we present some concepts about boundary theory, and see how they can be applied to random walks in Gromov hyperbolic spaces. 
	
	Let $G$ be a discrete group, and let $\mu$ be an admissible probability measure on $G$, i.e. whose support generates $G$ as a semigroup. Let $(\Omega, \mathbb{P}) $ be the probability space $(G^{\mathbb{N}}, \mu^{\otimes \mathbb{N}})$, with the product $\sigma$-algebra. The space $\Omega$ is called the space of (forward) increments. Again denote by $S : \Omega \rightarrow \Omega $ the shift $S((\omega_i)_i) = (\omega_{i+1})_i$. The application 
	\begin{equation*}
		(n, \omega) \in \mathbb{N} \times \Omega \mapsto Z_n(\omega) = \omega_1 \omega_2 \dots \omega_n,
	\end{equation*}
	defines the random walk on $G$ generated by the measure $\mu$. 
	
	Let $\check{\mu}$ be the measure on $G$ defined by $\check{\mu} = \iota_\ast\mu $, where $\iota(g) =g^{-1}$, and define $\check{\Omega}$ the space $(G^{\mathbb{N}}, \check{\mu}^{\otimes \mathbb{N}})$. Let $(\check{Z}_n)$ be the corresponding random walk. As in Section \ref{section bord poisson}, define $(B_+, \nu_+)$ the Poisson-Furstenberg boundary associated to the random walk $(Z_n)$ and $(B_-, \nu_-) $ the Poisson-Furstenberg boundary associated to $(\check{Z}_n)$. Recall that by Theorem \ref{thm paire de bords}, $(B_-, B_+)  $ is a $G$-boundary pair, and that they come equipped with natural $G$-equivariant measurable factors 
	\begin{eqnarray}
		\bnd_+ : \Omega \rightarrow B_+ \text{ and } \bnd_- : \check{\Omega} \rightarrow B_- \nonumber. 
	\end{eqnarray}
	
	\subsection{Convergence to the boundary}
	
	Let $(X,d)$ be a separable $\delta$-hyperbolic metric space in the sense of Gromov, and assume that $(X,d)$ is $\alpha$-almost geodesic. Assume that $G$ acts on $X$ by isometries, and that the action of $G$ on $X$ is \textit{non-elementary}, i.e. $G$ does not fix a bounded subset of $X$ and does not fix a point nor a pair of points in $\bdg X$. Let $o \in X$ be a basepoint. 
	The goal of this section is to prove that $\mathbb{P}$-almost surely, the random walk $(Z_n o)_n$ converges at infinity, to a point $\xi \in \bdg X$. This is known to be true in geodesic spaces by Maher-Tiozzo \cite[Theorem 1.1]{maher_tiozzo18}, and the extension to almost geodesic spaces is likely known to the authors. Furthermore, Gouëzel shows that it holds in general Gromov hyperbolic spaces, and that the random walk has almost surely positive drift \cite[Theorem 1.2]{gouezel22}, without moment condition. There is no novelty in the result that we present in this section, but the proof is completely different from what was done in \cite{gouezel22}, and the simplicity of the argument might prove useful if one considers other situations like ergodic cocycles. The methods we present here are described in \cite[Section 2]{bader_caprace_furman_sisto22}. 
	\newline
	
	Recall that $\Xhc$ is a compact metrizable space, hence by Proposition \ref{prop existence stationnaire compact}, there exists a measure $\nu \in \prob(\Xhc)$ that is $\mu$-stationary. As a consequence of Theorem \ref{thm mesures limites}, we have the following. 
	
	\begin{eprop}\label{prop mesures limites hyp}
		There exists a measurable map $\tilde{\psi}: \omega \in \Omega \rightarrow \prob(\Xhc)$ such that for $\mathbb{P}$-almost every $\omega \in \Omega$, and for every $g \in G$, 
		\begin{eqnarray}
			Z_n (\omega)g_\ast\nu \underset{n \to \infty}{\longrightarrow} \tilde{\psi}(\omega) \nonumber
		\end{eqnarray}
		in the weak-$\ast$ topology. Moreover, we have the decomposition $\nu = \int_{\Omega} \tilde{\psi}(\omega) d \mathbb{P}(\omega)$. 
	\end{eprop}
	
	By construction of the limit measures, for $\omega = (\omega_1, \omega_2, \dots) \in \Omega$, we have 
	\begin{eqnarray}
		{\omega_{1}}_\ast \tilde{\psi}(S\omega) =  \tilde{\psi}(\omega) \nonumber. 
	\end{eqnarray}
	
	In particular, $\tilde{\psi}$ passes to the quotient on $B_+$, and gives a $G$-equivariant measurable map 
	\begin{eqnarray}
		\psi_+ :  (B_+, \nu_+) &\rightarrow &\prob(\Xh) \nonumber \\
		b_+ &\mapsto& \tilde{\psi}(\omega) \nonumber, 
	\end{eqnarray}
	where $\omega$ is any trajectory in the class of $b_+$. 
	\newline 
	
	The following is now a rewriting of some results in \cite{bader_caprace_furman_sisto22}. Recall that 
	\begin{eqnarray}
		\Map_G (B_- \times B_+, \prob(\Xh)) \nonumber
	\end{eqnarray}
	denotes the set of measurable $G$-equivariant maps between $B_- \times B_+$ and $\prob(\Xh)$. We are going to show that the fact that $(B_-, B_+)$ is a boundary pair gives many restrictions to the boundary map $\psi_+$. 
	
	The proof of the following result uses the existence of a barycenter map, see \cite[Section 2.C]{bader_caprace_furman_sisto22}. For every $ \varepsilon \in (0, \frac{1}{2})$, there exists a measurable $G$-equivariant application 
	\begin{eqnarray}
		\beta_\varepsilon : \prob(\bdd(X)) \rightarrow \bdd(X). \nonumber
	\end{eqnarray}
	
	\begin{eprop}[{\cite[Claim 2.14]{bader_caprace_furman_sisto22}}]\label{prop no gmap bdd}
		In our situation, 
		\begin{eqnarray}
			\Map_G (B_- \times B_+, \prob(\bdd(X))) = \emptyset \nonumber. 
		\end{eqnarray}
		And therefore, 
		\begin{eqnarray}
			\Map_G ( B_+, \prob(\bdd(X))) = \emptyset . \nonumber
		\end{eqnarray}
	\end{eprop}
	
	\begin{proof}
		If there existed a measurable $G$-map $f : B_- \times B_+ \rightarrow  \prob(\bdd(X))$, then composing with $\beta_\varepsilon$ would give a measurable $G$-map $B_- \times B_+ \rightarrow  \bdd(X)$. But recall that $\bdd(X)$ is a metric space (with the Hausdorff topology) on which $G$ acts by isometries. Then by metric ergodicity, any such map is essentially constant. As a consequence, there exists a bounded set in $X$ that is stabilized by the action of $G$. This is a contradiction with the non-elementarity of the action. 
	\end{proof}

	\begin{ecor}
		\begin{eqnarray}
			\Map_G ( B_+, \prob(\Xhbc)) = \emptyset \text{ and } \Map_G ( B_-, \prob(\Xhbc)) = \emptyset. \nonumber
		\end{eqnarray}
	\end{ecor}
	
	\begin{proof}
		If there existed a measurable $G$-map $B_+\rightarrow \prob(\Xhbc)$, composing with the $G$-equivariant measurable map $I : \Xhbc \rightarrow  \bdd(X)$ given by Lemma \ref{lemma bdd subset} would get a measurable $G$-map $B_+ \rightarrow \prob(\bdd(X))$, which is forbidden by Proposition \ref{prop no gmap bdd}. 
	\end{proof}
	
	As a consequence, for $\nu_{\pm}$ almost every $b_{\pm} \in B_{\pm}$, $\tilde{\psi}(b_{\pm})$ is actually a probability measure on $\Xhuc$. We can then compose by the $\iso(X)$-equivariant continuous map $\pi : \Xhuc \rightarrow \bdg X $ to obtain a measurable $G$-map 
	\begin{eqnarray}
		\psi_{\pm} : B_{\pm}\rightarrow \prob (\bdg X) \nonumber. 
	\end{eqnarray}
	
	In particular, there exists a $\mu$-stationary measure $\nu \in \prob(\bdg X)$. Indeed, the measure 
	\begin{eqnarray}
		\nu = \int \psi_+(b) d\nu_+ (b) \nonumber
	\end{eqnarray}
	is $\mu$-stationary by Proposition \ref{prop application de bord}.
	
	Now with the use of ergodicity arguments, we can prove that this map is essentially unique. The following Lemma is proven in \cite[Lemma 2.12]{bader_caprace_furman_sisto22}, and does not use any geodesicity argument. The idea is that given $3$ distinct boundary points $\xi_1, \xi_2, \xi_3 \in \bdg X$, one can associate the coarse center of the ideal triangle $\Delta(\xi_1, \xi_2, \xi_3)$ in an equivariant and measurable way. 
	
	\begin{elem}\label{lem gmap bdg3 X bdd}
		Let $G < \iso(X)$ be a countable group. Then there is a Borel $G$-map 
		\begin{eqnarray}
			\tau : \bdg X^{(3)} \rightarrow \bdd(X) \nonumber. 
		\end{eqnarray}
	\end{elem}
	
	We can now prove that the maps $\psi_{\pm} : B_{\pm}\rightarrow \prob (\bdg X)$ have values in the set of Dirac measures. 
	
	\begin{elem}\label{lem image dirac measure}
		For almost every $b \in B_{\pm}$, we have that $\psi_{\pm}(b) = \delta_{\xi^+(b)}$, with $\xi^+(b) \in \bdg X$. 
	\end{elem}
	
	\begin{proof}
		Let $\Psi : B_- \times B_+ \rightarrow \prob((\bdg X)^3)$ given by 
		\begin{eqnarray}
			\Psi(b_-, b_+) = \psi_- (b_-) \otimes \psi_+ (b_+) \otimes \frac{1}{2} (\psi_- (b_-) + \psi_+ (b_+)). \nonumber
		\end{eqnarray}
		Since $G$ acts on $B_- \times B_+$ ergodically, the image of $\psi$ is essentially contained in $\prob((\bdg X)^{(3)}) $ or in $\prob(\Delta(\bdg X))$, where $\Delta(\bdg X)$ is the subset of $(\bdg X)^3$ for which at least two boundary points coincide. If $\Psi$ has values in $\prob((\bdg X)^{(3)}) $, then composing with the $G$-map $\tau $ of Lemma \ref{lem gmap bdg3 X bdd} gives a measurable $G$-map 
		\begin{eqnarray}
			B_- \times B_+ \rightarrow \prob(\bdd(X)) \nonumber, 
		\end{eqnarray}
		which is impossible by Proposition \ref{prop no gmap bdd}. 
		
		Therefore, for almost every $(b_-, b_+)$, $\Psi(b_-, b_+)$ must be atomic with at most two atoms and this implies that $\psi_{\pm}(b)$ is almost surely a Dirac measure. 
	\end{proof}
	
	Therefore, there exist measurable $G$-maps 
	\begin{eqnarray}
		\phi_{\pm} : B_{\pm} \rightarrow \bdg X . \nonumber
	\end{eqnarray}
	
	Let us prove that these maps are actually essentially unique. First, check that the measurable $G$-map 
	\begin{eqnarray}
		\phi_{\bowtie} &:& B_{-}\times B_+ \rightarrow (\bdg X)^2 \nonumber \\
		&& (b_-, b_+) \mapsto (\phi_-(b_-), \phi_+(b_+)) \nonumber
	\end{eqnarray}
	has essential values in $\bdg X^{(2)}$. Indeed, the set 
	\begin{eqnarray}
		\{(b_-, b_+) \; | \; \phi_+(b_+)= \phi_-(b_-)\} \nonumber
	\end{eqnarray}
	has measure $0$ or $1$ by ergodicity. Suppose it has full measure, If we let the coordinates vary separately, then it would imply that $\phi_-$ and $\phi_+$ are both essentially constant. As $\phi_+$ and $\phi_-$ are $G$-maps, this would give a fixed point in $\bdg X$, which is impossible because the action is non-elementary. 
	
	Let $\phi'_+ : B_+ \rightarrow \bdg X$ be another measurable $G$-map. Again, by ergodicity, $\phi_+$ and $\phi'_+$ are either essentially the same or essentially distinct. If they are essentially distinct, the measurable $G$-map 
	\begin{eqnarray}
		b \in B_+ \mapsto \frac{1}{2} (\delta_{\phi_+(b)} + \delta_{\phi'_+(b)}) \nonumber
	\end{eqnarray}
	has values in $\prob(\bdg X)$, which is impossible because any such map has essential values in the set of Dirac measures by Lemma \ref{lem image dirac measure}. As a consequence, $\phi_+$ and $\phi_-$ are essentially unique. 
	\newline 
	
	Notice that in the argument above, we never used the fact that $(B_-, B_+)$ are Poisson-Furstenberg boundaries. Instead, we only used ergodic arguments, and amenability of the actions $G\curvearrowright B_{\pm}$. Indeed, the following Theorem shows that these results remain true for any $G$-boundary pair. 
	
	\begin{ethm}[{\cite[Theorem 2.3]{bader_caprace_furman_sisto22}}]\label{thm uniq bdry map hyp}
		Let $(B_-, B_+) $ be any $G$-boundary pair, and $G \curvearrowright X$ as above. Then there exists $\phi_- : B_- \rightarrow \bdg X$ and $\phi_+ : B_+ \rightarrow \bdg X$ such that the map 
		\begin{eqnarray}
			\phi_{\bowtie} : (b_-, b_+ )\in (B_- \times B_+) \mapsto (\phi_-(b_-), \phi_+(b_+))
		\end{eqnarray} 
		is essentially contained is the set of distinct pairs of points of the boundary $\bdg(X)^{(2)}$. Moreover: 
		\begin{enumerate}
			\item[(i)] $\Map_G ( B_-, \prob(\bdg X)) = \{\delta \circ \phi_-\}$ and $\Map_G ( B_+, \prob(\bdg X)) = \{\delta \circ \phi_+\}$. 
			\item[(ii)] $\Map_G ( B_- \times B_+, \bdg X) = \{\phi_-\circ \text{pr}_-, \phi_+\circ \text{pr}_+\}$ 
			\item[(iii)] $\Map_G ( B_- \times B_+, \bdg X^{(2)}) = \{\phi_{\bowtie}, \tau\circ\phi_{\bowtie}\}$, where $\tau (\xi,\eta) = (\eta, \xi)$. 
		\end{enumerate} 
	\end{ethm}
	
	In view of Proposition \ref{prop mesures limites hyp}, Theorem \ref{thm uniq bdry map hyp} yields the following result. 
	
	\begin{ecor}\label{cor lim measures hyp}
		For $\mathbb{P}$-almost every $\omega \in \Omega$, 		
		\begin{eqnarray}
			Z_n (\omega)_\ast \nu \underset{n\rightarrow \infty}{\longrightarrow}\delta_{\phi_+(b)} \nonumber,  
		\end{eqnarray}
		where $b = \bnd_+(\omega)$. 
	\end{ecor}
	
	In \cite{maher_tiozzo18}, a great part of the proof for the convergence of the random walk to the boundary amounts to showing that Corollary \ref{cor lim measures hyp} holds. Indeed, the conclusion will follow from the following argument. 
	
	\begin{eprop}
		Let $G$ be a discrete countable group acting non-elementarily by isometries on an almost geodesic, $\delta$-hyperbolic separable metric space $(X,d)$. Let $\mu$ be an admissible probability measure on $G$, and let $\nu$ be a $\mu$-stationary measure on $\bdg X$. If $\mathbb{P}$-almost surely, ${Z_n}_\ast \nu$  converges to a Dirac measure $\delta_\lambda$ for $\lambda \in \bdg X$, then for any $o \in X$, $Z_n o \rightarrow \lambda$. 
	\end{eprop}
	
	\begin{proof}
		The action is non-elementary, hence the stationary measure $\nu$ is non-atomic by Lemma \ref{lem non atomic}. As the support of $\nu$ is a non-empty closed subset of $ \bdg X $, there exists $\lambda_1, \lambda_2 \in \bdg X$ a pair of distinct points in the support of $\nu $. Take shadows $S_1$ and $S_2$ such that $\lambda_i $ is contained in the interior of the closure $U_i = \overline{S_i} \subseteq \bdg X$. Because $\lambda_1, \lambda_2$ are in the support of $\nu$, $\nu(U_i ) >0$. Moreover, we can choose such shadows so that $U_i $ are disjoint. 
		
		Set $\varepsilon := \min \{\nu(U_1 ), \nu(U_2 )\}$. Let $S = S_o (x, R)$ be a shadow such that $V:= \overline{S}$ contains $\lambda$ in its interior. Then by convergence of the measures $Z_n \nu \rightarrow \delta_\lambda$, there exists $n_0$ such that for any $n \geq n_0 $, $Z_n \nu (V) \geq 1 - \varepsilon/2$. 
		
		In particular, for any $n \geq n_0$, ${Z_n}^{-1}V $ intersects both $U_1$ and $U_2$. As a consequence, there exists $\xi_1 \in U_1$, $\xi_2 \in U_2$ such that for all $n \geq n_0 $, $Z_n \xi_1$ and $Z_n \xi_ 2$ belong to $V$.
		
		As $X$ is $\alpha$-almost geodesic, Proposition \ref{prop exist qgeod hyp} implies that there exists a $(1,C)$-quasigeodesic line from $\xi_1$ to $\xi_2$, where $C = C(\delta, \alpha)$, which we call $\gamma$. By property of the Gromov product, there is a constant $C = C(\delta, \alpha)$ such that 
		\begin{eqnarray}\label{eq dist shadow}
			d(o, \gamma) = (\xi_1, \xi_2)_o + C(\delta, \alpha) =: A
		\end{eqnarray} Now by weak convexity of shadow \ref{cor conv shadows}, there is a constant $C'(\delta, \alpha)$ such that any $(1, C)$ quasigeodesic between points of $S$ is contained in the slightly larger shadow $S_o (x, R + C'(\alpha, \delta))$. Combining this fact with \eqref{eq dist shadow} shows that $Z_n o $ is contained in the larger shadow $S_o (x, R+A+ C'(\delta, \alpha))$. As this holds for every shadow $S$, we have the convergence. 
	\end{proof}

	We have proved the following theorem. 
	\begin{ethm}\label{thm cv rw hyp}
		Let $G$ be any discrete countable group and $G \curvearrowright (X,d)$ be a non-elementary action on a $\alpha$-almost geodesic, Gromov-hyperbolic space. Let $\mu$ be an admissible measure on $G$. Then for any basepoint $o \in X$, the random walk $(Z_n o)$ generated by $\mu$ converges almost surely to a boundary point $\psi_+(b) \in \bdg X$. Moreover, the hitting measure is the unique stationary measure, and is given by 
		\begin{eqnarray}
			\nu = \int_{B_+} \delta_{\psi_+(b)} d\nu_+(b) \nonumber. 
		\end{eqnarray}
	\end{ethm}

	We end this section by proving that almost surely, there exists a subsequence $n_k $ such that $(Z_{n_k}^{-1} (\omega)o)_k $ converges to a point of the Gromov boundary $\bdg X$. We will need this fact in Section \ref{section uniq stat cat}.

	First, recall that the measure $\mui$ is defined by $\mui = \iota_\ast \mu$, where $\iota (g) = g^{-1}$. Define $\mui_n = \mui^{\ast n} $ as the $n$-th convolution product of $\mui$ with itself. Let $o \in (X, d_L)$, and consider the probability measure $\delta_o  \in \prob (\overline{X_L})$ on the horocompactification $\overline{X_L}$. 
	For all $n \geq 1$, define $\tilde{\mu}_n = \mui_n \ast \delta_o $, which is a probability measure on $\overline{X_L}$. Consider the Cesàro averages 
	\begin{eqnarray}
		\overline{\mu}_n := \frac{1}{n} (\tilde{\mu}_1 + \dots + \tilde{\mu}_n) \nonumber. 
	\end{eqnarray}
	As the space $\prob(\Xhc)$ is weakly-$\ast$ compact, there exists a subsequence $n_k$ such  that $\overline{\mu}_{n_k}$ converges to a measure $\nui$, and this measure is $\mui$-stationary by construction, \cite[Lemma 4.3]{maher_tiozzo18}. Due to Theorem \ref{thm cv rw hyp}, the measure $\nui $ is the unique $\mui$-stationary measure and 
	\begin{eqnarray}
		\nui = \int_{\Omega} \delta_{\psi_-(\omega)} d\check{P}(\omega)\nonumber
	\end{eqnarray}
	By Theorem \ref{thm uniq bdry map hyp}, $\check{\nu}$ charges only $\Xhuc$. 
	
	The following is known as the Portmanteau Lemma, and is a classical result in measure theory.
	
	\begin{elem}\label{lem portmanteau}
		Let $Y$ be a metric space, $P_n$ a sequence of probability measures on $Y$, and $P$ a probability measure on $Y$. Then the following are equivalent: 
		\begin{itemize}
			\item $P_n \rightarrow_n P$  in the weak$-\ast$ topology; 
			\item $\underset{n \rightarrow \infty}{\liminf} \ P_n(O) \geq P(O)$ for every open set $O \subseteq Y$. 
			\item $\underset{n \rightarrow \infty}{\limsup} \ P_n(F) \leq P(F)$ for every closed set $F \subseteq Y$.
		\end{itemize}
	\end{elem}
	
	We are now ready to prove the following. 
	\begin{eprop}\label{prop subseq inverse}
		For $\mathbb{P}$-almost every sample path $\omega = (\omega_i) \in \Omega$, there exists a subsequence $n_k $ such that $(Z_{n_k}^{-1} (\omega)o)_k $ converges to a point of the Gromov boundary $\bdg X_L$. 
	\end{eprop}
	
	\begin{proof}
		Fix $\varepsilon >0$. By \cite[Lemma 4.10]{maher_tiozzo18}, there exists a finite descending shadow $\mathcal{S} = (\mathcal{S}_M)_M$ such that for each $M$, $\check{\nu}(\Sigma_M) \geq 1 - \varepsilon$. Let $Z$ be the set 
		\begin{eqnarray}
			Z := \{\omega \in \Omega \; | \; b_{Z_n^{-1}(\omega)o} \text{ does not have limit points in }\Xhuc\} \nonumber. 
		\end{eqnarray}
		By Lemma \ref{lem shadow subsequence}, if $(\omega_i)$ is such that $b_{Z_n^{-1}(\omega)} $ does not have limit points in $\Xhuc$, then there exists $M\geq 0$ such that $b_{Z_n^{-1}(\omega)o} $ does not belong to $\Sigma_M$ for any $M$. As a consequence, we have 
		\begin{eqnarray}
			Z \subseteq \underset{M}{\bigcup}\,  \underset{n}{\bigcap}\,  \{ \omega \; | \; b_{Z_n^{-1}(\omega ) o} \notin \Sigma_M\}\nonumber. 
		\end{eqnarray}
		By definition, for any  measurable set $A \subseteq \Xhuc$, 
		\begin{eqnarray}
			\mathbb{P}(b_{Z_n^{-1}o} \in A) &=& \mathbb{P}( (\omega_i)\in \Omega \; | \; b_{\omega_n^{-1}\dots \omega_n^{-1} o} \in A) \nonumber\\
			& = & \tilde{\mu}_n (A). \nonumber
		\end{eqnarray}
		As a consequence, if we set $Y_M = \Xhuc \setminus \Sigma_M$, $\mathbb{P}(Z) \leq \sup\inf \tilde{\mu}_n (Y_M)$. By Definition of the Cesàro averages, $ \inf_n \overline{\mu}_n (Y_M) \geq \inf_n \tilde{\mu}_n (Y_M)$, hence 
		\begin{eqnarray}
			\mathbb{P}(Z) \leq \sup\inf \tilde{\mu}_n (Y_M). \nonumber
		\end{eqnarray}
		There exists a subsequence $n_k$ such that $\overline{\mu}_{n_k}$ converges to $\check {\nu}$ for the weak-$\ast$ topology. By the Portmanteau Lemma \ref{lem portmanteau}, and since $Y_M $ is closed, we obtain for every $M$:
		\begin{eqnarray}
			\limsup_{k \rightarrow \infty} \overline{\mu}_{n_k} (Y_M) \leq \check{\nu}(Y_M) \nonumber.
		\end{eqnarray}
		Therefore, $\mathbb{P}(Z) \leq \check{\nu}(Y_M)$. But we chose $\mathcal{S}$ such that for each $M$, $\check{\nu}(\Sigma_M) \geq 1 - \varepsilon$. As a consequence $\mathbb{P}(Z) \leq \varepsilon $. As $\varepsilon$ was arbitrary, the result follows from Lemma \ref{lem horobdry gromov bdry}. 
	\end{proof}

	\section{Random walks in $\cat$(0) spaces}\label{section rw cat}
	
	Let $G$ be a discrete countable group, and let $\mu \in \prob(G)$ be an admissible probability measure on $G$. Let $(X,d)$ be a complete and separable $\cat$(0) space, and let $G \curvearrowright X$ be an action by isometries. We assume furthermore that $G$ admits a pair of elements that act on $X$ as contracting isometries with disjoint fixed points, in the sense of Definition \ref{def isom contract}. In particular, the action $G \curvearrowright X$ is non-elementary in the usual sense: there is no $G$-invariant flat in $X$. 
	
	Recall that thanks to the curtain models constructed by Petyt, Spriano and Zalloum in \cite{petyt_spriano_zalloum22}, there exists a family of hyperbolic spaces $X_L = (X,d_L)$ on which $G$ acts by isometries, see Theorem \ref{theorem hyperbolic models intro}. Since $G$ acts non-elementarily on $X$ with a contracting element, Theorem \ref{thm contract loxodromic} states that there exists $L$ such that $G$ acts on $X_L$ with two independent loxodromic isometries. For the rest of the section, fix such a $L > 0$. We recall that Theorem \ref{theorem hyperbolic models intro} states that $(X,d_L)$ is $\delta$-Gromov hyperbolic and $\alpha$-almost geodesic, for constants $\delta, \alpha$ depending only on $L$. To avoid any confusion, in what follows, a geodesic segment $[x,y]$ is always assumed to be in the $\cat$(0) space $(X,d)$, and $B(o, R)$ denotes the open metric ball of center $o$ and radius $R  >0$ for the $\cat$(0) metric $d$. 
	
	\subsection{Boundaries of the hyperbolic models}
	
	In this section, we prove that the random walk $(Z_n o )$ converges almost surely to the visual boundary $\bd X$. This result will follow from Theorem \ref{thm cv rw hyp}, as soon as we link the Gromov boundary $\bdg X_L$ of $X_L$ with the visual boundary $\bd X$. We are led to extend Theorem \ref{equivariant embedding of boundaries} to the case where $X$ is just a (possibly non-proper) complete $\cat$(0) space. We refer to Section \ref{section modèle hyp} for the vocabulary about the hyperbolic models $(X, d_L)$. 
	
	Let us introduce some definitions. 
	\begin{eDef}
		We say that a geodesic ray $\gamma : [0, \infty) \rightarrow X$ \textit{crosses} a curtain $h $ if there exists $t_0 \in [0, \infty)$ such that $h$ \emph{separates} $\gamma(0) $ from $\gamma ([t_0, \infty))$. Alternatively, we may say that $h$ separates $\gamma(0)$ from $\gamma(\infty)$. Similarly, we say that a geodesic line $\gamma: \mathbb{R} \rightarrow X$ crosses a hyperplane $h$ if there exists $t_1, t_2 \in \mathbb{R}$ such that $h$ separates $\gamma ((-\infty, t_1])$ from $\gamma ([t_2, \infty))$. We say that $\gamma $ crosses a chain $c = \{h_i\}$ if it crosses each individual curtain $h_i$. 
	\end{eDef}
	
	Recall that the space $B_L$ is defined to be the set of all geodesic rays $\gamma : [0, \infty) \rightarrow X$ emanating from $o \in X$ that cross an infinite $L$-chain. We can define a more useful topology on $B_L$ with the following basis of open sets. For a geodesic ray $\gamma$ emanating from $o $ such that $\gamma(\infty) = \xi \in \bd X$, and $h $ a curtain dual to $\gamma$, define 
	\begin{eqnarray}
		U_h(\xi) := \{ \eta \in \bd X \; | \; \gamma_o^\eta \text{ crosses $h$}\}, \nonumber
	\end{eqnarray}
	where $\gamma_o^\eta$ is the unique geodesic ray based at $o$, in the class of $\eta$. We say that $U \subseteq X$ is open for the \textit{curtain topology} if for every $\xi \in U$, there exists a curtain $h $ dual to $\gamma_o^\xi $ such that $U_h(\xi) \subseteq U$. The following result links this topology with the usual visual topology on $\bd X$. 
	
	\begin{ethm}[{\cite[Theorem 8.8]{petyt_spriano_zalloum22}}]\label{thm curtain topology}
		The identity map 
		\begin{eqnarray}
			(\bd X, \mathcal{T}_{\text{cone}}) \rightarrow (\bd X,\mathcal{T}_{\text{curtain}})\
		\end{eqnarray} 
		is continuous, and the topologies are the same on any $B_L$.
	\end{ethm} 
	
	We will use the following Lemma a several times. 
	\begin{elem}[{\cite[Lemma 2.21]{petyt_spriano_zalloum22}}]\label{lem dual chain}
		Let $L, n \in \mathbb{N}$, and let $\{h_1, \dots, h_{(4L +10)n}\}$ be an $L$-chain separating $A$, $B \subseteq X$. Take $x \in A$, $y \in B$. Then $A$ and $B$ are separated by an $L$-chain of size $\geq n+1 $ dual to $[x,y]$. 
	\end{elem}
	
	As a consequence of Lemma \ref{lem bottleneck intro} and Lemma \ref{lem dual chain}, if two geodesic rays with the same starting point cross an infinite $L$-chain $c$, then they are asymptotic, and hence equal. 
	
	\begin{erem}\label{infinite dual chain}
		Since curtains are not convex, it is not obvious that any geodesic ray $\gamma$ meeting a given curtain $h$ must cross it ($\gamma$ could meet $h$ infinitely often). However, by \cite[Corollary 3.2]{petyt_spriano_zalloum22} if $\gamma$ is a geodesic ray that meets every element of an infinite $L$-chain $c = \{h_i\}_{i \in \mathbb{N}}$, then $\gamma$ must cross $c$: for every $i$, there exists $t_i \in [0, \infty)$ such that $h_i$ separates $\gamma(0) $ from $\gamma ([t_0, \infty))$.
	\end{erem}
	
	In the case of the contact graph associated to a $\cat$(0) cube complex $X$, we had the existence of an $\iso(X)$-equivariant embedding of the boundary of the contact graph into the Roller boundary $\partial_{\mathcal{R}}X$, see Theorem \ref{bord contact}. The following result is the analogue in the context of $\cat$(0) spaces. 
	
	\begin{ethm}\label{thm homeomorphism bords XL}
		Let $X$ be a complete $\cat$(0) space. Then the identity map $\iota_L : (X,d) \rightarrow (X,d_L) $ induces an $\iso(X)$-equivariant homeomorphism of the boundary $\partial_L : B_L \rightarrow \bdg X_L$. 
	\end{ethm}
	
	The proof is similar to that of \cite[Proposition 8.9]{petyt_spriano_zalloum22}, where the authors only treat the proper case. We reproduce their argument, but extend the result to any complete $\cat$(0) space. 
	
	We begin with a Lemma. 
	\begin{elem}[{\cite[Corollary 8.11]{petyt_spriano_zalloum22}}]\label{lem dual chain qgeod}
		Let $P : [0, \infty) \rightarrow X$ be a $(Q,C)$-quasigeodesic ray of $X_L $. Then there is a sequence $(x_i)_{i \in \mathbb{N}} \subseteq P $ and a $L$-chain $\{c_i\}_{i \in \mathbb{N}}$ such that for every $n \in \mathbb{N}$, $\{c_1, \dots, c_n\} $ separates $o $ from $x_n $. 
	\end{elem}
	
	The following proposition will be essential. 
	
	\begin{eprop}\label{prop inverse map}
		Let $(x_n) $ be a sequence in $X$ such that $(x_n)$ converges to a point $\xi_L$ in the Gromov boundary $\bdg X_L$ (for the hyperbolic topology). Then $(x_n)$ converges to a point $\xi \in \bd X$ in the visual boundary of $X$. 
	\end{eprop}

	\begin{proof}[Proof of Proposition \ref{prop inverse map}]
		As $X_L$ is a $\alpha $-almost geodesic space, Proposition \ref{prop exist qgeod hyp} implies that $X_L$ is a $(Q,C) $ quasigeodesic space. In particular there exists a $(Q,C)$-quasigeodesic $P : [0 , \infty) \rightarrow X_L $ emanating from $o$ and representing $\xi_L$. By Lemma \ref{lem dual chain qgeod}, there exists a sequence $(y_i) $ in $P$ and an infinite $L$-chain $\{c_i\}$ such that for every $n \in \mathbb{N}$, $\{c_1, \dots, c_n\} $ separate $o $ from $ y_n$. Since $x_n \rightarrow \xi_L$, for every $n \in \mathbb{N}$, there exists $p_n \geq 0 $ such that for all $m \geq p_n $, $\{c_0, \dots, c_{(4L + 10)n} \} $ is a $L$-chain separating $o $ from $x_m$. 
		
		Now, define the geodesic segments $\gamma_n = [o, x_n]$ (for the $\cat$(0) distance) parametrized by arc-length. In other words for every $t \in [0, \infty)$, $\gamma_n(t) $ is the point at distance $t$ from $o$ on the geodesic segment $[0, x_n]$ if $t \leq d(o, x_n)$ and $\gamma_n (t) = x_n$ if $t \geq d(o, x_n)$. We are going to prove that for every $R> 0$, $\gamma_n(R)$ is a Cauchy sequence. 
		
		Let $R>0$ be fixed. By Lemma \ref{lem dual chain}, there exists a $L$-chain of size $n+1$, dual to the geodesic segment $[o, x_{p_n}]$, and separating $c_1 $ from $c_{(4L+10)n}$. Denote this chain by $\{h_1, \dots, h_{n+1}\}$. Since curtains are thick, the number of curtains in $\{c'_i\}$ that intersect the metric ball $B(o, R) $ is less than $R+1$. In particular, for $n $ large enough and for every $m \geq p_n$, the number of curtains in $\{h_1, \dots, h_{n+1}\}$ separating $\gamma_m (R)$ and $x_m $ is greater than $n -R$. Let $n \geq R + 3$, so that $\{h_{n-R+1}, \dots, h_{n+1}\}$ is a $L$-chain separating $B(o, R)$ from $x_m $ for every $m \geq p_n$. 
		
		Let $m \geq p_n$, and take $t^n_m \in [0, d(o, x_m)]$ such that $\gamma_m(t^n_m) \in h_n$, which is possible since $\gamma_m$ crosses $h_n$, see Figure \ref{figure prop inverse map}. By the ``bottleneck'' Lemma \ref{lem bottleneck intro}, for all $m , m'\geq p_n$, $d(\gamma_m(t^n_m), \gamma_{m'}(t^n_{m'})) \leq 4L+3$. Without loss of generality, we can assume that $t^n_m \leq t^n_{m'} $. Note that because curtains are thick, and since the projections do not increase distances (Proposition \ref{prop projection cat}), $d(o, \gamma_m(t_m^n)) \geq d(o, h_n)$, hence $t^n_m \geq n$. By convexity of the $\cat$(0) metric (Proposition \ref{prop convexite cat}), we have 
		\begin{eqnarray}
			d(\gamma_m(t^n_m), \gamma_{m'}(t^n_{m})) \leq 4L+3. \nonumber
		\end{eqnarray}
		
		The function $t \rightarrow d(\gamma_m(t), \gamma_{m'}(t))$ is convex, hence 
		\begin{eqnarray}
			d(\gamma_m(R), \gamma_{m'}(R)) &\leq& \frac{R}{t^n_m} d(\gamma_m(t^n_m), \gamma_{m'}(t^n_{m}))\nonumber \\
			& \leq & \frac{R}{t^n_m} (4L+3) \nonumber \\
			& \leq & \frac{R}{n} (4L+3) \nonumber. 
		\end{eqnarray}
		
		As a consequence, for every $\varepsilon >0$, there exists $p_\varepsilon \in \mathbb{N}$ large enough such that for all $m, m' \geq p_n$, $d(\gamma_m(R), \gamma_{m'}(R)) \leq \varepsilon$. Therefore, for all $R >0$, $\gamma_m(R)$ is a Cauchy sequence, and because $X$ is complete, the sequence $(\gamma_n)_n$ converges to a geodesic ray. 
	\end{proof}
	
	\begin{figure}
		\centering
		\begin{center}
			\begin{tikzpicture}[scale=1]
				\draw (-3,0) -- (4,0)  ;
				\draw (2.8,2) -- (2.8,-1.5)  ;
				\draw (-2.5,2) -- (-2.5,-1.5)  ;
				\draw (1.2,2) -- (1.2,-1.5)  ;
				\draw (2,2) -- (2,-1.5)  ;
				\draw (2,0.5) node[right]{$\leq 4L+3$} ;
				\draw (-3, 0) to[bend right = 15] (5, 3);
				\draw (-3, 0) to[bend left = 15] (4.5, -2 );
				\draw (-3,0) node[below left]{$o$} ;
				\draw (2, -1.5) node[below]{$h_{n-1}$} ;
				\draw (1.2, -1.5) node[below]{$h_n$} ;
				\draw (-2.5, -1.5) node[below]{$h_1$} ;
				\draw (2.8, -1.5) node[below]{$h_{n+1}$} ;
				\draw (4,0) node[below right]{$x_{p_n}$} ;	
				\draw (2,-0.8) node[above right]{$\gamma(t^n_{m'})$} ;
				\draw (2,1.2) node[above left]{$\gamma(t^n_{m})$} ;
				\draw (5,3) node[right]{$x_{m}$};
				\draw (4.5, -2 ) node[right]{$x_{m'}$};
				\filldraw[black] (2,-0.8) circle(1.5pt);
				\filldraw[black] (2,1.2) circle(1.5pt);
				\filldraw[black] (4.5,-2) circle(1.5pt);
				\filldraw[black] (4,0) circle(1.5pt);
				\filldraw[black] (5,3) circle(1.5pt);
			\end{tikzpicture}
		\end{center}
		\caption{Proof of Proposition \ref{prop inverse map}.}\label{figure prop inverse map}
	\end{figure}
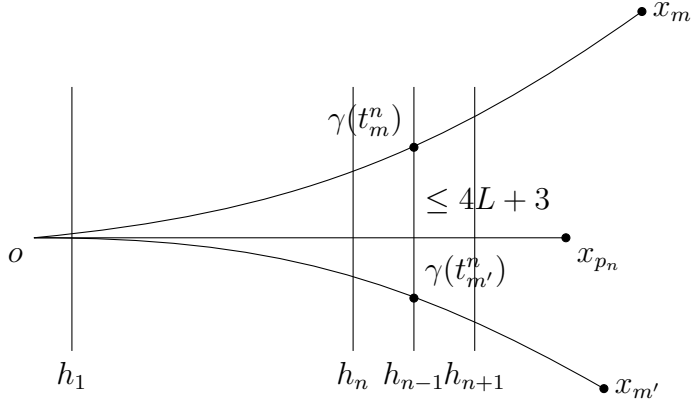

	The rest of the proof of Theorem \ref{thm homeomorphism bords XL} is now very similar to what was done in \cite[Proposition 8.9]{petyt_spriano_zalloum22}. We include it for clarity. 	
	
	\begin{proof}[Proof of Theorem \ref{thm homeomorphism bords XL}]
		The map $\iota_L : (X,d) \rightarrow (X,d_L)$ is $(1,1)$-coarsely Lipschitz, hence by \cite[Lemma 6.18]{incerti-medici_zalloum21}, the existence and continuity of $\partial_L$ follows from the fact that $X_L$ is $\alpha$-almost geodesic. 
		\newline
		
		If $\gamma$ crosses an infinite $L$-chain, then there is an infinite $L$-chain $\{h_i\}$ dual to $\gamma$. Orient this chain so that $o \in h^{-1}_i $ for all $i$. Let $\gamma '$ be a distinct geodesic ray, emanating from $o $. Then by Lemma \ref{lem bottleneck intro} and there must be a $k $ such that $\gamma' \subseteq h_k^{-1}$. In particular, if $x_n \in h_n \cap \gamma$, then $d_L(\gamma', x_n) \geq n-k$, and $\gamma'$ does not lie in a finite $X_L$ neighbourhood of $\gamma$. This argument shows that the map $\partial_L$ is injective.
		\newline
		
		Let $P : [0, \infty) \rightarrow X_L$ be a quasigeodesic ray (for the metric $d_L$) emanating from $o$, and let $(x_n)$ be an unbounded sequence on $P$. By Lemma \ref{lem dual chain qgeod}, there exists an infinite $L$-chain $\{c_i\}$ such that for every $n \in \mathbb{N}$, $\{c_1, \dots, c_n\} $ separate $o $ from $ x_n$. Proposition \ref{prop inverse map} shows that $[o, x_n]$ converges to a geodesic ray $\gamma$. By construction, $\gamma$ crosses infinitely many curtains $\{c_i\}$, hence belongs to $B_L$. Note that by Proposition \ref{cor stab qgeod hyp}, any $(1,C)$-quasigeodesic ray representing the same endpoint $\xi_L$ as $P$ lies in a uniform neighbourhood of $P$, with universal constants. Therefore, $\iota(\gamma)$ is an unparametrized rough geodesic, and lies in a uniform neighbourhood of any quasigeodesic ray between $\iota(o)$ and $\xi_L$. Then, the choice of $\gamma$ does not depend upon the choice of $P$, and the application $\partial_L$ is surjective. 
		\newline
		
		It remains to prove that the inverse map is continuous. This part is completely similar to what is presented in $(4)$ of the proof of \cite[Proposition 8.9]{petyt_spriano_zalloum22}, which does not use any properness assumption on $X$. 
	\end{proof}

	\subsection{Convergence to the boundary}
	
	In order to use the results concerning random walks in hyperbolic spaces, we must show that the action of a group $G$ on a proper $\cat$(0) space with rank one isometries induces a non-elementary action on some hyperbolic model $(X, d_L)$. 
	
	\begin{eprop}\label{prop non elem}
		Let $G$ be a group acting by isometries on a complete $\cat$(0) space $(X,d)$, and assume that $G$ contains a pair of independent contracting element for this action. Then there exists $L \in \mathbb{N}$ such that $G$ acts on the hyperbolic space $(X, d_L)$ by isometries, with a pair of independent loxodromic isometries. 
	\end{eprop}
	
	\begin{proof}
		The action $G \curvearrowright (X,d)$ contains a pair of independent contracting elements $g,h \in G$. By Theorem \ref{thm contract loxodromic}, there exists $L \in \mathbb{N}$ such that $g$ and $h$ act on $(X, d_L)$ as loxodromic isometries. As $g $ and $h$ are independent, their fixed points form four distinct points of the visual boundary $\bd X$. Now seen in $X_L = (X, d_L)$, their fixed points sets must also form four distinct points of $\bdg X_L$ because of the homeomorphism $\partial_L : B_L \longrightarrow \bdg X_L$ in Theorem \ref{thm homeomorphism bords XL}. This means that the action $G \curvearrowright X_L$ is non-elementary. 
	\end{proof}
	
	We can now prove the convergence to the boundary. In the following, even when not specified, the $\cat$(0) spaces we consider are always assumed separable. Indeed, we need the visual boundary to be a standard Borel set.
	
	\begin{ethm}\label{thm cv cat}
		Let $G$ be a discrete countable group and $G \curvearrowright X$ an action by isometries on a complete separable $\cat$(0) space $X$. Let $\mu \in \prob(G) $ be an admissible probability measure on $G$, and assume that $G $ contains a pair of independent contracting isometries. Then for every $o \in X$, and for $\mathbb{P}$-almost every $\omega \in \Omega$, the random walk $(Z_n (\omega) o)_n $ converges to a boundary point $z^{+}(\omega) \in \bd X$. Moreover, there exists $L \geq 0 $ such that almost surely, $z^+(\omega) \in B_L$. 
	\end{ethm}
	
	\begin{proof}
		Due to Proposition \ref{prop non elem}, there exists $L$ such that $G$ acts on the hyperbolic model $X_L$ by isometries, with a pair of loxodromic isometries. By Theorem \ref{thm cv rw hyp}, the random walk $(Z_n o )_n$ converges to a point of the boundary $z_L^+(\omega)\in \bdg X_L$ almost surely. Now due to Proposition \ref{prop inverse map},$(Z_n o )$ converges almost surely to a point of the boundary $z^+(\omega) = \partial_L^{-1}(z_L^+(\omega))\in \bd X$. 
	\end{proof}
	
	Moreover, as a corollary of Theorems \ref{thm uniq bdry map hyp} and \ref{thm homeomorphism bords XL}, we have the following result. 
	
	\begin{ecor}
		Let $G$ be a discrete countable group and $G \curvearrowright X$ an action by isometries on a complete $\cat (0)$ space $X$. Let $L$ be such that $G$ acts on $X_L$ non-elementarily. Let $(B_-, B_+)$ be a $G$-boundary pair. Then there exist essentially unique measurable $G$-equivariant maps $\phi_- : B_- \rightarrow B_L$ and $\phi_+ : B_+ \rightarrow B_L$ such that the map 
		\begin{eqnarray}
			\phi_{\bowtie} : (b_-, b_+ )\in (B_- \times B_+) \mapsto (\phi_-(b_-), \phi_+(b_+))
		\end{eqnarray} 
		is essentially contained is the set of distinct pairs of points of the boundary $(B_L)^{(2)}$. Moreover: 
		\begin{enumerate}
			\item[(i)] $\Map_G ( B_-, \prob(B_L)) = \{\delta \circ \phi_-\}$ and $\Map_G ( B_+, \prob(B_L)) = \{\delta \circ \phi_+\}$. 
			\item[(ii)] $\Map_G ( B_- \times B_+, B_L) = \{\phi_-\circ \text{pr}_-, \phi_+\circ \text{pr}_+\}$ 
			\item[(iii)] $\Map_G ( B_- \times B_+, (B_L)^{(2)}) = \{\phi_{\bowtie}, \tau\circ\phi_{\bowtie}\}$, where $\tau (\xi,\eta) = (\eta, \xi)$. 
		\end{enumerate} 
	\end{ecor}

	We end this section with a geometric result, that shows that the random walk converges to points that are isolated in the angular metric. 
	
	\begin{eprop}\label{prop angle B_L}
		Let $\xi \in \bd X$ be a point of the subspace $B_L$, and let $\eta \neq \xi $ be any other point in $\bd X$. Then $\angle (\xi, \eta) = \pi$. 
	\end{eprop}
	
	\begin{proof}
		Let $o \in X$, and let $\gamma : [0, \infty) \rightarrow X$ be a geodesic ray based at $o $ representing $\xi$. By \cite[Lemma 2.21]{petyt_spriano_zalloum22}, there is an infinite $L$-chain $\{h_i\}_{i \in \mathbb{N}}$ dual to $\gamma$. Let $\gamma' : [0, \infty) \rightarrow X$ be a geodesic ray representing $\eta$ and such that $\gamma'(0) = o $. By \cite[Lemma 3.1]{petyt_spriano_zalloum22}, if $\gamma'$ meets an infinite number of curtains in $\{h_i\}$, it must cross the chain, and hence $\xi = \eta$. As a consequence, there exists $k \geq 1$ such that $\gamma' \subseteq h_{k-1}^{-}$. Let $t_0 >0$ such that for all $t \geq t_0$, $\gamma(t) \in h_{k+1}$. By Lemma \ref{lem bottleneck intro}, for all $t \geq t_0$ and $t' \in [0, \infty)$, there exists a point $p_{t,t'} \in [\gamma(t), \gamma(t')] \cap h_{k}$ such that, if $q_{t,t'} := \pi_\gamma(p_{t,t'})$, we have $d(p_{t,t'}, q_{t,t'}) \leq 2L+1$. 
		
		Now recall that by \cite[Proposition II.9.8]{bridson_haefliger99}, 
		\begin{eqnarray}
			\angle (\xi, \eta) = \lim_{t,t' \rightarrow \infty} \overline{\angle}_o (\gamma(t), \gamma'(t'))  \nonumber. 
		\end{eqnarray}
		
		Denote $\overline{\angle}_o (\gamma(t), \gamma'(t')) $ by $\alpha_{t,t'}$. Let $\overline{\Delta} := \overline{\Delta} (\overline{o}, \overline{\gamma(t)}, \overline{\gamma'(t')})$ be the Euclidean comparison triangle of $\Delta (o, \gamma(t), \gamma'(t'))$. By the Euclidean law of cosines, 
		\begin{eqnarray}\label{eq al kashi}
			d(\gamma(t), \gamma'(t'))^2 &=& d(o, \gamma(t))^2 +d(o, \gamma'(t'))^2 - 2 d(o, \gamma(t))d(o, \gamma'(t')) \cos(\alpha_{t,t'}) \nonumber \\
			&  = & t^2 + t'^2 - 2t t' \cos(\alpha_{t,t'})
		\end{eqnarray}
		Since $q_{t,t'}$ belongs to the pole of $h_k$ there exists $A \geq  0$ such that $d(q_{t,t'}, o ) \in  [A, A+1]$ for all $t \geq t_0, t' \geq 0$. By the ``bottleneck'' lemma, 
		\begin{eqnarray}
			t' - (A + 2L + 1) \leq d(\gamma'(t'), p_{t,t'}) \leq t' + A +1 + 2 L+1 \nonumber, 
		\end{eqnarray}
		and 
		\begin{eqnarray}
			t-( A +1 + 2L + 1 )\leq d(\gamma(t), p_{t,t'}) \leq t+ A+ 2 L+1 \nonumber
		\end{eqnarray}
		As a consequence, 
		\begin{eqnarray}
			t +t' - (2A+ 4L + 3) \leq d(\gamma'(t'), \gamma(t)) \leq t +t' + 2A + 4 L+3 \nonumber, 
		\end{eqnarray}
		Let $C := 2A + 4L +3$. For $t,t' $ large enough, we then have 
		\begin{eqnarray}
			t^2 + t'^2 + C^2 + 2 tt' - 2C(t + t') \leq d(\gamma'(t'), \gamma(t))^2 
			\leq t^2 + t'^2 + C^2 + 2 tt' + 2C(t + t')  \nonumber. 
		\end{eqnarray}
		Combining this with \eqref{eq al kashi}, we get
		\begin{eqnarray}
			C^2 + 2 tt' - 2C(t+t') \leq - 2 tt' \cos(\alpha_{t,t'})
			\leq  C^2 + 2 tt' + 2C(t+t')  \nonumber. 
		\end{eqnarray}
		
		As $t, t' \rightarrow \infty$, we obtain that $\cos(\alpha_{t,t'}) \rightarrow -1$, hence $\alpha_{t,t'} \rightarrow \pi$, which proves the proposition. 
		
	\end{proof}
	
	In fact, if the space $X$ is proper, we can show that any point $\xi_L \in B_L$ is a visibility point, see \cite[Lemma 8.4]{petyt_spriano_zalloum22}.  
	
	\begin{ecor}\label{cor angle lim pts}
		Let $\nu \in \prob(\bd X)$ be the hitting measure of the random walk $(Z_n o)$. Then for any pair of points of distinct points $\xi \neq \eta$, then $\nu \otimes \nu$-almost surely, $\angle(\xi, \eta) = \pi$. If moreover $X $ is proper, then $\xi $ and $\eta$ are almost surely joined by a geodesic line. 
	\end{ecor}
	
	\begin{proof}
		The first result is just a combination of Theorem \ref{thm cv cat} and Proposition \ref{prop angle B_L}. The second assertion uses the fact that if moreover the space is proper, points of $B_L$ are visibility points \cite[Lemma 8.4]{petyt_spriano_zalloum22}.
	\end{proof}
	
	\begin{erem}
		If $X$ is proper, then actually almost every pair of boundary points is joined by a rank one geodesic line. It was done in \cite[Corollary 4.7]{LeBars22}.
	\end{erem}

	\subsection{Uniqueness of the stationary measure}\label{section uniq stat cat}

	From now on, let $G$ be a discrete countable group and $G \curvearrowright X$ an action by isometries on a complete $\cat (0)$ space $X$, and assume that $G$ admits a pair of independent contracting isometries. Let $\mu \in \prob(G) $ be an admissible probability measure on $G$. Fix $L > 0$ such that $G$ acts on the hyperbolic space $X_L = (X,d_L)$ with a pair of independent loxodromic isometries. By Theorem~\ref{thm cv cat}, there exists a stationary measure $\tilde{\nu}$ on $\bd X$ given by the hitting measure of the random walk. The goal of this section is to show that` $\tilde{\nu}$ is the unique $\mu$-stationary measure on $\overline{X} = X \cup \bd X$. For the rest of the section, let $\nu$ be any stationary measure on $\bd X$. In order to prove that $\nu$ is unique, we show that the measures $\nu_\omega $ given by the Theorem \ref{thm mesures limites} are in fact Dirac measures $\delta_{\psi_+(\omega)}$, and that they do not depend on $\nu $. 
	\newline
	
	The following result was proven by Karlsson, and establishes a weak notion of convergence action for $\cat$(0) spaces. 
	
	\begin{ethm}[{\cite[Theorem 27]{karlsson05}}]\label{thm pi/2 cv karlsson}
		Let $X$ be a complete $\cat$(0) space. Let $o \in X$, and let $g_n$ be a sequence of isometries such that $g_n o \rightarrow \xi^+ \in \bd X$ and $g_n^{-1} o \rightarrow \xi^{-}\in \bd X$. Then for any $\eta \in \bd X $ with $\angle (\eta, \xi^{-}) > \pi/2 $, we have 
		\begin{eqnarray}
			\limsup_n \angle (g_n \eta, \xi^+) \leq \pi/2 \nonumber. 
		\end{eqnarray}
	\end{ethm}

	Using this result, we can prove that there is a unique $\mu$-stationary measure on the boundary.

	\begin{ethm}\label{thm uniqueness stat measure cat}	
		Let $G$ be a discrete countable group and $G \curvearrowright X$ a an action by isometries on a complete $\cat $(0) space $X$ such that $G$ contains a pair of independent contracting elements. Let $\mu \in \prob(G) $ be an admissible probability measure on $G$. Then there exists a unique $\mu$-stationary measure $\nu \in \prob(\overline{X})$, which is the hitting measure of the random walk $(Z_n (\omega)o)_n$.  
	\end{ethm}
	
	\begin{proof}[Proof of Theorem \ref{thm uniqueness stat measure cat}]
		Let $L$ be such that $G$ acts non-elementarily by isometries on $X_L$. By Theorem \ref{thm cv cat}, the sequence $(Z_{n_k} o)_n $ converges almost surely to a point of the boundary $\xi^+(\omega) \in B_L \subseteq \bd X$. By Lemma \ref{prop subseq inverse}, for $\mathbb{P}$-almost every $\omega$, there exists a subsequence $n_k$ such that the the sequence $(Z_n^{-1}o)$ in $X_L$ converges to a point of the Gromov boundary $\bdg X_L$. By Proposition \ref{prop inverse map}, this implies that $\mathbb{P}$-almost surely, there exists a subsequence such that the sequence $(Z_n^{-1}o)$ in $X$ converges to a point $\xi^- (\omega) \in B_L$ of the visual boundary $\bdg X_L$. Now by Proposition \ref{prop angle B_L}, we have that 
		\begin{eqnarray}
			\{\eta \in \bd X \; | \; \angle(\eta, \xi^+(\omega)) < \pi\} = \{\xi^+(\omega)\}  \nonumber
		\end{eqnarray}
		and 
		
		\begin{eqnarray}
			\{\eta \in \bd X \; | \; \angle(\eta, \xi^-(\omega))< \pi\} = \{\xi^-(\omega)\} \nonumber.
		\end{eqnarray}

		Let $\nu$ be a $\mu$-stationary measure. For $\mathbb{P}$-almost every $\omega \in \Omega$, let $\xi^+(\omega), \xi^-(\omega) \in \bd X$ be such that $Z_n (\omega) o \rightarrow \xi^+ (\omega) $ and $Z_{n_k}^{-1}(\omega)o \rightarrow_k \xi^- (\omega) \in \bd X$. By Theorem \ref{thm pi/2 cv karlsson}, for all measurable set $F \subseteq \bd X \setminus \{\xi^- (\omega)\}$, and for all open set $U$ containing $\xi^+(\omega)$, we have that $Z_{n_k}(\omega) F \rightarrow_k U$ in the sense that there exists $k_0 \geq 0$ such that for all $k \geq k_0$, $Z_{n_k} (\omega) F \subseteq U$.  As $\nu$ is $\mu$-stationary, it is non-atomic by Lemma \ref{lem non atomic} and $\nu(\xi^-(\omega)) = 0$. As a consequence, for all measurable set $F$ of $\bd X$,  $Z_{n_k}(\omega)_\ast\nu(F)$ converges to $1$ if $\xi^{+}(\omega) \in F$ and to $0$ otherwise. In other words, $Z_{n_k}(\omega) \nu \rightarrow_k \delta_{\xi^{+}(\omega)}$ in the weak-$\ast$ topology. 
		
		Now by Theorem \ref{thm mesures limites}, there exists an (essentially well-defined) measurable $G$-map $\omega \mapsto \nu_\omega$ such that $\mathbb{P}$-almost surely 
		\begin{eqnarray}
			Z_n(\omega)_\ast\nu \rightarrow \nu_\omega \in \prob(\bd X) \nonumber 
		\end{eqnarray}
		in the weak-$\ast$ topology. By uniqueness of the limit, $\nu_\omega = \delta_{\xi^{+}(\omega)}$ almost surely. Moreover, $\nu$ satisfies the decomposition 
		\begin{eqnarray}
			\nu = \int_{\Omega} \nu_\omega d \mathbb{P}(\omega) \nonumber. 
		\end{eqnarray}
		As a consequence, $\nu$ can be written 
		\begin{eqnarray}
			\nu = \int_{\Omega} \delta_{\xi^{+}(\omega)} d \mathbb{P}(\omega) \nonumber, 
		\end{eqnarray}
		where $\delta_{\xi^+(\omega)} $ does not depend on $\nu$. The theorem follows. 
	\end{proof}

	\subsection{Positive drift and consequences}
	
	In this section, study the speed at which the random walk converges to the boundary. 
	Let $G$ be a discrete countable group, $G \curvearrowright X$ an action by isometries on a complete $\cat (0)$ space $X$, and assume that $G$ admits a pair of independent contracting isometries. Let $\mu \in \prob(G) $ be an admissible probability measure on $G$. Fix $L > 0$ such that $G$ acts on the hyperbolic space $X_L = (X,d_L)$ with a pair of independent loxodromic isometries. By Theorem~\ref{thm cv cat}, there exists a stationary measure $\tilde{\nu}$ on $\bd X$ given by the hitting measure of the random walk. 
	
	We refer to Section \ref{section drift} for the definitions. Assume that $\mu$ has finite first moment $\int d(go, o) d\mu(g) < \infty$. As a consequence of Kingman subadditive Theorem \ref{thm kingman}, there is a constant $\lambda$ such that for $\mathbb{P}$-almost every sample path $(Z_n(\omega)x)_n$ we have 
	\begin{eqnarray}
		\frac{1}{n}d(Z_n(\omega) o , o) \underset{n\rightarrow\infty}{\longrightarrow} \lambda:= \inf_n \frac{1}{n} \int d(Z_n(\omega) o, o) d\mathbb{P}(\omega) \in [0, \infty)\nonumber. 
	\end{eqnarray}
	In general, the \textit{escape rate} (or \textit{drift}) of the random walk is defined as 
	\begin{eqnarray}
		l_X(\mu) :=\lim_{n\rightarrow \infty} \frac{1}{n}\int_\Omega d(Z_n(\omega)o, o) d \mathbb{P}(\omega) \nonumber.
	\end{eqnarray}
	if $\mu$ has finite first moment, and $l_X(\mu)= \infty $ otherwise. 
	We recall the following result, due to Gouëzel. 
	
	\begin{ethm}[{\cite[Theorem 1.2]{gouezel22}}]\label{thm gouezel 1.2}
		Let $G$ be a discrete countable group and $G \curvearrowright (Y,d_Y)$ be a non-elementary action on a Gromov-hyperbolic space $Y$, and let $o \in Y$. Let $\mu$ be an admissible probability measure on $G$. Recall that the drift $ l_Y(\mu)$ is positive. Let $r < l_Y(\mu)$. Then there exists $\kappa > 0 $ such that 
		\begin{eqnarray}
			\mathbb{P}(d_Y(o, Z_n (\omega) o ) \leq r n) \leq \exp(-\kappa n) \nonumber. 
		\end{eqnarray}
	\end{ethm}	
	Note that there is no moment assumption on $\mu$ in Theorem \ref{thm gouezel 1.2}. In particular, the escape rate $l_Y(\mu)$ is positive (possibly infinite). 
	
	Using the fact that between any two points $x,y \in X$, $d_L(x,y) \leq d(x,y)$, the drift of the random walk in $X$ is positive. 
	
	\begin{ethm}\label{thm drift cat}
		Let $G$ be a discrete countable group and $G \curvearrowright X$ a an action by isometries on a complete $\cat $(0) space $X$ such that $G$ contains a pair of independent contracting elements. Let $\mu \in \prob(G) $ be an admissible probability measure on $G$. Then almost surely, $l_X(\mu) > 0$. If moreover $\mu$ has finite first moment, then $l_X(\mu) < \infty $, and almost surely
		\begin{eqnarray}
			\frac{1}{n}d(o, Z_n o) \underset{n}{\rightarrow} l_X(\mu) \nonumber. 
		\end{eqnarray}
	\end{ethm}

	\begin{proof}
		Let $L \geq 0$ be such that $G$ acts on $X_L$ by isometries with a pair of independent loxodromic isometries. By Gouëzel's Theorem \ref{thm gouezel 1.2}, $l_{X_L}(\mu)$ is almost surely positive. In other words 
		\begin{eqnarray}
			\liminf_n \frac{1}{n}\int_\Omega d_L(Z_n(\omega)o, o) d \mathbb{P}(\omega) > 0\nonumber.
		\end{eqnarray}
		But by definition of the chain metric $d_L \leq d$, and the result follows. 
	\end{proof}

	We can now add an application that is a reformulation of \cite[Theorem 2.1]{karlsson_margulis}, now that we know that the drift is positive. It states that we have a geodesic tracking of the random walk. 
	
	\begin{ecor}[{Sublinear tracking}]\label{cor sublinear tracking}
		Let $G$ be a discrete countable group and $G \curvearrowright X$ be an action by isometries on a complete $\cat (0)$ space $X$ such that $G$ contains a pair of independent contracting elements. Let $\mu \in \prob(G) $ be an admissible probability measure with finite first moment on $G$. Then for almost every $\omega \in \Omega$, there is a unique geodesic ray $\gamma^\omega : [0, \infty) \rightarrow X$ starting at $o$ such that 
		\begin{equation}
			\lim_{n\rightarrow \infty} \frac{1}{n} d(\gamma^\omega(\lambda n), Z_n(\omega)o) = 0, \nonumber
		\end{equation}
		where $\lambda $ is the (positive) drift of the random walk. 
	\end{ecor}

	\subsection{Proportion of contracting elements}
	
	In this Section, we prove that the probability that the step $(Z_n)$ is a contracting element goes to 1 as $n $ goes to infinity. Let then $G$ be a discrete countable group and $G \curvearrowright X$ an action by isometries on a complete $\cat (0)$ space $(X,d)$, and assume that $G$ admits a pair of independent contracting isometries for this action. Let $\mu \in \prob(G) $ be an admissible probability measure on $G$, and let $L$ be such that $G$ acts on the hyperbolic space $X_L$ in a non-elementary way. For random walks in geodesic hyperbolic spaces, Maher and Tiozzo show that the proportion of loxodromic isometries in the random walk $(Z_n)$ goes to 1 \cite[Theorem 1.4]{maher_tiozzo18}, but thanks to Bonk and Schramm \cite{bonk_schramm00}, this result extends to non-geodesic hyperbolic spaces. Combining this with Theorem \ref{thm contract loxodromic}, stating that a an element of $G$ that acts on $X_L$ as a loxodromic isometry is a contracting element of $X$, we prove the following result. 
	
	\begin{ethm}[{Proportion of contracting elements}]\label{thm prop contracting}
		Let $G$ be a discrete countable group and $G \curvearrowright X$ an action by isometries on a complete $\cat (0)$ space $(X,d)$, and assume that $G$ admits a pair of independent contracting isometries for this action. Let $\mu \in \prob(G) $ be an admissible probability measure on $G$. Then the proportion of contracting elements in $(Z_n)_n$ goes to $1$ as $n $ goes to $\infty$, i.e. 
		\begin{eqnarray}
			\mathbb{P}(\omega \, : \, Z_n (\omega) \text{ is a contracting isometry }) \underset{n\rightarrow \infty}{\rightarrow} 1 \nonumber.
		\end{eqnarray}
		If the support of $\mu$ is bounded in $X$, then there exist $c < 1$ and $ K, L>0$ such that for all $n $, 
		\begin{eqnarray}
			\mathbb{P}(\omega \, : \, Z_n (\omega) \text{ is not contracting }) \leq Kc^n  \nonumber.
		\end{eqnarray}
	\end{ethm}
	
	Let $(Y, d_Y)$ be a $\delta$-hyperbolic metric space in the sense of Gromov, and let $g \in \iso(Y)$ be an isometry of $Y$. Let $o \in Y$ be a basepoint. The \textit{translation length} of $g$ is defined to be 
	\begin{eqnarray}
		\ell(g) := \lim_{n \rightarrow \infty} \frac{1}{n}d_Y(o, g^n o) \nonumber. 
	\end{eqnarray}
	This limit is well-defined, and does not depend on the choice of the basepoint. An isometry $g $ is loxodromic if and only if its translation length $\ell(g)$ is $>0$. 
	
	In a hyperbolic space, we can estimate the translation length of an isometry in terms of its displacement on a given point, and a well-chosen Gromov product. 
	
	\begin{eprop}
		Let $(Y, d) $ be a $\delta$-Gromov hyperbolic space, and $o \in Y$ a point. Then there exists a constant $C:= C(\delta)$ depending only on $\delta$ such that the following holds. If $g \in \iso(Y)$ is an isometry satisfying the inequality 
		\begin{eqnarray}
			d_Y(o, go) \geq 2 (g o |g^{-1}o)_o + C,\nonumber
		\end{eqnarray}
		then the translation length of $g$ can be estimated 
		\begin{eqnarray}
			\ell(g) \geq d_Y(o, go) - 2 (g o |g^{-1}o)_o + O(\delta). \nonumber
		\end{eqnarray}
	\end{eprop}
	
	\begin{proof}
		If $(Y, d) $ is a geodesic hyperbolic space, this result is exactly \cite[Proposition 5.8]{maher_tiozzo18}. Denote by $C$ this universal constant. 
		
		Now in the general case, let $\iota : Y \rightarrow (Z, d_Z) $ be an isometric embedding of $Y$ into a geodesic, complete and $\delta$-hyperbolic metric space given by Theorem \ref{thm isom embed geod hyp space}. For convenience, we write $\widetilde{g} : \iota(Y) \rightarrow Z$, the isometry $\widetilde{g} := \iota \circ g \circ \iota^{-1}$, $\widetilde{g^{-1}} := \iota \circ g^{-1} \circ \iota^{-1}$ and $z := \iota(o) \in Z$.  
		
		By definition, $\iota(go) = \widetilde{g}z$ and $\widetilde{g^{-1}} = (\widetilde{g})^{-1}$. Moreover, since $\iota$ is an isometric embedding, 
		\begin{eqnarray}
			d_Z(\widetilde{g}z, z) = d_Y(go, o) \nonumber
		\end{eqnarray}
		and
		\begin{eqnarray}
			(\widetilde{g}z|(\widetilde{g})^{-1}z)_z = (go|g^{-1}o)_o \nonumber.
		\end{eqnarray}
		In particular, if one assumes that 
		\begin{eqnarray}
			d_Y(go, o) - 2  (g o |g^{-1}o)_o \geq C \nonumber, 
		\end{eqnarray}
		then 
		\begin{eqnarray}
			d_Z(\widetilde{g}z, z) - 2  (\widetilde{g}z|(\widetilde{g})^{-1}z)_z \geq C \nonumber. 
		\end{eqnarray}
		Following the same lines as in the proof of \cite[Proposition 5.8]{maher_tiozzo18}, we can conclude that 
		\begin{eqnarray}
			\lim_n \frac{1}{n}d_Z(z, \widetilde{g}^n z) \geq d_Z(\widetilde{g}z, z) - 2  (\widetilde{g}z|(\widetilde{g})^{-1}z)_z + O(\delta) . \nonumber
		\end{eqnarray}
		In particular, we have the lower bound 
		\begin{eqnarray}
			\ell(g) = \lim_n \frac{1}{n}d_Y(o, g^no) \geq d_Y(go, o) - 2  (g o |g^{-1}o)_o + O (\delta) \nonumber. 
		\end{eqnarray}
		
	\end{proof}
	
	With this in hand, the rest of the proof of \cite[Theorem 1.4]{maher_tiozzo18} remains true for almost geodesic, $\delta$-hyperbolic metric spaces. 
	
	\begin{ethm}[{\cite[Theorem 1.4]{maher_tiozzo18}}]\label{thm trans length hyp}
		Let $G$ be a discrete countable group and $G \curvearrowright Y$ a non-elementary action by isometries on a $\delta$-hyperbolic, almost geodesic metric space $(Y,d)$. Let $\mu \in \prob(G) $ be an admissible probability measure on $G$. Then the translation length $\ell(Z_n(\omega))$ of the group element $Z_n(\omega)$ grows almost surely at least linearly, i.e. there exists $L > 0$ such that  
		\begin{eqnarray}
			\mathbb{P}(\omega \, : \, \ell(Z_n (\omega)) \leq Ln ) \underset{n\rightarrow \infty}{\rightarrow} 0 \nonumber.
		\end{eqnarray}
		If $G$ is finitely generated, then there exist $c < 1$ and $ K, L>0$ such that for all $n $, 
		\begin{eqnarray}
			\mathbb{P}(\omega \, : \, \ell(Z_n (\omega)) \leq Ln ) \leq Kc^n  \nonumber.
		\end{eqnarray}
	\end{ethm}
	
	The key in proving this Theorem is to give a lower bound on $d(o, Z_n o )$, which is given by the positive drift \cite[Theorem 1.3]{gouezel22}, and an upper bound for the Gromov product $(Z_n^{-1}o | Z_n )_o$. We omit this proof, since it would just reproduce the argument given by Maher and Tiozzo in \cite[Lemmas 5.9, 5.10 and 5.11]{maher_tiozzo18}, with minor modifications to fit the context of almost geodesic spaces. If $G$ is finitely generated, the strategy given by Maher in \cite{maher12} gives the linear growth of the translation length with exponential decay. 
	
	\begin{proof}[Proof of Theorem \ref{thm prop contracting}]
		By Proposition \ref{prop non elem}, there exists $L > 0$ such that $G$ acts on $X_L = (X, d_L)$ by isometries and non-elementarily. The space $X_L$ is almost geodesic, $\delta$-hyperbolic by Theorem \ref{theorem hyperbolic models intro}, and we can apply Theorem \ref{thm trans length hyp}: seen as isometries of $X_L$, the translation length of $(Z_n) $ grows almost surely at least linearly. By Theorem \ref{thm contract loxodromic}, an element $g \in G$ that acts as a loxodromic isometry on $X_L$ is a contracting isometry of $X$. This gives the result. 
	\end{proof}

	\section{Central limit Theorem}\label{section clt cat}
	Let $G$ be a discrete countable group and $G \curvearrowright X$ an action by isometries on a complete $\cat (0)$ space $(X,d)$, and assume that $G$ admits a pair of independent contracting isometries for this action. Let $\mu \in \prob(G) $ be an admissible probability measure on $G$. Fix $L > 0$ such that $G$ acts on the hyperbolic space $X_L = (X,d_L)$ with a pair of independent loxodromic isometries. This section is dedicated to study further limit laws for the random walk $(Z_n o)$ on $X$ generated by the measure $\mu$. Namely, we prove that if $\mu$ has finite second moment, $(Z_n o)$ satisfies a central limit theorem. 
	
	\subsection{General strategy} \label{strategy}
	
	In order to prove our main result, we use a strategy that is largely inspired by the works of Benoist and Quint on linear spaces \cite{benoist_quint16CLTlineargroups} and hyperbolic spaces \cite{benoist_quint16}. They developed a method for proving central limit theorems for cocycles, relying on results due to Brown in the case of martingales \cite{brown71}. In this section, we introduce some background and describe this strategy. 
	
	\subsubsection{Centerable cocycle} 
	
	Let $G$ be a discrete group, $Z$ a measurable compact $G$-space and $c$ a cocycle $c : G \times Z \rightarrow \mathbb{R}$, meaning that $c(g_1g_2, x) = c(g_1, g_2 x ) + c(g_2, x)$, and assume that $c$ is continuous. Let $\mu$ be a probability measure on $G$. 
	
	\begin{eDef}
		Let $c$ be a continuous cocycle $c : G \times Z \rightarrow \mathbb{R}$. We say that $c$ has \textit{constant drift} $c_\mu$ if $c_\mu = \int_G c(g, x) d\mu(g)$ is essentially constant in $x$. We say that $c$ is \textit{centerable} if there exists a bounded measurable map $\psi : Z \rightarrow \mathbb{R} $ and a cocycle $c_{0} : G \times Z \rightarrow \mathbb{R}$ with constant drift $c_{0, \mu} = \int_G c_0(g, x) d\mu(g)$ such that 
		\begin{eqnarray}
			c(g, x) = c_0(g,x) + \psi(x) - \psi(gx).
		\end{eqnarray}
		We say that $c$ and $c_0$ are cohomologous. In this case, the \textit{average} of $c$ is defined to be $c_{0, \mu}$. 
	\end{eDef}

	\begin{erem}\label{rem average cocycle}
		Let $\nu \in \prob(Z)$ be a $\mu$-stationary measure, and let $c : G \times Z \rightarrow \mathbb{R}$ be a centerable continuous cocycle: for $g \in G, x \in Z$, $c(g, x) = c_0(g,x) + \psi(x) - \psi(gx)$ with $c_0$ having constant drift and $\psi$ bounded measurable. Then 
		\begin{eqnarray}
			\int_{G\times Z} c(g, x) d\mu(g) d\nu (x) &  = & \int_{G\times Z} c_0(g, x) d\mu(g) d\nu (x) + \int_Z \psi(x) d\nu(x)  \nonumber \\
			& & - \int_{G \times Z} \psi(gx) d\mu(g) d\nu(x) \nonumber \\
			& = & \int_{G} c_0(g, x) d\mu(g) + \int_Z \psi(x) d\nu(x)  - \int_{G \times Z} \psi(gx) d\mu(g) d\nu(x) \nonumber \\
			& = &  \int_{G} c_0(g, x) d\mu(g) + \int_Z \psi(x) d\nu(x) \nonumber \\
			& &- \int_{ Z} \psi(x) d\nu(x) \text{ because $\nu$ is $\mu$-stationary} \nonumber \\
			& = & c_{0, \mu} \text{ because $c_0$ has constant drift}. \nonumber
		\end{eqnarray}
		Hence the average of $c$ is given by $ c_{0, \mu} =\int c(g, x) d\mu(g) d\nu(x) $, which explains the terminology and shows that the value $c_{0, \mu}$ does not depend on the choices of $c_0$ and $\psi$. 
	\end{erem}
	
	The reason why we study limit laws on cocycles is the following result. This version is borrowed from Benoist and Quint, who improved previous results from Brown about central limit theorems for martingales \cite{brown71}.
	
	\begin{ethm}[{\cite[Theorem 3.4]{benoist_quint16CLTlineargroups}}]\label{thm criterium clt benoist quint}
		Let $G$ be a discrete countable group acting by homeomorphisms on a compact metrizable space $Z$. Let $c : G \times Z \rightarrow \mathbb{R}$ be a continuous cocycle such that $\int_G \sup_{x \in Z} | c(g,x)|^2 d\mu (g) < \infty $. Let $\mu $ be a Borel probability measure on $G$. Assume that $c$ is centerable with average $\lambda_c$ and that there exists a unique $\mu$-stationary probability measure $\nu $ on $Z$. 
		
		Then the random variables $\frac{1}{\sqrt{n}}(c(Z_n, x) - n \lambda_c) $ converge in law to a Gaussian law $N_\mu$. In other words, 
		for any bounded continuous function $F$ on $\mathbb{R}$, one has
		\begin{eqnarray}
			\int_G F \big( \frac{c(g,x) - n \lambda_c}{\sqrt{n}}\big) d(\mu^{\ast n})(g) \longrightarrow \int_{\mathbb{R}} F(t) dN_\mu (t) \nonumber. 
		\end{eqnarray}
		Moreover, if we write $c (g, z) = c_0(g, z) + \psi(z) - \psi(g z) $ with $\psi $ continuous and $c_0$ with constant drift $c_\mu$, then the covariance 2-tensor of the limit law is 
		\begin{eqnarray}
			\int_{G \times Z} (c_0(g, z) - c_\mu)^2 d\mu(g) d\nu(z) \nonumber. 
		\end{eqnarray}
	\end{ethm}
	
	\begin{erem}
		In Theorem \ref{thm criterium clt benoist quint}, the fact that $Z$ is compact is only used in order to say that the continuous function $\psi$ is bounded. It remains true if one considers a standard Borel space $Z$, and $\psi $ a bounded measurable function $Z \rightarrow \R$ such that  
		\begin{eqnarray}
			c (g, z) = c_0(g, z) + \psi(z) - \psi(g z), \nonumber
		\end{eqnarray}
		with $c_0$ a cocycle with constant drift, see \cite[Remark 5.6]{fernos_lecureux_matheus21}. 
	\end{erem}
	
	\subsubsection{Busemann cocycle and strategy}
	
	For the rest of the section, let $G$ be a discrete group and $G \curvearrowright X$ a non-elementary action by isometries on a complete $\cat (0)$ space $X$. Let $\mu \in \prob(G) $ be an admissible probability measure on $G$ with finite second moment, and assume that $G $ contains a pair of independent contracting isometries. Let $o \in X$ be a basepoint of the random walk. Theorems \ref{thm cv cat} and \ref{thm drift cat} ensure that the random walk $(Z_n(\omega)o)_n$ converges to a point of the boundary and that the drift $\lambda = \lim_n \frac{1}{n} d(Z_n(\omega) o , o)$ is well-defined and almost surely positive. 
	
	We denote by $\mui$ the probability measure on $G$ defined by $\mui(g) = \mu(g^{-1})$. Let $(\check{Z}_n)_n$ be the right random walk associated to $\mui$. Since $\mu $ is admissible and has finite second moment, so does $\mui$. We can then apply Theorems \ref{thm uniqueness stat measure cat}, \ref{thm cv cat} and \ref{thm drift cat} to $\mui$. We will denote by $\nui $ the unique $\mui$-stationary measure on $\overline{X} = X \cup \bd X$, and by $\check{\lambda} $ the positive drift of the random walk $(\check{Z}_n o)_n$. 
	
	\begin{erem}
		One can check that
		\begin{eqnarray}
			\check{\lambda} &=& \inf_n \frac{1}{n} \int d(go, o) d\mui^{\ast n}(g) \nonumber \\
			& = & \inf_n \frac{1}{n} \int d(o, g^{-1} o) d\mui^{\ast n}(g) \nonumber \\
			& = & \inf_n \frac{1}{n} \int d(o, g o) d\mu^{\ast n}(g) \nonumber,
		\end{eqnarray}
		hence $\lambda = \check{\lambda}$.
	\end{erem}
	
	In our context, the continuous cocycle that we consider is the Busemann cocycle on the visual bordification of the $\cat$(0) space $X$: for $x \in \overline{X}, \ g \in G$ and $o \in X$ a basepoint, 
	
	\begin{equation*}
		\beta(g,x) = b_x (g^{-1} o).
	\end{equation*}
	
	It is straightforward to show that $\beta$ is continuous. Observe that for all $g_1, g_2 \in G, \, \xi \in Y$, horofunctions satisfy a cocycle relation: 
	\begin{eqnarray}
		b_\xi (g_1g_2 o ) & = &  \lim_{x_n \rightarrow \xi}d(g_1g_2 , x_n) - d(x_n, o)  \nonumber \\
		& = & \lim_{x_n \rightarrow \xi}d(g_2 , g_1^{-1}x_n) - d(g_1 o, x_n ) +   d(g_1 o, x_n ) - d(x_n, o) \nonumber \\ 
		& = & \lim_{x_n \rightarrow \xi}d(g_2o , g_1^{-1}x_n) - d(o, g_1^{-1}x_n ) +   d(g_1 o, x_n ) - d(x_n, o) \nonumber \\ 
		& = & b_{g_1^{-1}\xi} (g_2 o) + b_\xi (g_1 o). \label{cocycle horof}
	\end{eqnarray}
	By Equation \eqref{cocycle horof}, $\beta$ satisfies the cocycle relation $\beta(g_1g_2, x) = \beta(g_1, g_2 x ) + \beta(g_2, x)$. 
	\newline 
	
	We use the Busemann cocycle because of the following result. 
	
	\begin{eprop}\label{prop approx bornee horof}
		Let $(x_n )$ be a sequence that converge to $\xi \in B_L $. Let $\eta \neq \xi $ be another boundary point in $B_L$. Then there exists $C$ such that for all $n $
		\begin{eqnarray}
			|b_\eta (x_n) - d(o, x_n) | \leq C. \nonumber
		\end{eqnarray}
	\end{eprop}
	
	\begin{proof}
		Let $\gamma^\xi : [0, \infty)  \rightarrow X$ and $\gamma^\eta : [0, \infty)  \rightarrow X$ be the geodesic rays based at $o$, representing $\xi$ and $\eta$ respectively. By definition of the horofunction $b_\eta  $, for all $n $:
		\begin{eqnarray}
			b_\eta (x_n) = \lim_t d(\gamma^\eta(t), x_n) - d(\gamma^\eta(t), o) \nonumber. 
		\end{eqnarray}
		By the triangular inequality, for all $n  \in \mathbb{N}$, 
		\begin{eqnarray}
			d(o, x_n) \geq d(\gamma^\eta(t), x_n) - d(\gamma^\eta(t), o)\nonumber,
		\end{eqnarray}
		hence $d(o, x_n) \geq b_\eta (x_n)$. 
		
		As $\xi $ and $\eta $ belong to $B_L$, Lemma \ref{lem dual chain} shows that there exist infinite $L$-chains $\{h_i\}$ and $\{h'_i\}$ dual to $\gamma^\xi $ and $\gamma^\eta$ respectively. Orient these chains such that for all $i $, $h_i \subseteq h^{-}_{i+1} $ and $h'_i \subseteq h'^{-}_{i+1} $. As $\xi \neq \eta$, there exists $i \in \mathbb{N}$ and $n_0$ such that $\gamma^\eta \subseteq h^-_i$ and $x_n \in h'^{-}_i$ for all $n \geq n_0$. Without loss of generality, we can assume that for all $n \geq n_0$, $x_n \in h^+_{i+2}$. Also, let $t_0$ be such that $\gamma^\eta(t) \in h'^{+}_{i+2}$ for all $t \geq t_0$. 
		
		Let $t \geq t_0$ and $n \geq n_0$. Let $q \in [\gamma^\eta(t), x_n]\cap h'_{i+1}$ and 
		$p \in [\gamma^\eta(t), x_n]\cap h_{i+1}$. By the ``bottleneck'' Lemma \ref{lem bottleneck intro}, we have 
		\begin{eqnarray}
			d(q, \pi_{\gamma^\eta}(q)) \leq 2L+1 \nonumber
		\end{eqnarray}
		and 
		\begin{eqnarray}
			d(p, \pi_{\gamma^\xi}(p)) \leq 2L+1 \nonumber.
		\end{eqnarray}
		Note that $\pi_{\gamma^\xi}(p) $ and $\pi_{\gamma^\eta}(q)$ belong to the poles of $h_{i+1}$ and $h'_{i+1}$ respectively. Define $A := d(\pi_{\gamma^\eta}(q),o)$ and $B := d(\pi_{\gamma^\xi}(p),o)$. We get 
		\begin{eqnarray}
			d(\gamma^\eta(t), x_n) & = & d(\gamma^\eta(t), q) + d(q, p) + d(p, x_n) \nonumber\\
			&\geq & d(\gamma^\eta(t), o) - d(o, q)+ d(q, p) + d(o, x_n) - d(o, p) \nonumber\\
			& \geq & d(\gamma^\eta(t), o) - (A + 2L+1)+ d(\pi_{\gamma^\eta}(q), \pi_{\gamma^\xi}(p)) -(4L+2)  \nonumber \\
			& & +d(o, x_n) - (B+2L +1). \nonumber
		\end{eqnarray}
		Denote $d(\pi_{\gamma^\eta}(q), \pi_{\gamma^\xi}(p))$ by $C$. Note that because poles have diameter 1, if $y $ belongs to the pole of $h_{i+1}$, then $d(o, y) \in [B-1, B+1]$, and similarly if $z $ belongs to the pole of $h'_{i+1}$, then $d(o, y) \in [A-1, A+1]$. Also, in this case $d(y, z) \in [C - 2, C +2]$. As a consequence, for all $t \geq t_0$, $n \geq n_0$, 
		\begin{eqnarray}
			d(\gamma^\eta(t), x_n) \geq d(\gamma^\eta(t), o) - (A + 2L+2)+ C-2 -(4L+2) + d(o, x_n) - (B+2L +2) \nonumber. 
		\end{eqnarray}
		Therefore, for all $t \geq t_0$, $n \geq n_0$, 
		\begin{eqnarray}
			d(\gamma^\eta(t), x_n) - d(\gamma^\eta(t), o) - d(o, x_n) \geq  D\nonumber, 
		\end{eqnarray}
		for $D$ a real number that does not depend on $t, n$. Passing to the limit, we obtain
		\begin{eqnarray}
			b_\eta(x_n) - d(o, x_n) \geq  D\nonumber 
		\end{eqnarray}
		and the result follows. 		
	\end{proof}

	\begin{eprop}\label{prop approx displacement horo}
		For every $o \in X$, $\mathbb{P}$-almost every $\omega \in \Omega $ and $\nu$-almost every $\xi \in \bd X$, we have that 
		\begin{equation*}
			\lambda = \lim_{n \rightarrow \infty} \frac{1}{n} b_\xi(Z_n(\omega) o).
		\end{equation*}
	\end{eprop}
	
	\begin{proof}
		We know by Theorem \ref{thm cv cat} that $(Z_n o) $ converges almost surely to a point $\xi^+ (\omega) \in  B_L$. Therefore the hitting measure $\nu$ charges only $B_L$. By Theorem \ref{thm drift cat}, we know that almost surely 
		\begin{eqnarray}
			\lim_{n \rightarrow \infty} \frac{1}{n} d(o, Z_n o ) = \lambda >0 \nonumber. 
		\end{eqnarray}
		Now by Proposition \ref{prop approx bornee horof}, for any other other point $\eta \neq \xi^+(\omega)\in B_L$, there exists $C$ such that almost surely
		\begin{eqnarray}
			|b^o_\eta (Z_n o) - d(o, Z_n o) | \leq C. \nonumber
		\end{eqnarray}
		Since $\nu$ is non-atomic by Lemma \ref{lem non atomic}, we then have that for $\nu$-almost every $\eta \in \bd X$,  
		\begin{equation}\label{equation approx lambda horo}
			\lambda = \lim_{n \rightarrow \infty} \frac{1}{n} b^o_\xi (Z_n(\omega) o) \nonumber
		\end{equation}
		almost surely. 
		
		Now if we take a different basepoint $z \in X$,
		\begin{eqnarray}
			d(Z_nz, z ) &\leq &d(Z_n z, Z_no) + d(Z_no, o ) + d(x, z ) \nonumber \\
			& \leq & d(Z_nx, x ) + 2 d(o, z ),  \nonumber
		\end{eqnarray}
		hence $|d(Z_n z, z ) - d(Z_no, o )| \leq 2 d(o, z )$.
		Similarly, $|h^x(Z_nz) - h^x_\xi(Z_nx)| \leq 2 d(o, z )$ and if we change the basepoint, $|h^x_\xi - h^z_\xi | \leq d(o, z)$, so equation \eqref{equation approx lambda horo} does not depend on the choice of the basepoint. 
	\end{proof}

	Proposition \eqref{prop approx displacement horo} shows that the cocycle $\beta(Z_n(\omega),x)$ "behaves" like $d(Z_n(\omega) o,o)$. Thus it makes sense to try and apply Theorem \ref{thm criterium clt benoist quint} to the Busemann cocycle $\beta(g,x)$. 
	\newline 
	
	Henceforth, we will assume that $\mu$ is an admissible probability measure on $G$ with finite second moment $\int_{G} d(go, o)^2 d\mu(g) < \infty$.
	
	The following proposition summarizes some properties of the Busemann cocycle. It shows that obtaining a central limit theorem on $\beta$ will imply our main result. 
	
	\begin{eprop}\label{Busemann cocyle prop}
		Let $G$ be a discrete group and $G \curvearrowright X$ an action by isometries on a complete $\cat (0)$ space $X$, and assume that $G$ contains a pair of independent contracting elements. Let $\mu \in \prob(G) $ be an admissible probability measure on $G$ with finite second moment. Let $o \in X$ be a basepoint of the random walk. Let $\lambda$ be the (positive) drift of the random walk, and $\beta : G \times \overline{X} \rightarrow \mathbb{R}$ be the Busemann cocycle $\beta(g, x) = b_x(g^{-1} o)$. Then 
		\begin{enumerate}
			\item $\int_{G} \sup_{x \in \overline{X}}|\beta (g, x)|^2 d\mu(g) < \infty$ and $\int_{G} \sup_{x \in \overline{X}}|\beta (g, x)|^2 d\mui(g) < \infty$;\label{moment}
			\item For $\nu$-almost every $\xi \in \bd X$, $\lambda = \lim_n \frac{1}{n} \beta (Z_n (\omega), \xi)$ $\mathbb{P}$-almost surely; \label{drift cocycle}
			\item $\mathbb{P}$-almost surely, $\lambda = \int_{G \times \overline{X}} \beta(g, x) d\mu(g)d\nu(x) = \int_{G \times \overline{X}} \beta(g, x) d\mui(g)d\nui(x)$. \label{average Busemann}
		\end{enumerate}
	\end{eprop}
	
	\begin{proof}
		As a consequence of Proposition \ref{prop approx bornee horof}, for $\nu$-almost every $x  \in \bd X$, and $\mathbb{P}$-almost every $\omega \in \Omega$, there exists $C > 0 $ such that for all $n \geq 0$ we have 
		\begin{eqnarray}
			|\beta (Z_n(\omega)^{-1}, x) - d(Z_n(\omega) o,o)| < C.  \nonumber
		\end{eqnarray}
		Because the action is isometric and $\mu$ has finite second moment $\int_{G} d(g o, o)^2 d\mu(g)~<~\infty$, we obtain 
		\begin{eqnarray}
			\int_{G} \sup_{x \in \overline{X}}|\beta (g, x)|^2 d\mu(g)<\infty. \nonumber
		\end{eqnarray} 
		With the same argument: 
		\begin{eqnarray}
			\int_{G} \sup_{x \in \overline{X}}|\beta (g, x)|^2 d\mui(g) < \infty. \nonumber
		\end{eqnarray} 
		
		Item \ref{drift cocycle} is just Proposition \ref{prop approx displacement horo}. 
		\newline 
		
		The ideas in the proof of \ref{average Busemann} are classical. We give the details for the convenience of the reader. 
		
		Let $T : (\Omega \times \overline{X}, \mathbb{P} \times \nui ) \rightarrow (\Omega \times \overline{X}, \mathbb{P} \times \nui )$ be defined by $T(\omega, \xi ) \mapsto (S\omega, \omega_0^{-1} \xi)$, with $S((\omega_i)_{i \in \mathbb{N}}) = (\omega_{i+1})_{i \in \mathbb{N}}$ the usual shift on $\Omega$. Because $\nui$ is the unique $\mui$-stationary measure, we can apply \cite[Proposition 1.14]{benoist_quint}, which says that $T$ preserves the measure $\mathbb{P}\times \check{\nu}$  and is an ergodic transformation. Define $H : \Omega \times \overline{X} \rightarrow \mathbb{R}$ by 
		\begin{equation*}
			H(\omega, \xi) = b_\xi (\omega_0 o) = \beta(\omega_0^{-1}, \xi). 
		\end{equation*}
		By \ref{moment}, it is clear that $\int |H(\omega, \xi)| d\mathbb{P}(\omega)d\nui(\xi) < \infty$. 
		
		By cocycle relation \eqref{cocycle horof} one gets that 
		\begin{eqnarray}
			b_\xi (Z_n o) =  \sum_{k=1}^{n} b_{Z_k^{-1}\xi} (\omega_k o )  = \sum_{k=1}^{n} H(T^k (\omega, \xi)) \nonumber. 
		\end{eqnarray}
		
		Then $\beta(Z_n(\omega)^{-1}, \xi) = \sum_{k=1}^{n} H(T^k (\omega, \xi))$, and by \ref{drift cocycle}, 
		\begin{eqnarray}
			\lambda = \lim_n \frac{1}{n} \sum_{k=1}^{n} H(T^k (\omega, \xi)).\nonumber
		\end{eqnarray}
		
		Now, by Birkhoff ergodic theorem, one obtains that almost surely, 
		\begin{eqnarray}
			\lambda &=& \int_{\Omega \times \overline{X}} H(\omega, \xi) d\mathbb{P}(\omega)d\nui(x). \nonumber \\
			& = & \int_{\Omega \times \overline{X}} b_\xi(\omega_0 o) d\mathbb{P}(\omega)d\nui(x) \nonumber \\
			& = & \int_{G \times \overline{X}} \beta(g^{-1}, \xi) d\mu(g)d\nui(x) \nonumber \\
			& =& \int_{G \times \overline{X}} \beta(g, \xi) d\mui(g)d\nui(x) \nonumber.
		\end{eqnarray} 
		
		The previous computations can be done similarly for $\mu $ and $\nu $, hence we also have that 
		\begin{eqnarray}
			\lambda = \int_{G \times \overline{X}} \beta(g, x) d\mu(g)d\nu(x). \nonumber
		\end{eqnarray} 
	\end{proof}
	
	In order to apply Theorem \ref{thm criterium clt benoist quint} on the Busemann cocycle $\beta$, it remains to show that $\beta$ is centerable. If this is the case, by \ref{average Busemann} and Remark \ref{rem average cocycle}, its average must be the positive drift $\lambda$. In other words, we need to show that there exists a bounded measurable function $\psi : \overline X \rightarrow \mathbb{R}$ such that the cohomological equation
	\begin{eqnarray}
		\beta(g, x) = \beta_0(g,x) + \psi(x) - \psi(gx). \label{cohom equation}
	\end{eqnarray}
	is verified. Then, proving the central limit theorem in our context amounts to finding such a $\psi $ that is well defined and bounded. This will be done by computing nice estimates on the random walk, with the help of the hyperbolic model $X_L$.

	\subsection{Geometric estimates}\label{Section geom estimates}

	In this section, we prove our main Theorem, following the strategy explained in Section \ref{strategy}. First, we will compute some geometric estimates on the random walk that we will need later on. This is where we use the specific contraction properties provided by the curtains and the hyperbolic models exposed in Section \ref{section modèle hyp}. The goal is ultimately to prove that the candidate $\psi$ for the cohomological equation is well-defined and bounded. 
	
	Let $G$ be a discrete group and $G \curvearrowright X$ a non-elementary action by isometries on a complete $\cat (0)$ space $X$, and assume that $G $ contains a contracting element. Let $o \in X$ be a basepoint of the random walk. Recall that $B_L$ is defined to be the subspace of $\bd X$ consisting of all geodesic rays $\gamma : [0, \infty) \rightarrow X$ starting from $o$ and such that there exists an infinite $L$-chain crossed by $\gamma$. By Theorem \ref{thm homeomorphism bords XL}, there exists an $\iso(X)$-equivariant embedding $\mathcal{I}_L : \bdg X_L \rightarrow \bd X$, whose image lies in $B_L$. 
	
	\begin{eprop}\label{geometric lemma}
		Let $(g_n)$ be a sequence of isometries of $G$, and let $o \in X$, $x, y \in \bd X$. Assume that there exists $\lambda, \epsilon, A >0 $ such that:
		\begin{enumerate}[label= (\roman*)]
			\item $ \{g_n o \}_n $ converges in $(\overline{X_L}, d_L)$ to a point of the boundary $z_L \in \bdg X_L$, whose image in $\bd X $ by the embedding $\mathcal{I}_L$ is not $y$; \label{assum cv}
			\item $d_L(g_n o, o) \geq An$; \label{assum drift hyp}
			\item $|b_x (g_n^{-1} o ) - n \lambda | \leq \varepsilon n$; \label{assum gnx}
			\item $|b_y (g_n o ) - n \lambda | \leq \varepsilon n$; \label{assum y}
			\item $|d(g_n o , o) - n \lambda | \leq \varepsilon n$. \label{assum drift cat}
		\end{enumerate}	
		Then, one obtains: 
		\begin{enumerate}
			\item $(g_nx |g_n o )_o \geq (\lambda - \varepsilon) n$; \label{geom estimate 1}
			\item $(y | g_n o )_o \leq \varepsilon n$.\label{geom estimate 2} 
		\end{enumerate}
		If moreover $A \geq 2(4L+10) \varepsilon$, then we have: 
		\begin{enumerate}[resume]
			\item $(y| g_n x )_o \leq \varepsilon n + (2L+ 1)$. \label{geom estimate 3}
		\end{enumerate}
	\end{eprop}
	
	The proof of points \ref{geom estimate 1} and \ref{geom estimate 2} is straightforward, so we begin by these. 
	
	\begin{proof}[Proof of estimates \ref{geom estimate 1} and \ref{geom estimate 2}]
		A simple computation gives that 
		\begin{eqnarray}
			(g_n x | g_n o)_o = \frac{1}{2}(b_x (g^{-1}_n o) + d(g_no, o)) \nonumber
		\end{eqnarray}
		Then using assumptions \ref{assum gnx} and \ref{assum drift cat} gives immediately that $(g_nx | g_no ) \geq (\lambda - \epsilon ) n $, which proves \ref{geom estimate 1}. 	
		
		Now, by definition,  
		\begin{eqnarray}
			(y |g_n o ) = \frac{1}{2}(d(g_no, o) - b_y (g_no))\nonumber
		\end{eqnarray}
		Then by assumptions \ref{assum y} and \ref{assum drift cat}, we obtain \ref{geom estimate 2}. 
	\end{proof}
	
	The proof of point \ref{geom estimate 3} is the hard part. We prove it in two steps. First, we show that under the assumptions, for $n$ large enough, there exist at least three $L$-separated curtains dual to $[o, g_no]$ separating $\{g_n o, g_nx\}$ on the one side and $\{o, y \}$ on the other, see Figure \ref{figure separating hyperplanes}. Then we show that the presence of these hyperplanes implies the result. 
	
	\begin{figure}
		\centering
		\begin{center}
			\begin{tikzpicture}[scale=1.2]
				\draw (0,0) -- (4,0)  ;
				\draw (2.8,2) -- (2.8,-1)  ;
				\draw (1.2,2) -- (1.2,-1)  ;
				\draw (2,2) -- (2,-1)  ;
				\draw (0,0) node[below left]{$o$} ;
				\draw (4,0) [-stealth] to[bend left = 20] (5, 3);
				\draw (0,0) [-stealth] to[bend right = 20] (-1, 2.5);
				\draw (-1, 2.5) to[bend right = 55] (5,3);
				\draw (-1, 2.5) node[above left]{$y$} ;
				\draw (4,0) node[below right]{$g_no$} ;	
				\draw (5,3) node[above]{$g_n x$};
			\end{tikzpicture}
		\end{center}
		\caption{A "hyperbolic-like" 4 points inequality in Proposition \ref{geometric lemma}.}\label{figure separating hyperplanes}
	\end{figure}
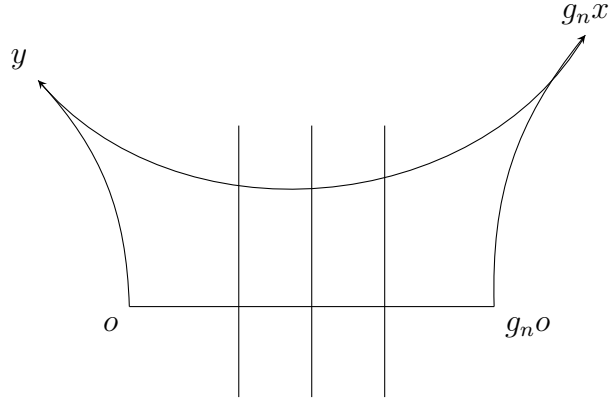
	
	By assumption \ref{assum drift hyp}, there exists an $L$-chain of size $S(n) $ that separates $o$ and $g_no$. By Lemma \ref{lem dual chain}, there exists an $L$-chain dual to $[o, g_no ]$ of size greater than or equal to $\lfloor\frac{S(n)}{4L + 10}\rfloor$ that separates $o$ and $g_no$. Denote by $c_n = \{h^n_i\}_{i=1}^{S'(n)}$ a maximal $L$-chain dual to $[o, g_no]$, separating $o $ and $g_n o $, and orient the half-spaces so that $h_{i-1} \subseteq h_i ^{-}$ for all $i$. When the context is clear, we might omit the dependence in $n$ for ease of notations, and just write $\{h_i\}_{i=1}^{S'(n)}$ for a maximal $L$-chain dual to $[o, g_n o ]$. Recall that $S(n) \geq An$, hence $c_n$ must be of length $S'(n) \geq A'n$, where $A' = \frac{A}{4 L+10}$. 
	
	\begin{elem}\label{estimate y}
		Under the assumptions of Proposition \ref{geometric lemma}, there exists a constant $C$ such that for all $n \in \mathbb{N}$, the number of $L$-separated hyperplanes in $c_n$ that do not separate $\{o, y\}$ and $\{g_n o\} $ is less than $C$. 
	\end{elem}
	\begin{proof}[Proof of Lemma \ref{estimate y}]
		By assumption, $ \{g_n o \}_n $ converges in $(\overline{X_L}, d_L)$ to a point of the boundary $z_L \in \bdg X_L$. By Theorem \ref{equivariant embedding of boundaries}, there exists an $\iso(X)$-equivariant embedding $\mathcal{I}_L : \bdg X_L \rightarrow \bd X$, whose image lies in $B_L$. Denote by $z := \mathcal{I}_L(z_L)$ the image in $\bd X$ of the limit point $z_L$ by this embedding. 
		
		Denote by $\beta : [0, \infty) \rightarrow X$ a geodesic ray joining $o$ to $z$. Since $z \in B_L$, there exists an infinite $L$-chain $c = \{k_i\}_{i \in \mathbb{N}}$ that separate $o$ from $z$. Note that because of Lemma \ref{lem dual chain} and Remark \ref{infinite dual chain}, we can assume that $c$ is a chain of curtains which is dual to the geodesic ray $\beta$. Since $\{g_n o \}_n $ converges in $(\overline{X_L}, d_L)$ to $z_L \in \bdg X_L$, and $z$ is the image of $z_L$ by the equivariant embedding $\mathcal{I}$, Proposition \ref{prop inverse map} implies that $\{g_no\}_n $ converges to $z$ in $X$. The fact that $z \in B_L$ implies that for all $i \in \mathbb{N}$, there exists $n_0 \in \mathbb{N}$ such that for all $n\geq n_0$, $k_i$ separates $o $ from $g_n o $. Now, we denote by $\gamma : [0, \infty) \rightarrow X$ the geodesic ray that represents $y \in \bd X $. See figure \ref{figure estimate y}. 
		
		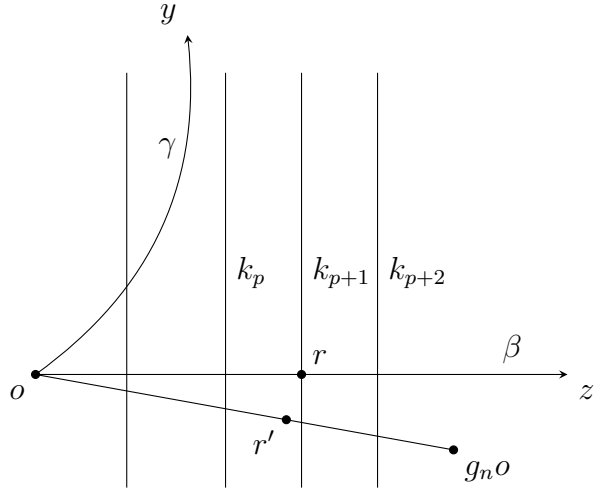
\begin{figure}
			\centering
			\begin{center}
				\begin{tikzpicture}[scale=1]
					\draw (0,0) [-stealth]-- (7,0)  ;
					\draw (1.2,4) -- (1.2,-1.5)  ;
					\draw (2.5,4) -- (2.5,-1.5)  ;
					\draw (3.5,4) -- (3.5,-1.5)  ;
					\draw (4.5,4) -- (4.5,-1.5)  ;
					\draw (0,0) node[below left]{$o$} ;
					\draw (2, 3) node[left]{$\gamma$} ;
					\draw (7,0) node[below right]{$z$} ;
					\draw (2, 4.5) node[above left]{$y$} ;
					\draw (2.5, 1) node[above right]{$k_p$} ;
					\draw (3.5, 1) node[above right]{$k_{p+1}$} ;
					\draw (3.5, 0) node[above right]{$r$} ;
					\draw (4.5, 1) node[above right]{$k_{p+2}$} ;
					\draw (5.5,-1) node[below right]{$g_no$} ;
					\draw (6,0) node[above right]{$\beta$} ;
					\filldraw[black] (0,0) circle (1.5pt) ; 
					\draw  (3.3,-0.6) node[below left]{$r'$} ;
					\filldraw[black] (3.5,0) circle (1.5pt) ; 
					\filldraw[black] (3.3,-0.6) circle (1.5pt) ; 
					\filldraw[black] (5.5,-1) circle (1.5pt) ; 
					\draw (0, 0 ) [-stealth] to[bend right]	(2, 4.5);
					\draw (0, 0 ) -- (5.5,-1);
				\end{tikzpicture}
			\end{center}
			\caption{Illustration of Lemma \ref{estimate y}.}\label{figure estimate y}
		\end{figure}
		
		Due to Remark \ref{infinite dual chain}, meeting $c$ infinitely often is equivalent to crossing it, then since $y \neq z$, there exists $p \in \mathbb{N} $ such that $\gamma \subseteq k_p^{-}$. Now consider $n_0$ such that for $n\geq n_0$, $g_n o \in k_{p+2}^{+}$. Fix $n \geq n_0$. Recall that $c_n$ is a maximal $L$-chain dual to $[o, g_n o]$ separating $o$ and $g_n o $. 
		
		Denote by $r \in \beta$ a point in the pole of $k_{p+1}$, and denote by $r'= r'(n)$ the projection of $r$ onto the geodesic $[o, g_n o ]$. Then by Lemma \ref{lem bottleneck intro}, 
		\begin{equation*}
			d(o, r'(n)) \leq d(o , r) + 2L +1 \nonumber. 
		\end{equation*} 
		Recall that curtains are thick, i.e. $d(h^+, h^-) = 1$ for all curtain $h$. Consequently, the number of curtains in $c_n$ that separate $o $ and $r' (n) $ is $\leq d(o , r) + 2L +2$. We emphasize that this number does not depend on $n\geq n_0$, because for all $n \geq n_0$, $g_n \in k_{p+2}^{+}$ and the previous equation holds.  
		
		Recall that $\gamma \subseteq k_{p}^{-}$, so in particular $\gamma \subseteq k_{p+1}^{-}$. Then by star convexity of the curtains (Lemma \ref{star convexity}), every curtain in $c_n$ whose pole belongs to $[r'(n), g_n o]$ separates $\{o,y\}$ from $g_n o $. Then by the previous argument, the number of curtains that do not separate $\{o, y\}$ from $\{g_no\}$ is less than $d(o , r'(n))$. In particular, the number of curtains that do not separate $\{o, y\}$ from $\{g_no\}$ is less than $d(o, r) + 2L+2 $. Since this quantity does not depend on $n$, we have proven the Lemma. 
	\end{proof}
	
	Now, for a fixed $n$, let us give an estimate for the number of curtains in $ c_n = \{h^n_1 , \dots , h^n_{S'(n)}\} $ that separate $o $ and $g_n x $. When $n$ is fixed, we omit the dependence in $n$ and just write $c_n = \{h_1 , \dots , h_{S'(n)}\}$ to ease the notations. Let $\gamma_n : [0, \infty) \rightarrow X$ be the geodesic ray joining $o $ and $g_n x$. Let us take $k_0= k_0(n)$ (depending on $n$) large enough so that for all $k \geq k_0$, 
	\begin{equation}
		| (g_n o | g_n x )_o - (g_no | \gamma_n(k)_o) | \leq 1.   \nonumber
	\end{equation}
	
	\begin{elem}\label{estimate gnx}
		Under the assumptions of Proposition \ref{geometric lemma}, the number of $L$-separated hyperplanes in $c_n$ that separate $\{o\}$ and $\{g_n o, \gamma_n(k_0)\} $ is unbounded in $n$. More precisely, for all $M \in \mathbb{N}$, there exists $n_0$ such that for all $n \geq n_0$, the number of $L$-separated hyperplanes in $c_n$ that separate $\{o\}$ and $\{g_n o, \gamma_n(k)\} $ is greater than $M$ for all $k \geq k_0$. 
	\end{elem}
	
	\begin{proof}[Proof of Lemma \ref{estimate gnx}]
		Let $k \geq k_0$. Suppose that the number of curtains in $ c_n = \{h_1 , \dots , h_{S'(n)}\} $ separating $o $ and $\gamma_n(k) $ is less than or equal to $p \in  [0, S'(n)-4]$. Then $\{h_{p+2} , \dots , h_{S'(n)}\} $ is an $L$-chain separating $\{o, \gamma_n(k)\} $ and $\{g_n o\} $. We then denote by $r(n)$ a point on $h_{p+3} \cap [\gamma_n(k), g_no]$ and by $r'(n)$ the projection of $r(n)$ onto $[o, g_n o]$, see Figure \ref{figure estimate gnx}. By hypothesis on $k$,
		\begin{eqnarray}
			2((g_nx | g_n o)_o - 1) &\leq& 2(\gamma_n(k)| g_no)_o \nonumber \\
			&=& d(\gamma_n(k) , o ) + d(g_n o , o ) - d(g_n o, \gamma_n (k) ). \nonumber
		\end{eqnarray} 
		Now by the ``bottleneck'' Lemma \ref{lem bottleneck intro} and the triangular inequality, 
		\begin{eqnarray}
			2(\gamma_n(k)| g_no)_o & = & d(\gamma_n(k) , o ) + d(g_n o , o ) - (d(g_n o, r(n) ) + d(r(n) , \gamma_n(k))) \nonumber \\
			&\leq & d(\gamma_n(k) , o ) + d(g_n o , o ) - (d(g_n o, r' (n)) - (2L + 1) + d(r (n), \gamma_n(k))) \nonumber \\
			&\leq & d(r (n) , o ) + d(g_n o , o ) - d(g_n o, r'(n) ) + 2L + 1 \nonumber \\ 
			&\leq & d(r' (n) , o ) + 2L + 1 + d(g_n o , o ) - d(g_n o, r'(n) ) + 2L + 1 \nonumber \\ 
			& \leq & 2 d(r' (n), o ) + 2(2L + 1). \nonumber
		\end{eqnarray} 
		
		Because the pole of a curtain is of diameter 1, $d(o , r'(n)) \leq d(g_n o, o) - (S'(n) - (p +1))$. However, by assumptions \ref{assum drift hyp} and \ref{assum drift cat} of Lemma \ref{geometric lemma}, one gets that $d(g_no , o ) \leq (\lambda + \varepsilon) n $ and $S(n) \geq An $. Recall that by Lemma \ref{lem dual chain}, this means that $S'(n) \geq A'n$, where $A' = \frac{A}{4L +1}$. Combining this with the previous result yields
		\begin{eqnarray}
			& & (g_nx | g_n o)_o - 1 \leq d(o, r'(n)) + 2L+1 \nonumber \\
			& \Rightarrow& (\lambda - \varepsilon) n  - 1 \leq (\lambda + \varepsilon )n - (A'n - (p+1)) + 2L+1 \text{ by Lemma \ref{geometric lemma}, \ref{geom estimate 1}}\nonumber  \\
			& \Rightarrow & 0 \leq (2\varepsilon - A' ) n + 2L + p+3.   \nonumber
		\end{eqnarray}
		If $A' > 2 \varepsilon$, there exists $n_0 $ large enough such that for all $n \geq n_0$, the above inequality gives a contradiction. As a consequence, if $A' > 2 \varepsilon$, or equivalently if $ A > 2(4L +10) \varepsilon$, there exists $n_0$ such that for all $n \geq n_0$, the number of curtains in $c_n$ separating ${o } $ and $\{\gamma_n (k), g_no\}$ is greater than $p$.
	\end{proof}
	
	\begin{figure}
		\centering
		\begin{center}
			\begin{tikzpicture}[scale=2]
				\draw (0,0) -- (4,0)  ;
				\draw (3.6,2) -- (3.6,-0.5)  ;
				\draw (2.5,2) -- (2.5,-0.5)  ;
				\draw (3,2) -- (3,-0.5)  ;
				\draw (1.5,2) -- (1.5,-0.5)  ;
				\draw (0,0) node[below left]{$o$} ;
				\draw (2.2, 2.5) to[bend right = 40] (4,0);
				\draw (0,0) [-stealth]to[bend right = 20] (2.3,2.8);
				\draw (2.2, 2.5) node[below left]{$\gamma_n(k)$} ;
				\draw (2.3, 2.8) node[above right]{$g_n x$} ;
				\draw (3,0.5) node[above right]{$r(n)$} ;
				\filldraw[black] (2.2, 2.5) circle (1 pt) ;
				\filldraw[black] (3, 0.5) circle (1 pt) ;
				\filldraw[black] (0,0) circle (1 pt) ;
				\filldraw[black] (2.85,0) circle (1 pt) ;
				\filldraw[black] (4,0) circle (1 pt) ;
				\draw (3,-0.5) node[below ]{$h_{p+3}$} ;
				\draw (1.5,-0.5) node[below ]{$h_p$} ;
				\draw (2.85,0) node[above]{$r'(n)$} ;
				\draw (2.5,-0.5) node[below ]{$h_{p+2}$} ;
				\draw (4,0) node[below right]{$g_no$} ;	
			\end{tikzpicture}
		\end{center}
		\caption{Illustration of Lemma \ref{estimate gnx}.}\label{figure estimate gnx}
	\end{figure}
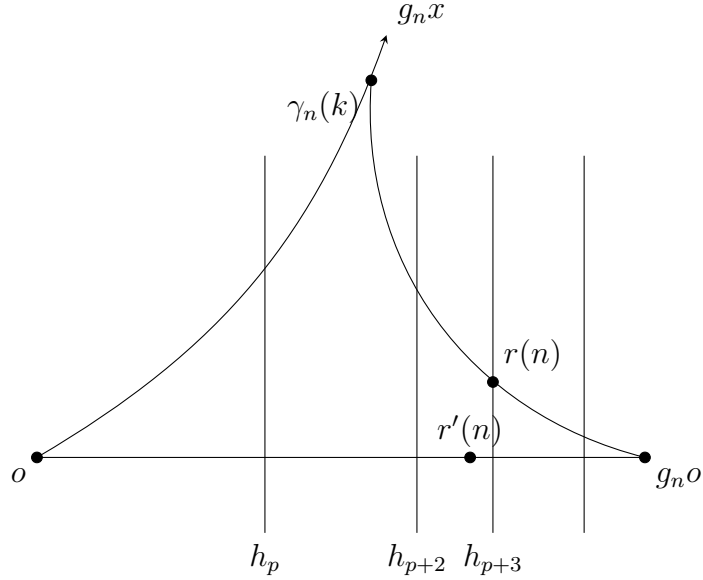
	
	We can now conclude the proof of Proposition \ref{geometric lemma}. 
	
	\begin{proof}[Proof of estimate \ref{geom estimate 3}]
		Recall that we denote by $\gamma : [0, \infty) \rightarrow X$ the geodesic ray based at $o$ that represents $y \in \bd X $ and by $\gamma_n : [0, \infty) \rightarrow X$ the geodesic ray joining $o $ and $g_n x$. Combining Lemma \ref{estimate y} and Lemma \ref{estimate gnx}, we get that if $A > 2(4L+10)\varepsilon$, there exists $n_0$, $k_0$ such that for all $n \geq n_0$ and all $k\geq k_0$, $c_n$ contains at least 3 pairwise $L$-separated curtains that separate $\{o, \gamma(k)\} $ on the one side and $\{g_no, \gamma_n(k)\}$ on the other. Call these hyperplanes $\{h_1, h_2, h_3\}$ and arrange the order so that $h_i \subseteq h_{i+1}^{-}$. Denote by $m_k (n) \in h_2$ some point on the geodesic segment joining $\gamma(k) $ to $\gamma_n(k)$, and $m'_k(n)$ belonging to the geodesic segment $[o, g_no]$ such that $d(m_k(n), m'_k(n)) \leq 2 L + 1$, see Figure \ref{figure proof geom estimate}
		
		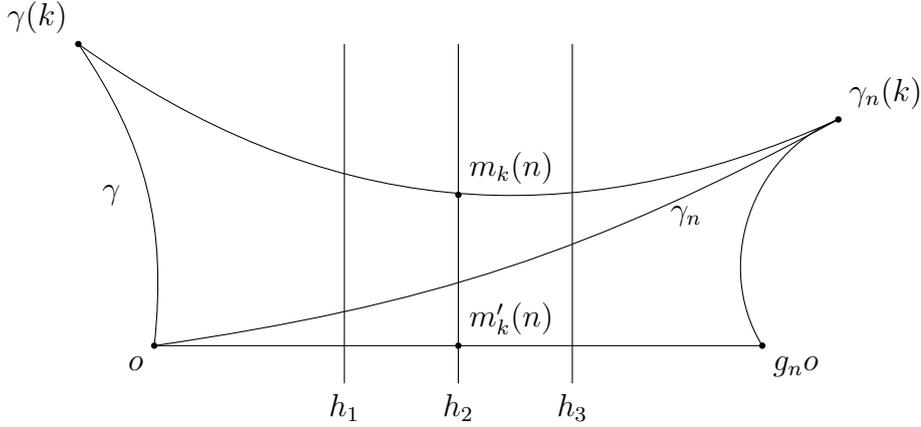
\begin{figure}
			\centering
			\begin{center}
				\begin{tikzpicture}[scale=1]
					\draw (0,0) -- (8,0)  ;
					\draw (2.5,4) -- (2.5,-0.5)  ;
					\draw (4,4) -- (4,-0.5)  ;
					\draw (5.5, 4) -- (5.5,-0.5)  ;
					\draw (0,0) node[below left]{$o$} ;
					\draw (-0.3, 2) node[left]{$\gamma$} ;
					\draw (0, 0) to[bend right = 20](-1, 4) ;
					\draw (0, 0) to[bend right = 10] (9,3) ;
					\draw (-1, 4) to[bend right = 30] (9,3);
					\draw (8,0) to[bend left = 50] (9,3);
					\draw (-1, 4) node[above left]{$\gamma(k)$} ;
					\draw (2.5, -0.5) node[below]{$h_1$} ;
					\draw (4, -0.5) node[below]{$h_2$} ;
					\draw (5.5, -0.5) node[below]{$h_3$} ;
					\draw (8,0) node[below right]{$g_no$} ;	
					\draw (9,3) node[above right]{$\gamma_n(k)$};
					\draw (4,2) node[above right]{$m_k(n)$};
					\draw (4,0) node[above right]{$m'_k(n)$};
					\draw (7,2) node[below]{$\gamma_n$};
					\filldraw[black] (-1, 4) circle (1 pt) ;
					\filldraw[black] (4,0) circle (1 pt) ;
					\filldraw[black] (4, 2.0) circle (1 pt) ;
					\filldraw[black] (9,3) circle (1 pt) ;
					\filldraw[black] (8,0) circle (1 pt) ;
					\filldraw[black] (0,0) circle (1 pt) ;
				\end{tikzpicture}
			\end{center}
			\caption{Proof of Proposition \ref{geometric lemma}}\label{figure proof geom estimate}
		\end{figure}
		Then we have 
		\begin{eqnarray}
			2(\gamma(k)| \gamma_n(k))_o & = & d(\gamma(k) , o ) + d(o, \gamma_n(k)) - d(\gamma(k), \gamma_n(k)) \nonumber\\
			& \leq & d(\gamma(k) , o ) + d(o, m'_k(n)) + d(m'_k(n), m_k(n) ) \nonumber \\
			& & + \, d(m_k (n), \gamma_n(k)) - d(\gamma(k), \gamma_n(k)) \text{ by the triangular inequality }\nonumber\\
			& \leq & d(\gamma(k) , o ) + d(o, m'_k(n))  - d(\gamma(k), m_k(n)) + 2L + 1 \nonumber
		\end{eqnarray}
		by Lemma \ref{lem bottleneck intro}. Since $m'_k$ is on $[o , g_no]$, $d(o , m'_k)= d(o, g_no ) - d(g_no , m'_k(n))$. 
		We then have:
		\begin{eqnarray}
			2(\gamma(k)| \gamma_n(k))_o & \leq & d(\gamma(k) , o ) + d(o, g_no ) - d(g_no , m'k(n)) - d(\gamma(k), m_k(n)) + 2L + 1 \nonumber \\
			& = & d(\gamma(k) , o ) + d(o, g_no ) - (d(g_no , m'_k(n)) + d(\gamma(k), m_k(n))) + 2L + 1. \nonumber
		\end{eqnarray}
		Now observe that 
		\begin{eqnarray}
			d(\gamma(k), g_no) &\leq& d(g_no , m'_k(n)) + d(\gamma(k), m_k(n)) + d(m_k, m'_k) \nonumber \\
			& \leq & d(g_no , m'_k(n)) + d(\gamma(k), m_k(n)) + 2L + 1 \text{ by Lemma \ref{lem bottleneck intro}}, \nonumber
		\end{eqnarray}
		hence $d(\gamma(k), g_n o ) - (2 L + 1) \leq d(g_no , m'_k(n)) + d(\gamma(k), m_k(n))$. Then 
		\begin{eqnarray}
			2(\gamma(k)| \gamma_n(k))_o & \leq & d(\gamma(k) , o ) + d(o, g_no ) - (d(\gamma(k), g_n o ) - (2 L + 1)) + 2L + 1 \nonumber \\
			& = & d(\gamma(k) , o ) + d(o, g_no ) - d(\gamma(k), g_n o ) + 2(2 L + 1) \nonumber \\
			& = & 2(\gamma(k)| g_n o )_o + 2(2 L + 1) \nonumber.
		\end{eqnarray}
		As $k \rightarrow \infty$, one obtains that $(g_nx| y)_o \leq (g_no |y)_o + (2 L + 1)$, and the result follows from the geometric estimate \ref{geom estimate 2}. 
	\end{proof}

	\subsection{Proof of the Central Limit Theorem}
	
	In this section, we prove the main result of this section. The setting is the same: we let $G$ be a discrete countable group and $G \curvearrowright X$ an action by isometries on a complete $\cat (0)$ space $X$. Let $\mu \in \prob(G) $ be an admissible probability measure on $G$ with finite second moment, and assume that $G $ contains a pair of independent contracting isometries for this action. Let $o \in X$ be a basepoint of the random walk. We begin this section by summarizing what we know. 
	\newline
	
	By Proposition \ref{prop non elem}, there exists a number $L \geq 0$ such that $G$ acts by isometries on $X_L = (X, d_L) $ non-elementarily. Then one can consider the random walk $(Z_n(\omega) o)_n$ as a random walk on $(X, d_L) $. The model $(X, d_L) $ is hyperbolic, so we can apply the results of Maher and Tiozzo \cite{maher_tiozzo18} and Gouëzel \cite{gouezel22} summarized in Section \ref{section rw courbure non positive}. In particular, due to Theorem \ref{thm cv rw hyp}, the random walk $(Z_n o)_n$ in $X_L$ converges to a point of the Gromov boundary $\bdg X_L$ of $(X, d_L)$. 
	
	Moreover, since we assume $\mu $ to have finite first moment (for the action on the $\cat$(0) space $X$), and since $d(x, y ) \geq d_L(x, y )$ for all $x, y \in X$, the measure $\mu $ is also of finite first moment for the action on the hyperbolic model $(X, d_L)$. By Theorem \ref{thm gouezel 1.2}, the drift $\tilde{\lambda}$ of the random walk $(Z_n \tilde{o})_n$ is almost surely positive. In other words, we have that $\mathbb{P}$-almost surely, 
	\begin{equation}
		\lim_{n \rightarrow \infty} \frac{1}{n} d(Z_n(\omega) \tilde{o}, \tilde{o}) = \tilde{\lambda} >0. \nonumber
	\end{equation}
	
	Due to Theorem \ref{thm uniqueness stat measure cat}, there exists a unique $\mu$-stationary probability measure $\nu$ on $\overline{X}$. If we define $\mui \in \prob(G)$ by $\mui (g) = \mu(g^{-1})$, $\mui $ is still admissible and of finite second moment. We denote by $\nui$ the unique $\mui$-stationary measure on $\overline{X}$. 
	
	We recall that the Busemann cocycle $\beta : G \times \overline{X} \rightarrow \mathbb{R}$ is defined by:
	\begin{equation*}
		\beta(g,x) = b_x (g^{-1} o).
	\end{equation*}
	
	Our goal is to apply Theorem \ref{thm criterium clt benoist quint} to the Busemann cocycle $\beta$. The results of Section \ref{strategy} show that proving a central limit theorem for the random walk $(Z_n(\omega)o)_n$ amounts to proving that $\beta$ is centerable. As in the works of \cite{benoist_quint16}, \cite{horbez18} and \cite{fernos_lecureux_matheus21}, the natural candidate for solving the cohomological equation \eqref{cohom equation} is the function:
	\begin{equation*}
		\psi ( x ) = -2 \int_{\overline{X}} (x | y)_o d\nui(y).
	\end{equation*}
	
	\begin{eprop}\label{prop centerable cocycle}
		Let $G$ be a discrete group and $G \curvearrowright X$ an action by isometries on a complete $\cat (0)$ space $X$. Let $\mu \in \prob(G) $ be an admissible probability measure on $G$ with finite second moment, and assume that $G $ contains a pair of independent contracting isometries for this action. Let $o \in X$ be a basepoint of the random walk. Then the map
		$\psi(x) = \int_{\overline{X}} (x |y)_o \, d\nui(y)$ is well-defined, Borel and essentially bounded.  
	\end{eprop}

	In order to show that $\psi$ is well-defined and bounded, we need the following statement, which resembles \cite[Proposition 4.5]{benoist_quint16}.

	\begin{eprop}\label{estimates}
		Let $G$ be a discrete countable group and $G \curvearrowright X$ a non-elementary action by isometries on a proper $\cat (0)$ space $X$. Let $\mu \in \prob(G) $ be an admissible probability measure on $G$ with finite second moment, and assume that $G $ contains a rank one element. Let $o \in X$ be a basepoint for the random walk $(Z_n(\omega)o)_n$. Let $\lambda$ be the (positive) drift of the random walk, and $\nu$ a $\mu$-stationary measure on $\overline{X}$. Assume that there exists $a >0 $ and $(C_n)_n \in \ell^1 (\mathbb{N})$ such that for almost every $x, y \in \overline{X}$, we have, for every $n$: 
		\begin{enumerate}
			\item $\mathbb{P}((Z_n o | Z_n x )_o \leq an) \leq C_n $; \label{estimate 1}
			\item $\mathbb{P}((Z_n o | y )_o \geq an) \leq C_n $; \label{estimate 2}
			\item $\mathbb{P}((Z_n x | y )_o \geq an) \leq C_n $. \label{estimate 3}
		\end{enumerate}
		Then one has: 
		\begin{eqnarray}
			\sup_{x \in \overline{X}} \int_{\overline{X}} (x |y)_o d\nu(y) < \infty \nonumber.
		\end{eqnarray}
	\end{eprop}
	\begin{proof}
		Suppose that there exist $a>0$, $(C_n)_n \in \ell^1 (\mathbb{N})$ such that for almost every $x, y \in \overline{X}$, we have estimates \ref{estimate 1} to \ref{estimate 3}. We get: 
		\begin{eqnarray}
			\nu(\{x \in X | (x|y) \geq an \}) & = & \int_{\overline{X}} \mu^{\ast n }(\{g\in G \, | \, (gx |y )_o \geq an \})d\nu(x) \text{ by $\mu$-stationarity} \nonumber \\
			& \leq & \int_{\overline{X}} C_n d\nu(x) = C_n \text{ by estimate \ref{estimate 3}}.\nonumber 
		\end{eqnarray}
		Then, define $A_{n, y } := \{ x \in \overline{X} \, | \, (x|y )_o \geq an\}$, so that by splitting along the subsets $A_{n-1, y } - A_{n, y }$, one gets 
		\begin{eqnarray}
			\int_{\overline{X}} (x |y)_o d\nu(x) & \leq & \sum_{n \geq 1} a n (\nu(A_{n-1, y }) - \nu(A_{n, y })) \nonumber \\
			& \leq &  \sum_{n \geq 1} an (C_{n-1} - C_n) \nonumber \\
			& = & a + \sum_{n \geq 1} aC_n (n+1 - n) < \infty. \nonumber
		\end{eqnarray}
	\end{proof}
	
	We want to show that estimates from Proposition \ref{estimates} hold. As we will see, estimates \ref{estimate 1} and \ref{estimate 2} are quite straightforward to check using the positivity of the drift. Most of the work concerns estimate \ref{estimate 3}.

	Combining Proposition \ref{Busemann cocyle prop} with Theorem \ref{thm drift cat} and \cite[Proposition 3.2]{benoist_quint16CLTlineargroups}, one obtains the following:
	
	\begin{eprop}\label{proba control}
		Consider the same assumptions as in Proposition \ref{prop centerable cocycle}. Then, for every $\varepsilon >0 $, there exists $(C_n)_n \in \ell^1 (\mathbb{N})$ such that for any $x \in \overline{X}$, 
		\begin{eqnarray}
			& & \mathbb{P}(|\beta(Z_n, x) - n \lambda | \geq \varepsilon n) \leq C_n;\label{prob estimate 1}\\
			& & \mathbb{P}(|\beta(Z^{-1}_n, x) - n \lambda | \geq \varepsilon n) \leq C_n ;  \label{prob estimate 2} \\
			& & \mathbb{P}(|d(Z_n o, o) - n \lambda | \geq \varepsilon n) \leq C_n.  \label{prob estimate 3}
		\end{eqnarray}
	\end{eprop}
	
	\begin{proof}
		Recall that by Proposition \ref{Busemann cocyle prop}, $\beta $ is a continuous cocycle such that
		\begin{eqnarray}
			& & \int_G \sup_{x \in \overline{X}} |\beta(g,x)|^2 d\mu(g) < \infty \text{ and }
			\int_G \sup_{x \in \overline{X}} |\beta(g,x)|^2 d\mui(g) < \infty. \nonumber
		\end{eqnarray}
		Moreover, 
		\begin{eqnarray}
			\lambda = \int_{G \times \overline{X}} \beta(g, x) d\mu(g)d\nu(x)= \int_{G \times \overline{X}} \beta(g, x) d\mui(g)d\nui(x). \nonumber
		\end{eqnarray} 
		We can then apply \cite[Proposition 3.2]{benoist_quint16CLTlineargroups}: for every $\varepsilon >0 $, there exists a sequence $(C_n) \in \ell^1 (\mathbb{N})$ such that for every $x \in \overline{X}$, 
		\begin{eqnarray}
			\mathbb{P}\big(\omega \in \Omega \, : \, \big|\frac{\beta(Z_n(\omega), x)}{n} - \lambda\big| \geq \epsilon\big) \leq C_n \nonumber. 
		\end{eqnarray}
		The same goes for $\mui $ and $\nui$, which gives estimates \eqref{prob estimate 1} and \eqref{prob estimate 2}. 
		
		Estimate \eqref{prob estimate 3} is then a straightforward consequence of Proposition \ref{prop approx displacement horo}. 
	\end{proof}
	
	The following Lemma will also be important in the proof of Proposition \ref{prop centerable cocycle}. 
	
	\begin{elem}\label{Lemma drift hyp}
		Consider the same assumptions as in Proposition \ref{prop centerable cocycle}. Then there exists $A > 0$ and $(C_n) \in \ell^1 (\mathbb{N}) $ such that 
		\begin{eqnarray}
			\mathbb{P}\big( d_L (Z_n o, o) < An\big) \leq C_n \nonumber. 
		\end{eqnarray}
	\end{elem}
	
	\begin{proof}
		The action $G \curvearrowright (X,d)$ is non elementary and contains a pair of independent contracting isometries, hence by Proposition \ref{prop non elem}, there exists $L$ such that the action $G \curvearrowright (X, d_L)$ is non-elementary as the loxodromic isometries $g$ and $h$ are independent. We can then apply Theorem \ref{thm gouezel 1.2}, which gives the Lemma. 
	\end{proof}
	
	Let us now complete the proof of Proposition \ref{prop centerable cocycle}.
	\begin{proof}[Proof of Proposition \ref{prop centerable cocycle}]
		By Proposition \ref{prop non elem} and Theorem \ref{thm cv rw hyp}, there exists $L >0 $ such that $(Z_n(\omega) o )_n$ converges in $(X_L, d_L)$ to a point $z_L$ of the boundary. By Theorem \ref{thm uniqueness stat measure cat}, there is a unique $\mu$-stationary measure $ \nu$ on $\bd X$, and this measure is non-atomic by Lemma \ref{lem non atomic}.
		
		Fix $A$ as in Lemma \ref{Lemma drift hyp}, and $(C_n)_n \in \ell^1 (\mathbb{N})$ such that
		\begin{eqnarray}
			\mathbb{P}\big( d_L (Z_n o, o) < An\big) < C_n \nonumber. 
		\end{eqnarray} 
		Now take $0 < \varepsilon < \min(\frac{A}{2(4L+10)}, \lambda/2)$. Due to Proposition \ref{proba control}, there exists a sequence $C'_n \in \ell^1( \mathbb{N})$ such that 
		\begin{eqnarray}
			& & \mathbb{P}(|\beta(Z_n, x) - n \lambda | \geq \varepsilon n) \leq C'_n\nonumber\\
			& & \mathbb{P}(|\beta(Z^{-1}_n, x) - n \lambda | \geq \varepsilon n) \leq C'_n  \nonumber \\
			& & \mathbb{P}(|d(Z_n o, o) - n \lambda | \geq \varepsilon n) \leq C'_n. \nonumber 
		\end{eqnarray}
		We can assume that $C_n = C'_n $ for all $n$. Then for $\nu$-almost every $x, y \in \bd X$, we have the quantitative assumptions in Proposition \ref{geometric lemma}: with 
		\begin{enumerate}[label= (\roman*)]
			\item $ \{Z_no  \}_n $ converges almost surely in $(\overline{X_L}, d_L)$ to a point of the boundary $z_L \in \bdg X_L$, whose image in $\bd X $ by the embedding $\mathcal{I}_L$ is not $y$; 
			\item $\mathbb{P}\big( d_L(Z_n o, o) \geq An \big) \geq 1 - C_n$; 
			\item $\mathbb{P}\big(|b_x (Z_n^{-1} o ) - n \lambda | \leq \varepsilon n\big) \geq 1-C_n$; 
			\item $\mathbb{P}\big(|b_y (Z_n o ) - n \lambda | \leq \varepsilon n\big) \geq 1- C_n$; 
			\item $\mathbb{P}\big(|d(gZ_n o , o) - n \lambda | \leq \varepsilon n\big) \geq 1- C_n$. 
		\end{enumerate}	
		We  obtain that for $\nu$-almost every $x, y \in \bd X$, the probability that these estimates are not satisfied is bounded above by $4 C_n$. Now choosing $a \in (\varepsilon, \lambda - \varepsilon)$, we get that for $n$ large enough,
		\begin{enumerate}
			\item $\mathbb{P}\big((g_nx |g_n o )_o \geq a n\big) \geq 1 - 4C_n$; 
			\item $\mathbb{P}\big((y | g_n o )_o \leq a n\big)\geq 1 - 4C_n$;
			\item $\mathbb{P}\big((y| g_n x )_o \leq a n\big) \geq 1 - 4 C_n$.
		\end{enumerate}
		
		Since the sequence $(4C_n)_n$ is still summable, we can apply Proposition \ref{estimates}, that states that the function $\psi $ defined by 
		\begin{equation*}
			\psi ( x ) = -2 \int_{\overline{X}} (x | y)_o d\nui(y)
		\end{equation*}
		is well-defined, and Borel by Fubini. Moreover, $\psi $ is essentially bounded:
		\begin{eqnarray}
			\sup_{x \in \overline{X}} \int_{\overline{X}} (x |y)_o d\nu(y) < \infty \nonumber.
		\end{eqnarray}
		
	\end{proof}
	
	\begin{ecor}\label{cor centerable cocycle}
		Under the same assumptions as in Proposition \ref{prop centerable cocycle}, the cocycle $\beta(g, x) = h_x (g^{-1} o )$ is centerable.
	\end{ecor}
	
	\begin{proof}
		By Proposition \ref{prop centerable cocycle}, the function $\psi $ defined by 
		\begin{equation*}
			\psi ( x ) = -2 \int_{\overline{X}} (x | y)_o d\nui(y).
		\end{equation*}
		is well-defined, Borel and essentially bounded. Also, as observed in \cite[Lemma 1.2]{benoist_quint16}, a quick computation shows that for all $g \in G$, $x,y \in \overline{X}$:
		\begin{eqnarray}
			h_x(g^{-1}o) = -2(x| g^{-1}y)_o + 2(gx|y)_o + h_y (go) \nonumber. 
		\end{eqnarray}
		Fix $x \in X$. Integrate this equality on $(G \times \bd X, \mu \otimes \nui)$ gives 
		\begin{eqnarray}
			\int_G  \beta(g,x) d\mu(g) &=& -2\int_G \int_{\bd X}(x| g^{-1}y)_o d\mu(g) d\nui(y) \nonumber \\
			& & + 2\int_G \int_{\bd X}(gx|y)_o d\mu(g) d\nui(y)+ \int_G \int_{\bd X} \beta(g^{-1}, x)d\mu(g) d\nui(y) \nonumber \\
			&  = & -2\int_G \int_{\bd X}(x| gy)_o d\mui(g) d\nui(y) - \int_G \psi(gx)d\mu(g) \nonumber \\
			& & +\int_G \int_{\bd X} \beta(g, x)d\mui(g) d\nui(y) \nonumber.
		\end{eqnarray}
		But $\int_G \int_{\bd X}(x| gy)_o d\mui(g) d\nui(y)=  \int_{\bd X}(x| y)_o d\nui(y)$ because $\nui$ is $\mui$-stationary. Also, by point \ref{average Busemann} in Proposition \ref{Busemann cocyle prop}, we have that 
		\begin{eqnarray}
			\int_G \int_{\bd X} \beta(g, x)d\mui(g) d\nui(y) = \lambda.
		\end{eqnarray}
		Combining these, we get: 
		\begin{eqnarray}
			\int_G \beta(g,x) d\mu(g) = \psi(x) - \int_G \psi(gx)d\mu(g) + \lambda. \nonumber
		\end{eqnarray}
		
		Hence if we define $\beta_0 (g, x) = \beta(g, x) - \psi(x) + \psi(gx) $, we obtain that for all $x \in \overline{X}$, 
		\begin{eqnarray}
			\int_G \beta_0(g, x) d\mu(g)= \lambda, 
		\end{eqnarray}
		and the cocycle $\beta_0$ has constant drift $\lambda$. Then by Remark \ref{rem average cocycle}, $\beta$ is centerable with average $\lambda$, as wanted. 
	\end{proof}
	
	We can now state the following. 
	\begin{ethm}\label{thm CLT cat}
		Let $G$ be a discrete group and $G \curvearrowright X$ an action by isometries on a complete $\cat $(0) space $X$. Let $\mu \in \prob(G) $ be an admissible probability measure on $G$ with finite second moment, and assume that $G $ contains a pair of independent contracting isometries for this action. Let $o \in X$ be a basepoint of the random walk. Let $\lambda$ be the (positive) drift of the random walk. Then the random variables $(\frac{1}{\sqrt{n}}(d(Z_n o, o) - n \lambda))_n $ converge in law to a non-degenerate Gaussian distribution $N_\mu$. Furthermore, the variance of $N_\mu$ is given by 
		\begin{eqnarray}
			\int_{G \times \bd X} (b_x(g^{-1} o) - \psi(x) + \psi(gx) - \lambda)^2 d\mu(g)d\nu(x).
		\end{eqnarray}
	\end{ethm}
	
	\begin{proof}
		By Proposition \ref{cor centerable cocycle}, the cocycle $\beta$ is centerable, with average $\lambda$. By Theorem \ref{thm uniqueness stat measure cat}, the measure $\nu$ is the unique $\mu$-stationary measure on $\overline{X}$, and we can apply Theorem \ref{thm criterium clt benoist quint}: the random variables $(\frac{1}{\sqrt{n}}(\beta(Z_n(\omega),x) - n \lambda))_n$ converge to a Gaussian law $N_\mu $. But thanks to Proposition \ref{prop approx displacement horo}, this is equivalent to the convergence of the random variables $(\frac{1}{\sqrt{n}}d(Z_n(\omega)o, o) - n \lambda)_n$ to a Gaussian law. Moreover, by Theorem \ref{thm criterium clt benoist quint} and Proposition \ref{Busemann cocyle prop}, the covariance 2-tensor of the limit law is given by 
		
		\begin{eqnarray}
			\int_{G \times \bd X} (\beta_0(g, z) - \lambda)^2 d\mu(g) d\nu(z) \nonumber, 
		\end{eqnarray}
		where $\beta_0 (g,x)= \beta(g,x) - \psi(x) + \psi(gx)$. 
	\end{proof}
	
	In order to finish the proof of Theorem \ref{thm CLT cat}, it only remains to show that the limit law is non-degenerate. This is what we do in the next Proposition. 
	
	\begin{eprop}\label{prop non deg limit law cat}
		With the same assumptions and notations as in Theorem \ref{thm CLT cat}, the covariance 2-tensor of the limit law satisfies:
		\begin{eqnarray}
			\int_{G \times \bd X} (\beta_0(g, z) - \lambda)^2 d\mu(g) d\nu(z) >0 \nonumber. 
		\end{eqnarray}
		In particular, the limit law $N_\mu$ of the random variables $(\frac{1}{\sqrt{n}}(d(Z_n o, o) - n \lambda))_n $ is non-degenerate. 
	\end{eprop}
	
	In the course of the proof, we shall use the following fact. We give the proof for completeness. 
	
	\begin{elem}\label{lem attracting pt supp}
		We use the same assumptions and notations as in Theorem \ref{thm CLT cat}. Let $g \in G$ be a contracting isometry of $X$, and let $\xi^+ \in \bd X$ be its attracting fixed point at infinity. Then $\xi^+ \in \supp(\nu)$, where $\nu$ is the unique $\mu$-stationary measure on $\overline{X}$. 
	\end{elem}
	\begin{proof}
		Denote by $\xi^-  \in \bd X$ the repelling fixed point in of $g$. The isometry $g$ is contracting, hence by Theorem \ref{thm contract dyn NS} it acts on $\bd X$ with North-South dynamics. This means that for every neighbourhood $U$ of $\xi^{+}$, $V$ of $\xi^{-}$ in $\bd X$, there exists $k$ such that for all $n \geq k$, $g^n (\bd X  - V) \subseteq U $ and $g^{-n} (\bd X  - U) \subseteq V $. By Lemma \ref{lem non atomic}, $\nu$ is non-atomic, hence there exists a neighbourhood $V$ of $\xi ^{-}$ such that $\nu (\bd X - V) > 0$. Fix such a $V$, and let $U$ be any neighbourhood of $\xi^{+}$. Take $k$ large enough so that for all $n \geq k$, $g^n (\bd X  - V) \subseteq U $. Since $\mu$ is admissible, there exists $p' \in \mathbb{N}$ such that $g^k \in \supp(\mu^{\ast p'})$. Check that $\nu$ is still $\mu^{\ast p'}$-stationary, therefore 
		\begin{eqnarray}
			\sum_{h \in G} \nu(h^{-1}U) \mu^{\ast p'} (h) = \nu(U) \nonumber. 
		\end{eqnarray}
		In particular, by North-South dynamics, 
		\begin{eqnarray}
			\nu (U) \geq \nu (g ^{-k}U) \mu^{\ast p'} (g^k) \geq \nu (\bd X - V) \mu^{\ast p'} (g^k) > 0 \nonumber. 
		\end{eqnarray} 
		This is true for every neighbourhood $U$ of $\xi^{+}$, hence $\xi^{+} \in \supp(\nu)$. 
	\end{proof}

	\begin{proof}[Proof of Proposition \ref{prop non deg limit law cat}]
		Let $g$ be a contracting isometry in $G$. Recall that $g$ has an axis $\gamma \subseteq X$ on which $g$ acts as a translation, and let $\xi^{+}, \xi^{-}$ be its attracting and repelling fixed points in $\bd X$ respectively. We let $l(g) = \lim \frac{d(g^n o,o)}{n}$ be the translation length of $g$ in $(X,d)$. Observe that 
		\begin{eqnarray}
			l(g) = \lim_n \frac{b_{\xi^{+}}(g^{-n} o)}{n}, \label{translation cocycle}
		\end{eqnarray}
		where $b_{\xi^{+}} $ is the horofunction centred on $\xi^+$ and based at $o $. 
		Indeed, if $o$ belongs to $\gamma$, then $b_{\xi^{+}}(g^{-n} o) = d(g^n o , o)$ and equation \eqref{translation cocycle} is true. If $o$ does not belong to $\gamma$, take $o' \in \gamma$, and by triangular inequality, 
		\begin{eqnarray}
			|b_{\xi^{+}}(g^{-1} o) - b_{\xi^{+}}^{o'}(g^{-1}o') |\leq 2 d(o, o') \nonumber,
		\end{eqnarray}
		where $b_{\xi^{+}}^{o'}$ is the horofunction with basepoint $o'$. Since $l(g) = \lim_n \frac{1}{n} b_{\xi^{+}}^{o'}(g^{-n} o')$, we obtain that $l(g) = \lim_n \frac{1}{n} b_{\xi^{+}}(g^{-n} o)$. 
		\newline 
		
		Suppose by contradiction that $\int_{G \times \bd X} (\beta_0(h, z) - \lambda)^2 d\mu(h) d\nu(z) = 0$. This means that for almost every $\xi \in \supp(\nu)$ and $h \in \supp(\mu)$, 
		\begin{eqnarray}
			b_\xi (h^{-1} o) - \lambda = \psi (\xi) - \psi(h\xi) \nonumber. 
		\end{eqnarray} 
		Since $\psi $ is bounded and continuous, we get that for every $\xi \in \supp(\nu)$ and every $h \in \supp(\mu)$, $|b_\xi (h^{-1} o) - \lambda| \leq 2 \| \psi \|$. 
		\newline 
		
		Now consider the random walk generated by $\mu^{\ast p}$, for $p \geq 1$. Observe that $\mu^{\ast p }$ is still admissible of finite second moment and that $\nu$ is still a $\mu^{\ast p }$-stationary measure on $\bd X$. We can then apply Theorems \ref{thm drift cat} and \ref{thm CLT cat}, so that the random walk generated by $\mu^{\ast p}$ converges to the boundary with positive drift $l_X (\mu^{\ast p }) = p \lambda >0$. By the previous argument, for almost every $\xi \in \supp(\nu)$ and every $h \in \supp(\mu^{\ast p})$, 
		\begin{eqnarray}\label{eq non deg}
			|b_\xi (h^{-1} o) - p\lambda| \leq 2 \| \psi \|. 
		\end{eqnarray} 
		
		Let $g$ be a contracting element in $G$, and let $\xi^+ $ be its attracting fixed point. Because $\mu $ is admissible, there exists $m $ such that $\mu^{\ast m} ( g) > 0$. Then by Equation \eqref{eq non deg}, for all $n \geq 1$, $|b_{\xi^{+}}(g^{-n}o) - nm \lambda | \leq 2 \| \psi \|$. By Lemma \ref{lem attracting pt supp}, we can apply equation \eqref{translation cocycle}, and we obtain that 
		\begin{eqnarray}
			\lim_n \frac{b_{\xi^{+}}(g^{-n}o)}{n} =  l(g) = m \lambda \nonumber. 
		\end{eqnarray}
		But there also exists $q \in \mathbb{N}^\ast$ such that $1 \in \supp(\mu^{\ast q})$, hence $g \in \supp(\mu^{\ast (m+q)})$ and by the same argument, $l(g) = (m + q) \lambda$. Since by Theorem \ref{thm drift cat}, $\lambda$ is positive, we get a contradiction. 
	\end{proof}

	\chapter{Stationary measures on $\tilde{A}_2$-buildings}\label{chapter rw immeuble}
	
	In this chapter, we study non-elementary group actions on locally finite buildings of type $\tilde{A}_2$. Using techniques from boundary theory, we prove that there is a unique stationary measure $\nu$ supported on the set of chambers of the building at infinity $\bdinf$. In a joint work with J. Lécureux and J. Schillewaert, we then prove that any such action admits a hyperbolic element, and we derive a local-to-global result for fixed points. Here, we consider simplicial metric buildings, but we believe that the results hold in the non-discrete, metrically complete separable case too, as most of the proofs remain true in this context. It is to be noted that similar local-to-global results were already obtained. In \cite{schillewaert_struyve_thomas22}, the authors prove that if a group acts on a complete affine building $X$ of type $\tilde{A}_2$ or $\tilde{C}_2$ in such a way that each element is elliptic, then there is a global fixed point in the bordification $\overline{X}$. In \cite{osajda_przytycki21}, Osajda and Przytycki prove a similar result for any $\cat$(0) triangle complex, hence also valid for locally finite $\tilde{A}_2$-buildings. However, the approaches are completely different, as the techniques presented here involve random walks. In Section \ref{section prelim A2}, we recall some vocabulary on affine buildings, and we introduce the notion of measurable metric fields. In Section \ref{section meas struc A2}, we build the necessary framework in order to apply the theory of strong boundaries to our situation. In Section \ref{section bd map immeuble}, we study $G$-maps $B \rightarrow \ch(\bdinf)$ between the Poisson-Furstenberg boundary and the set of chambers at infinity. Using these properties, we derive in Section \ref{section uniq stat meas A2} the main result of the paper \cite{le-bars23immeuble}: we prove that there exists a unique stationary measure supported on the chambers at infinity. Last, we show in Section \ref{section hyp eleme A2} that any such action admits a hyperbolic element, and a local-to-global principle. This last part was done in collaboration with J. Lécureux and J. Schillewaert in \cite{le-bars_lecureux_schillewaert23}. 
	\vspace{1cm}

	\minitoc
	\vspace{1cm}
	
	\section{Metric fields and $\tilde{A}_2$-buildings}\label{section prelim A2}
	\subsection{Preliminaries on $\tilde{A}_2$-buildings}
	We refer to section \ref{section immeuble affine} for an introduction to the general theory of buildings. Let $\Delta$ be a discrete affine building of type $\tilde{A}_2$, and let $\Sigma $ be a fundamental apartment of $\Delta$. In this case, $\Sigma$ is the 2-dimensional simplicial complex afforded by the tessellation of the Euclidean plane $\R^2 $ by equilateral triangles, and any apartment of $\Delta$ is a subcomplex isomorphic to $\Sigma$. Recall that $\Delta$ is a building if and only if 
	\begin{itemize}
		\item any two simplices are contained in an apartment;
		\item If $A$ and $A'$ are two apartments, there exists an isomorphism of simplicial complexes $A \rightarrow A'$ fixing $A \cap A'$ pointwise. 
	\end{itemize}
	A simplex of $\Delta$ of dimension 2 is called a chamber. 
	
	Between any two points $x,y \in X$, define the distance between $x$ and $y$ to be the Euclidean distance in an apartment that contains both. This metric is well defined, and turns $\Delta$ into a $\cat$(0) space, which we call $(X,d)$: it is the canonical geometric realization of $\Delta$.  We denote its $\cat$(0) boundary by $\bd X$. The set of vertices $X^{(0)} $ of $X$ is endowed with a colouring $X^{(0)} \rightarrow \{0,1,2\} $ such that any chamber (2-dim simplex) has exactly one vertex of each colour. 
	
	When two simplices $A, B \in X$ belong to the same apartment, one can define the projection from $A $ to $B$ (for example using sign sequences). This definition can be extended to the whole building, by choosing an apartment that contains both simplices. In the following, this projection will be written $\proj_B (A)$. 
	
	Recall that a wall in $X$ is a geodesic line which is a union of edges. 
	
	\begin{eDef}
		A \textit{sector} based at a vertex $o$ in an apartment $\Sigma$ is a connected component in $\Sigma$ of the complement of the union of the walls passing through $o$. A \textit{sector} in $X$ is a subset of $X$ which is isometric to a sector based at $o$ in $\Sigma$. 
	\end{eDef}
	
	A boundary point $v \in \bd X$ is called a vertex at infinity if it is the endpoint of a geodesic ray asymptotic to a wall in $X$. Two sectors are said to be equivalent if they contain a common subsector. The equivalence class of a sector is the ideal boundary of a sector, and is also called a \textit{Weyl chamber}, two of which are said to be adjacent if they contain a common vertex at infinity. The graph $\bdinf $ is the graph whose vertices are the vertices at infinity, and for which there is an edge between two vertices if they lie in a common Weyl chamber. $\bdinf $ is then a spherical building, called the building at infinity of $X$. The set of edges in $ \bdinf$ is denoted $\ch (\bdinf)$: it is then the set of equivalence classes of sectors. In the same way, $\bdinf$ is endowed with a colouring that divides it into two types, and the corresponding subsets are denoted $\bdinf_1$ and $\bdinf_2$. 
	
	The set of chambers at infinity $\ch (\bdinf)$ can be endowed with a topology, which makes it a compact, totally disconnected space. Recall that for every vertex $x \in \Delta$, and every chamber at infinity $C \in \bdinf$, there exists a unique sector $Q(x, C)$ based at $x$ in the equivalence class of $C$. A basis of open neighbourhoods in $\ch (\bdinf)$ is then given by 
	\begin{eqnarray}
		\Omega_x(y) := \{ C \in \ch(\bdinf) \, | \, y \in Q(x, C)\}, \nonumber
	\end{eqnarray}
	for $x \in X$ and $C \in \ch(\bdinf)$. 
	
	With this topology, the action of $G $ a group of automorphisms of $\Delta$ can be extended to a continuous action on $\ch(\bdinf)$. 
	From the geometric viewpoint, the boundary of a $\cat$(0) space can be endowed with the Tits metric, with which it becomes a $\cat(1)$ space and on which its isometry group still acts by isometries. The matter of existence of fixed points for group action on a $\cat(1)$ space has been widely studied, see \cite{caprace_lytchak10}, \cite{caprace_monod13}, \cite{balser_lytchak05}. The following is taken from \cite{balser_lytchak05}, and will be useful later.

	\begin{eprop}[{\cite[Proposition 1.4]{balser_lytchak05}}]\label{circumcenter}
		Let $X$ be a $\cat$(1) space of finite dimension and of radius $r \leq \pi/2$. Then $X$ has a circumcenter which is fixed by every isometry of $X$. 
	\end{eprop}

	Recall that a point $v \in  \bd X$ is an equivalence class of rays, two rays $r_1$ and $r_2$ being equivalent, for which we write $r_1 \sim r_2$, if they contain subrays that lie in a common apartment and are parallel in this apartment.  We will say that two geodesic rays $r_1$ and $r_2$ are \textit{strongly asymptotic}, and write $r_1 \simeq r_2$, if their intersection contains a geodesic ray.  For two $\sim$-equivalent geodesic rays $r_1$ and $r_2$ that represent the boundary point $v \in \bd X$, we define their distance to be:
	
	\begin{equation}
		d_v (r_1, r_2) := \underset{s}{\inf} \underset{t \rightarrow \infty}{\lim} d(r_1(t+s), r_2(t)). 
	\end{equation}
	Note that it defines a distance, and that two strongly asymptotic rays  $r_1 $ and $r_2$ satisfy $d_v(r_1, r_2) =0$. Then this distance does not depend on the $\simeq $-strongly asymptotic classes of ray amongst rays from the same $\sim$-equivalence class. Again, this will only be interesting if we consider $v \in \bdinf$ a vertex at infinity. 
	\begin{eprop}[{\cite[Proposition 2.13]{bader_caprace_lecureux19}}]
		Let $v \in \bdinf$ be a vertex at infinity. The metric space $(T_v, d_v)$ of asymptotic classes of rays in the class of a vertex at infinity $v$ is a thick tree, and if $X$ is locally finite of order $q$, $T_v$ is regular of order $q+1$. 
	\end{eprop}

	The tree $T_v$ is called the panel tree at $v$. Note that this notion can be generalized in higher dimension, and that it can be endowed with a natural affine building structure, see \cite[Section 4.3]{remy_trojan21} and \cite[Section 5]{caprace_lecureux11}. A key point is that the space of ends of a given panel tree $T_v$ is in bijection with the residue $\res(v)$ of the chambers at infinity containing $v$ ($\res(v)$ can also be called the star at infinity containing $v$). For $v$ a vertex at infinity, denote by $\Aut(X)_v$ the group of automorphisms of $X$ that stabilize $v$. We recall the following result, see \cite[Proposition 4.5]{remy_trojan21}.
	
	\begin{eprop}\label{bijection ends chambers}
		Let $v $ be a vertex in $\bdinf$. There is a $\Aut(X)_v $-equivariant homeomorphism between the space of ends of the panel tree $T_v$ and the set $\res(v)$ of chambers of $\bdinf$ containing $v$.  
	\end{eprop}
	
	Note that for all $\xi \in \bd X$, there is a well-defined and continuous application defined by 
	\begin{eqnarray}
		\pi_\xi : X &\longrightarrow& T_\xi \nonumber \\
		x &\longmapsto & \germ_\infty(x + \xi) \nonumber. 
	\end{eqnarray}
	
	where $\germ_\infty (x + \xi)$ is here defined as the (strongly) asymptote class of the geodesic ray based at $x$, in the direction of $\xi$.

	From now on, we will denote by $T_1 =\underset{u \in \bdinf_1}{\bigsqcup}T_u$ the union of all the panel trees $T_u$, for $u \in \bdinf_1$, namely: 
	\begin{eqnarray}
		T_1 =  \{ (u, r) \, | \,  u \in \bdinf_1, r \text{ is an asymptotic class of ray in the equivalence class of } u\} \nonumber.
	\end{eqnarray}
	
	Note that for each $v \in \Delta_1^\infty$, there is a distance on $T_v$ defined by \begin{equation}\label{eq metric panel tree}
		d_v (\xi, \eta) := \underset{s}{\inf} \underset{t \rightarrow \infty}{\lim} d(r_1(t+s), r_2(t)), 
	\end{equation}
	
	where $r_1$ and $r_2$ are two geodesic rays in the equivalence class of $v$, and in the asymptote class of $\xi$ and $\eta$ respectively. We will prove in Lemma \ref{lem t1 measurable structure} that $T_1$ can be endowed with a measurable structure, compatible with the topology.

	Similarly, denote by $\dt_1 $ the union of all the ends of the panel trees $T_u, u \in \bdinf_1$. Namely, 
	\begin{equation*}
		\dt_1 = \{ (u, C) \, | \,  u \in \bdinf_1, C \in \partial T_u\}. \nonumber
	\end{equation*}
	Due to Proposition \ref{bijection ends chambers}, we can identify $\dt_1$ with the space $\res(\bdinf_1)$, where  
	\begin{equation*}
		\res(\bdinf_1) = \{ (u, C), u \in \bdinf_1, C \in \res(u)\}. 
	\end{equation*}
	
	Last, define $\partial_i T_u$ the set of unordered $i$ distinct points in $\partial T_u $ and $\partial_i T_1$ the union of all $\partial_i T_u$, $u \in \bdinf_1$. In other words, $\partial_i T_1$ is the set
	\begin{equation*}
		\{ (u , \xi_1, \dots, \xi _i ), u \in \bdinf_1, \xi_k \in \partial T_u\}, 
	\end{equation*}
	up to reorganizing the order of the $(\xi_1, \dots, \xi _i) \in (\partial T_u)^i$, where $\xi_i $ are pairwise distinct. We will come back on these spaces in Section \ref{section meas struc A2}, and we will endow them with a measurable structure.

	\subsection{Boundary theory}\label{section bd theory immeuble}
	
	In this Section, we review some notions of boundary theory that we will need later. We refer the reader to Section \ref{section bords ergodiques} for the vocabulary of fiberwise isometric actions, and to \cite{duchesne13}, \cite{bader_duchesne_lecureux16}. In particular, we will need some definitions and important properties in the theory of strong boundaries. These results were presented in Section \ref{section action moyennable} and Section \ref{section bords ergodiques}. 
	\newline 
	
	We will use the following characterization of amenable actions, see Proposition \ref{prop application de bord}. Recall that $(A, \eta)$ is an amenable $G$-space if $(A, \eta)$ is a standard Borel space such that there exists a non-singular and amenable action $G \curvearrowright (A, \eta)$, see Section \ref{section action moyennable}. 
	
	\begin{eprop}
		Let $G$ be a locally compact second countable group, and let $(A, \eta)$ be an amenable $G$-space. Let $(X,d)$ be a compact metrizable space such that $G \curvearrowright (X, d)$ is an isometric action. Then  there exists a measurable $G$-equivariant map $(A, \eta) \rightarrow \prob(X)$ which is essentially well-defined. 
	\end{eprop}
	
	We recall the definition of a $G$-boundary pair, introduced in Section \ref{section bords ergodiques}. 
	
	\begin{eDef}
		Let $(B_{-}, \nu_{-})$ and $(B_{+}, \nu_{+})$ be $G$-spaces. We say that $(B_{-}, B_{+} )$ is a $G$- boundary pair if: 
		\begin{itemize}
			\item the actions $G \curvearrowright B_{+}$ and $G \curvearrowright B_{-}$ are amenable; 
			\item the $G$-maps $B_{-} \times B_{+} \rightarrow B_{-} $ and $B_{-} \times B_{+} \rightarrow B_{+} $ are relatively metrically ergodic. 
		\end{itemize}
		We say that a $G$-space $(B, \nu)$ is a $G$-boundary if $(B,B)$ is a $G$-boundary pair. 
	\end{eDef}
	
	A reason why it is so convenient to use $G$-boundaries is that they give rigidity results on the group $G$, one of which is the existence of Furstenberg maps. Here, we are dealing with spaces with finite geometric dimension, hence with finite telescopic dimension, see \cite{caprace_lytchak10}. 
	
	\begin{ethm}[{\cite[Theorem 1.1]{bader_duchesne_lecureux16}}]\label{furst map}
		Let $X$ be a $\cat$(0) space of finite telescopic dimension and let $G$ be a locally compact second countable group acting continuously by isometries on $X$ without invariant flats. If $(B, \nu)$ is a $G$-boundary, then there exists a measurable $G$-map $B \rightarrow \bd X$. 
	\end{ethm}
	
	For instance, Theorem \ref{thm paire de bords} states that the Poisson-Furstenberg boundary $(B(G, \mu), \nu)$ associated to an admissible measure is a $G$-boundary.

	\subsection{Metric fields and equivariant theorem}
	
	In the course of the proof, we will use the vocabulary of measurable metric fields, a notion introduced by B. Duchesne in \cite{duchesne13}. This notion bears many similarities with the language of fiberwise isometric actions, as explained in \cite{duchesne_lecureux_pozetti23}. A good introduction of these notions is done in \cite{bader_duchesne_lecureux16}. 
	
	\subsubsection{Measurable metric fields}
	Let $(A, \eta)$ be a measurable space. A measurable metric field $\mathbf{X}$  over $A$ can be thought of as a way of attaching a metric space $X_a$ to each point $a \in A$  in a measurable way. 
	
	\begin{eDef}
		Let $(A, \eta)$ be a Lebesgue space (a standard Borel space endowed with a Borel probability measure). Let $\bold{X}$ be a collection $\bold{X} = \{(X_a, d_a)\}_{a \in A}$ of complete separable metric spaces over $(A, \eta)$, and let $\{x^n\}_{n \in \mathbb{N}}$ be a countable set of elements of $\underset{a \in A}{\Pi} X_a $ such that: 
		\begin{itemize}
			\item $\forall n, m \in \mathbb{N}, \, a \mapsto d_a(x^n_a, x^m_a)$ is measurable; 
			\item for $\eta$-almost every $a \in A$, $\{x^n_a\}_{n \in \mathbb{N}}$ is dense in $X_a$. 
		\end{itemize}
		The family $\{x^n\}_{n \in \mathbb{N}}$ is called a \emph{fundamental family} for $\bold{X}$. The data $((A, \eta), \bold{X}, \{x^n\}_{n \in \mathbb{N}})$ is called a \emph{metric field} over $(A, \eta)$. 
		
		A \emph{section} $s : (A, \eta) \rightarrow \bold{X} $ is a map such that for every $ a \in A$, $s(a) \in X_a$ and for every element $x^n $ of the fundamental family, $a \mapsto d_a (x^n_a, s(a))$ is measurable. 
	\end{eDef}
	
	Fix a locally compact second countable group $G$ endowed with its Haar measure, and let $(A, \eta)$ be a $G$-space. 
	Now let $\bold{X} = \{X_a\}_{a \in A}$ be a metric field over $A$. We say that $G$ acts on $\bold{X}$ by the cocycle $\alpha$ if: 
	\begin{itemize}
		\item for all $g \in G$ and almost every $a \in A$, $\alpha (g, a) \in \iso(X_a, X_{ga})$; 
		\item for all $g,g' \in G$ and almost every $a \in A$, $\alpha (gg', a) = \alpha(g, g'a) \alpha(g', a)$; 
		\item for all element $x^n, x^m $ of the fundamental family associated to $\bold{X}$, the map $(g, a) \mapsto d_a(x^n_a, \alpha (g, g^{-1}a) x^m_{g^{-1}a})$ is measurable. 
	\end{itemize}
	A section $s : (A, \eta) \rightarrow \{X_a\}_{a \in A} $ is then said to be \emph{$\alpha$-invariant} if for all $g \in G$ and almost every $a \in A$, $s(ga) = \alpha (g, a) s(a)$. 
	
	\begin{eDef}\label{def subfield}
		Let $\bold{X} = \{X_a\}_{a \in A}$ be a metric field over $(A, \eta)$. We say that $\bold{X}$ is a \emph{$\cat$($\kappa$)-field} if for almost every $a$, $X_a$ is a $\cat(\kappa)$ space. Similarly, if for almost every $a$, $X_a$ is a Euclidean space, we say that $\bold{X}$ is a \emph{Euclidean field}. A \emph{subfield} of a $\cat$(0) field is a collection $\{Y_a\}_a $ of non-empty closed convex subsets such that for every section $x $ of $\bold{X}$, the function $a \mapsto d (x_a, Y_a)$ is measurable. A subfield is called \emph{invariant} if for all $g$, and $\eta$-almost every $a \in A$, 
		\begin{eqnarray}
			\alpha(g , g^{-1} a)Y_{g^{-1}a } = Y_a \nonumber. 
		\end{eqnarray}
	\end{eDef}
	
	Let $\bold{X}$ be a $\cat$(0)-field over $(A, \eta)$. There is a way to define the \emph{boundary field} $\partial \bold{X}$ of $\bold{X}$, but one must be careful with the measurable structure and the separability conditions. For instance, the boundary of the
	hyperbolic plane $\mathbb{H}^2$, endowed with the Tits metric, is an uncountable discrete space. However, it is possible if one considers the compactification by horofunctions, as explained in \cite[Section 9.2]{duchesne13}. 
	
	Let $G$ be a locally compact second countable group and $B$ an ergodic $G$-space. Let $\bold{X}$ be a measurable field of $\cat$(0) spaces of finite telescopic dimension endowed with a $G$-cocycle $\alpha : G \times B \rightarrow \iso(\bold{X})$. Fix a fundamental family $\{x^n\}$ for $\bold{X}$. For $b \in B$, let $C_b$ be the locally convex linear space of convex functions that vanish at $x^0_b$, and endow $C_b$ with the topology of pointwise convergence. Let $L_b$ be the subset of $C_b$ which consists of functions that are moreover 1-Lipschitz. A metric on $L_b$ is given by 
	\begin{eqnarray}
		D_b(f,g) = \sum \frac{|f(x^i_b) - g(x^i_b)|}{2^i d_b(x^i_b, x^0_b)} \nonumber. 
	\end{eqnarray}
	Let $\iota_b : X_b \to L_b $ be the map defined by 
	\begin{eqnarray}
		\iota_b (x) : y \in X_b \mapsto d_b (x, y ) - d(x, x^0_b) \nonumber,
	\end{eqnarray}
	which is indeed a 1-Lipschitz map on $X_b$. Define $K_b $ to be the closed convex hull of $\iota(X_b) $ in $L_b$. Then $\bold{K} := (K_b)_b$ is a measurable field of compact spaces over $B$, with a fundamental family given by $\{\iota_b(x^i_b)\}_i$. The metric field $\bold{K}$ corresponds to the horocompactification of the metric field $\bold{X}$. 
	\newline 
	
	Let $\bold{X}$ is a measurable field of proper metric spaces endowed with a $G$-cocycle $\alpha : G \times B \rightarrow \iso(\bold{X})$. Fix a fundamental family $\{x^n\}$ for $\bold{X}$. Then there exists a measurable structure on the collection $\prob(\bold{X}) := \{\prob(X_b)\}_b$, coming from a Borel structure on the set $C(\bold{X}):= \{C(X_b)\}_b$  of Banach spaces of continuous functions over $\bold{X}$. We refer to \cite[Theorem 2.19]{anderegg_henry14}
	for details. 
	\newline 
	
	The link between measurable metric fields and fiberwise isometric actions is given by the following Lemma. 
	\begin{elem}[{\cite[Lemma 4.11]{duchesne_lecureux_pozetti23}}]\label{lem metric field fiber action}
		Let $G$ be a countable discrete group and $\bold{X} = \{(X_a, d_a)\}_{a \in A}$ be a measurable metric field over a Lebesgue $G$-space $(A, \eta)$, and let us assume that $G$ acts on $\bold{X}$ with some cocycle $\alpha$. Then there is a $G$-invariant Borel subset $A_0 \subseteq A$ of full $\eta$-measure, a standard Borel structure on $\mathbb{X} := \bigsqcup_{a \in A_0} X_a$, and a Borel map $p : \mathbb{X} \rightarrow A_0$ such that $p$ admits a $G$-fiberwise isometric action. Moreover, the fiber $p^{-1}(a) $ is $X_a$ with the metric $d_a$. 
		
		If $s$ is an invariant section of $\bold{X}$, then $s$ corresponds canonically to a $G$-equivariant measurable map $A_0 \rightarrow \mathbb{X}$. 
	\end{elem}

	\subsubsection{Adams-Ballmann equivariant theorem for amenable actions}
	
	Recall that a flat is a closed convex subspace which is isometric to some finite dimensional Euclidean space, and that a point is a flat of dimension $0$. In \cite[Theorem 1]{adams_ballmann98}, Adams and Ballmann prove a general dichotomy for amenable groups acting on Hadamard spaces. 
	
	\begin{ethm}[{Equivariant Adams-Ballmann Theorem}]
		Let $X$ be a locally compact complete Hadamard space. If $G$ is an amenable group of isometries of $X$, then at least one of the following proposition hold: 
		\begin{enumerate}
			\item[(i)] $G$ fixes a point in $\bd X$; 
			\item[(ii)] $X$ contains a $G$-invariant flat. 
		\end{enumerate}
	\end{ethm}
	
	The following result is a (measurable) version of the equivariant Adams-Ballmann Theorem, but for amenable actions, and was proven by Duchesne.
	
	\begin{ethm}[{\cite[Theorem 1.8]{duchesne13}}]\label{thm meas adams ballmann}
		Let $G$ be a locally compact second countable group and $B$ an amenable ergodic $G$-space. Let $\bold{X}$ be a measurable metric field of Hadamard spaces of finite telescopic dimension over $B$. If $G$ acts on $\bold{X}$, then either there is an invariant section of the boundary field $\partial \bold{X}$, or there is an invariant Euclidean subfield of $\bold{X}$. 
	\end{ethm}
	
	In Section \ref{section bd map immeuble}, we will use this correspondence to study the properties of $G$-equivariant measurable maps $B \rightarrow \bdinf$.

	In the course of the proof, we will use the fact that taking the circumcenters in a subfield of bounded closed convex subsets gives a measurable map. 
	
	\begin{eprop}[{\cite[Lemma 3.2]{anderegg_henry14}}]\label{prop section circ borel}
		Let $\bold{X}$ be a metric field of proper $\cat$(0) spaces over a Borel space $(A, \eta)$, and let $\{B_a\}_{a \in A}$ be a Borel subfield of bounded subsets of $\bold{X}$. Then the circumradius function $a \in A \mapsto r_a(B_a)$ is Borel, and the section of circumcenters 
		\begin{eqnarray}
			\text{circ} : \bold{B} &\rightarrow& \bold{X} \nonumber \\
			B_a & \mapsto & \text{circ} (B_a), \nonumber
		\end{eqnarray}
		where $\text{circ} (B_a)$ is the circumcenter of $B_a$, is Borel. 
	\end{eprop}
	
	\section{Measurable structures and barycenter}\label{section meas struc A2}
	
	Let $X$ be a thick building of type $\tilde{A}_2$ and let $G$ be a locally compact second countable group acting continuously by isometries on $X$ without fixed points nor invariant flats. Up to taking a subgroup of index 2, we assume that $G$ preserves the type of the vertices at infinity. Let $\mu $ be an admissible symmetric measure on $G$, and let $(B, \nu)$ be the Poisson boundary associated to $(G, \mu)$. In this Section, we give an explicit construction of different tools that we will need in the proof of our main theorem. We will try to be careful with the measurable structures, so that the results of Section \ref{section bd theory immeuble} can apply. 
	
	\subsection{Measurable structures on the panel trees}
	Recall that for every $u  \in \bdinf_1$, $T_u $ is the panel tree associated to the vertex $u$. Endow $T_u$ with the separable metric $d_u$ given by Equation \ref{eq metric panel tree}. 
	
	\begin{elem}\label{lem t1 measurable structure}
		The collection of panel trees $T_1 = \bigsqcup_u T_u$ admits a standard Borel structure and the natural projection $p : T_1 \rightarrow \bdinf_1 $ admits a $G$-fiberwise isometric action. 
	\end{elem}
	
	\begin{proof}
		The collection $\{(T_u, d_u)\}_{u \in \bdinf_1}$ is a family of complete, separable metric spaces. For the rest of the proof, identify each tree $T_u$ with its set of vertices. Let us construct a fundamental family for $\{(T_u, d_u)\}_{u}$ over $\bdinf_1$. Let $\{y_n\}_n$ be the (countable) set of vertices of the building $X$. For every $n $, and every $u \in \bdinf_1$, consider 
		\begin{eqnarray}
			x^n_u := \pi_u(y_n) = \germ_\infty (y + u), \nonumber
		\end{eqnarray}
		where $\germ_\infty (y + u)$ is defined as the class (for the strong asymptote relation) of the geodesic ray based at $y$ in the direction of $u$. Observe that for every $u $, $\{x^n_u\}_n$ is dense in $T_u$. Fix $n, m \in \mathbb{N}$. Then the function 
		\begin{eqnarray}
			u \in \bdinf_1 \mapsto d_u (x^n_u, x^m_u) \nonumber
		\end{eqnarray}
		is measurable. Indeed, for every $ y  \in X$, the application $\xi \in \bd X \mapsto \gamma_y^\xi$ is continuous for the topology of uniform convergence on compact sets, where $\gamma_y^\xi$ is the unique geodesic ray based at $y$ representing $\xi $. Recall that by definition, 
		\begin{eqnarray}
			d_u (x^n_u, x^m_u) = \underset{s}{\inf} \lim_{t \rightarrow \infty} d(\gamma_{y_n}^u (t+ s), \gamma_{y_m}^u(t)). \nonumber
		\end{eqnarray} 
		Since we consider the infimum of measurable functions, we have that for all $n, m $, $u \in \bdinf_1 \mapsto d_u (x^n_u, x^m_u)$ is then measurable. Therefore, the collection $\{x^n\}$ is a fundamental family for the metric field $\{T_u\}_u$ over $\bdinf_1$.
		
		Let us now prove that $G$ acts on the metric field $\{T_u\}_u$. If $u \in \bdinf_1 $ is a vertex at infinity and $\xi \in T_u$ is a class (for the strong asymptote relation) of geodesic rays parallel to $u$, the element $g \in G$ acts as a cocycle $\alpha $ with 
		\begin{eqnarray}
			\alpha (g, u) (u, \xi) = g \cdot(u, \xi) = (gu, g \xi) \nonumber, 
		\end{eqnarray}
		where $g\xi$ is the class (for the strong asymptote relation) of $g \gamma$, for any geodesic ray $\gamma $ in the class of $\xi$. It is clear for all $g \in G$, 
		\begin{eqnarray}
			\alpha (g, u) \in \iso(T_u, T_{gu}) \nonumber. 
		\end{eqnarray}
		Since the action of an element $g $ on each fiber $T_u$ does not depend on the vertex $u$, $\alpha  $ satisfies the cocycle relation: for all $g, g' \in G$, $u \in \bdinf_1$, $\xi \in T_u$, 
		\begin{eqnarray}
			g' \cdot (g \cdot(u, \xi)) = (g'g ) \cdot (u, \xi). \nonumber
		\end{eqnarray} 
		It remains to prove that $G$ acts on this metric field in a measurable way. 
		
		Take $n,  m \in \mathbb{N}$. Observe that for $g \in G$, $u \in \bdinf_1$, 
		\begin{eqnarray}
			g \cdot x^m_{g^{-1}u}&=& g \cdot \pi_{g^{-1}u}(y^m) \nonumber \\
			&=& g \cdot \germ_\infty (y^m + g^{-1}u) \nonumber \\
			&=& \germ_\infty (gy^m + u) \in T_u\nonumber.
		\end{eqnarray}
		As $G$ acts continuously on $X$ and preserves the type of the vertices at infinity, we obtain that 
		\begin{eqnarray}
			(g, u) \mapsto d_u (x^n_u , g \cdot x^m_{g^{-1}u}) \nonumber
		\end{eqnarray}
		is measurable.

		Now by Lemma \ref{lem metric field fiber action}, there is a standard Borel structure on $T_1 := \bigsqcup_{u \in \bdinf_1} T_u$, and a Borel map $p : T_1 \rightarrow \bdinf_1$ such that $p$ admits a $G$-fiberwise isometric action and such that $p^{-1}(u) = T_u$ for every $u \in \bdinf_1$. 
	\end{proof}
	
	Recall that we denote by $\dt_1 $ the union of all the boundaries of the panel trees $T_u, u \in \bdinf_1$. Namely, 
	\begin{equation*}
		\dt_1 = \{ (u, C) \, | \,  u \in \bdinf_1, C \in \partial T_u\}. \nonumber
	\end{equation*}
	
	For $(u ,\xi ) \in T_1$, that is, $u\in \bdinf_1$ and $\xi \in T_u$, we can define the following metric on $\dt_u$:
	\begin{eqnarray}
		d_{u, \xi}(C, D) := \exp(-(C |D)_\xi),\nonumber
	\end{eqnarray} 
	with $(C,D) \in \dt_u$, and where $(. \, | \,. )_\xi$ the Gromov product on the tree $T_u$ based at $\xi \in T_u$.
	
	The next Lemma states that the natural projection $T_1 \times \dt_1 \rightarrow T_1 $ admits a $G$-fiberwise isometric action for this metric.

	\begin{elem}\label{lem dt1 measurable structure}
		The collection $T_1 \times \dt_1 = \bigsqcup_{u \in \bdinf_1} T_u \times \dt_u$ admits a standard Borel structure and the natural projection $q : T_1 \times \dt_1 \rightarrow T_1 $ admits a $G$-fiberwise isometric action. 
	\end{elem}
	
	\begin{proof}
		Endow $T_1$ with the measurable structure given by Lemma \ref{lem t1 measurable structure}. Consider the family $\{T_u \times \dt_u\}_{u \in \bdinf_1}$, and call $q$  the projection on the first coordinate 
		\begin{eqnarray}
			q &:& \{T_u \times \dt_u\}_{u \in \bdinf_1} \longrightarrow T_1 \nonumber \\
			&&(u , \xi, C) \longmapsto (u, \xi) \in T_u. \nonumber
		\end{eqnarray}
		For every $u \in \bdinf_1 $, $\xi \in T_u$, the metric $d_{u, \xi}$ on the fiber $q^{-1}\{(u, \xi)\} = \{(u , \xi)\} \times \dt_u$ is complete and separable. For convenience, we denote the metric spaces $(q^{-1}\{(u, \xi)\}, d_{u, \xi})$ by $\dt_{u, \xi}$. In order to prove that this is a metric field over $T_1$, we need to construct a fundamental family. Consider a countable dense family $\{C_n\}$ of chambers in $\ch(\bdinf)$. For every $u \in \bdinf_1$, define 
		\begin{eqnarray}
			C^n_u := \proj_{u} (C_n) \in \res(u). \nonumber
		\end{eqnarray}
		The projection onto the vertex $u$ is continuous. Moreover, if one fixes $C \in \ch(\bdinf)$, the application $u \in \bdinf_1 \mapsto \proj_{u} (C)$ is continuous. By Proposition \ref{prop homeo panel tree ch}, $C^n_u \in \dt_u$. Consequently, for every $(u, \xi) \in T_1$, define 
		\begin{eqnarray}
			\mathcal{C}^n_{u, \xi} := (u, \xi, C^n_u) \in T_u \times \dt_u \nonumber. 
		\end{eqnarray}
		For all $(u, \xi) $, the collection $\{\mathcal{C}^n_{u, \xi}\}_n$ is dense in the fiber $\dt_{u, \xi}$. Fix $n , m \in \mathbb{N}$, and let $(u , \xi ) \in T_1$. We have that 
		\begin{eqnarray}
			d_{u, \xi}(\mathcal{C}^n_{u, \xi}, \mathcal{C}^m_{u, \xi}) &=& d_{u, \xi}(C^n_u, C^m_u)\nonumber \\
			& = & \exp (-(C^n_u |C^m_u)_\xi) \nonumber \\
			& = & \exp (-(\proj_u(C_n) |\proj_u(C^m))_\xi) \nonumber. 
		\end{eqnarray}
		By composition, the application 
		\begin{eqnarray}
			(u, \xi) \in T_1 \longmapsto d_{u, \xi}(\mathcal{C}^n_{u, \xi}, \mathcal{C}^m_{u, \xi}) \nonumber 
		\end{eqnarray}
		is measurable. The collection $\{\mathcal{C}^n_{u, \xi}\}_n$ is then a fundamental family for the metric field $\{T_u \times \dt_u\}_{u \in \bdinf_1}$ over $T_1 $. 
		
		By Lemma \ref{lem t1 measurable structure}, the projection $T_1 \rightarrow \bdinf_1$ admits a $G$-fiberwise isometric action. For $g \in G$, and $(u, \xi, C) \in \{T_u \times \dt_u\}$, 
		\begin{eqnarray}
			\alpha (g, (u, \xi)) (u, \xi, C) = g \cdot (u, \xi, C) = (gu, g\xi, gC), \nonumber
		\end{eqnarray}
		where $(gu, g\xi) $ was defined in the proof of Lemma \ref{lem t1 measurable structure}, and $gC \in \dt_{gu}$ is the boundary point of $\dt_{gu}$ that corresponds to the chamber $gC \in \res(gu)$ by Proposition \ref{prop homeo panel tree ch}. 
		With this action, for all $g \in G$, 
		\begin{eqnarray}
			\alpha (g, (u, \xi)) \in \iso(\dt_{u, \xi}, \dt_{g \cdot (u, \xi)}). \nonumber
		\end{eqnarray}
		Moreover, the projection $q$ is $G$-equivariant. Last, for every $n, m \in \mathbb{N}$, the measurability of
		\begin{eqnarray}
			(u, \xi) \in T_1 \longmapsto d_{u, \xi}(\mathcal{C}^n_{u, \xi}, g \cdot \mathcal{C}^m_{g^{-1}\cdot(u, \xi)}) \nonumber 
		\end{eqnarray}
		comes from the continuity of the $G$-action on the chambers $\ch(\bdinf)$. This proves that $G$ acts on the metric field $\{T_u \times \dt_u\}_{u \in \bdinf_1}$ with a natural cocycle. By Lemma \ref{lem metric field fiber action}, there is a standard Borel structure on $T_1 \times \dt_1 = \bigsqcup_u T_u \times \dt_u$ such that the projection $T_1 \times \dt_1 \rightarrow T_1 $ is Borel and admits a $G$-fiberwise isometric action. 
	\end{proof}
	
	\subsection{Barycenter map}
	
	Consider the set $(\dt_1)^2 $ defined by 
	\begin{equation*}
		(\dt_1)^2 = \{ (u , C_1, C_2), u \in \bdinf_1, C_1, C_2 \in \partial T_u\}. 
	\end{equation*}
	Denote by $\mathcal{D}(\dt_1)$ the diagonal: 
	\begin{equation*}
		\mathcal{D}(\dt_1) = \{ (u , C_1, C_2), u \in \bdinf_1, C_1 = C_2 \in \partial T_u\} \subseteq (\dt_1)^2.  
	\end{equation*} 
	Define the application:
	\begin{eqnarray}
		\tau &:& (\dt_1)^2 \to (\dt_1)^2  \nonumber \\
		& & (u , C_1, C_2) \in (\dt_u)^2  \mapsto (u , C_2, C_1) \in (\dt_u)^2 \nonumber. 
	\end{eqnarray}
	Let $(\partial T_1)^{(2)}$ be the measurable set $(\dt_1)^2 -\mathcal{D}(\dt_1)$, and define the space $\partial_2 T_1$ as the set $(\partial T_1)^{(2)}/ \tau $ of unordered pair of distinct points of $\dt_1$. Similarly, we define by $\partial_n T_1$ as the space of unordered $n$-tuples of points of $\dt_1$ that are pairwise distinct. Endow $\partial_2 T_1$ and $\partial_n T_1$ with the product measurable structure coming from $\dt_1$. 
	\newline
	
	In the course of the proof, we will need the existence of a well-defined barycenter map on the set $\partial_3 T_1$. 
	
	\begin{eprop}\label{prop bary panel trees}
		For all $n \geq 3$, there exists a measurable $G$-equivariant map $\partial_n T_1 \to T_1$. 
	\end{eprop}
	\begin{proof}
		Let $u \in \bdinf_1$. Let $S\subseteq \partial_n T_u$ be a set of $n\geq 3$ distinct points in $\partial T_u$. Consider the function 
		\begin{eqnarray}
			F_S : x \in T_u \mapsto \sum_{C\neq C' \in S} (C| C')_x \nonumber. 
		\end{eqnarray}
		It is standard to check that the function $F_S$ is proper and convex, see for example \cite[Section 2]{burger_mozes96}. Therefore its minimum $\min(F_S) = \{ x \in T_u \; | \; F_S(x) = \inf F_S\}$ is non-empty, bounded and closed. Since $T_u$ is a tree, we can take the circumcenter of this set, which gives a continuous map
		\begin{eqnarray}
			\bary_u : \partial_n T_u \to T_u. \nonumber
		\end{eqnarray}
		We define
		\begin{eqnarray}
			\bary &:& \partial_n T_1 \to T_1 \nonumber \\
			& & (u, C_1, \dots, C_n) \in\partial_n T_u \mapsto \bary_u(u, C_1, \dots, C_n) \in T_u \nonumber. 
		\end{eqnarray}
		Let $G$ act on $\partial_n T_1$ with the natural diagonal action. Then the map $\bary$ is $G$-equivariant: 
		\begin{eqnarray}
			\bary (g \cdot (u, C_1, \dots, C_n)) &=& \bary_{gu} (g \cdot (u, C_1, \dots, C_n)) \nonumber \\
			& =& g \cdot \bary (u, C_1, \dots, C_n) \nonumber. 
		\end{eqnarray}
		Moreover, $\bary$ is measurable for the measurable structures defined by Lemma \ref{lem t1 measurable structure} and \ref{lem dt1 measurable structure}. Indeed, if we take a measurable section $s : \bdinf_1 \to \partial_n T_1 $, then the map 
		\begin{eqnarray}
			u \in \bdinf_1 \mapsto \bary_u (s(u)) \nonumber 
		\end{eqnarray}
		is the composition of two maps. First, the collection $\{\min(F_{s(u)})\}_{u \in \bdinf_1}$ is a Borel subfield of non-empty closed convex bounded subsets of the metric field $T_1$, in the sense of Definition \ref{def subfield}. Indeed, for any other section $s'$, 
		\begin{eqnarray}
			u \in \bdinf_1 \mapsto d_u (s'(u), \min(F_{s(u)})) \nonumber
		\end{eqnarray}
		is measurable by the definition of $u \mapsto F_{s(u)}$. 
		Next, by Proposition \ref{prop section circ borel}, taking the section of circumcenters gives a Borel map 
		\begin{eqnarray}
			\{\min(F_{s(u)})\}_{u \in \bdinf_1} \rightarrow T_u \nonumber. 
		\end{eqnarray}
		This ends the proof. 
	\end{proof}
	
	Let $Y$ be a proper $\cat(-1) $ space, and let $\prob_3( \bdg Y)$ be the space of positive measures on $\bdg Y$ whose support contains at least three points. In \cite[Proposition 2.1]{burger_mozes96}, Burger and Mozes prove that there exists a $\iso(Y)$-equivariant map 
	\begin{eqnarray}
		\bary : \prob_3 (\bdg Y) \to Y
	\end{eqnarray}
	whose restriction to any $\iso (Y)$-orbit is continuous. We shall now prove that there exists a measurable version of this fact, for the metric field $T_1$.  
	\newline 
	
	Let $u \in \bdinf_1$ be a vertex at infinity of type 1 and let $\prob_3( \partial T_u)$ be the space of positive measures on $\partial T_u$ whose supports contain at least three points.
	\begin{elem}\label{lem bary proba tu}
		There exists a measurable map $\varphi_u : \prob_3( \partial T_u) \rightarrow \prob(T_u)$. 
	\end{elem}
	
	\begin{proof}
		Let $\nu \in \prob_3( \partial T_u)$. For the moment, assume that $\nu$ is non-atomic. Then consider the probability measure $\tilde{\nu} := \nu \otimes \nu \otimes \nu $ on $(\partial T_u)^3$. Let $\mathcal{D}(\partial T_u)^3 $ be the space of triples $(u, \xi_1, \xi_2, \xi_3) \in (\partial T_u)^3$ such that at least two entries coincide. The set $\mathcal{D}(\partial T_u)^3 $ is measurable. Since $\nu$ is non-atomic, $\tilde{\nu}$ gives zero measure to $\mathcal{D}(\partial T_u)^3 $ and we can see $\tilde{\nu}$ as a measure on $(\partial T_1)^{(3)} $. Passing to the quotient, we can see $\tilde{\nu}$ as a measure on $\partial_3 T_1 $, which we do without changing the notation. Now, thanks to the barycenter map defined in Proposition \ref{prop bary panel trees}, $\varphi_u (\nu) := (\bary_u)_\ast \tilde{\nu}$ is a probability measure on $T_u $.

		Assume now that $\nu$ is purely atomic, that is, for every $\xi \in \supp(\nu)$, $\nu(\xi) > 0$. Since the support of $\nu$ is of cardinal at least $3$, $\nu$ has at least 3 atoms $\xi_1, \xi_2, \xi_3$. Again, let $\tilde{\nu} := \nu \otimes \nu \otimes \nu $ be the product probability measure on $(\partial T_u)^3$. By definition, $\tilde{\nu} (\xi_1, \xi_2, \xi_3) > 0$ and in particular $\tilde{\nu}((\partial T_1)^{(3)} ) >0 $. Recall that taking the restriction of a measure to an open subset is a measurable operation. We can then consider
		\begin{eqnarray}
			\bar{\nu} = \frac{1}{\tilde{\nu}((\partial T_1)^{(3)} )} \tilde{\nu}_{|(\partial T_1)^{(3)} }, \nonumber
		\end{eqnarray} 
		which is a probability measure on the set of triples of pairwise distinct points in $\dt_u$. Again, we can see $\bar{\nu}$ as a probability measure on $\partial_3 T_1$, and applying the barycenter map $\bary_u$ from Proposition \ref{prop bary panel trees}, one obtains a probability measure $\varphi_u (\nu)$ on $T_u$.

		If now $\nu \in \prob_3( \partial T_u)$ is any probability measure, we can decompose $\nu$ as a sum $\nu = \nu_c + \nu_a$, where $\nu_c$ is non-atomic and $\nu_a$ is purely atomic. By the previous case, we can assume that the non-atomic part is nonzero $\nu_c > 0$. Therefore, the probability measure $\nu_c / \nu_c (\partial T_u)$ is non-atomic, and by the first case we can measurably associate a probability measure on $T_u$.
	\end{proof}
	We denote by $\prob_3(\partial T_1)$ the subfield of $\prob(\dt_1)$ of measurable sections $u \in \bdinf_1\mapsto \nu_u \in \prob_3(\partial T_u)$ for the measurable structure on $\partial T_1$. 
	\begin{elem}\label{lem bary proba t1}
		There exists a measurable and $G$-equivariant map $\varphi : \prob_3( \partial T_1) \rightarrow \prob(T_1)$.
	\end{elem}
	
	\begin{proof}
		On every fiber $u \in \bdinf_1$, the application $\varphi_u : \prob_3( \partial T_u) \rightarrow \prob(T_u)$ is measurable. Moreover, every operation that we did in the proof of Lemma \ref{lem bary proba tu} was $G$-equivariant, i.e. if $\nu = \{\nu_u\}_{u \in \bdinf_1} \in \prob_3(\dt_1)$, then for every $u \in \bdinf_1$, 
		\begin{eqnarray}
			g \cdot (\varphi_u (\nu_u)) = \varphi_{gu} (\nu_{gu}). \nonumber
		\end{eqnarray}
		Measurability of this map comes from Proposition \ref{prop bary panel trees} and the operations in the proof of Lemma \ref{lem bary proba tu}. This proves the measurability and $G$-equivariance of the application. 
	\end{proof}
	
	For self-containment, we now recall the classical construction of a barycenter map for the set of measures on a tree. Let $\varepsilon < \frac{1}{2}$, $\nu \in \prob(T_u)$ and $\xi \in T_u $. For any $R >0$, define 
	\begin{eqnarray}
		F^\nu_\xi (R) := \nu(B(\xi, R)) \nonumber, 
	\end{eqnarray}
	where $B(\xi, R)$ is the ball of radius $R$ and center $\xi$ in $T_u$ for the metric $d_u$ given by equation \eqref{eq metric panel tree}.
	It is clear that $F^\nu_\xi (R) \to 1$ as $R \to \infty $, hence for all $\xi \in T_u $, there exists a well-defined and finite 
	\begin{eqnarray}
		R^\nu_{\xi, \varepsilon} := \inf \{ R \; | \; F^\nu_\xi (R) > 1 - \varepsilon\} \nonumber. 
	\end{eqnarray}
	Also, denote by $R^\nu_{\varepsilon} := \inf_{\xi \in T_u} R^\nu_{\xi, \varepsilon}$. We can then define the measurable application 
	\begin{eqnarray}
		\tilde{\beta}^u_\varepsilon (\nu):= \overline{\{\xi \in T_u \; | \; R^\nu_{\xi, \varepsilon} < R^\nu_{\varepsilon} +1\}} \nonumber. 
	\end{eqnarray}
	Check that $\tilde{\beta}^u_\varepsilon (\nu) $ is a closed bounded set on $T_u$. Indeed, if $\xi_1, \xi_2 \in \tilde{\beta}^u_\varepsilon (\nu) $, then by definition 
	\begin{eqnarray}
		\nu(\xi \in T_u \; | \; B(\xi_i, R^\nu_{\varepsilon} +1)) > 1/2 \text{ for $i = 1, 2$}\nonumber,
	\end{eqnarray}
	and then $B(\xi_1, R^\nu_{\varepsilon} +1)) \cap B(\xi_2, R^\nu_{\varepsilon} +1)) \neq \emptyset$. Now since trees are $\cat$(0) spaces, Proposition \ref{prop section circ borel} states that we can measurably associate to each bounded set its circumcenter. We then have a measurable application 
	\begin{eqnarray}
		\beta^u_\varepsilon : \nu \in \prob(T_u) \mapsto \text{circ} (\tilde{\beta}^u_\varepsilon (\nu) ) \in T_u. \nonumber
	\end{eqnarray}
	From the definition, $\beta_\varepsilon^u$ is $\iso(T_u)$-equivariant. 
	
	Considering the metric field $\prob(T_1)$ over $\bdinf_1$, what we have obtained is that the map 
	\begin{eqnarray}
		\beta_\varepsilon : \prob(T_1) &\to& T_1 \nonumber \\
		\{\nu_u\}_{u \in \bdinf_1} &\mapsto & \{\beta^u_\varepsilon (\nu_u)\}_{u \in \bdinf_1} \nonumber
	\end{eqnarray}
	is well-defined, measurable, and $G$-equivariant for the natural $G$-action on $\prob(T_1)$. Composing $\beta_\varepsilon $ and $\varphi $ from Lemma \ref{lem bary proba t1}, we have obtained the following result. 
	
	\begin{eprop}\label{prop bary proba 3 T1}
		There exists a measurable and $G$-equivariant map $\prob_3( \partial T_1) \rightarrow T_1$.
	\end{eprop}

	We define $T_2$ as the collection of panel trees over the vertices at infinity $\bdinf_2$ of type 2, and we proceed similarly for $\dt_2$ and $\partial_i T_2$. All the results in this section apply for these metric fields in the same way.

	\section{Boundary maps in $\tilde{A}_2$-buildings}\label{section bd map immeuble}
	
	Let $X$ be a thick building of type $\tilde{A}_2$ and let $G$ be a locally compact second countable group acting continuously by isometries on $X$ without fixed points nor invariant flats. For the rest of this chapter, let $\mu $ be a symmetric measure on $G$, which is \emph{spread-out}, i.e. there is $n$ such that $\mu^n $ and $\haar_G$ are non-singular, and such that there is no proper closed subgroup of full $\mu$-measure. We shall say that a measure satisfying such conditions is \emph{admissible}. Let $(B, \nu)$ be the Poisson boundary associated to $(G, \mu)$. By Theorem \ref{furst map}, there exists a measurable $G$-map $\psi : B \rightarrow \bd X$. 
	
	By Proposition \ref{prop isom autom immeuble}, and since the building at infinity $\bdinf$ is thick, any isometry of $X$ preserves the Weyl chambers. In particular the action of $G $ on $X$ induces an action of $G$ on the chambers at infinity $\ch(\bdinf)$. Up to taking a subgroup of index 2, we assume that $G$ preserves the types of the vertices at infinity, so that $G$ acts on $\bdinf_1$ and on $\bdinf_2$. Moreover, the partition of the visual boundary $\bd X $ between regular points (that is, classes of regular geodesic rays in the sense of Definition \ref{def regular geod immeuble}), and singular points (that is, classes of singular geodesic rays) is measurable because $\bdinf_1$ and $\bdinf_2$ are closed in $\bd X$. In other words, we have the measurable decomposition $\bd X =  \bdinf_1 \sqcup \bdinf_2 \sqcup (\bd X - \bdinf_1 \cup \bdinf_2 )$, and the set of regular points $\bd X - \bdinf_1 \cup \bdinf_2 $ projects equivariantly onto $\ch(\bdinf)$. 
	\newline 
	
	We showed in Section \ref{section mesure stat intro} that in order to study stationary measures on $\ch(\bdinf)$, one needs to characterize the $G$-maps $B \rightarrow \prob(\ch(\bdinf))$. The purpose of this section is to study such boundary maps. More specifically, we will prove that there exists a unique boundary map $B \to \ch(\bdinf)$. First, we prove that such a map exists. 
	
	\begin{ethm}\label{thm bd map immeuble}
		Let $X$ be a thick building of type $\tilde{A}_2$ and let $G$ be a locally compact second countable group acting continuously by isometries on $X$ without invariant flats nor fixed point at infinity. Assume that $G$ preserves the types of the vertices of the building at infinity $\bdinf$. Let $\mu $ be an admissible symmetric measure on $G$, and let $(B, \nu)$ be the Poisson boundary associated to $(G, \mu)$. Then there exists a measurable map $B \rightarrow \ch (\bdinf)$ which is $G$-equivariant. 
	\end{ethm}

	\subsection{A trichotomy for the boundary map}
	Thanks to Theorem \ref{thm furst map}, there exists a Furstenberg map $\psi : B \to \bd X$. The purpose of this section is to investigate the characteristics of this map, using the nice ergodic properties of $B$ and the geometric structure of $\bd X$. If almost surely, $\psi(b)$ is in a chamber of the building at infinity, we get a $G$-map $B \to \ch (\bdinf)$, which is what we want. This leads us to study the case where the target space is the set of vertices at infinity. 
	
	\begin{elem}\label{disjonction}
		Consider the same assumptions as in Theorem \ref{thm bd map immeuble}, and assume that there exists a $G$-map $B \rightarrow \bdinf_1$. Then there is either a $G$-map $B \rightarrow T_1$, or a $G$-map $B \rightarrow \dt_1$, or a $G$-map $B \rightarrow \partial_2 T_1$. 
	\end{elem}
	\begin{proof}
		Let $\phi$ be a $G$-map $\phi: B \rightarrow \bdinf_1$, then by Lemma \ref{lem t1 measurable structure}, the panel trees $\{T_{\phi(b)}\}_{b \in B}$ endowed with their natural metrics form a measurable metric field over $B$. Moreover, there is a natural $G$-action on $T_1$ over $\bdinf_1 $. Each individual $T_u$ is $\cat$(0), complete and of finite telescopic dimension, so we can apply the measurable equivariant Theorem \ref{thm meas adams ballmann} to the metric field $T_1$ over $\bdinf_1$. Therefore there either there exists an invariant section of the boundary field $\partial T_1$, or there is an invariant Euclidean subfield of $T_1$. Note that by Lemma \ref{lem metric field fiber action}, an invariant section of the boundary field gives a measurable $G$-map
		\begin{eqnarray}
			B \rightarrow \dt_1 \nonumber.
		\end{eqnarray} 
		Otherwise, there is an invariant Euclidean subfield. If this subfield is of dimension $0$ (the subfield consists of points), it gives a $G$-map $B \rightarrow T_1$, and if this subfield is of dimension $1$ (it consists of lines), it gives a $G$-map $B \rightarrow \partial_2 T_1$. 
	\end{proof}

	Now our goal is to rule out two of the three previous possibilities. 
	\begin{elem}\label{no T1}
		Under the same assumptions as in Theorem \ref{thm bd map immeuble}, there is no $G$-map $B \rightarrow T_1$. 
	\end{elem}
	
	\begin{proof}
		Let us assume that there is a $G$-map $\psi : B \rightarrow T_1$. Consider the natural projection $p : T_1 \rightarrow \bdinf_1$, and denote $\psi_1 = p \circ \psi : B \rightarrow \bdinf_1$. By Lemma \ref{lem t1 measurable structure}, $p$ is measurable and $G$-equivariant, hence by composition so is $\psi_1$. Take $u,v \in \bdinf_1$ two vertices at infinity of type $1$. Recall that for the boundary of an $\tilde{A}_2 $-building, if $u\neq v$ are both of the same type, the projection $\proj_u (v) $ is a chamber in the residue $\res(u)$. This map is continuous in $v$, and measurable in $u$. Thanks to Proposition \ref{bijection ends chambers}, we can identify the chamber $\proj_u (v) $ to an endpoint of $\dt_u$. 
		
		Observe that because $\psi_1$ is $G$-equivariant, the set $\{ (b, b') \; | \; \psi_1 (b) = \psi_1(b') \}$ is $G$-invariant. Since the action of $G$ on $B \times B$ is doubly ergodic, this set is either null or conull. If $\nu \otimes \nu$ almost surely, $\psi_1 (b) = \psi_1(b')$, then varying $b$ and $b'$ separately gives that $\psi_1$ is almost surely constant. But since $\psi_1$ is $G$-equivariant, it means that there is a $G$-fixed point in $\bd X$, which is forbidden by non-elementarity of the $G$-action. As a consequence, $\psi_1(b) \neq \psi_1(b') $ almost surely.
		
		This discussion allows us to define the map:
		\begin{eqnarray}
			\phi: & B\times B &\longrightarrow \dt_1 \nonumber \\
			& (b, b') & \longmapsto (\psi_1(b), \proj_{\psi_1(b)} (\psi_1(b'))) \in \dt_{\psi_1(b)} \nonumber. 
		\end{eqnarray} 
		This application is measurable by composition. As in Lemma \ref{lem t1 measurable structure}, the group $G$ acts on $\dt_1$ via its implicit action on the chambers, using the bijection given by Proposition \ref{bijection ends chambers}. With this action, $\phi$ is $G$-equivariant. Now consider
		\begin{eqnarray}
			\Phi : & B\times B &\longrightarrow T_1 \times \dt_1 \nonumber \\
			& (b, b') & \longmapsto ( \psi(b), \phi(b,b')) \in T_{\psi_1(b)} \times \dt_{\psi_1(b)} \nonumber, 
		\end{eqnarray}
		where  $ T_1 \times\dt_1 $ is the metric field over $T_1$ defined in Lemma \ref{lem dt1 measurable structure}. 
		
		The application $\Phi$ is measurable because $\psi$ and $\phi $ are measurable. The group $G$ acts on $T_1 \times \dt_1$ via the diagonal action $g \cdot (u, \xi, C) = (gu, g\xi, gC)\in T_{gu} \times \dt_{gu}$. Thus, for $g \in G$, and almost every $(b, b') \in B \times B$,
		\begin{eqnarray}
			\Phi(g b , g b')  & = & (\psi (gb ), \proj_{\psi_1(gb)} (\psi_1(gb'))) \nonumber \\
			& = & (g \psi(b ), \proj_{g \psi_1(b)} (g\psi_1 (b'))) \text{ because $\psi $ and $\psi_1$ are $G$-equivariant} \nonumber \\
			& = & (g \psi (b ), g \proj_{\psi_1(b)} (\psi_1(b'))) \nonumber. 
		\end{eqnarray}
		
		As a consequence, $\Phi$ is a $G$-map. We then have the following commutative diagram:
		\[
		\begin{tikzcd}
			B \times B \arrow[r, "\Phi"] \arrow[d, "\pi_B"] 
			& T_1 \times \dt_1 \arrow[d, "\pi_1"] \\
			B \arrow[r, "\psi"]
			& T_1	
		\end{tikzcd}
		\]
		
		By Lemma \ref{lem dt1 measurable structure}, $T_1 \times \dt_1 \rightarrow T_1$ admits a fiberwise isometric $G$-action. But $B$ is a $G$-boundary by Theorem \ref{thm paire de bords}, hence the projection on the first factor $\pi_B : (b, b') \in B \times B \mapsto b \in B$ is relatively isometrically ergodic. Therefore, there exists an invariant section $s : B \longrightarrow T_1 \times \dt_1$ such that $\nu \times \nu $-almost surely, $\Phi (b, b') = s \circ \pi_B(b,b') = s(b)$. 
		\[
		\begin{tikzcd}
			B \times B \arrow[r, "\Phi"] \arrow[d, "\pi_B"] 
			& T_1 \times \dt_1 \arrow[d, "\pi_1"] \\
			B \arrow[r, "\psi"] \arrow[ur, dashed,"s"]
			& T_1	
		\end{tikzcd}
		\]
		
		In particular, the projection $\phi(b,b')= \proj_{\psi_1(b)}(\psi_1(b'))$ does not depend on $b'$. We then have an essentially well defined measurable map $u : B \rightarrow \bdinf_2$ such that for almost every $b \in B$, $u(b) \in \bdinf_2$ is the unique vertex of type 2 belonging to the chamber $\proj_{\psi_1(b)}(\psi_1(b'))$. Then almost surely, $\psi_1(b')$ belongs to a chamber in $\res(u(b))$, see Figure \ref{figure no t1}. 
		
		In other words, there exists a subset $B_0 \times B_1$ of $B \times B$ of full $\nu \times \nu$ measure such that for all $(b,b') \in B_0 \times B_1$, $\psi_1(b')$ belongs to a chamber in $\res(u(b))$. But for all $b \in B_0$, $\res(u(b))$ is a subset of $\bd X$ of circumradius $\leq \pi/2$. The application $\psi_1 $ is then  a $G$-map from $B_1$ to a subset of $\bd X$ of circumradius $\leq \pi/2$, so by Proposition \ref{circumcenter}, $G$ fixes a point in $\bd X$, in contradiction with the assumption that the action is non-elementary.  
		
	\end{proof}
	
	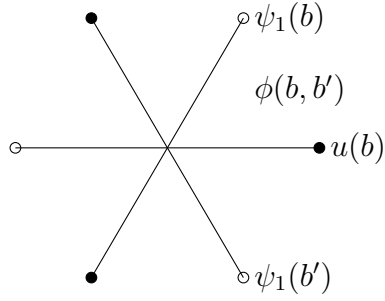
\begin{figure}[h]
		\centering
		\begin{center}
			\begin{tikzpicture}[scale=2]
				\draw (-1,0) -- (1,0)  ;
				\draw (-0.5, -0.86) -- (0.5, 0.86)  ;
				\draw (-0.5, 0.86) -- (0.5, -0.86)  ;
				\draw (0.5, 0.86) circle(1pt) ;
				\draw (0.5, -0.86) circle(1pt) ;
				\draw (-1,0) circle(1pt) ;
				\filldraw (-0.5, 0.86) circle(1pt) ;
				\filldraw (-0.5, -0.86) circle(1pt) ;
				\filldraw (1,0) circle(1pt) ;
				\draw (1,0) node[right]{$u(b)$} ;
				\draw (0.5,0.38) node[right]{$\phi(b,b')$} ;
				\draw (0.5, 0.86) node[right]{$\psi_1(b)$} ;
				\draw (0.5, -0.86) node[right]{$\psi_1(b')$} ;				
			\end{tikzpicture}
		\end{center}
		\caption{An apartment of the building at infinity $\bdinf$, in the proof of Lemma \ref{no T1}.}\label{figure no t1}
	\end{figure}
	
	It remains to prove that there is no $G$-map $B \rightarrow \partial_2 T_1$. 
	
	\begin{elem}\label{no d2T1}
		Under the same assumptions as in Theorem \ref{thm bd map immeuble}, there is no $G$-map $B \rightarrow \partial_2T_1$. 
	\end{elem}
	
	\begin{proof}
		Assume that there exists a measurable $G$-map $\psi : B  \longrightarrow \partial_2 T_1$. Denote by $p : \partial_2 T_1 \rightarrow \bdinf_1$ the natural projection, and $p \circ \psi $ by $\psi_1$. The application $\psi_1 : B \rightarrow \bdinf_1$ is measurable by composition, and $G$-equivariant. 
		
		Again, we define  the application $\phi$ on $B \times B$ by 
		\begin{eqnarray}
			\phi : (b,b') \in B\times B \mapsto \proj_{\psi_1(b)} (\psi_1(b')) \in \dt_{\psi_1(b)} \nonumber.
		\end{eqnarray}
		
		As before, the application $\phi$ is $G$-equivariant, and measurable. Observe that the set $\{(b,b') \in B\times B' \,  | \, \phi(b,b') \in \psi (b)\}$ is $G$-invariant. By double ergodicity of $G \curvearrowright B $, Theorem \ref{thm double ergod}, its measure is full or null. Assume that almost surely, $\phi(b,b') \notin \psi(b)$. Then we get a measurable $G$-map from $B \times B$ to the space of distinct triples of points of $\dt_1$:
		\begin{eqnarray}
			\Psi :&&  	B \times B  \longrightarrow  \partial_3 T_1  \nonumber \\
			&&(b,b')  \longmapsto  (\psi(b), \phi(b,b')) \in \partial_3 T_{\psi_1(b)}. \nonumber
		\end{eqnarray}
		Now consider the barycenter map	
		\begin{eqnarray}
			\bary : && \partial_3 T_1  \longrightarrow T_1 \nonumber \\
			&&(C_1, C_2, C_3) \in \partial_3 T_u  \longmapsto  \bary_u(C_1, C_2, C_3) \in T_u \nonumber, 
		\end{eqnarray}
		By Proposition \ref{prop bary panel trees}, the barycenter map just defined is measurable and $G$-equivariant, hence $\bary \circ \Psi $ is a measurable $G$-map. It implies that the following diagram
		\[
		\begin{tikzcd}
			B \times B \arrow[r, "\bary \circ \Psi"] \arrow[d, "\pi_B"] 
			& T_1 \arrow[d, "\pi_1"] \\
			B \arrow[r, "\psi_1"]
			& \bdinf_1	
		\end{tikzcd}
		\]
		is commutative. But $T_1 \rightarrow \bdinf_1$ admits a fiberwise isometric $G$-action by Lemma \ref{lem t1 measurable structure}. Therefore, by relative isometric ergodicity of the Poisson boundary, there exists an invariant section $s : B \longrightarrow T_1$, which contradicts Lemma \ref{no T1}.
		
		As a consequence, we have that almost surely, $\phi(b,b') \in\psi(b)$. This means that almost surely, $\phi(b,b')$ does not depend on $b'$. But the image of $\phi(b,b')$ lies in $\res(\psi_1(b))$, which is a subset of $\bd X $ of circumradius $\leq \pi/2$. Now by the same arguments as for Lemma \ref{no T1}, we get a $G$-fixed point in the boundary and a contradiction. 
	\end{proof}
	
	It is clear that Lemmas \ref{no T1} and \ref{no d2T1} remain true for the vertices of type 2: there is no $G$-map $B \rightarrow T_2$ nor $G$-maps $B \rightarrow \partial_2 T_2$. 
	\newline 
	
	We can now prove that there exists a measurable $G$-map from $B$ to the set of chambers at infinity. 
	
	\begin{proof}[Proof of Theorem \ref{thm bd map immeuble}]
		By assumption, $G$ acts continuously without invariant flats on the affine building $X$, and $X$ is a complete $\cat$(0) space of finite telescopic dimension. We can then apply Theorem \ref{furst map}: there exists a measurable $G$-map $\psi : B \rightarrow \bd X$.
		
		Observe that the set $\{b \in B \; | \; \psi (b) \in \bd X - (\bdinf_1 \sqcup \bdinf_2)\}$ is measurable, and $G$-invariant. By ergodicity of $G \curvearrowright B $, its measure is null or conull. If for almost every $b \in B$, $\phi (b) $ is a regular point, then we get a $G$-map $\psi : B \rightarrow \ch (\bdinf)$, which is what we want. 
		Otherwise, by ergodicity of the action $G \curvearrowright B$, we get either a map $ \psi : B \rightarrow \bdinf_1$ or $ \psi : B \rightarrow \bdinf_2$. Assume the former, the argument being identical in both cases. 
		
		We can apply Lemma \ref{disjonction}: there is either a $G$-map $B \rightarrow T_1$, or a $G$-map $B \rightarrow \dt_1$, or a $G$-map $B \rightarrow \partial_2 T_1$. Along with Lemmas \ref{no T1} and \ref{no d2T1}, there is a $G$-equivariant measurable map $B \rightarrow \dt_1$. But by Proposition \ref{bijection ends chambers}, the chambers are in bijection with $\dt_1$, hence this gives a $G$-equivariant map $ B \rightarrow \ch (\bdinf)$. 
	\end{proof}
	
	\subsection{Uniqueness of the boundary map}
	What was done in the previous section does not rely on the particular structure of the Poisson-Furstenberg boundary. In fact, for any $G$-boundary $(\tilde{B}, \tilde{\nu})$, the results of Section \ref{section bd map immeuble} hold. 
	
	In this section, we prove that the measurable map given by Theorem \ref{thm bd map immeuble} is actually unique. 
	
	\begin{ethm}\label{thm uniqueness chamber map}
		Let $X$ be a thick building of type $\tilde{A}_2$ and let $G$ be a locally compact second countable group acting continuously by isometries on $X$ without invariant flats. Let $(B, \nu_B)$ be a $G$-boundary. Then there exists a unique measurable $G$-equivariant boundary map $\psi : B \longrightarrow \ch(\bdinf)$. 
	\end{ethm}
	\begin{proof}
		The existence of such a boundary map $\psi $ is given by Theorem \ref{thm bd map immeuble}. Recall that for each chamber $C \in \ch({\bdinf})$, $C$ contains a unique vertex of type 1 and a unique vertex of type 2. Hence the boundary map $\psi : B \longrightarrow \ch(\bdinf)$ induces two $G$-maps $\psi_1 : B \longrightarrow \bdinf_1$ and $\psi_2 : B \longrightarrow \bdinf_2$. It is then sufficient to prove that those are the only $G$-maps $ B \longrightarrow \bdinf_1$ and $ B \longrightarrow \bdinf_2 $. The argument being the same in both cases, we only prove that $\psi_1: B \longrightarrow \bdinf_1$ is unique. Assume that we have another $G$-map $\psi_1' : B  \longrightarrow \bdinf_1$. By ergodicity, $\psi_1(b) \neq \psi'_1 (b)$ almost surely. Since $\psi_1(b)$ and $\psi'_1(b)$ are of type 1, there is a unique vertex of type 2 adjacent to both. Call this vertex $u(b)$. Now $\psi_1(b) $ and $u(b)$ belong to a unique common chamber $C_1(b)$, and the same goes for $\psi_2(b) $ and $u(b)$, whose common chamber we denote by $C_2(b)$. Eventually, using Proposition \ref{bijection ends chambers}, this yields a $G$-map 
		\begin{eqnarray}
			B  &\longrightarrow & \partial_2 T_2 \nonumber \\
			b & \longmapsto &(u(b), C_1(b), C_2(b)) \nonumber, 
		\end{eqnarray}
		which is in contradiction with Lemma \ref{no d2T1}. Hence $\psi_1$ is unique, proving the result. 
	\end{proof}
	
	\section{Stationary measures on the chambers at infinity}\label{section uniq stat meas A2}
	\subsection{Uniqueness of the stationary measure}
	The main result of this section is the uniqueness of the stationary measure $\nu$ on the set of chambers $\ch(\bdinf)$ of the spherical building at infinity of $X$. 
	
	\begin{ethm}\label{thm uniqueness stationary measure immeuble}
		Let $X$ be a building of type $\tilde{A}_2$ and let $G$ be a locally compact second countable group acting continuously by isometries on $X$ without fixed points nor invariant flats. Let $\mu $ be an admissible symmetric measure on $G$. Then there is a unique $\mu$-stationary measure $\nu$ on  $\ch(\bdinf)$. Moreover, we have the decomposition 
		\begin{eqnarray}
			\nu = \int_{b \in B} \delta_{\psi(b)} d\nu_B(b) \nonumber, 
		\end{eqnarray}
		where $\psi$ is the unique boundary map given by Theorem \ref{thm uniqueness chamber map}.
	\end{ethm}
	
	Before proving this theorem, recall the following classical result, which goes back to H.~Furstenberg \cite[Lemma 1.33]{furstenberg73}. It was described in Theorem \ref{thm bdry map = stat measure}

	\begin{elem}\label{furst}
		Let $B(G, \mu)$ be the Poisson-Furstenberg boundary associated to the probability measure $\mu$ on $G$, and let $M$ be a Lebesgue $G$-space. Then for any $\mu$-stationary measure $\nu \in \prob(M)$, there is a $G$-equivariant map $ \phi : B \rightarrow \prob(M)$ such that $\nu = \int \phi (b) d \mathbb{P}(b)$. Conversely, if $\phi : B \rightarrow \prob(M)$ is measurable and $G$-equivariant, then the measure $\nu = \int \phi (b) d \mathbb{P}(b)$ is $\mu$-stationary. 
	\end{elem}

	\begin{proof}[Proof of Theorem \ref{thm uniqueness stationary measure immeuble}]
		By Lemma \ref{furst}, it is enough to prove that there is a unique $G$-equivariant map $ B \longrightarrow \prob(\ch(\bdinf))$. Let $\psi : B \longrightarrow \ch (\bdinf)$ be the measurable $G$-equivariant map given by Theorem \ref{thm bd map immeuble}. Observe that the map 
		\begin{eqnarray}
			b \in B \mapsto \delta_{\psi(b)} \in \prob(\ch(\bdinf)) \nonumber
		\end{eqnarray}
		is measurable and $G$-equivariant. Let $\phi : B \rightarrow \prob(\ch(\bdinf))$ be any $G$-equivariant map. The goal is then to show that almost surely, $\phi $ is given by $\phi(b) = \delta_{\psi(b)}$. 
		
		Recall that each chamber in $\bdinf$ contains exactly one vertex of type 1 and one vertex of type 2, so denote by $\pi_1 : \ch(\bdinf) \longrightarrow \bdinf_1 $ and $\pi_2 : \ch(\bdinf) \rightarrow \bdinf_2 $ the corresponding maps. Let $\psi_i: B \rightarrow \bdinf_i$ be defined by $\psi_i= \pi_i \circ \psi : B \rightarrow \bdinf_i$, for $i = 1, 2 $. By composition, $\psi_i$ is a measurable, $G$-equivariant map. Consider also the map 
		\begin{eqnarray}
			\phi_1 : b \in B \mapsto  ({\proj_{\psi_1(b)}})_\ast \phi(b) \in \prob(\ch(\bdinf)) \nonumber.
		\end{eqnarray} 
		The map $\phi_1$ is a measurable $G$-equivariant map between $B$ and the probability measures on the chambers, but now the support of $\phi_1(b)$ is contained in $\res({\psi_1(b)})$. Now identify $\res(\psi_1(b))$ with the ends $\dt_{\psi_1(b)}$ of the panel tree at $\psi_1(b)$ given by Proposition \ref{bijection ends chambers}. Therefore, $\phi_1$ can be written as 
		\begin{eqnarray}
			\phi_1 : B &\rightarrow &\prob(\dt_1) \nonumber \\
			b & \mapsto &\phi_1(b) \in \prob(\dt_{\psi_1(b)}). \nonumber 
		\end{eqnarray} 
		
		By ergodicity of the Poisson boundary, the cardinal $k(b)$ of the support of $\phi_1(b)$ is almost surely constant. If almost surely $k(b) \geq 3$, then the barycenter map constructed in Proposition \ref{prop bary proba 3 T1} associates to the support of $\phi(b)$ a canonical point in $T_{\psi_1(b)}$. This gives a measurable $G$-equivariant map $B \longrightarrow T_1$, which is impossible due to Lemma \ref{no T1}. 
		\newline 
		
		If the support of the measure $\phi_1(b)$ is almost surely of cardinal $2$, then we have a $G$-map $B \longrightarrow \partial_2 T_1$, which is in contradiction with Lemma \ref{no d2T1}. 
		\newline 
		
		Then $\phi_1(b)$ has to be a Dirac mass, which we denote by $\delta_{\xi(\omega) }$ for $(\psi_1 (\omega), \xi (\omega )) \in \dt_{\psi_1(\omega)}$. Again, by Proposition \ref{prop homeo panel tree ch}, this gives a $G$-map $B \longrightarrow \ch(\bdinf)$. But by Theorem \ref{thm uniqueness chamber map}, any such map is unique and equal to $\psi$. Therefore, $\phi_1(b)$ is the Dirac mass at $\psi(b)$. 
		
		What we obtained so far is that for any $C \in \ch(\bdinf)$, 
		\begin{eqnarray}\label{eq uniq stat measure A2}
			\phi_1(b) (\proj_{\psi_1(b)} (C)) &=& \delta_{\psi(b)} (\proj_{\psi_1(b)} (C)).
		\end{eqnarray}
		By definition, $ \phi_1(b)  = ({\proj_{\psi_1(b)}})_\ast \phi(b) $. Then by Equation \eqref{eq uniq stat measure A2}, the support of $\phi(b)$ is contained in the set of chambers $C \in \ch(\bdinf)$ such that $ \proj_{\psi_1(b)} (C) = \psi(b) $. 
		
		By repeating the same argument as before for $\psi_2 : B \rightarrow \bdinf_2$, we obtain that $\phi(b) $ must also be supported on the set of chambers $C \in \ch(\bdinf)$ such that $ \proj_{\psi_2(b)} (C) = \psi(b) $. Combining these two results, we get that $\phi(b)$-almost surely, $\proj_{\psi_1(b)} (C) = \proj_{\psi_2(b)} (C) = \psi(b)$, hence $\phi(b)$ is supported on the chamber $\psi(b)$, meaning that $\phi(b)$ is essentially equal to $\delta_{\psi(b)}$. This proves the uniqueness of $\psi$, therefore of the stationary measure. 
	\end{proof}
	
	\subsection{Opposite chambers at infinity}
	The following is an important feature of the stationary measure $\nu$. Recall that two chambers $C, C'$ in $\bdinf$ are opposite if they are at maximal gallery distance, i.e. if $d_\Gamma (C, C') = 3$. Since apartments are convex, the convex hull of two chambers is contained in any apartment containing both of them. As a consequence, two opposite chambers are contained in a unique apartment. Denote by $\nu$ the unique $\mu$-stationary measure on $\ch(\bdinf)$ given by Theorem \ref{thm uniqueness stationary measure immeuble}. 
	
	\begin{eprop}\label{prop proba opposite chambers}
		Then $\nu \otimes \nu$-almost every pair of chambers in $\ch(\bdinf)$ are opposite. 
	\end{eprop}
	
	\begin{proof}
		Check that due to Lemma \ref{lem non atomic}, $\nu$ is non-atomic. Consider the $G$-equivariant map $\psi : B \rightarrow \ch(\bdinf)$ given by Theorem \ref{thm bd map immeuble}, and the following measurable map : 
		\begin{eqnarray}
			\phi : & B\times B &\longrightarrow W \nonumber \\
			& (b, b') & \longmapsto \delta(\psi(b), \psi(b')) \nonumber, 
		\end{eqnarray}
		where $\delta$ is the Weyl distance function associated to the Coxeter complex $(W, S)$ of type $A_2 $. Then by $G$-equivariance of $\psi$, $\phi$ is a $G$-invariant measurable map. By double ergodicity of the Poisson boundary, $\phi $ is then essentially constant. 
		
		If its essential value is $e$, then almost surely, $\psi(b) = \psi(b')$, and there is a $G$-fixed chamber on the boundary, which is impossible because the action is non-elementary. 
		
		Let us denote by $l_S$ the word metric on $W$ given by the set $S$. From now on, let $B_0, B_1 \subseteq B$ be conull sets such that for all $(b,b') \in B_0 \times B_1$, $\phi(b,b') $ is the essential value of $\phi$. If almost surely, $l_S(\phi(b,b'))= 1 $, then almost surely $\psi(b)$ and $\psi(b')$ are adjacent. Since $\phi$ is constant, then almost surely $\psi(b)$ and $\psi(b')$ contain a vertex of constant type. Let $b \in B_0$, denote by $u(b) \in \psi(b)$ the unique vertex of type $1$. Then for all $b' \in B_1$, $\psi(b') \in \res(u(b))$. Again, because the action is non-elementary, this is impossible by Proposition \ref{circumcenter} since $\res(u(b))$ has radius $\leq \pi/2$. 
		
		If $l_S(\phi)$ is almost surely equal to 2. Then the gallery between $\psi(b)$ and $\psi(b')$ is of constant type $(1,2)$ or $(2,1)$. But this is impossible because $\delta(\psi(b'), \psi(b))= \delta(\psi(b), \psi(b'))^{-1}$, so by ergodicity, we have that $\delta(\psi(b), \psi(b'))= \delta(\psi(b), \psi(b'))^{-1}$. 
		
		Consequently, $\psi(b)$ and $\psi(b')$ are almost surely opposite. The Proposition then follows from the decomposition 
		
		\begin{eqnarray}
			\nu = \int_{b \in B} \delta_{\psi(b)} d\mathbb{P}(b) \nonumber 
		\end{eqnarray}
		given by Theorem \ref{thm uniqueness stationary measure immeuble}. 
	\end{proof}

	\section{Hyperbolic elements}\label{section hyp eleme A2}
	
	In this section, we present the results obtained in a joint work with J.~Lécureux and J.~Schillewaert. We recall that a hyperbolic element is a semi-simple axial isometry. The main result is the following.  
	
	\begin{ethm}[{Existence of hyperbolic elements}]\label{mainthm}
		Let $G$ be a group with a non-elementary isometric action on a metrically complete locally finite $\tilde{A}_2$-building $X$. Then there exists a hyperbolic element in $G$.
	\end{ethm}
	
	From this we deduce a fixed point theorem if the group is finitely generated. 
	
	\begin{ethm}[{Fixed point}]\label{local-global}
		Let $G$ be a finitely generated group acting on a simplicial building of type $\tilde{A}_2$. Assume that all elements of $G$ are elliptic. Then $G$ fixes a point of $X$.
	\end{ethm}
	
	\subsection{Convergence to the boundary}
	
	Let $X$ be a building of type $\tilde{A}_2$ and let $G$ be a discrete countable group acting continuously by isometries on $X$ without fixed points nor invariant flats. Our approach for proving Theorem \ref{mainthm} involves the study of a random walk $(Z_n o)$ generated by some well-chosen measure $\mu$ on $G$. More precisely, we can take $\mu \in \prob(G)$ a symmetric admissible probability measure, and assume that $\mu$ has finite second moment: 
	\begin{eqnarray}
		\sum_{g \in G} d(g o, o)^2 \mu(g) < \infty \nonumber. 
	\end{eqnarray}
	
	We are going to use the following theorem from H. Izeki. 
	\begin{ethm}[{\cite[Theorem A]{izeki22}}]\label{thm izeki}
		Let $X$ be a complete CAT(0) space which is either proper or of finite
		telescopic dimension, and let $G$ be a discrete countable group equipped with a symmetric
		and admissible probability measure $\mu$ with finite second moment. Assume that $G \curvearrowright X $ is a non-elementary action. Then the drift of the random walk is positive:
		\begin{eqnarray}
			l_X(\mu) > 0 \nonumber. 
		\end{eqnarray}
	\end{ethm}
	
	Using this result, we can prove that almost surely, the random walk converges to the boundary. We recall that in our notations, $(\Omega,\mathbb{P}) = (G^\mathbb{N}, \mu^{\otimes \mathbb{N}})$ is the increment space on $G$, endowed with the product $\sigma$-algebra. The group $G$ acts on $\Omega$ in the following way: for $g \in G$ and $\omega \in \Omega$, 
	\begin{eqnarray}
		g. (\omega_1, \omega_2, \dots) & =& (g\omega_1, \omega_2, \dots). \nonumber
	\end{eqnarray}
	Also, $\mathbb{N}$ acts on $\Omega$ by the transformation map:
	\begin{eqnarray}
		T : \omega  = (\omega_1, \omega_2, \dots) \in \Omega & \mapsto & (\omega_1\omega_2, \omega_3, \dots) \in\Omega \nonumber
	\end{eqnarray}
	Both the $G$-action and the $T$-action are on $\Omega$ are non-singular, and they commute. Last, recall that if we denote by $(B, \nu_B)$ the Poisson-Furstenberg boundary associated to $\mu$, we have a canonical measurable map 
	\begin{eqnarray}
		\bnd : \Omega \rightarrow B \nonumber
	\end{eqnarray}
	such that $\bnd_\ast \mathbb{P} = \nu_B$.

	\begin{ethm}[{Convergence to regular points}]\label{thm cv reg A2}
		Let $X$ be a locally finite building of type $\tilde{A}_2$ and let $G$ be discrete countable group acting continuously by isometries on $X$ without fixed points nor invariant flats. Let $\mu$ be a symmetric and admissible probability measure $\mu$ with finite second moment. Then the random walk $(Z_n(\omega)o)$ generated by $\mu$ converges almost surely to a regular point of the boundary $\xi^+(\omega) \in \bd X$. 
	\end{ethm}
	In the course of the proof, we will use the following convergence-type property, which was proven by Papasoglu and Swenson. For $\xi $ a point in $\bd X$, we write $B_T(\xi, \theta)$ the closed ball of radius $\theta$ around $\xi$ for the Tits metric.
	\begin{ethm}[{\cite[Lemma 19]{papasoglu_swenson09}}]\label{thm PS}
		Let $X$ be a proper $\cat (0)$ space, and $G \curvearrowright X$ an action by isometries. Let $x \in X$, $\theta \in [0, \pi]$ and $(g_n) \subseteq G$ be a sequence of isometries for which there exists $x \in X $ such that $g_n(x) \rightarrow \xi \in \bd X $ and $ g^{-1}_n(x) \rightarrow \eta \in \bd X$. Then for all compact subset $K \subseteq \bd X -  B_T(\eta, \theta)$ and for all open subset $U$ such that $B_T(\xi, \pi - \theta)\subseteq U$, there exists $n_0$ such that for all $n \geq n_0$, $g_n(K) \subseteq U$. 
	\end{ethm}
	\begin{proof}[Proof of Theorem \ref{thm cv reg A2}.]
		Due to Theorem \ref{thm izeki}, the random walk has positive drift $l_X(\mu) > 0$. By Karlsson and Margulis' Theorem \cite[Theorem 2.1]{karlsson_margulis}, the random walk $(Z_n o)$ converges almost surely to a point of the boundary $\xi (\omega)$. It remains to prove that $\xi (\omega)$ is almost surely a regular point. 
		
		Let $o \in X$ be a vertex. For every $C \in \ch(\bdinf)$, there is a unique Weyl chamber $Q(o, C)$ based at $o$ representing $C$. Recall that $Q(o, C)$ is isometric to the Euclidean cone $\mathcal{C}$ with angle $\pi/3$ based at 0 of the plane $\R^2$ (with the usual metric), as defined by
		\begin{eqnarray}
			\mathcal{C} := \{(x,y) \in \R^2_+ \, | \, y \leq \sqrt{3}x  \} \nonumber.
		\end{eqnarray} 
		In $\mathcal{C}$, define the geodesic ray $\bar{\gamma}(t) = (\frac{\sqrt{3}}{2}t, \frac{1}{2} t ) \in \R^2$ so that $\bar{\gamma}$ divides $\mathcal{C}$ in two isometric parts. In $Q(o, C)$, there is a unique geodesic ray corresponding to $\bar{\gamma}$ under the isometry $ \mathcal{C} \to Q(o, C)$, and we denote this geodesic ray by $\gamma$. The equivalence class of $\gamma$ gives a point of $C$, that we denote by $\text{mid}(C)$. From a geometric point of view,  $\text{mid}(C)$ is the middle of the chamber $ C$ (seen as a subset of $\bd X$) for the Tits metric. 
		
		Let $\nu$ be the unique $\mu$-stationary measure on the chambers at infinity $\ch(\bdinf)$. Define $\tnu$ as the image of $\nu$ under the map 
		\begin{eqnarray}
			 \text{mid} &:& \ch(\bdinf) \rightarrow \bd X \nonumber \\
			& & C \mapsto  \text{mid}(C), \nonumber
		\end{eqnarray}
		where  $\text{mid}(C)$ is the unique point of the chamber $C$ given by the construction above. 
		
		Note that since $\nu$ is $\mu$-stationary, Theorem \ref{thm mesures limites} implies that there exists a measurable map $ \omega \in \Omega \mapsto \nu_\omega \in \prob(\bd X) $ such that almost surely, 
		\begin{eqnarray}
			Z_n(\omega)_\ast \nu \rightarrow \nu_\omega \nonumber. 
		\end{eqnarray}
		Recall that by Theorem \ref{thm uniqueness chamber map}, we have that $ \nu_\omega = \delta_{\psi(\omega)}$, with $\psi(\omega)$ defined by Theorem~\ref{thm bd map immeuble}. 
		As the map $\text{mid}$ is $G$-equivariant and continuous, $\tnu$ is also $\mu$-stationary, and we immediately obtain that almost surely, 
		\begin{eqnarray}
			Z_n(\omega)_\ast \tnu \rightarrow \tnu_\omega \nonumber,
		\end{eqnarray}
		where 
		\begin{eqnarray}\label{eq nu tnu lim}
			\tnu_\omega = \text{mid}_\ast \nu_\omega= \delta_{\text{mid}_\ast\psi(\omega)}.
		\end{eqnarray} 
		
		Let $ \Omega' \subseteq \Omega $ such that for all $\omega \in \Omega'$, the random walk $Z_n(\omega) o $ converges to a point $\xi (\omega) \in \bd X$ of the boundary. For all $\omega \in \Omega' $, we can extract a subsequence $ \phi(n)$ such that $Z_{\phi(n)}(\omega)^{-1} o $ converges to a point $\eta  \in \bd X$. By Theorem \ref{thm PS}, we obtain that for all $\omega \in \Omega' $, and for all compact subset $K \subseteq \bd X -  B_T(\eta, \pi -\frac{\pi}{6} )$ and $U$ open subset containing $B_T(\xi(\omega), \pi/6)$, there exists $n_0$ such that for all $n \geq n_0$, $Z_{\phi(n)}(\omega)(K) \subseteq U$. 
		
		Now by Proposition \ref{prop proba opposite chambers},  $\nu \otimes \nu$-almost every pair of chambers in $\ch(\bdinf)$ are opposite. As a consequence, for $\tnu$-almost every $\eta \in \bd X$, $\tnu (B_T (\eta, \pi - \frac{\pi}{6}))=0$. 
		We then obtain that for all $\omega \in \Omega'$, 
		\begin{eqnarray}
			Z_{\phi(n)}(\omega)_\ast \tnu \rightarrow \delta_{\xi(\omega)} \nonumber. 
		\end{eqnarray}
		By uniqueness of the limit, $ \tnu_\omega = \delta_{\xi(\omega)}$. Combining this with Equation \eqref{eq nu tnu lim}, it means that $\xi(\omega )$ is a regular point of the boundary.
	\end{proof}
	
	\subsection{Local-to-global results}
	In this Section, we prove our local-to-global results.
	The following Lemma gives an easy criterium for an isometry to be hyperbolic in an affine building. 
	
	\begin{elem}\label{hyperbolic-creation}
		Suppose that there exists $g\in G$ and a point $o$ such that $go\neq o$ and  the segments $[g^{-1}o,o]$ and $[o,g(o)]$ are contained in  sectors $S_0$ and $S_1$ respectively, and $S_0$ and $S_1$ are based at $o$ and opposite at $o$. Then $g$ is hyperbolic.
	\end{elem}
	\begin{proof}
		First assume that the $g^{-1}o,o$ and $g(o)$ are contained in a line. Then we can repeatedly apply $g$ or $g^{-1}$ and we find that $g$ is hyperbolic and its translation axis. 
		
		If the segment $[o,go]$ was contained in a wall  then by applying $g^{-1}$ it is also the case of $[g^{-1}o,o]$. Hence $[g^{-1}o,o]$ is contained in one of the two walls of $S_0$. But since the action of $G$ on $X$ is type-preserving, there is only one of the two walls which is possible, and it is the wall which continues the segment $[o,go]$. Hence in this case we again get the result.
		
		So now let us assume that the segment $[o,go]$ is not contained in a wall, and therefore has a projection on $\res(o)$ which is exactly the chamber of $S_0$ containing $o$. 
		By \cite[Proposition 1.12]{parreau00} there exists an apartment $A$ containing $S_0$ and $S_1$. We thus find a sector $S'_0$ based at $g(o)$ which contains $S_0$ and thus $g^{-1}(o)$, $o$ and $g(o)$. Since $S_0$ and $S_1$ are opposite at $o$, the sectors $S''_0=gS_0$ and $S''_1=gS_1$ are opposite at $go$. Now the sectors $S'_0$ and $S''_0$ are both based at $go$ and contain the segment $[o,go]$. Since this segment is not contained in a wall, it follows that the projections of $S'_0$ and $S''_0$ on $\res(o)$ are the same, and therefore $S'_0$ and $S''_1$ are sectors based at $go$ which are opposite at $go$, and such that $S'_0$ contains $g^{-1}o,o,go$ and $S''_1$ contains $go,g^2 o$.

		Repeating this argument we obtain by induction that for all $n\geq 0$ there exists an apartment which contains $(g^io), -1\leq i \leq n$. More precisely, for every $n$  there exists opposite sectors $S_n$ and $S'_n$ which are based at $g^{n-1} o$ and such that $g^{n} o \in S'_n$, and $g^i o \in S_n$ for $-1\leq i\leq n-1$.
		
		Applying the cosine rule to the geodesic triangle $\Delta(o,g^{n-1}o,g^no)$ and using that $g$ is an isometry yields 
		\begin{eqnarray}
			d^2(o,g^n o)=d^2(o,go)+d^2(o,g^{n-1}o)-2d(o,go)d(o,g^{n-1}o)\cos \theta, \nonumber
		\end{eqnarray} 
		where $\theta$ is the angle at $g^{n-1}o$ between $o$ and $g^n o$. Since $o\in S_n$ and $g^{n-1} o \in S'_n$, and $S_n$ and $S'_n$ are opposite at $g^{n-1} o$, we see that $\theta\geq 2\pi/3$, and therefore
		$\cos \theta\leq -\frac{1}{2}$. It follows that we have $d^2(o,g^no)\geq d^2(o,go)+d^2(o,g^{n-1}o)+d(o,go)d(o,g^{n-1}o)$.
		Thus $d(o,g^no)\to \infty$ for $n\to \infty$. Therefore $g$ is not elliptic, and by Parreau's Theorem \ref{thm isom ss immeuble affine}, $g$ is hyperbolic.
	\end{proof}
	
	\begin{erem}
		A.~Parreau suggested to us that in fact, if  $go\neq o$ and  the segments $[g^{-1}o,o]$ and $[o,g(o)]$ are contained in  sectors $S_0$ and $S_1$ respectively, with $S_0$ and $S_1$ based at $o$ and opposite at $o$, then $g^{-1}o$, $o$ and $go$ must be aligned. This is due to the fact that $[g^{-1}o,o]$ and $[o,go]$ have same type. This idea simplifies the previous proof.
	\end{erem}
	For $o\in X$, and $c$ a chamber in $\res(o)$, let $\Omega_o(c)$ be the subset of points $\eta\in \overline X$ such that the projection of $\eta$ on $\res(o)$ is $c$; that is, $c$ is in the convex hull of $o$ and $\eta$, see Section \ref{section panel trees}. Note that $\Omega_o(c)$ is a closed open subset of $\overline X$. 
	
	For the rest of this section, we let $\mu$ be a symmetric and admissible probability measure on $G$, with finite second moment, and we denote by $\nu$ the unique stationary measure on $\ch(\bdinf)$ defined by Theorem \ref{thm uniqueness stationary measure immeuble}. Similarly, we will denote by $\nui$ the unique stationary measure associated to $\mui = \iota_\ast \mu$, where $\iota(g) = g^{-1}$. Since $\mu$ is symmetric, we immediately have that $\nui = \nu$. As before, we denote by $(Z_n (\omega))$ the random walk generated by $\mu$ on $G$. 
	
	The following result is a standard application of \cite[Corollary 2.7]{benoist_quint}.

	\begin{elem}\label{lem:nu(omega)}
		For every chamber $c$ in $\res(o)$, 
		we have almost surely
		\begin{eqnarray}
			\lim_{n\to +\infty} \frac{1}{n} | \{ k\leq n\mid Z_k^{-1}o \in \Omega_o(c) \} | =\nu(\Omega_o(c)) \nonumber
		\end{eqnarray}
	\end{elem}
	
	\begin{proof}
		Let $\varphi$ be the characteristic function of $\Omega_o(c)$. As $\Omega_o(c)$ is closed and open, $\varphi$ is continuous. Let $\nui$ be the unique $\mui $-stationary measure on $\ch(\bdinf)$ given by Theorem \ref{thm uniqueness stationary measure immeuble}. Since $\mu$ is symmetric, we have that $\nui = \nu$. By \cite[Corollary 2.7]{benoist_quint} we obtain
		\begin{eqnarray}
			\lim_{n\to +\infty} \frac{1}{n}\sum_{k=1}^n \varphi(Z_k^{-1} o) = \nui(\varphi) = \nu(\varphi).  \nonumber
		\end{eqnarray}
		The result follows.
	\end{proof}
	
	If $C,C'$ are two chambers at infinity, and $o\in X$, then say that $C$ and $C'$ are \emph{opposite at $o$} if $\proj_o(C)$ and $\proj_o(C')$ are opposite chambers in $\res(o)$. Note that $C$ and $C'$ are opposite (in the usual sense) if and only if there exists an $o$ such that $C$ and $C'$ are opposite at $o$. Indeed, one can simply take for $o$ a vertex in an apartment containing $C$ and $C'$. We denote by $\Opp_o (C)$ the set of chambers at that are opposite to $C$ at $o$.

	
	%

	\begin{eprop}\label{prop:probaopposite}
		For almost every $\omega \in \Omega$, there exists $o$ such that almost surely, $Z_n (\omega)o$ and $Z_n^{-1} (\omega) o$ are in opposite sectors based at $o$ for infinitely many $n $. 
	\end{eprop}
	
	\begin{proof}

		Let $x\in X $ be a basepoint. We know by Theorem \ref{thm cv reg A2} that $(Z_n (\omega)x) $ converges almost surely to a regular point $\xi^+(\omega) $ in the boundary $\bd X$. Let $\Omega' \subseteq \Omega $ be a conull set such that, for all $\omega \in \Omega'$, $(Z_n(\omega)x)_n $ converges to a regular point of the boundary. For all $\omega \in \Omega'$, denote by $C(\omega)$ the unique chamber that contains $\xi^+(\omega)$.
		By Proposition \ref{prop proba opposite chambers}, two chambers at infinity are almost surely opposite. We then have that for $\nu$-almost every $C \in \ch (\bdinf)$, 
		\begin{eqnarray}
			\nu \left( \underset{o \in X^{(o)}}{\bigcup} \Opp_o (C)\right) = 1 \nonumber. 
		\end{eqnarray} 
		Because $X$ is locally finite, $X^{(o)}$ is countable. Then, for $\nu$-almost every $C \in \ch(\bdinf)$, there must exist a vertex $o \in X^{(o)} $ such that $\nu(\Opp_o (C))>0$.  
		In particular, there exists $\Omega''$ of full $\mathbb{P}$-measure such that for every $\omega \in \Omega''$, $(Z_n(\omega)x)_n $ converges to a regular point in the chamber $C(\omega)$ and there exists a vertex $o(\omega) $ for which $\nu(\Opp_o (C(\omega)))>0$.
		Since for all $\omega \in \Omega''$, $(Z_n (\omega )o(\omega))$ converges to a point in $C(\omega)$, the projection of $Z_n (\omega)o(\omega)$ on $\res(o(\omega))$ is a constant chamber $c(\omega)$ for $n$ sufficiently large. 
		By our choice of $o(\omega)$, there exists $c' := c'(\omega)$ opposite  to $c(\omega)$ such that $\nu(\Omega_o(c'))>0$. By Lemma \ref{lem:nu(omega)}, we have that 
		\begin{eqnarray}
			\lim_{n\to +\infty} \frac{1}{n} | \{ k\leq n\mid Z_k^{-1}o \in \Omega_o(c') \} | =\nu(\Omega_o(c')) > 0\nonumber.
		\end{eqnarray}
		In particular the cardinal of set $ | \{ n\mid Z_n^{-1} \in \Omega_o(c') \} |$ is almost surely unbounded. We proved that for almost every $\omega \in \Omega$, there exists $o \in X^{(o)}$ and opposite chambers $c$ and $c'$ such that for infinitely many $n$, the projection of $Z_n(\omega)o$ on $\res(o)$ is $c$ and the projection of $Z_n(\omega)^{-1}o$ on $\res(o)$ is $c'$. 
	\end{proof}

	We are now in ready to prove our main result.
	
	\begin{proof}[Proof of Theorem \ref{mainthm}]
		By Proposition \ref{prop:probaopposite}, for almost every $\omega \in \Omega$, there exists a vertex $o \in X$ and $n \in \mathbb{N}$ such that $Z_n (\omega)o$ and $Z^{-1}_n (\omega)o $ are in opposite sectors based at $o$. The result follows from Lemma \ref{hyperbolic-creation}. 
	\end{proof}
	
	\begin{erem}
		Using similar techniques, we believe we can prove that the proportion of hyperbolic isometries in the random walk goes to 1 almost surely as $n $ goes to infinity, in a similar fashion as Theorem \ref{thm prop contracting} for contracting elements in general $\cat$(0) spaces. 
	\end{erem}
	
	Using Theorem \ref{mainthm} we can derive our local-to-global fixed point result. First, we recall a fixed point property for trees due to Serre. 
	
	\begin{eprop}[{\cite[Corollary I.6.5.3]{serre80}}]\label{prop pt fixe serre}
		Let $G$ be a finitely generated group acting on a tree. If each element of $G$ has a fixed point, then $G$ has a global fixed point. 
	\end{eprop}
	
	\begin{proof}[Proof of Theorem \ref{local-global}]
		By Theorem \ref{mainthm}, we see that the action of $G$ on $X$ must be elementary. We proceed by induction on the dimension of $X$. If $X$ is a tree, then the result is given by Proposition \ref{prop pt fixe serre}. 
		
		By passing to a finite index subgroup, we can always assume that $G$ fixes a boundary vertex $\xi$ in $\partial_\infty X$, and therefore a facet $\sigma$ of the building at infinity. Hence $G$ acts on the corresponding panel tree $T_\xi$. The projection $\pi_\xi : X\to T_\xi$  is $G$-equivariant, so the action on $T_\xi$ is still elliptic.  By Proposition \ref{prop pt fixe serre}, $T_\xi$ has a $G$-fixed point, which we denote by $x_0$. 
		
		If $g\in G$, then it follows that $g$ fixes some point $x^g$ in $\pi_\sigma^{-1}(x_0)$. Hence it fixes a half-line $[x^g ,\xi)$, whose projection on $T_\xi$ is $x_0$. 
		
		Now let $g_1,\dots,g_n$ be a generating set for $G$. Then each $g_i$ fixes a half-line $[x^{g_i},\xi)$ which projects to $x_0$. Since all these lines have the same projection, it follows that any two of them intersect. Therefore there is a half-line contained in all the $[x^{g_i},\xi)$, which is fixed pointwise by all of the finitely many generators of $G$, and therefore by all of $G$.
	\end{proof}
	
	\selectlanguage{french}
	\addcontentsline{toc}{chapter}{Bibliographie}
	\small
	\bibliographystyle{alpha}
	\bibliography{bibliography}

\begin{thebibliography}{CFFT22}

\bibitem[AB98]{adams_ballmann98}
Scot Adams and Werner Ballmann.
\newblock Amenable isometry groups of {H}adamard spaces.
\newblock {\em Math. Ann.}, 312(1):183--195, 1998.

\bibitem[AB08]{abramenko_brown08}
Peter Abramenko and Kenneth~S. Brown.
\newblock {\em Buildings}, volume 248 of {\em Graduate Texts in Mathematics}.
\newblock Springer, New York, 2008.
\newblock Theory and applications.

\bibitem[AH14]{anderegg_henry14}
Martin Anderegg and Philippe Henry.
\newblock Actions of amenable equivalence relations on {$\rm CAT(0)$} fields.
\newblock {\em Ergodic Theory Dynam. Systems}, 34(1):21--54, 2014.

\bibitem[Bal95]{ballman95}
Werner Ballmann.
\newblock {\em Lectures on spaces of nonpositive curvature}, volume~25 of {\em
  DMV Seminar}.
\newblock Birkh\"{a}user Verlag, Basel, 1995.
\newblock With an appendix by Misha Brin.

\bibitem[BB95]{ballmann_brin95}
Werner Ballmann and Michael Brin.
\newblock Orbihedra of nonpositive curvature.
\newblock {\em Publications Math\'ematiques de l'IH\'ES}, 82:169--209, 1995.

\bibitem[BB00]{ballmann_brin00}
Werner Ballmann and Michael Brin.
\newblock Rank rigidity of {E}uclidean polyhedra.
\newblock {\em Amer. J. Math.}, 122(5):873--885, 2000.

\bibitem[BBE85]{ballman_brin_eberlein}
Werner Ballmann, Misha Brin, and Patrick Eberlein.
\newblock Structure of manifolds of nonpositive curvature 1.
\newblock {\em Annals of Mathematics}, 122(1):171--203, 1985.

\bibitem[BBF15]{bestvina_bromberg_fujiwara15}
Mladen Bestvina, Ken Bromberg, and Koji Fujiwara.
\newblock Constructing group actions on quasi-trees and applications to mapping
  class groups.
\newblock {\em Publ. Math. Inst. Hautes \'{E}tudes Sci.}, 122:1--64, 2015.

\bibitem[BBS85]{ballman_burns_spatzier}
Werner Ballmann, Misha Brin, and Ralf Spatzier.
\newblock Structure of manifolds of nonpositive curvature 2.
\newblock {\em Annals of Mathematics}, 122(2):205--235, 1985.

\bibitem[BCFS22]{bader_caprace_furman_sisto22}
Uri Bader, Pierre-Emmanuel Caprace, Alex Furman, and Alessandro Sisto.
\newblock Hyperbolic actions of higher-rank lattices come from rank-one
  factors, 2022.
\newblock arXiv:2206.06431.

\bibitem[BCL19]{bader_caprace_lecureux19}
Uri Bader, Pierre-Emmanuel Caprace, and Jean L\'{e}cureux.
\newblock On the linearity of lattices in affine buildings and ergodicity of
  the singular {C}artan flow.
\newblock {\em J. Amer. Math. Soc.}, 32(2):491--562, 2019.

\bibitem[BDL16]{bader_duchesne_lecureux16}
Uri Bader, Bruno Duchesne, and Jean L\'{e}cureux.
\newblock Furstenberg maps for {${\rm CAT}(0)$} targets of finite telescopic
  dimension.
\newblock {\em Ergodic Theory Dynam. Systems}, 36(6):1723--1742, 2016.

\bibitem[BE17]{erschler_bartholdi17}
Laurent Bartholdi and Anna Erschler.
\newblock Poisson-{F}urstenberg boundary and growth of groups.
\newblock {\em Probab. Theory Related Fields}, 168(1-2):347--372, 2017.

\bibitem[BF09]{bestvina_fujiwara09}
Mladen Bestvina and Koji Fujiwara.
\newblock A characterization of higher rank symmetric spaces via bounded
  cohomology.
\newblock {\em Geom. Funct. Anal.}, 19(1):11--40, 2009.

\bibitem[BF14]{bader_furman14}
Uri Bader and Alex Furman.
\newblock Boundaries, rigidity of representations, and {L}yapunov exponents.
\newblock In {\em Proceedings of the {I}nternational {C}ongress of
  {M}athematicians---{S}eoul 2014. {V}ol. {III}}. Kyung Moon Sa, Seoul, 2014.

\bibitem[BF22]{bader_furman22}
Uri Bader and Alex Furman.
\newblock An extension of {M}argulis's superrigidity theorem.
\newblock In {\em Dynamics, geometry, number theory---the impact of {M}argulis
  on modern mathematics}, pages 47--65. Univ. Chicago Press, Chicago, IL,
  [2022] \copyright 2022.

\bibitem[BH99]{bridson_haefliger99}
Martin~R. Bridson and Andr\'{e} Haefliger.
\newblock {\em Metric spaces of non-positive curvature}, volume 319 of {\em
  Grundlehren der mathematischen Wissenschaften [Fundamental Principles of
  Mathematical Sciences]}.
\newblock Springer-Verlag, Berlin, 1999.

\bibitem[BHM11]{blachere_haissinsky_mathieu11}
S\'{e}bastien Blach\`ere, Peter Ha\"{\i}ssinsky, and Pierre Mathieu.
\newblock Harmonic measures versus quasiconformal measures for hyperbolic
  groups.
\newblock {\em Ann. Sci. \'{E}c. Norm. Sup\'{e}r. (4)}, 44(4):683--721, 2011.

\bibitem[Bj{\"o}10]{bjorklund09}
Michael Bj{\"o}rklund.
\newblock Central limit theorems for {G}romov hyperbolic groups.
\newblock {\em J. Theoret. Probab.}, 23(3):871--887, 2010.

\bibitem[BL05]{balser_lytchak05}
Andreas Balser and Alexander Lytchak.
\newblock Centers of convex subsets of buildings.
\newblock {\em Ann. Global Anal. Geom.}, 28(2):201--209, 2005.

\bibitem[BM96]{burger_mozes96}
M.~Burger and S.~Mozes.
\newblock {${\rm CAT}$}(-{$1$})-spaces, divergence groups and their
  commensurators.
\newblock {\em J. Amer. Math. Soc.}, 9(1):57--93, 1996.

\bibitem[BMSS21]{boulanger_mathieu_sisto21}
Adrien Boulanger, Pierre Mathieu, Cagri Sert, and Alessandro Sisto.
\newblock Large deviations for random walks on {G}romov-hyperbolic spaces,
  2021.
\newblock arXiv:2008.02709.

\bibitem[Bow08]{bowditch08}
Brian~H. Bowditch.
\newblock Tight geodesics in the curve complex.
\newblock {\em Invent. Math.}, 171(2):281--300, 2008.

\bibitem[BQ11]{benoist_quint11}
Yves Benoist and Jean-Fran\c{c}ois Quint.
\newblock Mesures stationnaires et ferm\'{e}s invariants des espaces
  homog\`enes.
\newblock {\em Ann. of Math. (2)}, 174(2):1111--1162, 2011.

\bibitem[BQ16a]{benoist_quint16}
Yves Benoist and Jean-Fran\c{c}ois Quint.
\newblock Central limit theorem on hyperbolic groups.
\newblock {\em Izv. Ross. Akad. Nauk Ser. Mat.}, 80(1), 2016.

\bibitem[BQ16b]{benoist_quint}
Yves Benoist and Jean-Fran\c{c}ois Quint.
\newblock {\em Random walks on reductive groups}, volume~62 of {\em Ergebnisse
  der Mathematik und ihrer Grenzgebiete. 3. Folge. A Series of Modern Surveys
  in Mathematics [Results in Mathematics and Related Areas. 3rd Series. A
  Series of Modern Surveys in Mathematics]}.
\newblock Springer, Cham, 2016.

\bibitem[BQ16c]{benoist_quint16CLTlineargroups}
Yves Benoist and Jean-François Quint.
\newblock {Central limit theorem for linear groups}.
\newblock {\em The Annals of Probability}, 44(2):1308 -- 1340, 2016.

\bibitem[Bro71]{brown71}
B.~M. Brown.
\newblock {Martingale Central Limit Theorems}.
\newblock {\em The Annals of Mathematical Statistics}, 42(1):59 -- 66, 1971.

\bibitem[BS87]{burns_spatzier}
Keith Burns and Ralf Spatzier.
\newblock Manifolds of nonpositive curvature and their buildings.
\newblock {\em Publications Math\'ematiques de l'IH\'ES}, 65:35--59, 1987.

\bibitem[BS00]{bonk_schramm00}
M.~Bonk and O.~Schramm.
\newblock Embeddings of {G}romov hyperbolic spaces.
\newblock {\em Geom. Funct. Anal.}, 10(2):266--306, 2000.

\bibitem[CF10]{caprace_fujiwara10}
Pierre-Emmanuel Caprace and Koji Fujiwara.
\newblock Rank-one isometries of buildings and quasi-morphisms of {K}ac-{M}oody
  groups.
\newblock {\em Geom. Funct. Anal.}, 19(5), 2010.

\bibitem[CFFT22]{chawla_forghani_frisch_tiozzo22}
Kunal Chawla, Behrang Forghani, Joshua Frisch, and Giulio Tiozzo.
\newblock The {P}oisson boundary of hyperbolic groups without moment
  conditions, 2022.
\newblock arXiv:2209.02114.

\bibitem[Cha09]{charignon09}
Cyril Charignon.
\newblock Compactifications polygonales d'un immeuble affine, 2009.
\newblock arXiv:0903.0502.

\bibitem[Cho22a]{choi22a}
Inhyeok Choi.
\newblock Random walks and contracting elements {I}: {D}eviation inequality and
  {L}imit laws, 2022.
\newblock arXiv:2207.06597.

\bibitem[Cho22b]{choi22b}
Inhyeok Choi.
\newblock Random walks and contracting elements {II}: {T}ranslation length and
  {Q}uasi-isometric embedding, 2022.
\newblock arXiv:2212.12119.

\bibitem[CK00]{croke_kleiner00}
Christopher~B. Croke and Bruce Kleiner.
\newblock Spaces with nonpositive curvature and their ideal boundaries.
\newblock {\em Topology}, 39(3):549--556, 2000.

\bibitem[CL10]{caprace_lytchak10}
Pierre-Emmanuel Caprace and Alexander Lytchak.
\newblock At infinity of finite-dimensional {CAT}(0) spaces.
\newblock {\em Math. Ann.}, 346(1):1--21, 2010.

\bibitem[CL11]{caprace_lecureux11}
Pierre-Emmanuel Caprace and Jean L\'{e}cureux.
\newblock Combinatorial and group-theoretic compactifications of buildings.
\newblock {\em Ann. Inst. Fourier (Grenoble)}, 61(2):619--672, 2011.

\bibitem[CM09a]{caprace_monod1}
Pierre-Emmanuel Caprace and Nicolas Monod.
\newblock Isometry groups of non-positively curved spaces: discrete subgroups.
\newblock {\em J. Topol.}, 2(4):701--746, 2009.

\bibitem[CM09b]{caprace_monod2}
Pierre-Emmanuel Caprace and Nicolas Monod.
\newblock Isometry groups of non-positively curved spaces: structure theory.
\newblock {\em J. Topol.}, 2(4):661--700, 2009.

\bibitem[CM13]{caprace_monod13}
Pierre-Emmanuel Caprace and Nicolas Monod.
\newblock Fixed points and amenability in non-positive curvature.
\newblock {\em Math. Ann.}, 356(4):1303--1337, 2013.

\bibitem[CMR20]{ciobotaru_muhlerr_rousseau_20}
Corina Ciobotaru, Bernhard M\"{u}hlherr, and Guy Rousseau.
\newblock The cone topology on masures.
\newblock {\em Adv. Geom.}, 20(1):1--28, 2020.
\newblock With an appendix by Auguste H\'{e}bert.

\bibitem[CS11]{caprace_sageev11}
Pierre-Emmanuel Caprace and Michah Sageev.
\newblock Rank rigidity for {CAT}(0) cube complexes.
\newblock {\em Geom. Funct. Anal.}, 21(4), 2011.

\bibitem[CZ13]{caprace_zadnik13}
Pierre-Emmanuel Caprace and Ga\v{s}per Zadnik.
\newblock Regular elements in {${\rm CAT}(0)$} groups.
\newblock {\em Groups Geom. Dyn.}, 7(3):535--541, 2013.

\bibitem[Dav15]{davis15}
Michael~W. Davis.
\newblock The geometry and topology of {C}oxeter groups.
\newblock In {\em Introduction to modern mathematics}, volume~33 of {\em Adv.
  Lect. Math. (ALM)}, pages 129--142. Int. Press, Somerville, MA, 2015.

\bibitem[DGO17]{dahmani_guirardel_osin17}
F.~Dahmani, V.~Guirardel, and D.~Osin.
\newblock Hyperbolically embedded subgroups and rotating families in groups
  acting on hyperbolic spaces.
\newblock {\em Mem. Amer. Math. Soc.}, 245(1156):v+152, 2017.

\bibitem[DK18]{drutu_kapovich18}
Cornelia Dru\c{t}u and Michael Kapovich.
\newblock {\em Geometric group theory}, volume~63 of {\em American Mathematical
  Society Colloquium Publications}.
\newblock American Mathematical Society, Providence, RI, 2018.
\newblock With an appendix by Bogdan Nica.

\bibitem[DLP23]{duchesne_lecureux_pozetti23}
Bruno Duchesne, Jean L\'{e}cureux, and Maria~Beatrice Pozzetti.
\newblock Boundary maps and maximal representations on infinite-dimensional
  {H}ermitian symmetric spaces.
\newblock {\em Ergodic Theory Dynam. Systems}, 43(1):140--189, 2023.

\bibitem[Duc13]{duchesne13}
Bruno Duchesne.
\newblock Infinite-dimensional nonpositively curved symmetric spaces of finite
  rank.
\newblock {\em Int. Math. Res. Not. IMRN}, (7):1578--1627, 2013.

\bibitem[Duc16]{duchesne16}
Bruno Duchesne.
\newblock Groups acting on spaces of non-positive curvature.
\newblock 2016.
\newblock arXiv:1603.04573.

\bibitem[DY20]{dussaule_yang20}
Matthieu Dussaule and Wenyuan Yang.
\newblock The {H}ausdorff dimension of the harmonic measure for relatively
  hyperbolic groups, 2020.
\newblock arXiv:2010.07671.

\bibitem[Ebe96]{eberlein96}
Patrick~B. Eberlein.
\newblock {\em Geometry of nonpositively curved manifolds}.
\newblock Chicago Lectures in Mathematics. University of Chicago Press,
  Chicago, IL, 1996.

\bibitem[EH90]{eberlein_heber}
Patrick Eberlein and Jens Heber.
\newblock A differential geometric characterization of symmetric spaces of
  higher rank.
\newblock {\em Publications Math\'ematiques de l'IH\'ES}, 71:33--44, 1990.

\bibitem[FK60]{furstenberg_kesten60}
H.~Furstenberg and H.~Kesten.
\newblock {Products of Random Matrices}.
\newblock {\em The Annals of Mathematical Statistics}, 31(2):457 -- 469, 1960.

\bibitem[FLM18]{fernos_lecureux_matheus18}
Talia Fern\'{o}s, Jean L\'{e}cureux, and Fr\'{e}d\'{e}ric Math\'{e}us.
\newblock Random walks and boundaries of {$\rm CAT(0)$} cubical complexes.
\newblock {\em Comment. Math. Helv.}, 93(2), 2018.

\bibitem[FLM21]{fernos_lecureux_matheus21}
Talia Fernós, Jean Lécureux, and Frédéric Mathéus.
\newblock Contact graphs, boundaries, and a central limit theorem for {CAT}(0)
  cubical complexes, 2021.
\newblock arXiv:2112.10141.

\bibitem[Fur63]{furstenberg63}
Harry Furstenberg.
\newblock Noncommuting random products.
\newblock {\em Transactions of the American Mathematical Society},
  108(3):377--428, 1963.

\bibitem[Fur73]{furstenberg73}
Harry Furstenberg.
\newblock Boundary theory and stochastic processes on homogeneous spaces.
\newblock In {\em Harmonic analysis on homogeneous spaces ({P}roc. {S}ympos.
  {P}ure {M}ath., {V}ol. {XXVI}, {W}illiams {C}oll., {W}illiamstown, {M}ass.,
  1972)}, 1973.

\bibitem[Gar97]{garrett97}
Paul Garrett.
\newblock {\em Buildings and classical groups}.
\newblock Chapman \& Hall, London, 1997.

\bibitem[GdlH90]{ghys_dlharpe90}
\'{E}. Ghys and P.~de~la Harpe, editors.
\newblock {\em Sur les groupes hyperboliques d'apr\`es {M}ikhael {G}romov},
  volume~83 of {\em Progress in Mathematics}.
\newblock Birkh\"{a}user Boston, Inc., Boston, MA, 1990.
\newblock Papers from the Swiss Seminar on Hyperbolic Groups held in Bern,
  1988.

\bibitem[Gen19]{genevois19}
Anthony Genevois.
\newblock Hyperbolicities in {${\rm CAT}(0)$} cube complexes.
\newblock {\em Enseign. Math.}, 65(1-2):33--100, 2019.

\bibitem[Gen20]{genevois20}
Anthony Genevois.
\newblock Contracting isometries of {${\rm CAT}(0)$} cube complexes and
  acylindrical hyperbolicity of diagram groups.
\newblock {\em Algebr. Geom. Topol.}, 20(1):49--134, 2020.

\bibitem[GK20]{gouezel_karlsson20}
S\'{e}bastien Gou\"{e}zel and Anders Karlsson.
\newblock Subadditive and multiplicative ergodic theorems.
\newblock {\em J. Eur. Math. Soc. (JEMS)}, 22(6):1893--1915, 2020.

\bibitem[Gou22]{gouezel22}
S\'{e}bastien Gou\"{e}zel.
\newblock Exponential bounds for random walks on hyperbolic spaces without
  moment conditions.
\newblock {\em Tunis. J. Math.}, 4(4):635--671, 2022.

\bibitem[GQR22]{gekhtman_qing_rafi22}
Ilya Gekhtman, Yulan Qing, and Kasra Rafi.
\newblock Genericity of sublinearly {M}orse directions in {CAT}(0) spaces and
  the {T}eichmüller space, 2022.
\newblock arXiv:2208.04778.

\bibitem[GR85]{guivarch_raugi85}
Y.~Guivarc'h and A.~Raugi.
\newblock Fronti\`ere de {F}urstenberg, propri\'{e}t\'{e}s de contraction et
  th\'{e}or\`emes de convergence.
\newblock {\em Z. Wahrsch. Verw. Gebiete}, 69(2), 1985.

\bibitem[Gro87]{gromov87}
M.~Gromov.
\newblock Hyperbolic groups.
\newblock In {\em Essays in group theory}, volume~8 of {\em Math. Sci. Res.
  Inst. Publ.}, pages 75--263. Springer, New York, 1987.

\bibitem[Gui80]{guivarch80}
Y.~Guivarc'h.
\newblock Sur la loi des grands nombres et le rayon spectral d'une marche
  al\'{e}atoire.
\newblock In {\em Conference on {R}andom {W}alks ({K}leebach, 1979)
  ({F}rench)}, volume~74 of {\em Ast\'{e}risque}, pages 47--98, 3. Soc. Math.
  France, Paris, 1980.

\bibitem[Hag14]{hagen14}
Mark~F. Hagen.
\newblock Weak hyperbolicity of cube complexes and quasi-arboreal groups.
\newblock {\em J. Topol.}, 7(2):385--418, 2014.

\bibitem[Ham09]{hamenstadt09}
Ursula Hamenst\"{a}dt.
\newblock Rank-one isometries of proper {${\rm CAT}(0)$}-spaces.
\newblock In {\em Discrete groups and geometric structures}, volume 501 of {\em
  Contemp. Math.} Amer. Math. Soc., Providence, RI, 2009.

\bibitem[Hel62]{helgason62}
Sigur\dj~ur Helgason.
\newblock {\em Differential geometry and symmetric spaces}.
\newblock Pure and Applied Mathematics, Vol. XII. Academic Press, New
  York-London, 1962.

\bibitem[Hel01]{helgason01}
Sigurdur Helgason.
\newblock {\em Differential geometry, {L}ie groups, and symmetric spaces},
  volume~34 of {\em Graduate Studies in Mathematics}.
\newblock American Mathematical Society, Providence, RI, 2001.
\newblock Corrected reprint of the 1978 original.

\bibitem[Hor18]{horbez18}
Camille Horbez.
\newblock Central limit theorems for mapping class groups and
  {$\mathrm{Out}(F_N)$}.
\newblock {\em Geometry \& Topology}, 22(1):105 -- 156, 2018.

\bibitem[Hou21]{houdayer21}
Cyril Houdayer.
\newblock Lecture notes: Ergodic group theory, 2021.
\newblock M2 AAG, Université Paris-Saclay, Orsay.

\bibitem[Hou23]{houdayer23}
Cyril Houdayer.
\newblock Lecture notes: Superrigidity, 2023.
\newblock M2 AAG, Université Paris-Saclay, Orsay, 2022-2023.

\bibitem[IMZ21]{incerti-medici_zalloum21}
Merlin Incerti-Medici and Abdul Zalloum.
\newblock Sublinearly morse boundaries from the viewpoint of combinatorics,
  2021.
\newblock arXiv:2101.01037.

\bibitem[Ize22]{izeki22}
Hiroyasu Izeki.
\newblock Isometric group actions with vanishing rate of escape on {CAT}(0)
  spaces, 2022.
\newblock arXiv:2204.00206.

\bibitem[Kai00]{kaimanovich00}
Vadim~A. Kaimanovich.
\newblock The {P}oisson formula for groups with hyperbolic properties.
\newblock {\em Ann. of Math. (2)}, 152(3), 2000.

\bibitem[Kai03]{kaimanovich03}
V.~A. Kaimanovich.
\newblock Double ergodicity of the {P}oisson boundary and applications to
  bounded cohomology.
\newblock {\em Geom. Funct. Anal.}, 13(4):852--861, 2003.

\bibitem[Kar05]{karlsson05}
Anders Karlsson.
\newblock On the dynamics of isometries.
\newblock {\em Geom. Topol.}, 9:2359--2394, 2005.

\bibitem[Kin68]{kingman68}
J.~F.~C. Kingman.
\newblock The ergodic theory of subadditive stochastic processes.
\newblock {\em Journal of the Royal Statistical Society. Series B
  (Methodological)}, 30(3):499--510, 1968.

\bibitem[KL11]{karlsson_ledrappier11}
Anders Karlsson and Fran\c{c}ois Ledrappier.
\newblock Noncommutative ergodic theorems.
\newblock In {\em Geometry, rigidity, and group actions}, Chicago Lectures in
  Math., pages 396--418. Univ. Chicago Press, Chicago, IL, 2011.

\bibitem[KM99]{karlsson_margulis}
Anders Karlsson and Gregory Margulis.
\newblock A multiplicative ergodic theorem and nonpositively curved spaces.
\newblock {\em Communications in Mathematical Physics}, 208:107--123, 12 1999.

\bibitem[Kol28]{kolmogorov28}
A.~Kolmogoroff.
\newblock \"{U}ber die {S}ummen durch den {Z}ufall bestimmter unabh\"{a}ngiger
  {G}r\"{o}\ss en.
\newblock {\em Math. Ann.}, 99(1):309--319, 1928.

\bibitem[KV83]{kaimanovich_vershik}
V.~A. Ka\u{\i}manovich and A.~M. Vershik.
\newblock Random walks on discrete groups: boundary and entropy.
\newblock {\em Ann. Probab.}, 11(3):457--490, 1983.

\bibitem[LeB22a]{le-bars22b}
Corentin Le~Bars.
\newblock Central limit theorem on {CAT}(0) spaces with contracting isometries,
  2022.
\newblock arXiv:2209.11648.

\bibitem[LeB22b]{LeBars22}
Corentin Le~Bars.
\newblock Random walks and rank one isometries on {CAT}(0) spaces, 2022.
\newblock arXiv:2205.07594.

\bibitem[LeB23]{le-bars23immeuble}
Corentin Le~Bars.
\newblock Stationary measures on $\tilde{A}_2$-buildings, 2023.
\newblock in preparation.

\bibitem[LeBLS23]{le-bars_lecureux_schillewaert23}
Corentin Le~Bars, Jean Lécureux, and Jeroen Schillewaert.
\newblock Hyperbolic elements and boundaries of $\mathbb{R}$-buildings of type
  $\tilde{A}_2$, 2023.
\newblock in preparation.

\bibitem[Lea13]{leary13}
Ian~J. Leary.
\newblock A metric {K}an-{T}hurston theorem.
\newblock {\em J. Topol.}, 6(1):251--284, 2013.

\bibitem[LeP82]{lepage82}
Emile Le~Page.
\newblock Théorèmes limites pour les produits de matrices aléatoires.
\newblock In Herbert Heyer, editor, {\em Probability Measures on Groups}, pages
  258--303, Berlin, Heidelberg, 1982. Springer Berlin Heidelberg.

\bibitem[Mac62]{mackey1962}
George~W. Mackey.
\newblock Point realizations of transformation groups.
\newblock {\em Illinois J. Math.}, 6(2), 06 1962.

\bibitem[Mah12]{maher12}
Joseph Maher.
\newblock Exponential decay in the mapping class group.
\newblock {\em J. Lond. Math. Soc. (2)}, 86(2):366--386, 2012.

\bibitem[Man05]{manning05}
Jason~Fox Manning.
\newblock Geometry of pseudocharacters.
\newblock {\em Geom. Topol.}, 9:1147--1185, 2005.

\bibitem[Mar91]{margulis91}
G.~A. Margulis.
\newblock {\em Discrete subgroups of semisimple {L}ie groups}, volume~17 of
  {\em Ergebnisse der Mathematik und ihrer Grenzgebiete (3) [Results in
  Mathematics and Related Areas (3)]}.
\newblock Springer-Verlag, Berlin, 1991.

\bibitem[MM99]{masur_minsky99}
Howard~A. Masur and Yair~N. Minsky.
\newblock Geometry of the complex of curves. {I}. {H}yperbolicity.
\newblock {\em Invent. Math.}, 138(1):103--149, 1999.

\bibitem[MS20]{mathieu_sisto20}
P.~Mathieu and A.~Sisto.
\newblock Deviation inequalities for random walks.
\newblock {\em Duke Math. J.}, 169(5):961--1036, 2020.

\bibitem[MT18]{maher_tiozzo18}
Joseph Maher and Giulio Tiozzo.
\newblock Random walks on weakly hyperbolic groups.
\newblock {\em J. Reine Angew. Math.}, 742, 2018.

\bibitem[OP21]{osajda_przytycki21}
Damian Osajda and Piotr Przytycki.
\newblock Tits alternative for groups acting properly on $2$-dimensional
  recurrent complexes.
\newblock 2021.
\newblock arXiv:1904.07796.

\bibitem[Osi16]{osin2016}
D.~Osin.
\newblock Acylindrically hyperbolic groups.
\newblock {\em Trans. Amer. Math. Soc.}, 368(2):851--888, 2016.

\bibitem[Par00]{parreau00}
Anne Parreau.
\newblock Immeubles affines: construction par les normes et \'{e}tude des
  isom\'{e}tries.
\newblock In {\em Crystallographic groups and their generalizations
  ({K}ortrijk, 1999)}, volume 262 of {\em Contemp. Math.}, pages 263--302.
  Amer. Math. Soc., Providence, RI, 2000.

\bibitem[Pau14]{paulin}
Frédéric Paulin.
\newblock {\em Groupes et {G}éométries}.
\newblock 2014.
\newblock Cours de seconde année de Mastère, Département de Mathématiques
  d'{O}rsay.

\bibitem[Pet21]{petyt21}
Harry Petyt.
\newblock Mapping class groups are quasicubical, 2021.
\newblock arXiv:2112.10681.

\bibitem[PS09]{papasoglu_swenson09}
Panos Papasoglu and Eric Swenson.
\newblock Boundaries and {JSJ} decompositions of {CAT}(0)-groups.
\newblock {\em Geom. Funct. Anal.}, 19(2), 2009.

\bibitem[PSZ22]{petyt_spriano_zalloum22}
Harry Petyt, Davide Spriano, and Abdul Zalloum.
\newblock Hyperbolic models for {CAT}(0) spaces, 2022.
\newblock arXiv:2207.14127.

\bibitem[Ron09]{ronan09}
Mark Ronan.
\newblock {\em Lectures on buildings}.
\newblock University of Chicago Press, Chicago, IL, 2009.
\newblock Updated and revised.

\bibitem[Ros81]{rosenblatt81}
Joseph Rosenblatt.
\newblock Ergodic and mixing random walks on locally compact groups.
\newblock {\em Math. Ann.}, 257(1):31--42, 1981.

\bibitem[Rou09]{rousseau09}
Guy Rousseau.
\newblock Euclidean buildings.
\newblock In {\em G\'{e}om\'{e}tries \`a courbure n\'{e}gative ou nulle,
  groupes discrets et rigidit\'{e}s}, volume~18 of {\em S\'{e}min. Congr.},
  pages 77--116. Soc. Math. France, Paris, 2009.

\bibitem[Rou11]{rousseau11}
Guy Rousseau.
\newblock Masures affines.
\newblock {\em Pure Appl. Math. Q.}, 7(3, Special Issue: In honor of Jacques
  Tits):859--921, 2011.

\bibitem[RT21]{remy_trojan21}
Bertrand Rémy and Bartosz Trojan.
\newblock Martin compactifications of affine buildings, 2021.
\newblock arXiv:2105.14807.

\bibitem[Sag14]{sageev14}
Michah Sageev.
\newblock {CAT}(0) cube complexes and groups.
\newblock In {\em Geometric group theory}, volume~21 of {\em IAS/Park City
  Math. Ser.}, pages 7--54. Amer. Math. Soc., Providence, RI, 2014.

\bibitem[Ser80]{serre80}
Jean-Pierre Serre.
\newblock {\em Trees}.
\newblock Springer-Verlag, Berlin-New York, 1980.
\newblock Translated from the French by John Stillwell.

\bibitem[Sis18]{sisto18}
Alessandro Sisto.
\newblock Contracting elements and random walks.
\newblock {\em J. Reine Angew. Math.}, 742:79--114, 2018.

\bibitem[SS87]{sawyer_steger87}
Stanley Sawyer and Tim Steger.
\newblock The rate of escape for anisotropic random walks in a tree.
\newblock {\em Probab. Theory Related Fields}, 76(2):207--230, 1987.

\bibitem[SST20]{schillewaert_struyve_thomas22}
Jeroen Schillewaert, Koen Struyve, and Anne Thomas.
\newblock Fixed points for group actions on 2-dimensional affine buildings,
  2020.
\newblock arXiv:2003.00681.

\bibitem[Sta22]{stadler22}
Stephan Stadler.
\newblock Rank rigidity for {CAT}(0) spaces without 3-flats, 2022.
\newblock arXiv:2212.07100.

\bibitem[Swe99]{swenson99}
Eric~L. Swenson.
\newblock A cut point theorem for {${\rm CAT}(0)$} groups.
\newblock {\em J. Differential Geom.}, 53(2):327--358, 1999.

\bibitem[Tit72]{tits72}
J.~Tits.
\newblock Free subgroups in linear groups.
\newblock {\em J. Algebra}, 20:250--270, 1972.

\bibitem[Tit74]{tits74}
Jacques Tits.
\newblock {\em Buildings of spherical type and finite {BN}-pairs}.
\newblock Lecture Notes in Mathematics, Vol. 386. Springer-Verlag, Berlin-New
  York, 1974.

\bibitem[Tit86]{tits86}
Jacques Tits.
\newblock Immeubles de type affine.
\newblock In {\em Buildings and the geometry of diagrams ({C}omo, 1984)},
  volume 1181 of {\em Lecture Notes in Math.}, pages 159--190. Springer,
  Berlin, 1986.

\bibitem[Wei03]{weiss03}
Richard~M. Weiss.
\newblock {\em The structure of spherical buildings}.
\newblock Princeton University Press, Princeton, NJ, 2003.

\bibitem[Wei09]{weiss09}
Richard~M. Weiss.
\newblock {\em The structure of affine buildings}, volume 168 of {\em Annals of
  Mathematics Studies}.
\newblock Princeton University Press, Princeton, NJ, 2009.

\bibitem[Zim78]{zimmer78}
Robert~J. Zimmer.
\newblock Amenable ergodic group actions and an application to {P}oisson
  boundaries of random walks.
\newblock {\em J. Functional Analysis}, 27(3), 1978.

\bibitem[Zim84]{zimmer84}
Robert~J. Zimmer.
\newblock {\em Ergodic theory and semisimple groups}, volume~81 of {\em
  Monographs in Mathematics}.
\newblock Birkh\"{a}user Verlag, Basel, 1984.

\end{thebibliography}
\end{document}